\documentclass[final]{amsart}
\usepackage{a4wide}
\usepackage{amssymb}
\usepackage{amsmath}
\usepackage{pifont}
\usepackage{amsfonts}
\usepackage{latexsym}
\usepackage{verbatim}
\usepackage{float}
\usepackage{exscale}

\newcommand{\BIGOP}[1]{\mathop{\mathchoice%
{\raise-0.22em\hbox{\huge $#1$}}%
{\raise-0.05em\hbox{\Large $#1$}}{\hbox{\large $#1$}}{#1}}}
\newcommand{\bigtimes}{\BIGOP{\times}}
\newcommand{\BIGboxplus}{\mathop{\mathchoice%
{\raise-0.35em\hbox{\huge $\boxplus$}}%
{\raise-0.15em\hbox{\Large $\boxplus$}}{\hbox{\large $\boxplus$}}{\boxplus}}}

\renewcommand{\ldots}{\ensuremath{\dotsc}}

\newcommand{\Rplus}{{\mathbb R}_{>0}}

\def\epsilon{\varepsilon}
\def\hat{\widehat}

\def\undertilde#1{{\baselineskip=0pt\vtop
{\hbox{$#1$}\hbox{$\scriptscriptstyle\sim$}}}{}}
\def\underdtilde#1{{\baselineskip=0pt\vtop
{\hbox{$#1$}\hbox{$\scriptscriptstyle\approx$}}}{}}
\def\epsilon{\varepsilon}
\def\hat{\widehat}

\def\nabq{\undertilde{\nabla}_{q}}
\def\nabqi{\undertilde{\nabla}_{q_i}}
\def\nabqj{\undertilde{\nabla}_{q_j}}
\def\nabx{\undertilde{\nabla}_{x}}
\def\nabr1{\undertilde{\nabla}_{r_1}}
\def\nabr2{\undertilde{\nabla}_{r_2}}
\def\delx{\Delta_{x}}

\def\chit{\undertilde{\chi}}

\def\lae{\ell_{a}}

\def\ut{\undertilde{u}}
\def\utae{\undertilde{u}_{\epsilon,L}}
\def\utaed{\undertilde{u}_{\epsilon,L,\delta}}
\def\uta{\undertilde{u}_{\epsilon,L}}
\def\ute{\undertilde{u}_{\epsilon}}

\def\utaeD{\utae^{\Delta t}}
\def\utaeDm{\utae^{\Delta t,-}}
\def\utaeDp{\utae^{\Delta t,+}}

\def\vt{\undertilde{v}}
\def\wt{\undertilde{w}}
\def\xt{\undertilde{x}}
\def\qt{\undertilde{q}}
\def\pt{\undertilde{p}}
\def\bt{\undertilde{b}}
\def\ft{\undertilde{f}}
\def\gt{\undertilde{g}}

\def\nt{\undertilde{n}}

\def\etat{\undertilde{\eta}}
\def\xit{\undertilde{\xi}}

\def\Ct{\undertilde{C}}
\def\Ft{\undertilde{F}}

\def\Ht{\undertilde{H}}

\def\Lt{\undertilde{L}}

\def\St{\undertilde{S}}

\def\Vt{\undertilde{V}}
\def\Wt{\undertilde{W}}

\def\Yt{\undertilde{Y}}
\def\dq{\,{\rm d}\undertilde{q}}
\def\dx{\,{\rm d}\undertilde{x}}

\def\dt{\,{\rm d}t}
\def\zerot{\undertilde{0}}
\def\tautt{\underdtilde{\tau}}

\def\Att{\underdtilde{A}}

\def\Ctt{\underdtilde{C}}

\def\Itt{\underdtilde{I}}

\def\Ltt{\underdtilde{L}}

\def\unabtt{\underdtilde{\nabla}_{x}\,\ut}

\def\vnabtt{\underdtilde{\nabla}_{x}\,\vt}

\def\wnabtt{\underdtilde{\nabla}_{x}\,\wt}
\def\sigtt{\underdtilde{\sigma}}

\def\pae{p_{\epsilon,L}}
\def\hpsiae{\hat \psi_{\epsilon,L}}
\def\hpsiaed{\hat \psi_{\epsilon,L,\delta}}
\def\psiae{\psi_{\epsilon,L}}
\def\psia{\hat \psi_{\epsilon,L}}

\def\ptut{\frac{\partial \ut}{\partial t}}

\def\ptutae{\frac{\partial \utae}{\partial t}}

\def\dd {{\,\rm d}}

\def\qqquad{\qquad\quad}
\def\qqqquad{\qquad\qquad}

\newcommand{\nabxtt}{\underdtilde{\nabla}_{x}\,}

\newcommand{\bet}{\noalign{\vskip6pt plus 3pt minus 1pt}}

\newcommand{\grad}[2]{\nabla_{#1}#2} % Gradient symbol
\newcommand{\compEmb}{\hookrightarrow \!\!\!\rightarrow}
\newtheorem{example}{Example}[section]
\newtheorem{definition}{Definition}[section]
\newtheorem{lemma}{Lemma}[section]
\newtheorem{theorem}{Theorem}[section]
\newtheorem{remark}{Remark}[section]

\renewcommand{\theequation}{\arabic{section}.\arabic{equation}}

\newcounter{ind}
\def\eqlabstart{%
 \setcounter{ind}{\value{equation}}\addtocounter{ind}{1}%
 \setcounter{equation}{0}%
 \renewcommand{\theequation}{\arabic{section}.\arabic{ind}\alph{equation}}%
}
\def\eqlabend{%
 \renewcommand{\theequation}{\arabic{section}.\arabic{equation}}%
 \setcounter{equation}{\value{ind}}%
}

\newcounter{appendix}
\renewcommand\appendix{\par
  	\refstepcounter{appendix}
	   \setcounter{section}{0}
       \setcounter{theorem}{0}
	   \setcounter{equation}{0}
\renewcommand\thesection{\appendixname\ \Alph{section}}
\renewcommand\thesubsection{\Alph{section}.\arabic{subsection}}%
\renewcommand\theequation{\Alph{section}.\arabic{equation}}}%

\begin{document}

\markboth{John W. Barrett and Endre S\"{u}li}
{Existence and Equilibration of Global Weak Solutions for Dilute Polymers}

%%%%%%%%%%%%%%%%%%% Publisher's Area please ignore %%%%%%%%%%%%%%%%%%%%%%%
%
%\catchline{}{}{}{}{}
%
%%%%%%%%%%%%%%%%%%%%%%%%%%%%%%%%%%%%%%%%%%%%%%%%%%%%%%%%%%%%%%%%%%%%%%%%%%

\title[Existence and Equilibration of Global Weak Solutions for Dilute Polymers]
{Existence and equilibration of global weak solutions to finitely extensible nonlinear bead-spring chain models for dilute polymers}
%{EXISTENCE AND EQUILIBRATION OF GLOBAL WEAK SOLUTIONS TO KINETIC MODELS FOR DILUTE POLYMERS  I: FINITELY EXTENSIBLE NONLINEAR BEAD-SPRING %CHAINS}

\author{JOHN W. BARRETT}

\address{\footnotesize Department of Mathematics, Imperial College London\\
London SW7 2AZ, UK\\
{\tt j.barrett@imperial.ac.uk}}

\author{ENDRE S\"ULI}

\address{
Mathematical Institute, University of Oxford\\ Oxford OX1 3LB, UK\\
{\tt endre.suli@maths.ox.ac.uk}}

\maketitle

%\begin{history}
%\received{(8 April 2010)}
%%\revised{(Day Month Year)}
%\accepted{(17 September 2010)}
%%\comby{?? ??}
%\end{history}

\begin{abstract}
We show the existence of global-in-time weak solutions to a general class of coupled FENE-type bead-spring chain models that arise from the kinetic theory of dilute solutions of polymeric liquids with noninteracting polymer chains. The class of models involves the unsteady incompressible Navier--Stokes equations in a bounded domain in $\mathbb{R}^d$, $d = 2$ or $3$, for the velocity and the pressure of the fluid, with an elastic extra-stress tensor appearing on the right-hand side in the momentum equation. The extra-stress tensor stems from the random movement of
the polymer chains and is defined by the Kramers expression through the associated probability density function
that satisfies a Fokker--Planck-type parabolic equation, a crucial feature of which is the presence of a centre-of-mass diffusion term. We require no structural assumptions on the drag term in the Fokker--Planck equation; in particular, the drag term need not be corotational. With a square-integrable and divergence-free initial velocity datum $\undertilde{u}_0$ for the Navier--Stokes equation and a nonnegative initial probability density function $\psi_0$ for the Fokker--Planck equation, which has finite relative entropy with respect to the Maxwellian $M$, we prove, {\em via} a limiting procedure on certain regularization parameters, the existence of a global-in-time weak solution $t \mapsto (\undertilde{u}(t), \psi(t))$ to the coupled Navier--Stokes--Fokker--Planck system, satisfying the initial condition $(\undertilde{u}(0), \psi(0)) = (\undertilde{u}_0, \psi_0)$, such that $t\mapsto \undertilde{u}(t)$ belongs to the classical Leray space and $t \mapsto \psi(t)$ has bounded relative entropy with respect to $M$ and $t \mapsto \psi(t)/M$ has integrable Fisher information (w.r.t. the measure ${\rm d}\mu:= M(\undertilde{q})\,{\rm d}\undertilde{q}\,{\rm d}\undertilde{x}$)
over any time interval $[0,T]$, $T>0$. If the density of body forces $\undertilde{f}$ on the right-hand side of the Navier--Stokes momentum equation vanishes, then
a weak solution constructed  as above is such that
$t\mapsto (\undertilde{u}(t),\psi(t))$ decays exponentially in time to $(\undertilde{0},M)$ in the $\undertilde{L}^2 \times L^1$ norm, at a rate that is independent of $(\undertilde{u}_0,\psi_0)$ and of the centre-of-mass diffusion coefficient.

\medskip


\end{abstract}

\section{Introduction}
\label{sec:1}

This paper establishes the existence of global-in-time weak solutions
to a large class of bead-spring chain models with finitely
extensible nonlinear elastic (FENE) type spring potentials, ---
a system of nonlinear partial differential equations that arises
from the kinetic theory of dilute polymer solutions. The solvent is an
incompressible, viscous, isothermal Newtonian fluid confined to a
bounded open Lipschitz domain $\Omega \subset \mathbb{R}^d$, $d=2$ or $3$, with
boundary $\partial \Omega$. For the sake of simplicity of presentation,
we shall suppose that $\Omega$ has a `solid boundary' $\partial \Omega$;
the velocity field $\ut$ will then satisfy the no-slip boundary condition
$\ut=\zerot$ on $\partial \Omega$. The polymer chains, which are suspended
in the solvent, are assumed not to interact with each other. The
conservation of momentum and mass equations for the solvent then
have the form of the incompressible Navier--Stokes equations in
which the elastic {\em extra-stress} tensor $\tautt$ (i.e. \ the
polymeric part of the Cauchy stress tensor) appears as a source
term:

Given $T \in \mathbb{R}_{>0}$, find $\ut\,:\,(\xt,t)\in
\overline\Omega \times [0,T] \mapsto
\ut(\xt,t) \in {\mathbb R}^d$ and $p\,:\, (\xt,t) \in \Omega \times (0,T]
\mapsto p(\xt,t) \in {\mathbb R}$ such that
\begin{subequations}
\begin{alignat}{2}
\ptut + (\ut \cdot \nabx)\, \ut - \nu \,\delx \ut + \nabx p
&= \ft + \nabx \cdot \tautt \qquad &&\mbox{in } \Omega \times (0,T],\label{ns1a}\\
\nabx \cdot \ut &= 0        \qquad &&\mbox{in } \Omega \times (0,T],\label{ns2a}\\
\ut &= \zerot               \qquad &&\mbox{on } \partial \Omega \times (0,T],\label{ns3a}\\
\ut(\xt,0)&=\ut_{0}(\xt)    \qquad &&\forall \xt \in \Omega.\label{ns4a}
\end{alignat}
\end{subequations}
It is assumed that each of the equations above has been written in its nondimensional form;
$\ut$ denotes a nondimensional velocity, defined as the velocity field
scaled by the characteristic flow speed $U_0$; $\nu\in \mathbb{R}_{>0}$ is the reciprocal of the Reynolds number,
i.e. the ratio of the kinematic viscosity coefficient of the solvent and $L_0 U_0$, where
$L_0$ is a characteristic length-scale of the flow; $p$ is
the nondimensional pressure and $f$ is the nondimensional density of body forces.

In a {\em bead-spring chain model}, consisting of $K+1$ beads coupled with $K$ elastic
springs to represent a polymer chain, the extra-stress tensor
$\tautt$ is defined by the \textit{Kramers expression}
as a weighted average of $\psi$, the probability
density function of the (random) conformation vector $\qt := (\qt_1^{\rm T},\dots, \qt_K^{\rm T})^{\rm T}
\in \mathbb{R}^{Kd}$ of the chain (cf. (\ref{tau1}) below), with $\qt_i$
representing the $d$-component conformation/orientation vector of the $i$th spring.
The Kolmogorov equation satisfied by $\psi$ is a second-order parabolic equation,
the Fokker--Planck equation, whose transport coefficients depend
on the velocity field $\ut$. The domain $D$ of admissible conformation
vectors $D \subset \mathbb{R}^{Kd}$ is a $K$-fold
Cartesian product $D_1 \times \cdots \times D_K$ of balanced convex
open sets $D_i \subset \mathbb{R}^d$, $i=1,\dots, K$; the term
{\em balanced} means that $\qt_i \in D_i$ if, and only if, $-\qt_i \in D_i$.
Hence, in particular, $\undertilde{0} \in D_i$, $i=1,\dots,K$.
Typically $D_i$ is the whole of $\mathbb{R}^d$ or a bounded open $d$-dimensional ball
centred at the origin $\zerot \in \mathbb{R}^d$ for each $i=1,\dots,K$.
When $K=1$, the model is referred to as the {\em dumbbell model.}

Let $\mathcal{O}_i\subset [0,\infty)$ denote the image of $D_i$
under the mapping $\qt_i \in D_i \mapsto
\frac{1}{2}|\qt_i|^2$, and consider the {\em spring potential}~$U_i
\!\in\! C^2(\mathcal{O}_i;\mathbb {R}_{\geq 0})$, $i=1,\dots, K$.
Clearly, $0 \in \mathcal{O}_i$. We shall suppose that $U_i(0)=0$
and that $U_i$ is monotonic increasing and unbounded on $\mathcal{O}_i$ for each $i=1,\dots, K$.
The elastic spring-force $\Ft_i\,:\, D_i \subseteq \mathbb{R}^d \rightarrow \mathbb{R}^d$
of the $i$th spring in the chain is defined by
\begin{equation}\label{eqF}
\Ft_i(\qt_i) = U_i'(\textstyle{\frac{1}{2}}|\qt_i|^2)\,\qt_i, \qquad i=1,\dots,K.
\end{equation}

\begin{example}
\label{ex1.1} \em
In the Hookean dumbbell model $K=1$, and the spring force
is defined by ${\Ft}({\qt}) = {\qt}$, with ${\qt} \in {D}=\mathbb{R}^d$,
corresponding to ${U}(s)= s$, $s \in \mathcal{O} = [0,\infty)$.
This model is physically unrealistic
as it admits an arbitrarily large
extension.$\quad\diamond$
\end{example}

We shall therefore assume in what follows that $D$ is a Cartesian product of $K$
{\em bounded} open balls $D_i \subset \mathbb{R}^d$, centred at the
origin $\zerot \in \mathbb{R}^d$, $i=1,\dots, K,$ with $K \geq 1$.

We shall further suppose that
for $i=1,\dots, K$
there exist constants $c_{ij}>0$,
$j=1, 2, 3, 4$, and $\gamma_i > 1$ such that
the (normalized) Maxwellian $M_i$, defined by
\[
M_i(\qt_i) = \frac{1}{\mathcal{Z}_i} {\rm e}^{-U_i(\frac{1}{2}\,|\qt_i|^2)}, \qquad
\mathcal{Z}_i:= {\displaystyle \int_{D_i} {\rm e}^{-U_i(\frac{1}{2}|\qt_i|^2)} \dq_i}\,,
\]
and the associated spring potential $U_i$ satisfy %
\eqlabstart
\begin{eqnarray}
&& ~\hspace{6mm}c_{i1}\,[\mbox{dist}(\qt_i, \,\partial D_i)]^{\gamma_i} \leq  M_i(\qt_i) \,
\leq
c_{i2}\,[\mbox{dist}(\qt_i, \,\partial D_i)]^{\gamma_i} \quad \forall \qt_i \in
D_i, \label{growth1}\\
&& ~\hspace{15mm}c_{i3} \leq [\mbox{dist}(\qt_i,\,\partial D_i)] \,U_i'
(\textstyle{\frac{1}{2}}|\qt_i|^2)
\leq c_{i4}\quad \forall \qt_i \in D_i. \label{growth2}
\end{eqnarray}
\eqlabend
The Maxwellian in the model is then defined by
\begin{align}
M(\qt) := \prod_{i=1}^K M_i(\qt_i) \qquad \forall \qt:=(\qt_1^{\rm T},\ldots,\qt_K^{\rm T})^{\rm T} \in D
:= \bigtimes_{i=1}^K D_i.
\label{MN}
\end{align}
Observe that, for $i=1, \dots, K$,
\begin{equation}
M(\qt)\,\nabqi [M(\qt)]^{-1} = - [M(\qt)]^{-1}\,\nabqi M(\qt) =
\nabqi U_i(\textstyle{\frac{1}{2}}|\qt_i|^2)
=U_i'(\textstyle{\frac{1}{2}}|\qt_i|^2)\,\qt_i. \label{eqM}
\end{equation}
Since
$[U_i(\textstyle{\frac{1}{2}}|\qt_i|^2)]^2 = (-\log M_i(\qt_i) + {\rm Const}.)^2$,
it follows from (\ref{growth1},b) that (if $\gamma_i >1$,
as has been assumed here,)
\begin{equation}\label{additional-1}
\int_{D_i} \left[1 + [U_i(\textstyle{\frac{1}{2}}|\qt_i|^2)]^2
+ [U_i'(\textstyle{\frac{1}{2}}|\qt_i|^2)]^2\right] M_i(\qt_i) \, \dd
\qt_i < \infty, \qquad i=1, \dots, K.
\end{equation}

\begin{example}
\label{ex1.2} \em
In the FENE (finitely extensible nonlinear elastic)
dumbbell model $K=1$ and the spring force is given by
$
{\Ft}({\qt}) = (1 - |{\qt}|^2/b)^{-1}\,{\qt}$,
$\qt \in D = B(\zerot,b^{\frac{1}{2}})$,
corresponding to ${U}(s) = - \frac{b}{2}\log \left(1-\frac{2s}{b}\right)$,
$s \in \mathcal{O} = [0,\frac{b}{2})$.
Here $B(\zerot,b^{\frac{1}{2}})$ is a bounded open ball
in $\mathbb{R}^d$ centred at the origin
$\zerot \in \mathbb{R}^d$ and of fixed radius $b^{\frac{1}{2}}$, with $b>0$.
Direct calculations show that the Maxwellian $M$ and the
elastic potential $U$ of the FENE model satisfy
the conditions (\ref{growth1},b) with $K=1$ and $\gamma:= \frac{b}{2}$
provided that $b>2$. Thus, (\ref{additional-1}) also holds for $K=1$
and $b>2$.

It is interesting to note that in the (equivalent)
stochastic version of the FENE model a solution to the system of
stochastic differential equations
associated with the Fokker--Planck equation exists and has
trajectorial uniqueness if, and only if, $\gamma = \frac{b}{2}\geq 1$;
(cf. Jourdain, Leli\`evre \& Le Bris \cite{JLL2} for details).
Thus, in the general class of FENE-type bead-spring chain models considered here,
the assumption $\gamma_i > 1$, $i=1,\dots, K$, is the weakest reasonable
requirement on the decay-rate of $M_i$  in (\ref{growth1}) as
$\mbox{dist}(\qt_i,\partial D_i) \rightarrow 0$.
$\quad\diamond$
\end{example}

The governing equations of the general FENE-type bead-spring chain
model with centre-of-mass diffusion are
(\ref{ns1a}--d), where the extra-stress tensor $\tautt$ is defined by the \textit{Kramers
expression}:
\begin{equation}\label{tau1}
\tautt(\xt,t) = k\left( \sum_{i=1}^K \int_{D}\psi(\xt,\qt,t)\, \qt_i\,
\qt_i^{\rm T}\, U_i'\left(\textstyle \frac{1}{2}|\qt_i|^2\right)
{\dd} \qt -
\rho(\xt,t)\,\Itt\right),
\end{equation}
with the density of polymer chains located at $\xt$ at time $t$ given by
\begin{equation}\label{rho1}
\rho(\xt,t) = \int_{D} \psi(\xt, \qt,t)
{\dd}\qt.
\end{equation}
The probability density function $\psi$ is a solution of the Fokker--Planck equation
\begin{align}
\label{fp0}
&\frac{\partial \psi}{\partial t} + (\ut \cdot\nabx) \psi +
\sum_{i=1}^K \nabqi
\cdot \left(\sigtt(\ut) \, \qt_i\,\psi \right)
\nonumber
\\
\bet
&\hspace{0.1in} =
\epsilon\,\Delta_x\,\psi +
\frac{1}{2 \,\lambda}\,
\sum_{i=1}^K \sum_{j=1}^K
A_{ij}\,\nabqi \cdot \left(
M\,\nabqj \left(\frac{\psi}{M}\right)\right) \quad \mbox{in } \Omega \times D \times
(0,T],
\end{align}
with $\sigtt(\vt) \equiv \nabxtt \vt$, where $(\vnabtt)(\xt,t) \in {\mathbb
R}^{d \times d}$ and $\{\vnabtt\}_{ij} = \textstyle
\frac{\partial v_i}{\partial x_j}$.
The dimensionless constant $k>0$ featuring in \eqref{tau1} is a constant multiple of
the product of the Boltzmann constant $k_B$ and the absolute temperature $\mathtt{T}$.
In \eqref{fp0}, $\varepsilon>0$ is the centre-of-mass diffusion coefficient defined
as $\varepsilon := (\ell_0/L_0)^2/(4(K+1)\lambda)$ with
$\ell_0:=\sqrt{k_B \mathtt{T}/\mathtt{H}}$ signifying the characteristic microscopic length-scale
and $\lambda :=(\zeta/4\mathtt{H})(U_0/L_0)$,
where $\zeta>0$ is a friction coefficient and $\mathtt{H}>0$ is a spring-constant.
The dimensionless parameter $\lambda \in \Rplus$, called the Weissenberg number (and usually denoted by $\mathsf{Wi}$), characterizes
the elastic relaxation property of the fluid, and $A=(A_{ij})_{i,j=1}^K$ is the symmetric positive definite
\textit{Rouse matrix}, or connectivity matrix;
for example, $A = {\tt tridiag}\left[-1, 2, -1\right]$ in the case
of a linear chain; see, Nitta \cite{Nitta}.

\begin{definition} The collection of equations and structural hypotheses
(\ref{ns1a}--d)--\eqref{fp0} will be referred to throughout the paper as model
$({\rm P}_{\varepsilon})$, or as the {\em general FENE-type bead-spring chain model with centre-of-mass diffusion}.
\end{definition}

A noteworthy feature of equation (\ref{fp0}) in the model $({\rm P}_{\varepsilon})$
compared to classical Fokker--Planck equations for bead-spring models in the
literature is the presence of the
$\xt$-dissipative centre-of-mass diffusion term $\varepsilon
\,\Delta_x \psi$ on the right-hand side of the Fokker--Planck equation (\ref{fp0}).
We refer to Barrett \& S\"uli \cite{BS} for the derivation of (\ref{fp0}) in the case of $K=1$;
see also the article by Schieber \cite{SCHI}
concerning generalized dumbbell models with centre-of-mass diffusion,
and the recent paper of Degond \& Liu \cite{DegLiu} for a careful justification
of the presence of the centre-of-mass diffusion term through asymptotic analysis.
In standard derivations of bead-spring models the centre-of-mass
diffusion term is routinely omitted on the grounds that it is
several orders of magnitude smaller than the other terms in the
equation. Indeed, when the characteristic macroscopic
length-scale $L_0\approx 1$, (for example, $L_0 = \mbox{diam}(\Omega)$), Bhave,
Armstrong \& Brown \cite{Bh} estimate the ratio $\ell_0^2/L_0^2$ to
be in the range of about $10^{-9}$ to $10^{-7}$. However, the
omission of the term $\varepsilon \,\Delta_x \psi$ from (\ref{fp0})
in the case of a heterogeneous solvent velocity $\ut(\xt,t)$ is a
mathematically counterproductive model reduction. When $\varepsilon
\,\Delta_x\psi$ is absent, (\ref{fp0}) becomes a degenerate
parabolic equation exhibiting hyperbolic behaviour with respect to
$(\xt,t)$. Since the study of weak solutions to the coupled problem
requires one to work with velocity fields $\ut$ that have very
limited Sobolev regularity (typically $\ut \in
L^\infty(0,T;\Lt^2(\Omega)) \cap L^2(0,T; \Ht^1_0(\Omega))$), one is
then forced into the technically unpleasant framework of
hyperbolically degenerate parabolic equations with rough transport
coefficients (cf. Ambrosio \cite{Am} and DiPerna \& Lions \cite{DPL}).
The resulting difficulties are further
exacerbated by the fact that, when $D$ is bounded, a typical spring
force $\Ft(\qt)$ for a finitely extensible model (such as FENE)
explodes as $\qt$ approaches $\partial D$; see
Example~\ref{ex1.2} above. For these reasons, here we
shall retain the centre-of-mass diffusion term in (\ref{fp0}).
In order to emphasize that the positive centre-of-mass diffusion coefficient
$\varepsilon$ is {\em not} a mathematical artifact but the
outcome of the physical derivation of the model, in Section \ref{sec:2} and
thereafter the variables $\ut$ and $\psi$ have been labelled with the
subscript $\varepsilon$.

We continue with a brief
literature survey. Unless otherwise stated, the centre-of-mass
diffusion term is absent from the model considered in the cited reference
(i.e. $\varepsilon$ is set to $0$); also, in all references cited $K=1$,
i.e. a simple dumbbell model is considered rather than a bead-spring
chain model.

An early contribution to the existence and uniqueness of
local-in-time solutions to a family of dumbbell type polymeric
flow models is due to Renardy \cite{R}. While the class of
potentials $\Ft(\qt)$ considered by Renardy \cite{R} (cf.\
hypotheses (F) and (F$'$) on pp.~314--315) does include the case of
Hookean dumbbells, it excludes the practically relevant case of the
FENE dumbbell model (see Example~\ref{ex1.2} above). More recently, E, Li \&
Zhang \cite{E} and Li, Zhang \& Zhang \cite{LZZ} have revisited the
question of local existence of solutions for dumbbell models. A further
development in this direction is the work of Zhang \& Zhang \cite{ZZ},
where the local existence of regular solutions to FENE-type dumbbell
models has been shown. All of these papers require high regularity of the initial data.
Constantin \cite{CON} considered the Navier--Stokes equations
coupled to nonlinear Fokker--Planck equations describing the
evolution of the probability distribution of the particles
interacting with the fluid.
Otto \& Tzavaras \cite{OT} investigated the Doi model (which is
similar to a Hookean model (cf. Example \ref{ex1.1} above), except
that $D=S^2$) for suspensions of rod-like molecules in the dilute regime.
Jourdain, Leli\`evre \& Le Bris \cite{JLL2} studied the existence of
solutions to the FENE dumbbell model in the case of a simple Couette flow.
By using tools from the theory of stochastic differential equations, they showed the existence
of a unique local-in-time solution to the FENE dumbbell model for $d=2$
when the velocity field $\ut$ is unidirectional and of the particular form $\ut(x_1,x_2)
= (u_1(x_2),0)^{\rm T}$.

In the case of Hookean dumbbells ($K=1$), and assuming $\varepsilon=0$, the coupled
microscopic-macroscopic model described above yields, formally,
taking the second moment of $\qt \mapsto \psi(\qt,\xt,t)$, the fully macroscopic,
Oldroyd-B model of viscoelastic flow. Lions \& Masmoudi \cite{LM}
have shown the existence of global-in-time weak solutions to the Oldroyd-B model
in a simplified corotational setting (i.e. with $\sigma(\ut) = \nabxtt \ut$
replaced by $\frac{1}{2}(\nabxtt \ut - (\nabxtt u)^{\rm T})$)
by exploiting the propagation in time of the compactness of the
solution (i.e. the property that if one takes a sequence of weak solutions that
converges weakly and such that the corresponding sequence of
initial data converges strongly, then the weak
limit is also a solution) and the DiPerna--Lions \cite{DPL} theory of
renormalized solutions to linear hyperbolic
equations with nonsmooth transport coefficients. It is not
known if an identical global
existence result for the Oldroyd-B model also holds in the
absence of the crucial assumption
that the drag term is corotational. With $\varepsilon>0$, the coupled microscopic-macroscopic model above yields,
taking the appropriate moments in the case of Hookean dumbbells, a dissipative
version of the Oldroyd-B model. In this sense, the Hookean dumbbell model has a macroscopic closure:
it is the Oldroyd-B model when $\varepsilon=0$, and a dissipative version of Oldroyd-B
when $\varepsilon>0$ (cf. Barrett \& S\"uli \cite{BS}).
Barrett \& Boyaval \cite{barrett-boyaval-09} have proved a global existence result
for this dissipative Oldroyd-B model in two space dimensions.
In contrast, the FENE model is not known to have an
exact closure at the macroscopic level, though Du, Yu \& Liu \cite{DU} and Yu, Du \&
Liu \cite{YU} have recently considered the analysis of approximate closures of the FENE dumbbell model.
Lions \& Masmoudi \cite{LM2} proved the global existence of weak
solutions for the \textit{corotational} FENE dumbbell model,
once again corresponding to the
case of $\varepsilon=0$ and $K=1$, and the Doi model,
also called the rod model;
see also the work of Masmoudi \cite{M}. Recently, Masmoudi \cite{M10} has extended this analysis to the noncorotational case.

Previously, El-Kareh \& Leal \cite{EKL} had proposed a steady macroscopic model,
with added dissipation in the equation satisfied by the conformation tensor, defined as
$$\Att(\xt):=\int_D\qt\,\qt^{\rm T} U'(\frac{1}{2}|\qt|^2) \,\psi(\xt,\qt) {\dd}\qt,$$ in order to account
for Brownian motion across streamlines; the model
can be thought of as an approximate macroscopic closure of a
FENE-type micro-macro model with centre-of-mass diffusion.

Barrett, Schwab \& S\"uli \cite{BSS}
showed the existence of global weak solutions to the coupled microscopic-macroscopic
model (\ref{ns1a}--d),
(\ref{fp0}) with $\varepsilon=0$, $K=1$, an $\xt$-mollified
velocity gradient in the Fokker--Planck
equation and an $\xt$-mollified probability density function $\psi$ in the Kramers
expression, admitting a large class of potentials $U$
(including the Hookean dumbbell model and general FENE-type dumbbell models); in addition to these
mollifications, $\ut$ in the
$\xt$-convective term $(\ut\cdot\nabx) \psi$ in the
Fokker--Planck equation was also mollified.
Unlike Lions \& Masmoudi \cite{LM}, the arguments in
Barrett, Schwab \& S\"uli \cite{BSS} did {\em not} require that
the drag term $\nabq\cdot(\sigtt(\ut)\,\qt\, \psi)$ in the Fokker--Planck equation
was corotational in the FENE case.

In Barrett \& S\"uli \cite{BS},
we derived the coupled Navier--Stokes--Fokker--Planck model
with centre-of-mass diffusion stated above, in the case of $K=1$.
We established the existence of global-in-time weak solutions to a mollification
of the model for a general class of spring-force-potentials including in
particular the FENE potential. We justified also, through a rigorous
limiting process, certain classical reductions of this model appearing
in the literature that exclude the centre-of-mass diffusion term from the
Fokker--Planck equation on the grounds that the diffusion coefficient is
small relative to other coefficients featuring in the equation.
In the case of a corotational drag term we performed a
rigorous passage to the limit as the
mollifiers in the Kramers expression and the drag term converge to identity operators.

In Barrett \& S\"uli \cite{BS2} we showed the existence of global-in-time
weak solutions to the general class of noncorotational FENE type dumbbell models (including the standard
FENE dumbbell model) with centre-of-mass
diffusion, in the case of $K=1$, with microsropic  cut-off (cf. \eqref{cut1} and \eqref{betaLa} below) in the drag term
\begin{equation}\label{drag}
 \nabq\cdot(\sigtt(\ut)\,\qt\, \psi) =
\nabq\cdot \left[\sigtt(\ut) \, \qt\,M \left(\frac{\psi}{M}\right)\right].
\end{equation}
In this paper we prove the existence of global-in-time weak
solutions to the model {\em without} cut-off or mollification, in the general case of $K \geq 1$.
Since the argument is long and technical, we give a brief overview of the main steps of
the proof here.

{\em Step 1.} Following the approach in Barrett \& S\"uli \cite{BS2}
and motivated by recent papers of Jourdain, Leli\`evre, Le Bris \& Otto \cite{JLLO} and
Lin, Liu \& Zhang \cite{LinLZ}
(see also  Arnold, Markowich, Toscani \& Unterreiter \cite{AMTU},
and Desvillettes \& Villani \cite{DV})
concerning the convergence of the probability density function $\psi$ to its equilibrium
value $\psi_{\infty}(\xt,\qt):=M(\qt)$
(corresponding to the equilibrium value $\ut_\infty(\xt) :=\zerot$
of the velocity field) in the absence of body forces $\ft$,
we observe that if $\psi/M$ is bounded above then, for $L \in \mathbb{R}_{>0}$
sufficiently large, the drag term (\ref{drag}) is equal to
\begin{equation}\label{cut1}
\nabq\cdot \left[\sigtt(\ut) \, \qt\,M \,\beta^L\left
(\frac{\psi}{M}\right)\right],
\end{equation}
where $\beta^L \in C({\mathbb R})$ is a cut-off function defined as
\begin{align}
\beta^L(s) := \min(s,L).
\label{betaLa}
\end{align}
More generally, in the case of $K \geq 1$, in analogy with \eqref{cut1}, the drag term with cut-off is
defined by
$ \sum_{i=1}^K \nabqi \cdot \left(\sigtt(\ut) \, \qt_i\,M \,\beta^L\!\left(\frac{\psi}{M}\right)\right)$.
It then follows that, for $L\gg 1$, any solution $\psi$ of (\ref{fp0}),
such that $\psi/M$ is bounded above, also satisfies
\begin{eqnarray}
\label{eqpsi1aa}
&&\hspace{-6.5mm}\frac{\partial \psi}{\partial t} + (\ut \cdot\nabx) \psi + \sum_{i=1}^K \nabqi \cdot \left(\sigtt(\ut) \, \qt_i\,M \,\beta^L\!\left(\frac{\psi}{M}\right)\right)
\nonumber
\\
\bet
&&= \epsilon\,\Delta_x\,\psi + \frac{1}{2 \,\lambda}\,\sum_{i=1}^K\sum_{j=1}^K A_{ij} \nabqi \cdot \left(
M\,\nabqj\!\left(\frac{\psi}{M}\right)\right)\quad \mbox{in } \Omega \times D \times
(0,T]. ~~~
\end{eqnarray}
We impose the following boundary and initial conditions:
\begin{subequations}
\begin{alignat}{2}
&M \left[\frac{1}{2\,\lambda} \sum_{j=1}^K A_{ij}\,\nabqj\!\left(\frac{\psi}{M}\right)
- \sigtt(\ut) \,\qt_i\,\beta^L\!\left(\frac{\psi}{M}\right)
\right]\! \cdot \frac{\qt_i}{|\qt_i|}
=0\qquad &&~~\nonumber\\
&~ \qquad && \hspace{-7.1cm} \mbox{on }
\Omega \times \partial D_i\times \left(\bigtimes_{j=1,\, j \neq i}^K D_j\right)\times (0,T],
\mbox{~~for $i=1,\dots, K$,} \label{eqpsi2aa}\\
&\epsilon\,\nabx \psi\,\cdot\,\nt =0&&\hspace{-6.1cm}\mbox{on }
\partial \Omega \times D\times (0,T],\label{eqpsi2ab}\\
&\psi(\cdot,\cdot,0)=M(\cdot)\,\beta^L\!\left({\psi_{0}(\cdot,\cdot)}/{M(\cdot)}\right) \geq 0&&
\hspace{-2.7cm}\mbox{~on $\Omega\times D$},\label{eqpsi3ac}
\end{alignat}
\end{subequations}
where $\qt_i$ is normal to $\partial D_i$, as $D_i$
is a bounded ball centred at the origin,
and $\nt$ is normal to $\partial \Omega$;
$\psi_0$ is nonnegative,
defined on $\Omega\times D$,
with $\int_{D} \psi_0(\xt,\qt) \dd \qt = 1$ for a.e. $\xt  \in \Omega$, and
assumed to have finite relative entropy with
respect to the Maxwellian $M$; i.e. $\int_{\Omega \times D} \psi_0(\xt,\qt)
\log (\psi_0(\xt,\qt)/M(\qt)) \dq \dx < \infty$.
Clearly, if there exists $L>0$ such that $0 \leq \psi_0 \leq L\, M$,
then $M\,\beta^L(\psi_{0}/M) = \psi_0$. Henceforth $L >1$ is assumed.

\begin{definition}
The coupled problem (\ref{ns1a}--d), (\ref{tau1}), (\ref{rho1}),
(\ref{eqpsi1aa}), (\ref{eqpsi2aa}--c) will be referred to
as model $({\rm P}_{\varepsilon,L})$, or as the
{\em general FENE-type bead-spring chain
model with centre-of-mass diffusion and microscopic cut-off},
with cut-off parameter $L>1$.
\end{definition}

In order to highlight the dependence
on $\varepsilon$ and $L$, in subsequent sections the solution to
(\ref{eqpsi1aa}), (\ref{eqpsi2aa}--c) will be labelled $\psiae$.
Due to the coupling of
(\ref{eqpsi1aa}) to (\ref{ns1a}) through (\ref{tau1}), the velocity and the pressure
will also depend on $\varepsilon$ and $L$ and we shall therefore
denote them in subsequent
sections by $\ut_{\varepsilon,L}$ and $p_{\varepsilon,L}$.

The cut-off $\beta^{L}$ has a convenient property:
the couple $(\ut_{\infty},\psi_{\infty})$,  defined by
$\ut_{\infty}(\xt) := \zerot$ and $\psi_{\infty}(\xt,\qt):=M(\qt)$,
is still
an equilibrium solution of (\ref{ns1a}--d) with $\ft =\zerot$,
(\ref{tau1}), (\ref{rho1}),
(\ref{eqpsi1aa}), (\ref{eqpsi2aa}--c) for all $L>0$.
Thus, unlike the truncation of
the (unbounded) potential proposed in El-Kareh \& Leal \cite{EKL},
the introduction of
the cut-off function $\beta^L$ into the Fokker--Planck equation (\ref{fp0})
does not alter the equilibrium solution $(\ut_{\infty},\psi_\infty)$ of the
original Navier--Stokes--Fokker--Planck system.
In addition, the boundary conditions for $\psi$
on $\partial\Omega\times D\times(0,T]$ and $\Omega \times \partial D
\times (0,T]$ ensure that
$\int_{D}\psi(\xt,\qt,t) \dd \qt  =
\int_{D}\psi(\xt,\qt,0) \dd \qt$
for a.e. $\xt \in \Omega$ and a.e. $t \in \mathbb{R}_{\geq 0}$.

\textit{Step 2.} Ideally, one would like to pass to the limit $L\rightarrow \infty$ in problem $({\rm P}_{\varepsilon,L})$ to
deduce the existence of solutions to $({\rm P}_{\varepsilon})$. Unfortunately,
such a direct attack at the problem is (except in the special case of $d=2$, or in the absence of
convection terms from the model,) fraught with technical difficulties. Instead,
we shall first (semi)discretize problem $({\rm P}_{\varepsilon,L})$
by an implicit Euler scheme with respect to $t$, with step size $\Delta t$;
this then results in a time-discrete version $({\rm P}^{\Delta t}_{\varepsilon,L})$ of $({\rm P}_{\varepsilon,L})$. By using Schauder's fixed point theorem, we will show in
Section \ref{sec:existence-cut-off} the existence of solutions to $({\rm P}^{\Delta t}_{\varepsilon,L})$.
In the course of the proof, for technical reasons, a further cut-off, now from below, is required,
with a cut-off parameter $\delta \in (0,1)$, which we shall let pass to $0$ to complete the proof of existence of solutions to $({\rm P}^{\Delta t}_{\varepsilon,L})$ in the limit of $\delta \rightarrow 0_+$ (cf.
Section \ref{sec:existence-cut-off}).
Ultimately, of course, our aim is to show existence of weak solutions to the
general FENE-type bead-spring chain model with centre-of-mass diffusion, $({\rm P}_{\varepsilon})$, and that demands passing to
the limits $\Delta t \rightarrow 0_+$ and $L \rightarrow \infty$; this then brings us
to the next step in our argument.

{\em Step 3.} We shall link the time step $\Delta t$ to the cut-off parameter $L>1$ by
demanding that $\Delta t = o(L^{-1})$, as $L \rightarrow \infty$, so that the
only parameter in the problem $({\rm P}^{\Delta t}_{\varepsilon,L})$ is the cut-off parameter (the centre-of-mass
diffusion parameter $\epsilon$ being fixed). By using special energy estimates, based on testing the
Fokker--Planck equation in $({\rm P}^{\Delta t}_{\varepsilon,L})$
with the derivative of the relative entropy with respect to
the Maxwellian of the general FENE-type bead-spring
 chain model, we show that $\uta^{\Delta t}$
can be bounded, independent of $L$.
Specifically $\ut^{\Delta t}_{\varepsilon,L}$ is bounded in the norm of
the classical Leray space, independent of $L$;
also, the $L^\infty$ norm in time of the relative entropy of
$\psi^{\Delta t}_{\epsilon,L}$ and the $L^2$ norm in time
of the Fisher information of  $\hat \psi^{\Delta t}_{\epsilon,L}
:=\psi^{\Delta t}_{\epsilon,L}/M$ are bounded,
independent of $L$.
We then use these $L$-independent
bounds on the relative entropy and the Fisher information to derive $L$-independent bounds on the
time-derivatives of $\uta^{\Delta t}$ and $\psia^{\Delta t}$ in very weak, negative-order Sobolev norms.

\textit{Step 4.} The collection of $L$-independent bounds from Step 3
then enables us to extract a weakly convergent subsequence of solutions to
problem $({\rm P}^{\Delta t}_{\varepsilon,L})$ as $L \rightarrow \infty$.
We then apply a general compactness result in seminormed sets due to Dubinski{\u\i} \cite{DUB} (see also \cite{BS-DUB}),
which furnishes strong convergence of a subsequence of solutions
$(\ut^{\Delta t_k}_{\varepsilon,L_k}, \psi^{\Delta t_k}_{\epsilon,L_k})$
to $({\rm P}^{\Delta t}_{\varepsilon,L})$ with $\Delta t = o(L^{-1})$ as $L \rightarrow \infty$,
in $L^2(0,T;L^2(\Omega))
\times L^p(0,T; L^1(\Omega \times D))$ for any $p>1$. A crucial observation is that the set of
functions with finite Fisher information is not a linear space; therefore, typical Aubin--Lions--Simon
type compactness results (see, for example, Simon \cite{Simon})
do not work in our context; however, Dubinski{\u\i}'s compactness theorem, which applies
to seminormed sets in the sense of Dubinski{\u\i}, does, enabling us to pass to the limit
with the microscopic cut-off parameter $L$ in the model $({\rm P}^{\Delta t}_{\varepsilon,L})$, with
$\Delta t = o(L^{-1})$,
as $L \rightarrow \infty$, to finally deduce the existence of a weak solution to the general FENE-type bead-spring chain model with centre-of-mass diffusion,
$({\rm P}_{\varepsilon})$.

The paper is structured as follows. We begin, in Section~\ref{sec:2}, by stating $({\rm P}_{\varepsilon,L})$,
the coupled Navier--Stokes--Fokker--Planck system with centre-of-mass diffusion and microscopic
cut-off for a general class of FENE-type spring potentials. In Section~\ref{sec:existence-cut-off}
we establish the existence of solutions to the time-discrete problem $({\rm P}^{\Delta t}_{\varepsilon,L})$.
In Section~\ref{sec:entropy} we derive a set of $L$-independent bounds on $\uta^{\Delta t}$ in the classical Leray space, together with $L$-independent bounds on the relative entropy of $\psi^{\Delta t}_{\epsilon,L}$ and Fisher information
of $\psia^{\Delta t}$. We then use
these $L$-independent bounds on spatial norms to obtain $L$-independent bounds on very weak
norms of time-derivatives of $\uta^{\Delta t}$ and $\psia^{\Delta t}$.
Section~\ref{sec:dubinskii} is concerned with the application of Dubinski{\u\i}'s
theorem to our problem; and the extraction of a strongly convergent subsequence, which we shall then use
in Section~\ref{sec:passage.to.limit} to pass to the limit with the cut-off parameter $L$
in problem $({\rm P}^{\Delta t}_{\varepsilon,L})$, with $\Delta t = o(L^{-1})$,
as $L \rightarrow \infty$,
to deduce the existence of a weak solution $(\ut_{\varepsilon},\psi_\varepsilon:=M\,\hat\psi_\epsilon)$
to problem $({\rm P}_{\varepsilon})$, the general FENE-type bead-spring chain model
with centre-of-mass diffusion.
Finally, in Section \ref{sec:decay}, we show using a logarithmic Sobolev inequality and the
Csisz\'ar--Kullback inequality that, when $\ft \equiv \zerot$, global weak solutions $t \mapsto (\ut_\epsilon(t),\psi_\epsilon(t))$
thus constructed decay exponentially in time to $(\zerot,M)$, at a rate that is independent
of the initial data for the Navier--Stokes and Fokker--Planck equations and of
the centre-of-mass diffusion coefficient $\varepsilon$.
We shall operate within Maxwellian-weighted Sobolev spaces, which provide the
natural functional-analytic framework for the problem.
Our proofs require special
density and embedding results in these spaces, which
are proved in Appendix C and Appendix D, respectively.

For an analogous set of existence and equilibration results for weak solutions
of Hookean-type bead-spring chain models for dilute polymers, we refer to Part II of the present paper \cite{BS2010-hookean}.

\section{The polymer model $({\rm P}_{\varepsilon,L})$}
\label{sec:2}
\setcounter{equation}{0}

Let $\Omega \subset {\mathbb R}^d$ be a bounded open set with a
Lipschitz-continuous boundary $\partial \Omega$, and suppose that
the set $D:= D_1\times \cdots \times D_K$ of admissible
conformation vectors $\qt := (\qt_1^{\rm T}, \ldots ,\qt_K^{\rm T})^{\rm T}$ in (\ref{fp0}) is
such that $D_i$, $i=1, \dots, K$, is an open ball
in ${\mathbb R}^d$, $d=2$ or $3$, centred at the origin with boundary $\partial D_i$
and radius $\sqrt{b_i}$, $b_i>2$; let
\begin{align}
\mbox{\small $\partial D := \bigcup_{i=1}^K
\left[\partial D_i \times \left(\bigtimes_{j=1,\, j \neq i}^K D_j\right)\right].$}
\label{dD}
\end{align}
Collecting (\ref{ns1a}--d), (\ref{tau1}), and (\ref{fp0}),
we then consider the following initial-boundary-value problem,
dependent on the parameter $L > 1$. As has been already emphasized in the
Introduction, the centre-of-mass diffusion coefficient $\varepsilon>0$ is a
physical parameter and is regarded as being fixed, although we systematically
highlight its presence in the model through our subscript notation.

(${\rm P}_{\epsilon,L}$)
Find $\utae\,:\,(\xt,t)\in \overline{\Omega} \times [0,T]
\mapsto \utae(\xt,t) \in {\mathbb R}^d$ and $\pae\,:\, (\xt,t) \in
\Omega \times (0,T] \mapsto \pae(\xt,t) \in {\mathbb R}$ such that
\begin{subequations}
\begin{eqnarray}
\ptutae + (\utae \cdot \nabx) \utae - \nu \,\delx \utae + \nabx \pae
&=& \ft + \nabx \cdot \tautt(\psiae)
\nonumber\\
%\bet
&& \qquad \;
\mbox{in } \Omega \times (0,T], \label{equ1}\\
%\bet
\nabx \cdot \utae &=& 0 ~\hspace{0.5cm}\mbox{in } \Omega \times (0,T],
\label{equ2}\\
%\bet
\utae &=& \zerot  ~\hspace{0.5cm}\mbox{on } \partial \Omega \times (0,T],
\label{equ3}\\
%\bet
\utae(\xt,0)&=&\ut_{0}(\xt) \!\hspace{0.5cm}\forall \xt \in \Omega,
\label{equ4}
\end{eqnarray}
\end{subequations}
where $\psiae\,:\,(\xt,\qt,t)\in
\overline{\Omega} \times \overline{D} \times [0,T]
\mapsto \psiae(\xt,\qt,t)
\in {\mathbb R}$, and
$\tautt(\psiae)\,:\,(\xt,t) \in \Omega \times (0,T] \mapsto
\tautt(\psiae)(\xt,t)\in \mathbb{R}^{d\times d}$ is the symmetric
extra-stress tensor defined as
\begin{equation}
\tautt(\psiae) := k \left( \sum_{i=1}^K
\Ctt_i(\psiae)\right)
- k\,\rho(\psiae)\, \Itt.
\label{eqtt1}
\end{equation}
Here $k \in \Rplus$,
$\Itt$ is the unit $d
\times d$ tensor,
\begin{subequations}
\begin{eqnarray}
\Ctt_i(\psiae)(\xt,t) &:=& \int_{D} \psiae(\xt,\qt,t)\, U_i'({\textstyle
\frac{1}{2}}|\qt_i|^2)\,\qt_i\,\qt_i^{\rm T} \dq,\qquad \mbox{and} \label{eqCtt}\\
\rho(\psiae)(\xt,t) &:=& \int_{D} \psiae(\xt,\qt,t)\dq. \label{eqrhott}
\end{eqnarray}
\end{subequations}
The Fokker--Planck equation with microscopic cut-off satisfied by $\psiae$ is:
\begin{align}
\label{eqpsi1a}
&\hspace{-2mm}\frac{\partial \psiae}{\partial t} + (\utae \cdot\nabx) \psiae +
\sum_{i=1}^K \nabqi
\cdot \left[\sigtt(\utae) \, \qt_i\,M\,\beta^L
\left(\frac{\psiae}{M}\right)\right]
\nonumber
\\
&\hspace{0.01in} =
\epsilon\,\Delta_x\,\psiae +
\frac{1}{2 \,\lambda}\,
\sum_{i=1}^K \sum_{j=1}^K
A_{ij}\,\nabqi \cdot \left(
M\,\nabqj \left(\frac{\psiae}{M}\right)\right) \quad \mbox{in } \Omega \times D \times
(0,T].
\end{align}
Here, for a given $L > 1$, $\beta^L \in C({\mathbb R})$ is defined by (\ref{betaLa}),
$\sigtt(\vt) \equiv \nabxtt \vt$, and
\begin{align}
\!\!\!\!\!A \in {\mathbb R}^{K \times K} \mbox{ is symmetric positive definite
with smallest eigenvalue $a_0 \in {\mathbb R}_{>0}$.}
\label{A}
\end{align}
We impose the following boundary and initial conditions:
\begin{subequations}
\begin{align}
&M\left[\frac{1}{2\,\lambda} \sum_{j=1}^K A_{ij}\, \nabqj \left(\frac{\psiae}{M}\right)
- \sigtt(\utae) \,\qt_i\,\beta^L\left(\frac{\psiae}{M}\right)
\right]\cdot \frac{\qt_i}{|\qt_i|}
=0 \nonumber \\
&\hspace{2.52cm}\mbox{on }
\Omega \times \partial D_i \times \left(\bigtimes_{j=1,\, j \neq i}^K D_j\right) \times (0,T],
\quad i=1, \dots, K,
\label{eqpsi2}\\
&\hspace{3.8cm}\epsilon\,\nabx \psiae\,\cdot\,\nt =0 \qquad \,\, \quad \mbox{on }
\partial \Omega \times D\times (0,T],\label{eqpsi2a}\\
&\qquad\psiae(\cdot,\cdot,0)=M(\cdot)\,\beta^L(\psi_{0}(\cdot,\cdot)/M(\cdot)) \geq 0 \qquad \mbox{on  $\Omega\times D$},\label{eqpsi3}
\end{align}
\end{subequations}
where $\nt$ is the unit outward normal to $\partial \Omega$.
The boundary conditions for $\psiae$
on $\partial\Omega\times D\times(0,T]$ and $\Omega \times \partial D
\times (0,T]$ have been chosen so as to ensure that
\begin{align}
\int_{D}\psiae(\xt,\qt,t) \dd \qt  =
\int_{D} \psiae(\xt,\qt,0) \dd \qt  \qquad \forall (\xt,t) \in \Omega \times (0,T].
\label{intDcon}
\end{align}
Henceforth, we shall write $\psia:= \psiae/M$, $\hat\psi_0 := \psi_0/M$.
Thus, for example, \eqref{eqpsi3} in terms of this compact notation becomes:
$\psia(\cdot,\cdot,0) = \beta^L(\hat\psi_0(\cdot,\cdot))$ on $\Omega \times D$.

The notation $|\cdot|$ will be used to signify one of the following. When applied to a real number $x$,
$|x|$ will denote the absolute value of the number $x$; when applied to a vector $\vt$,  $|\vt|$ will
stand for the Euclidean norm of the vector $\vt$; and, when applied to a square matrix $A$, $|A|$ will
signify the Frobenius norm, $[\mathfrak{tr}(A^{\rm T}A)]^{\frac{1}{2}}$, of the matrix $A$, where, for a square matrix
$B$, $\mathfrak{tr}(B)$ denotes the trace of $B$.

\section{Existence of a solution to the discrete-in-time problem}
\label{sec:existence-cut-off}
\setcounter{equation}{0}

Let
\begin{eqnarray}
&\Ht :=\{\wt \in \Lt^2(\Omega) : \nabx \cdot \wt =0\} \quad
\mbox{and}\quad \Vt :=\{\wt \in \Ht^{1}_{0}(\Omega) : \nabx \cdot
\wt =0\},&~~~ \label{eqVt}
\end{eqnarray}
where the divergence operator $\nabx\cdot$ is to be understood in
the sense of distributions on $\Omega$. Let $\Vt'$ be
the dual of $\Vt$.
Let
$\St: \Vt' \rightarrow \Vt$ be such that $\St \,\vt$
is the unique solution to the Helmholtz--Stokes problem
\begin{eqnarray}
&\displaystyle\int_{\Omega} \St\,\vt\cdot\, \wt \dx +
\int_{\Omega}
\nabxtt (\St\,\vt)
: \wnabtt \dx
= \langle \vt,\wt \rangle_V
\qquad \forall \wt \in \Vt,
\label{eqvn1}
\end{eqnarray}
where $ \langle \cdot,\cdot\rangle_V$ denotes the duality pairing
between $\Vt'$ and $\Vt$.
We note that
\begin{eqnarray}
\left\langle \vt, \St\,\vt \right\rangle_V =
\|\St\,\vt\|_{H^1(\Omega)}^2
\qquad \forall \vt \in
\Vt', %\supset (\Ht^1_0(\Omega))',
\label{eqvn2}
\end{eqnarray}
and $\|\St \cdot\|_{H^{1}(\Omega)}$ is a norm on $\Vt'$. More generally, let
$\Vt_\sigma$ denote the closure of the set of all divergence-free $\Ct^\infty_0(\Omega)$ functions
in the norm of $\Ht^1_0(\Omega)\cap \Ht^\sigma(\Omega)$, $\sigma \geq 1$, equipped with the Hilbert
space norm, denoted by $\|\cdot\|_{V_\sigma}$, inherited from $\Ht^\sigma(\Omega)$, and let $\Vt_\sigma'$
signify the dual space of $\Vt_\sigma$, with duality pairing $\langle \cdot , \cdot \rangle_{V_\sigma}$.
As $\Omega$ is a bounded Lipschitz domain, we have that $\Vt_1 = \Vt$ (cf. Temam \cite{Temam},
Ch. 1, Thm. 1.6).
Similarly, $ \langle \cdot , \cdot \rangle_{H^1_0(\Omega)}$ will
denote the duality pairing between $(\Ht^1_0(\Omega))'$ and $\Ht^1_0(\Omega)$.
The norm on $(\Ht^1_0(\Omega))'$  will be that induced from taking
$\|\nabxtt \cdot\|_{L^2(\Omega)}$ to be the norm on $\Ht^1_0(\Omega)$.

For later purposes, we recall the following well-known
Gagliardo--Nirenberg inequality. Let $r \in [2,\infty)$ if $d=2$,
and $r \in [2,6]$ if $d=3$ and $\theta = d \,\left(\frac12-\frac
1r\right)$. Then, there is a constant $C=C(\Omega,r,d)$,
such that, for all $\eta \in H^{1}(\Omega)$:
\begin{equation}\label{eqinterp}
\|\eta\|_{L^r(\Omega)}
\leq C\,
\|\eta\|_{L^2(\Omega)}^{1-\theta}
\,\|\eta\|_{H^{1}(\Omega)}^\theta.
\end{equation}

Let $\mathcal{F}\in C(\mathbb{R}_{>0})$ be defined by
$\mathcal{F}(s):= s\,(\log s -1) + 1$, $s>0$.
As $\lim_{s \rightarrow 0_+} \mathcal{F}(s) = 1$,
the function $\mathcal{F}$ can be considered
to be defined and continuous on $[0,\infty)$,
where it is a nonnegative, strictly convex function
with $\mathcal{F}(1)=0$.
We assume the following:
\begin{align}\nonumber
&\partial \Omega \in C^{0,1}; \quad \ut_0 \in \Ht; \quad
\hat \psi_0 := \frac{\psi_0}{M}\geq 0 \ {\rm ~a.e.\ on}\ \Omega \times D\quad \mbox{with}
\\ &
\mathcal{F}(\hat\psi_0)
\in L^1_M(\Omega \times D) \quad \mbox{and}
\quad \int_D M(\qt)\,\hat\psi_0(\xt,\qt)\,\dq = 1\quad \mbox{for a.e.\ } \xt \in \Omega;
\nonumber\\
& \gamma_i>1,\quad i= 1, \dots,  K \quad  \mbox{in (\ref{growth1},b)};
\quad \mbox{and} \quad \ft \in L^{2}(0,T;(\Ht^1_0(\Omega))').\label{inidata}
\end{align}
Here, $L^p_M(\Omega \times D)$, for $p\in [1,\infty)$,
denotes the Maxwellian-weighted $L^p$ space over $\Omega \times D$ with norm
\[
\| \hat \varphi\|_{L^{p}_M(\Omega\times D)} :=
\left\{ \int_{\Omega \times D} \!\!M\,
|\hat \varphi|^p \dq \dx
\right\}^{\frac{1}{p}}.
\]
Similarly, we introduce $L^p_M(D)$,
the Maxwellian-weighted $L^p$ space over $D$. Letting
\begin{eqnarray}
\| \hat \varphi\|_{H^{1}_M(\Omega\times D)} &:=&
\left\{ \int_{\Omega \times D} \!\!M\, \left[
|\hat \varphi|^2 + \left|\nabx \hat \varphi \right|^2 + \left|\nabq \hat
\varphi \right|^2 \,\right] \dq \dx
\right\}^{\frac{1}{2}}\!\!, \label{H1Mnorm}
\end{eqnarray}
we then set
\begin{eqnarray}
\quad \hat X \equiv H^{1}_M(\Omega \times D)
&:=& \left\{ \hat \varphi \in L^1_{\rm loc}(\Omega\times D): \|
\hat \varphi\|_{H^{1}_M(\Omega\times D)} < \infty \right\}. \label{H1M}
\end{eqnarray}
It is shown in Appendix C %in the extended version of this paper \cite{BS2010}
that
\begin{align}
C^{\infty}(\overline{\Omega \times D})
\mbox{ is dense in } \hat X.
\label{cal K}
\end{align}

We have from Sobolev embedding that
\begin{equation}
H^1(\Omega;L^2_M(D)) \hookrightarrow L^s(\Omega;L^2_M(D)),
\label{embed}
\end{equation}
where $s \in [1,\infty)$ if $d=2$ or $s \in [1,6]$ if $d=3$.
Similarly to (\ref{eqinterp}) we have,
with $r$ and $\theta$ as there,
that there is a constant $C$, depending only  on
$\Omega$, $r$ and $d$, such that
\begin{equation}\label{MDeqinterp}
\hspace{-0mm}\|\hat \varphi\|_{L^r(\Omega;L^2_M(D))}
\leq C\,
\|\hat \varphi\|_{L^2(\Omega;L^2_M(D))}^{1-\theta}
\,\|\hat \varphi\|_{H^1(\Omega;L^2_M(D))}^\theta ~~~
\mbox{$\forall\hat \varphi \in
H^{1}(\Omega;L^2_M(D))$}.
\end{equation}
In addition, we note that the embeddings
\begin{subequations}
\begin{align}
 H^1_M(D) &\hookrightarrow L^2_M(D) ,\label{wcomp1}\\
H^1_M(\Omega \times D) \equiv
L^2(\Omega;H^1_M(D)) \cap H^1(\Omega;L^2_M(D))
&\hookrightarrow
L^2_M(\Omega \times D) \equiv L^2(\Omega;L^2_M(D))
\label{wcomp2}
\end{align}
\end{subequations}
are compact if $\gamma_i \geq 1$, $i=1, \dots,  K$, in (\ref{growth1},b);
see Appendix D.
% in the extended version of this paper \cite{BS2010}.

Let $\hat X'$ be the dual space of $\hat X$ with $L^2_M(\Omega\times D)$
being the pivot space.
Then, similarly to (\ref{eqvn1}),
let $\mathcal{G}: \hat X' \rightarrow \hat X$
be such that $\mathcal{G}\, \hat \eta$ is the unique solution of
\begin{align}
&\int_{\Omega \times D}
M\,\biggl[ (\mathcal{G}\,\hat \eta) \,\hat \varphi
+ \nabq\, (\mathcal{G}\,\hat \eta) \cdot \nabq\,\hat \varphi
+ \nabx\, (\mathcal{G}\,\hat \eta) \cdot \nabx\,\hat \varphi \biggr]
\dq \dx
\label{CalG}
\nonumber
\\
&\hspace{2.14in}
=
\langle M\,\hat \eta,\hat \varphi \rangle_{\hat X} \qquad
\forall \hat \varphi \in \hat X,
\end{align}
where $\langle M\,\cdot, \cdot \rangle_{\hat X}$
is the duality pairing between $\hat X'$ and $\hat X$.
Then, as in (\ref{eqvn2}),
\begin{align}
\langle M \,\hat \eta, \mathcal{G}\, \hat \eta\, \rangle_{\hat X}
= \|\mathcal{G}\,\hat \eta \|_{\hat X}^2 \qquad
\forall \hat \eta \in \hat X',
\label{CalG1}
\end{align}
and $\|\mathcal{G} \cdot\|_{\hat X}$ is a norm on ${\hat X}'$.

We recall the Aubin--Lions--Simon compactness theorem, see, e.g.,
Temam \cite{Temam} and Simon \cite{Simon}. Let $\mathcal{B}_0$, $\mathcal{B}$ and
$\mathcal{B}_1$ be Banach
spaces, $\mathcal{B}_i$, $i=0,1$, reflexive, with a compact embedding $\mathcal{B}_0
\hookrightarrow \mathcal{B}$ and a continuous embedding $\mathcal{B} \hookrightarrow
\mathcal{B}_1$. Then, for $\alpha_i>1$, $i=0,1$, the embedding
\begin{eqnarray}
&\{\,\eta \in L^{\alpha_0}(0,T;\mathcal{B}_0): \frac{\partial \eta}{\partial t}
\in L^{\alpha_1}(0,T;\mathcal{B}_1)\,\} \hookrightarrow L^{\alpha_0}(0,T;\mathcal{B})
\label{compact1}
\end{eqnarray}
is compact.

Throughout we will assume that
(\ref{inidata}) hold,
so that (\ref{additional-1}) and (\ref{wcomp1},b) hold.
We note for future reference that (\ref{eqCtt}) and
(\ref{additional-1}) yield that, for
$\hat \varphi \in L^2_M(\Omega \times D)$,
\begin{align}
\label{eqCttbd}
\int_{\Omega} |\Ctt_i( M\,\hat \varphi)|^2\,\dx & =
\int_{\Omega} %\sum_{i=1}^d \, \sum_{j=1}^d
\left|
\int_{D} M\,\hat \varphi \,U_i'\,\qt_{i}\,\qt_{i}^{\rm T} \dq \right|^2 \dx
\nonumber
\\
%\bet
&\leq
\left(\int_{D} M\,(U_i')^2 \,|\qt_i|^4 \,\dq\right)
\left(\int_{\Omega \times D} M\,|\hat \varphi|^2 \dq \dx\right)
\nonumber \\
%\bet
& \leq C
\left(\int_{\Omega \times D} M\,|\hat \varphi|^2 \dq \dx\right),
%\nonumber
\qquad i=1, \dots,  K,
\end{align}
where $C$ is a positive constant.

We establish a simple integration-by-parts formula.

\begin{lemma} Let $\hat\varphi \in H^1_M(D)$ and suppose that $B \in \mathbb{R}^{d\times d}$ is a square
matrix such that $\mathfrak{tr}(B)=0$; then,
\begin{equation}\label{intbyparts}
\int_D M\, \sum_{i=1}^K (B\qt_i) \cdot \nabqi \hat\varphi \dd \qt = \int_D M\,\hat\varphi \sum_{i=1}^K\qt_i \qt_i^{\rm T} U_i'(\textstyle{\frac{1}{2}|\qt_i|^2}) : B \dq.
\end{equation}
\end{lemma}
\begin{proof}
By Theorem C.1 in Appendix C,
%of Barrett \& S\"uli \cite{BS2010},
the set  $C^\infty(\overline D)$ is dense in $H^1_M(D)$; hence, for any $\hat\varphi \in H^1_M(D)$ there exists a sequence $\{\hat\varphi_n\}_{n \geq 0} \subset C^\infty(\overline D)$ converging to $\hat\varphi$ in $H^1_M(D)$. As $M \in C^2(\overline{D})$ and vanishes on $\partial D$, the same is true of each of the functions $M\hat\varphi_n$, $n \geq 1$. On replacing $\hat\varphi$ by $\hat\varphi_n$ on both sides of
\eqref{intbyparts}, the resulting identity is easily verified by using the classical divergence theorem
for smooth functions, noting \eqref{eqM}, that $M\hat\varphi_n$ vanishes on $\partial D$ and that $\mathfrak{tr}(B)=0$.
Then, \eqref{intbyparts} itself follows by
letting $n \rightarrow \infty$, recalling the definition of the norm in $H^1_M(D)$ and hypothesis \eqref{additional-1}.
\end{proof}

We now formulate our discrete-in-time approximation of problem
(P$_{\epsilon,L})$ for fixed parameters $\epsilon \in (0,1]$
and $L > 1$. For any $T>0$ and $N \geq 1$, let
$N \,\Delta t=T$ and $t_n = n \, \Delta t$, $n=0, \dots,  N$.
To prove existence of a solution under minimal
smoothness requirements on the initial datum $\ut_0$ (recall (\ref{inidata})),
we introduce $\ut^0 = \ut^0(\Delta t) \in \Vt$  such that
\begin{alignat}{2}
\int_{\Omega} \left[ \ut^0 \cdot \vt + \Delta t\,
\nabxtt \ut^0 : \nabxtt \vt \right] \dx
&=   \int_{\Omega} \ut_0 \cdot \vt \dx \qquad
&&\forall \vt \in \Vt;
\label{proju0}
\end{alignat}
and so
\begin{equation}
\int_{\Omega} [\,|\ut^0|^2 + \Delta t \,|\unabtt^0|^2 \,]\dx %+
\leq
\int_{\Omega} |\ut_0|^2 \dx
\leq C.
\label{idatabd}
\end{equation}
In addition, we have
that $\ut^0$ converges to $\ut_0$ weakly in $\Ht$ in the limit
of $\Delta t \rightarrow 0_+$.
For $p\in [1,\infty)$, let
\begin{align}
\hat Z_p &:= \{ \hat \varphi \in L^p_M(\Omega \times D) :
\hat \varphi
\geq 0 \mbox{ a.e.\ on } \Omega \times D
\mbox{ and }
\nonumber \\ & \hspace{1.4in}
\int_D M(\qt)\,\hat \varphi(\xt,\qt) \dq \leq 1 \mbox{ for a.e. } \xt \in \Omega
\}.
\label{hatZ}
\end{align}

Analogously to defining  $\ut^0$ for a given initial velocity field
$\ut_0$, we shall assign a certain `smoothed' initial
datum, $\hat\psi^0 = \hat\psi^0(\Delta t)$, to the initial datum $\hat\psi_0$. The definition of $\hat\psi^0$
is delicate; it will be given in Section \ref{sec:passage.to.limit}. All we need to know for now is
that there exists a $\hat\psi^0$, independent of the cut-off parameter $L$, such that:
\begin{equation}\label{inidata-1}
\!~\!\hat\psi^0 \in \hat{Z}_1;\;\; \!{\small \left\{\begin{array}{rr}\mathcal{F}(\hat\psi^0)\! \in \!L^1_M(\Omega \times D);\\
                                                    \sqrt{\hat\psi^0} \in H^1_M(\Omega \times D);
                                    \end{array}\right.}
\;\;\!\!
\int_{\Omega \times D}\!\!\! M\, \mathcal{F}(\hat\psi^0)\dq \dx \leq \!\int_{\Omega \times D}\!\!\! M\, \mathcal{F}(\hat\psi_0)\dq \dx.\!\!\!\!\!\!
\end{equation}
The proofs of these properties will be given in Lemma \ref{psi0properties} in Section \ref{sec:passage.to.limit}.
It follows from (\ref{inidata-1}) and (\ref{betaLa}) that
$\beta^L(\hat \psi^0) \in \hat Z_2$; in fact, $\beta^L(\hat \psi^0) \in L^\infty(\Omega \times D) \cap H^1_M(\Omega \times D)$.

Our discrete-in-time  approximation of (P$_{\epsilon,L}$) is then defined as follows.

{\boldmath $({\rm P}_{\varepsilon,L}^{\Delta t})$}
Let $\utae^0 := \ut^0 \in \Vt$ and $\hpsiae^0 := \beta^L(\hat \psi^0) \in \hat Z_2$.
Then, for $n =1, \dots,  N$, given
$(\utae^{n-1},\hpsiae^{n-1}) \in \Vt \times \hat Z_2$,
find $(\utae^n,\hpsiae^n) \in
\Vt \times (\hat X \cap \hat Z_2)
$ such that
\begin{subequations}
\begin{align}
&\int_{\Omega}
\left[\frac{\utae^{n}-\utae^{n-1}}{\Delta t}
+ (\utae^{n-1} \cdot \nabx) \utae^{n} \right]\,\cdot\, \wt
\dx +
\nu\, \int_{\Omega}
\nabxtt \utae^n
: \wnabtt \dx
\nonumber
\\
&\hspace{1cm} = \langle \ft^n , \wt \rangle_{H^1_0(\Omega)}
- k\,\sum_{i=1}^K \int_{\Omega} \Ctt_i(M\,\hpsiae^n): \nabxtt
\wt \dx
\qquad \forall \wt \in \Vt,
\label{Gequn}
\end{align}
\begin{align}
&\int_{\Omega \times D} M\,\frac{\hpsiae^n
- \hpsiae^{n-1}}
{\Delta t}\,\hat \varphi \dq \dx
\nonumber
\\
\bet
&\hspace{0.1cm} + \int_{\Omega \times D}
M\,\sum_{i=1}^K  \left[\, \frac{1}{2\, \lambda}\,
\sum_{j=1}^K A_{ij}\,\nabqj \hpsiae^n
-[\,\sigtt(
\utae^n) \,\qt_i\,]\,\beta^L(\hpsiae^{n})\right]\,\cdot\, \nabqi
\hat \varphi \dq \dx\,
\nonumber \\
\bet
&\hspace{0.1cm} + \int_{\Omega \times D} M
\left[\epsilon\,\nabx \hpsiae^n
- \utae^{n-1}\,\hpsiae^n\right]\,\cdot\, \nabx
\hat \varphi\dq \dx=0
\qquad \forall \hat \varphi \in
\hat X;
\label{psiG}
\end{align}
\end{subequations}
where, for  $t \in [t_{n-1}, t_n)$, and $n=1, \dots,  N$,
\begin{align}
\ft^{\Delta t, +}(\cdot,t) =
\ft^n(\cdot) := \frac{1}{\Delta t}\,\int_{t_{n-1}}^{t_n}
\ft(\cdot,t) \dt \in (\Ht^1_0(\Omega))' \subset \Vt'.
\label{fn}
\end{align}
It follows from (\ref{inidata}) and (\ref{fn}) that
\begin{align}
\ft^{\Delta t, +} \rightarrow \ft \quad \mbox{strongly in } L^{2}(0,T;(\Ht^1_0(\Omega))')
\mbox { as } \Delta t \rightarrow 0_{+}.
\label{fncon}
\end{align}
Note that as the test function $\wt$ in \eqref{Gequn} is chosen to be divergence-free,
the term containing the density $\rho$ in the definition of $\tautt$ (cf. \eqref{eqtt1})
is eliminated from \eqref{Gequn}.

In order to prove the
existence of a solution to (P$_{\epsilon,L}^{\Delta t}$), we
require the following convex regularization
$\mathcal{F}_{\delta}^L \in
C^{2,1}({\mathbb R})$ of $\mathcal{F}$ defined, for any $\delta \in (0,1)$ and
$L>1$, by
\begin{align}
 &\mathcal{F}_{\delta}^L(s) := \left\{
 \begin{array}{ll}
 \textstyle\frac{s^2 - \delta^2}{2\,\delta}
 + s\,(\log \delta - 1) + 1
 \quad & \mbox{for $s \le \delta$}, \\
\mathcal{F}(s)\ \equiv
s\,(\log s - 1) + 1 & \mbox{for $\delta \le s \le L$}, \\
  \textstyle\frac{s^2 - L^2}{2\,L}
 + s\,(\log L - 1) + 1
 & \mbox{for $L \le s$}.
 \end{array} \right. \label{GLd}
\end{align}
Hence,
\begin{subequations}
\begin{align}
\quad &[\mathcal{F}_{\delta}^{L}]'(s) = \left\{
 \begin{array}{ll}
 \textstyle \frac{s}{\delta} + \log \delta - 1
 \quad & \mbox{for $s \le \delta$}, \\
 \log s & \mbox{for $\delta \le s \le L$}, \\
 \textstyle \frac{s}{L} + \log L - 1
& \mbox{for $L \le s$},
 \end{array} \right. \label{GLdp}\\
\quad
 &[\mathcal{F}_{\delta}^{L}]''(s) = \left\{
 \begin{array}{ll}
 {\delta}^{-1} \quad & \mbox{for $s \le \delta$}, \\
 s^{-1} & \mbox{for $\delta \le s \le L$}, \\
 L^{-1} & \mbox{for $L \le s$}. \,
\end{array} \right. \label{Gdlpp}
\end{align}
\end{subequations}
We note that
\begin{align}
\mathcal{F}^L_\delta(s) \geq \left\{
\begin{array}{ll}
\frac{s^2}{2\,\delta} &\quad \mbox{for $s \leq 0$},
\\
\frac{s^2}{4\,L} - C(L)&\quad \mbox{for $s \geq 0$};
\end{array}
\right.
\label{cFbelow}
\end{align}
and that
$[\mathcal{F}_{\delta}^{L}]''(s)$
is bounded below by $1/L$ for all $s \in \mathbb{R}$.
Finally, we set
\begin{align}
\beta^L_\delta(s) := ([\mathcal{F}_{\delta}^{L}]'')^{-1}(s)
= \max \{\beta^L(s),\delta\},
\label{betaLd}
\end{align}
and observe that $\beta^L_\delta(s)$
is bounded above by $L$ and bounded below by $\delta$ for all $s \in \mathbb{R}$. Note also that both $\beta^L$ and $\beta^L_\delta$ are Lipschitz continuous on $\mathbb{R}$, with Lipschitz constants equal to $1$.

\subsection{\boldmath Existence of a solution to  $({\rm P}_{\varepsilon,L}^{\Delta t})$}
\label{sec:existence-cut-off.1}

It is convenient to rewrite (\ref{Gequn}) as
\begin{equation}
b(\utae^n,\wt) =
\ell_b(\hpsiae^n)(\wt)
\qquad \forall \wt \in \Vt;
\label{bLM}
\end{equation}
where, for all $\wt_i \in \Ht^{1}_{0}(\Omega)$, $i=1,2$,
\begin{subequations}
\begin{align}
\!\!\!\!\!\!b(\wt_1,\wt_2) &\!:=\!\!
\int_{\Omega}\! \left[ \wt_1
+ \Delta t \, (\utae^{n-1} \cdot \nabx) \wt_{1} \right]\!\cdot
\wt_2 \dx + \Delta t\, \nu \!\int_{\Omega}
\nabxtt \wt_1
:\nabxtt \wt_2 \dx,\!\!
\label{bgen}
\end{align}
and, for all $\wt \in \Ht^1_0(\Omega)$ and
$\hat \varphi \in L^2_M(\Omega \times D)$,
\begin{align}
\!\!\!\!\ell_b(\hat \varphi)(\wt) &:=
\Delta t \,\langle \ft^n, \wt
\rangle_{H^1_0(\Omega)} +
\displaystyle\int_{\Omega}
\left[
\utae^{n-1} \cdot \wt  - \Delta t \,k\,\sum_{i=1}^K
\Ctt_i(M\,\hat \varphi) : \nabxtt
\wt \right]\! \dx.
\label{lbgen}
\end{align}
\end{subequations}
We note that, for all $\vt \in \Vt$ and all $\wt_1, \wt_2 \in
\Ht^{1}(\Omega)$, we have that
\begin{align}
\int_{\Omega} \left[ (\vt \cdot \nabx) \wt_1 \right]\,\cdot\,
\wt_2 \dx
\label{tripid}
= - \int_{\Omega} \left[ (\vt \cdot \nabx) \wt_2
\right]\,\cdot\, \wt_1 \dx
\end{align}
and hence $b(\cdot,\cdot)$ is a continuous nonsymmetric coercive bilinear
functional on $\Ht^1_0(\Omega) \times \Ht^1_0(\Omega)$.
In addition, thanks to \eqref{eqCttbd}, $\ell_b(\hat \varphi)(\cdot)$ is a continuous linear
functional on $\Ht^1_0(\Omega)$ for any $\hat\varphi \in L^2_M(\Omega\times D)$.

For $r>d$, let
\begin{align}
\Yt^r :=
\left\{\vt \in \Lt^{r}(\Omega) :
\int_{\Omega} \vt\,\cdot\,\nabx w \dx = 0 \quad \forall w
\in W^{1,\frac{r}{r-1}}(\Omega)\right\}.
\label{Ytr}
\end{align}
It is also convenient to rewrite (\ref{psiG}) as
\begin{align}
a(\hpsiae^n,\hat \varphi) = \lae(\utae^n,\beta^L(\hpsiae^n))
(\hat \varphi) \qquad \forall \hat \varphi \in \hat X,
\label{genLM}
\end{align}
where, for all $\hat \varphi_1,\,\hat \varphi_2 \in \hat X$,
\begin{subequations}
\begin{align}
a(\hat \varphi_1,\hat \varphi_2) &:= \int_{\Omega \times D} M\,\biggl(
\hat \varphi_1\,\hat \varphi_2 + \Delta t \left[\epsilon\,\nabx
\hat \varphi_1 - \utae^{n-1}\,\hat \varphi_1\right]\,\cdot\, \nabx
\hat \varphi_2 \label{agen}
\nonumber\\
& \hspace{1in} +\, \frac{\Delta t}{2\,\lambda} \,
\sum_{i=1}^K \sum_{j=1}^K A_{ij}\,
\nabqj
\hat \varphi_1 \, \cdot\, \nabqi
\hat \varphi_2 \biggr) \dq \dx,
\end{align}
and, for all $\vt \in \Ht^1(\Omega)$, $\hat \eta \in L^\infty(\Omega\times D)$
and $\hat \varphi \in \hat X$,
\begin{align}
\lae(\vt,\hat \eta)(\hat \varphi) &:=
\int_{\Omega \times D}
M \left[\hpsiae^{n-1}
\,\hat \varphi
+ \Delta t\,\sum_{i=1}^K [\,\sigtt(\vt)
\,\qt_i\,]\,\hat \eta\, \cdot\, \nabqi
\hat \varphi \right]\!\dq \dx.
\label{lgen}
\end{align}
\end{subequations}
It follows from
(\ref{Ytr}) and (\ref{embed}) that, for $r>d$,
\begin{eqnarray}
\int_{\Omega \times D} M\,
\vt\,\hat \varphi \,\cdot \nabx \hat \varphi
\dq \dx = 0 \qquad \forall \vt \in \Yt^r, \quad \forall \hat \varphi \in \hat
X.
\label{Ytra}
\end{eqnarray}
Hence
$a(\cdot,\cdot)$ is a
continuous
coercive bilinear functional on $ \hat X \times \hat X$.
In addition, we have that,
for all $\vt \in \Ht^1(\Omega)$, $\hat \eta \in L^\infty(\Omega \times D)$
and $\hat \varphi \in \hat X$,
\begin{align}
|\lae(\vt,\hat \eta)(\hat \varphi)| &\leq
\|\hpsiae^{n-1}\|_{L^2_M(\Omega \times D)}
\,\|\hat \varphi\|_{L^2_M(\Omega \times D)}
\nonumber \\
& \hspace{-0.5in}
+ \Delta t\,\left( \int_{D} M\,|\qt|^2 \dq \right)^{\frac{1}{2}}
\|\hat \eta\|_{L^\infty(\Omega \times D)}
\,\|\nabxtt \vt\|_{L^2(\Omega)}
\,\|\nabq \hat \varphi\|_{L^2_M(\Omega \times D)}.
\label{lgenbd}
\end{align}
Therefore, by noting that $\hpsiae^{n-1} \in \hat Z_2$
and recalling (\ref{MN}), it follows that $\ell_a(\vt,\hat \eta)(\cdot)$ is a continuous linear functional on $\hat X$
for all $\vt \in \Ht^1(\Omega)$ and $\hat \eta \in L^\infty(\Omega \times D)$.

In order to prove existence of a solution to (\ref{Gequn},b),
i.e. (\ref{bLM}) and (\ref{genLM}),
we consider a regularized system for a given $\delta \in (0,1)$:

Find $(\utaed^{n},\hpsiaed^n) \in \Vt \times \hat X$ such that
\begin{subequations}
\begin{alignat}{2}
b(\utaed^n,\wt) &=
\ell_b(\hpsiaed^n)(\wt)
\qquad &&\forall \wt \in \Vt,
\label{bLMd} \\
a(\hpsiaed^n,\hat \varphi) &= \lae(\utaed^n,\beta^L_\delta(\hpsiaed^n))
(\hat \varphi) \qquad &&\forall \hat \varphi \in \hat X.
\label{genLMd}
\end{alignat}
\end{subequations}
The existence of a solution to (\ref{bLMd},b) will be proved by
using a fixed-point argument. Given
$\hat \psi \in L^2_M(\Omega \times D)$,
let $(\ut^{\star}, \hat \psi^\star ) \in
\Vt \times \hat X$ be such that
\begin{subequations}
\begin{alignat}{2}
\qquad b(\ut^{\star},\wt) &=
\ell_b(\hat \psi)(\wt)
\qquad &&\forall \wt \in \Vt,
\label{fix4} \\
\qquad a(\hat \psi^{\star},\hat \varphi) &=
\lae(\ut^\star,\beta^L_\delta(\hat \psi))(\hat \varphi) \qquad
&&\forall \hat \varphi \in \hat X. \label{fix3}
\end{alignat}
\end{subequations}
The Lax--Milgram theorem yields the existence of a unique solution to (\ref{fix4},b),
and so the overall procedure (\ref{fix4},b) is well defined.

\begin{lemma}
\label{fixlem} Let $G: L^2_M(\Omega \times D) \rightarrow \hat X
\subset L^2_M(\Omega \times D)$ denote the
nonlinear map that takes the function $\hat{\psi}$ to $\hat \psi^{\star} = G(\hat \psi)$
{\em via} the procedure
{\rm (\ref{fix4},b)}. Then $G$ has a fixed point. Hence there exists a solution
$(\utaed^n,\hpsiaed^n) \in \Vt \times \hat X$ to {\rm (\ref{bLMd},b)}.
\end{lemma}

\begin{proof}
Clearly, a fixed point of $G$ yields a
solution of (\ref{bLMd},b).
In order to show that $G$ has a fixed point, we apply
Schauder's fixed-point theorem; that is, we need to show that:
(i)~$G:
L^2_M(\Omega \times D) \rightarrow
L^2_M(\Omega \times D)$ is continuous; (ii)~$G$ is compact; and
(iii)~there exists a $C_{\star} \in {\mathbb R}_{>0}$ such that
\begin{eqnarray}
\|\hat{\psi}\|_{L^{2}_M(\Omega\times D)} \leq C_{\star}
\label{fixbound}
\end{eqnarray}
for every $\hat{\psi} \in L^2_M(\Omega \times D)$ and $\kappa \in (0,1]$
satisfying $\hat{\psi} = \kappa\, G(\hat{\psi})$.

Let $\{\hat{\psi}^{(p)}\}_{p \geq 0}$ be such that
\begin{eqnarray}
\hat{\psi}^{(p)}
\rightarrow \hat{\psi} \qquad \mbox{strongly in }
L^{2}_M(\Omega\times D)\qquad \mbox{as } p \rightarrow \infty.
\label{Gcont1}
\end{eqnarray}
It follows immediately from (\ref{betaLd}) and (\ref{eqCttbd}) that
\begin{subequations}
\begin{alignat}{1}
M^{\frac{1}{2}}\,\beta^L_\delta(\hat{\psi}^{(p)})
\rightarrow M^{\frac{1}{2}}\,
\beta^L_\delta(\hat{\psi}) \quad
\mbox{strongly in }
L^r(\Omega\times D)\quad
\mbox{as } p \rightarrow \infty,
\label{betacon}
\end{alignat}
for all $r \in [1,\infty)$ and, for $i=1, \dots, K$,
\begin{alignat}{1}
&
\Ctt_i(M\,\hat{\psi}^{(p)}) \rightarrow \Ctt_i(M\,\hat{\psi}) \quad
\mbox{strongly in } L^{2}(\Omega)\quad
\mbox{as } p \rightarrow \infty.
\label{Cttcon}
\end{alignat}
\end{subequations}
In order to prove (i) above, we need to show that
\begin{eqnarray}
\hat{\eta}^{(p)}:=G(\hat{\psi}^{(p)}) \rightarrow G(\hat{\psi})
\quad \mbox{strongly in } L^2_M(\Omega\times D)\quad \mbox{as } p
\rightarrow \infty. \label{Gcont2}
\end{eqnarray}
We have from the definition of $G$ (see (\ref{fix4},b)) that, for all $p \geq 0$,
\begin{subequations}
\begin{align}
a(\hat \eta^{(p)}, \hat \varphi)&= \ell_a(\vt ^{(p)},
\beta^L_\delta(\hat \psi^{(p)}))(\hat \varphi)
\qquad \forall \hat \varphi \in \hat X,
\label{Gcont5}
\end{align}
where $\vt^{(p)} \in \Vt$ satisfies
\begin{align}
b(\vt^{(p)},\wt) &= \ell_b(\hat \psi^{(p)})(\wt)
\qquad \forall \wt \in \Vt.
\label{Gcont5a}
\end{align}
\end{subequations}

Choosing $\hat \varphi = \hat \eta^{(p)}$ in (\ref{Gcont5})
yields, on noting the simple identity
\begin{equation}
2\,(s_1-s_2)\,s_1 = s_1^2 + (s_1 -s_2)^2 -s_2^2 \qquad \forall
s_1, s_2 \in {\mathbb R}, \label{simpid}
\end{equation}
(\ref{A}), (\ref{Ytra}) and (\ref{betaLd})
that, for all $p \geq 0$,
\begin{align}
&\int_{\Omega \times D}
M\,\left [ |\hat \eta^{(p)}|^2
+ | \hat \eta^{(p)} - \hpsiae^{n-1}|^2
+ \frac{a_0\,\Delta t}{2\,\lambda}\,|\nabq \hat \eta^{(p)}|^2
+ 2\,\epsilon \, \Delta t\,|\nabx \hat \eta^{(p)}|^2 \right] \dq \dx
\label{Gcont6}
\nonumber
\\
& \qquad \leq \int_{\Omega \times D} M\,
|\hpsiae^{n-1}|^2 \dq \dx + C(L,\lambda,a_0^{-1})\,\Delta t\int_{\Omega}
|\nabxtt \vt^{(p)}|^2 \dx.
\end{align}
Choosing $\wt \equiv \vt^{(p)}$ in (\ref{Gcont5a}), and noting
(\ref{simpid}), (\ref{tripid}),
(\ref{eqCttbd}) %, (\ref{eqvn1}), a Poincar\'e inequality
and (\ref{Gcont1}) yields, for
all $p\geq 0$, that
\begin{align}
&\hspace{-3mm}\int_{\Omega} \left[ |\vt^{(p)}|^2 +
|\vt^{(p)}-\utae^{n-1}|^2  \right] \dx +
\Delta t \,\nu \, \int_{\Omega} |\nabxtt \vt^{(p)}|^2 \dx
\label{Gcont4}
\nonumber\\
&
\leq \int_\Omega |\utae^{n-1}|^2 \dx + C\,\Delta t \,
\|\ft^n\|_{(H^1_0(\Omega))'}^2
+ C\, \Delta t \, \int_{\Omega \times D}
M\,|\hat{\psi}^{(p)}|^2 \dq \dx
\leq C.
\end{align}
Combining (\ref{Gcont6}) and (\ref{Gcont4}),
we have for all $p \geq 0$ that
\begin{eqnarray}
\|\hat \eta^{(p)}\|_{\hat X}
+ \|\vt^{(p)}\|_{H^1(\Omega)} \leq C(L, (\Delta t)^{-1})\,.
\label{Gcont7}
\end{eqnarray}
It follows from
(\ref{Gcont7}), (\ref{embed}) and the compactness of the embedding (\ref{wcomp2}) that there exists a subsequence
$\{(\hat{\eta}^{(p_k)},\vt^{(p_k)})\}_{p_k \geq 0}$ and functions $\hat{\eta}\in \hat X$ and $\vt \in \Vt$ such that, as $p_k \rightarrow \infty$,
\begin{subequations}
\begin{alignat}{2}
\hat{\eta}^{(p_k)}
&\rightarrow
\hat{\eta}
\qquad
&&\mbox{weakly in }
L^{s}(\Omega ;L^2_M(D)),
\label{Gcont8a} \\
\bet
M^{\frac{1}{2}}\,\nabx \hat{\eta}^{(p_k)}
&\rightarrow M^{\frac{1}{2}}\,\nabx \hat{\eta} \qquad
&&\mbox{weakly in }
\Lt^{2}(\Omega \times D),
\label{Gcont8bx} \\
\bet
M^{\frac{1}{2}}\,\nabq \hat{\eta}^{(p_k)}
&\rightarrow M^{\frac{1}{2}}\,\nabq \hat{\eta}
\qquad
&&\mbox{weakly in }
\Lt^{2}(\Omega \times D),
\label{Gcont8b} \\
\bet
\hat{\eta}^{(p_k)}
&\rightarrow
\hat{\eta}
\qquad
&&\mbox{strongly in }
L^{2}_M(\Omega\times D),
\label{Gcont8as} \\
\bet
\vt^{(p_k)} &\rightarrow \vt \qquad
&&\mbox{weakly in }
\Ht^{1}(\Omega);
\label{Gcont8c}
\end{alignat}
\end{subequations}
where $s \in [1,\infty)$ if $d=2$ or $s \in [1,6]$ if $d=3$.
We deduce from (\ref{Gcont5a}), (\ref{bgen},b), (\ref{Gcont8c})
and (\ref{Cttcon})
that the functions $\vt \in \Vt$ and $\hat{\psi} \in \hat X$ satisfy
\begin{eqnarray}\qquad
b(\vt,\wt) &=\ell_b(\hat \psi)(\wt) \qquad \forall \wt \in \Vt.
\label{Gcont9}
\end{eqnarray}
It follows from (\ref{Gcont5}), (\ref{agen},b), (\ref{Gcont8a}--e) and
(\ref{betacon}) that $\hat \eta,\,\hat \psi \in \hat X$
and $\vt \in \Vt$ satisfy
\begin{align}
a(\hat{\eta},\hat \varphi)
&= \lae(\vt,\beta^L_\delta(\hat \psi))(\hat \varphi)
\qquad \forall \hat \varphi \in C^\infty(\overline{\Omega \times D}).
\label{Gcont11}
\end{align}
Then, noting that $a(\cdot,\cdot)$ is a bounded bilinear functional on $\hat X \times \hat X$,
that $\lae(\vt,\beta^L_\delta(\hat \psi))(\cdot)$ is a bounded linear functional on $\hat X$, and recalling (\ref{cal K}), we deduce that (\ref{Gcont11}) holds
for all $\hat \varphi \in \hat X$.
Combining this $\hat X$ version of (\ref{Gcont11}) and (\ref{Gcont9}), we have that
$\hat \eta = G(\hat \psi)\in \hat X$. Therefore
the whole sequence
\[\hat{\eta}^{(p)} \equiv G(\hat{\psi}^{(p)})
\rightarrow G(\hat{\psi})\qquad \mbox{strongly in $L^2_M(\Omega\times D)$},
\]
as $p \rightarrow \infty$, and so (i) holds.

Since the embedding $\hat X \hookrightarrow L^{2}_M(\Omega \times D)$
is compact, it follows that (ii) holds. It therefore remains to show that (iii) holds.

As regards (iii), $\hat{\psi} =
\kappa \, G(\hat{\psi})$ implies that $\{\vt,\hat{\psi}\} \in
\Vt \times \hat X$ satisfies
\begin{subequations}
\begin{alignat}{2}
b(\vt,\wt) &=
\ell_b(\hat \psi)(\wt)
\quad &&\forall \wt \in \Vt,
\label{fix4sig}
\\
a(\hat{\psi},\hat \varphi)&=\kappa\,\lae(\vt,\beta^L_\delta(\hat \psi))(\hat \varphi)\qquad
&&\forall \hat \varphi \in \hat X. \label{fix3sig}
\end{alignat}
\end{subequations}
Choosing $\wt \equiv \vt$ in (\ref{fix4sig})
yields, similarly to (\ref{Gcont4}), that
\begin{eqnarray}
\label{Gequnbhat}
&&\hspace{-2.5cm}\frac{1}{2}\,\displaystyle
\int_{\Omega} \left[ \,|\vt|^2 +
|\vt-\utae^{n-1}|^2 - |\utae^{n-1}|^2 \,\right]
\dx + \Delta t\, \nu \,
\int_{\Omega} |\nabxtt \vt|^2 \dx
\nonumber
 \\
%\bet
&&\hspace{-0.5cm}
= \Delta t \left[
\langle \ft^n, \vt \rangle_{H^1_0(\Omega)}
- k\,\sum_{i=1}^K
\int_{\Omega}
 \Ctt_i(M\,\hat{\psi}): \nabxtt \vt \dx \right].
\end{eqnarray}
Choosing $\hat \varphi = [\mathcal{F}_\delta^L]'(\hat{\psi})$ in (\ref{fix3sig})
and noting the convexity of $\mathcal{F}_\delta^L$, (\ref{betaLd})
and that $\vt$ is divergence-free, yield
\begin{align}
&
\int_{\Omega \times D} M\,\left[ \,\mathcal{F}_\delta^L (\hat{\psi})
- \mathcal{F}_\delta^L (\kappa \,\hpsiae^{n-1})
+ \Delta t \left[\varepsilon \nabx \hat\psi - \undertilde{u}^{n-1}_{\epsilon, L}\hat\psi \right] \cdot \nabx ([\mathcal{ F}_{\delta}^L]'(\hat{\psi}))
\right]  \dq \dx
\label{acorstab1}
\nonumber
\\
&
\quad + \frac{\Delta t}{2\,\lambda}\,
\sum_{i=1}^K \sum_{j=1}^K A_{ij}
\int_{\Omega \times D} M\,
\,\nabqj
\hat{\psi} \cdot \nabqi ([\mathcal{F}_\delta^L]'(\hat{\psi}))
\dq \dx \nonumber \\
&
\hspace{2cm}
\leq \kappa\,\Delta t \,\sum_{i=1}^K
\int_{\Omega \times D}
M\,\sigtt(\vt) \,\qt_i \cdot
\nabqi \hat{\psi}
\dq \dx
\nonumber
\\
&\hspace{2cm}
= \kappa\,\Delta t \,\sum_{i=1}^K
\int_{\Omega}
\Ctt_i(M\,\hat \psi) : \sigtt(\vt)  \dx,
\end{align}
where in the transition to the final inequality we applied \eqref{intbyparts} with
$B := \sigtt(\vt)$ (on account of it being independent of the variable $\qt$), together
with the fact that $\mathfrak{tr}(\sigtt(\vt))  = \nabx \cdot ~\vt = 0$, and recalled
\eqref{eqCtt}. Next, on noting \eqref{betaLd} and that $\undertilde{u}_{\epsilon, L}^{n-1}
\in \Vt$, it follows that
\begin{eqnarray}\label{convGdL}
%&&\hspace{-8mm}
\int_{\Omega \times D}\!\! M\, \ut^{n-1}_{\epsilon, L}\hat\psi \cdot \nabx ([\mathcal{F}^L_\delta]'(\hat\psi)) \dq \dx &=& \int_{\Omega \times D}\!\! M\, \ut^{n-1}_{\epsilon, L}\, \frac{\hat \psi}{\beta^L_\delta(\hat\psi)} \cdot \nabx \hat\psi \dq \dx\nonumber\\
%&&\hspace{38mm}
&=& \int_{\Omega \times D}\!\! M\, \ut^{n-1}_{\epsilon, L} \cdot \nabx (G^L_\delta (\hat\psi)) \dq \dx = 0,
\nonumber\\
\end{eqnarray}
where $G^L_\delta \in C^{0,1}(\mathbb{R})$ is defined by
\begin{equation}\label{Gdl}
G^L_\delta(s) := \left\{\begin{array}{ll}
\frac{1}{2\delta} s^2 + \frac{\delta - L}{2} & \qquad\mbox{if $s \leq \delta$}, \\
s - \frac{L}{2}                              & \qquad\mbox{if $s \in [\delta, L]$},\\
\frac{1}{2L} s^2                             & \qquad\mbox{if $s \geq L$};
\end{array}
\right.
\end{equation}
and so $[G^L_\delta]'(s) = s/\beta^L_\delta(s)$.  Combining (\ref{Gequnbhat}) and (\ref{acorstab1}), and noting \eqref{convGdL}, \eqref{Gdlpp} and (\ref{A})
%(\ref{eqvn1}) and a Poincar\'{e} inequality
yields that
\begin{align}
\label{Ek}
&
\frac{\kappa}{2}\,\displaystyle
\int_{\Omega} \left[ \,|\vt|^2 +
|\vt-\utae^{n-1}|^2 \,\right]
\dx + \kappa\,\Delta t\, \nu \,
\int_{\Omega} |\nabxtt \vt|^2 \dx
+k\,
\int_{\Omega \times D} M\,\mathcal{F}_\delta^L (\hat{\psi})
\dq \dx \nonumber
\\
&\qquad +k\,L^{-1}\,\Delta t\,
\int_{\Omega \times D} M\,
\left[\,\epsilon
\,|\nabx \hat{\psi}|^2 + \frac{a_0}{2\,\lambda} \,|\nabq
\hat{\psi}|^2 \right] \dq \dx \nonumber \\
&\hspace{1cm}\leq
\kappa\,\Delta t \,\langle \ft^n, \vt \rangle_{H^1_0(\Omega)} +
\frac{\kappa}{2}\,\int_{\Omega} |\utae^{n-1}|^2 \dx
+ k\,
\int_{\Omega \times D} M\,\mathcal{F}_\delta^L (\kappa \,\hpsiae^{n-1})
\dq \dx
\nonumber
\\
&\hspace{1cm}\leq
\frac{\kappa \,\Delta t\, \nu}{2} \,
\int_{\Omega} |\nabxtt \vt|^2 \dx
+ \frac{\kappa \,\Delta t}
{2\nu}\,
\|\ft^n\|_{(H^{1}_0(\Omega))'}^2 \nonumber\\
& \hspace{2cm} +
\frac{\kappa}{2}\,\int_{\Omega} |\utae^{n-1}|^2 \dx
+ k\,
\int_{\Omega \times D} M\,\mathcal{F}_\delta^L (\kappa \,\hpsiae^{n-1})
\dq \dx.
\end{align}

It is easy to show that $\mathcal{F}^L_\delta(s)$ is nonnegative for all
$s \in \mathbb{R}$, with  $\mathcal{F}^L_\delta(1)=0$.
Furthermore, for any $\kappa \in (0,1]$,
$\mathcal{F}^L_\delta(\kappa\, s) \leq \mathcal{F}^L_\delta(s)$
if $s<0$ or $1 \leq \kappa\, s$, and also
$\mathcal{F}^L_\delta(\kappa\, s) \leq \mathcal{F}^L_\delta(0) \leq 1$
if $0 \leq \kappa\, s \leq 1$.
Thus we deduce that
\begin{equation}\label{deltaL}
\mathcal{F}_\delta^L(\kappa\, s)
\leq \mathcal{F}_\delta^L(s)+ 1\qquad \forall s \in {\mathbb R},\quad
\forall \kappa \in (0,1].
\end{equation}
Hence, the bounds (\ref{Ek}) and (\ref{deltaL}), on noting (\ref{cFbelow}),
give rise to the desired bound (\ref{fixbound}) with $C_*$ dependent only on
$\delta$, $L$, $k$, $a_0$, $\nu$, $\ft$ and $\hpsiae^{n-1}$.
Therefore (iii) holds, and so $G$ has a fixed point, proving
existence of a solution
to (\ref{bLMd},b).
\qquad\end{proof}

Choosing $\wt \equiv \utaed^n$ in (\ref{bLMd})
and $\hat \varphi \equiv [\mathcal{F}_\delta^L]'(\hpsiaed^n)$ in (\ref{genLMd}),
and combining, then yields, as in
(\ref{Ek}), with $C(L)$ a positive constant, independent of $\delta$ and $\Delta t$,
\begin{align}\label{E1}
&
\frac{1}{2}\,\displaystyle
\int_{\Omega} \left[ \,|\utaed^n|^2 +
|\utaed^n-\utae^{n-1}|^2 \,\right] \dx
+k\, \int_{\Omega \times D} M\,\mathcal{F}_\delta^L (\hpsiaed^n)
\dq \dx
\nonumber
\\
&\quad\quad
+ \Delta t\, \biggl[ \frac{\nu}{2} \,
\int_{\Omega} |\nabxtt \utaed^n|^2 \dx
+k\,L^{-1}\,\epsilon\,
\int_{\Omega \times D} M
\,|\nabx \hpsiaed^n|^2
\dq \dx
\nonumber \\
& \quad \quad + \frac{k\,L^{-1}\,a_0}{2\,\lambda} \,
\int_{\Omega \times D}
M\, |\nabq
\hpsiaed^n|^2
 \dq \dx \biggr] \nonumber
\\
&\quad\quad\quad\leq
\frac{\Delta t}{2\nu} \,\|\ft^n \|_{(H^1_0(\Omega))'}^2 +
\frac{1}{2}\,\int_{\Omega} |\utae^{n-1}|^2 \dx
+ k\,
\int_{\Omega \times D} M\,\mathcal{F}_\delta^L (\hpsiae^{n-1})
\dq \dx
\nonumber \\
& \quad \quad \quad \leq C(L).
\end{align}

We are now ready to pass to the limit $\delta \rightarrow 0_+$, to deduce the existence of
%$(\utae^n, \hpsiae^n)$ that solves (\ref{Gequn},b) for $n=1,\dots, N$; and thereby, the existence of
a solution $\{(\utae^n,\hpsiae^n)\}_{n=1}^N$ to
({\rm P}$^{\Delta t}_{\epsilon,L}$), with $\utae^n \in \Vt$ and $\hpsiae^n \in
\hat X \cap \hat{Z}_2$, $n=1,\dots, N$.

\begin{lemma}
\label{conv}
There exists a subsequence (not indicated) of $\{(\utaed^{n},
\hpsiaed^{n})\}_{\delta >0}$, and functions $\utae^{n} \in \Vt$
and $\hpsiae^{n} \in \hat X \cap \hat Z_2$, $n \in \{1,\dots, N\}$,
such that, as $\delta \rightarrow 0_+$,
\begin{subequations}
\begin{eqnarray}
&\utaed^{n} \rightarrow \utae^{n} \qquad &\mbox{weakly in } \Vt, \label{uwconH1}\\
\bet
&\utaed^{n} \rightarrow \utae^{n}
\qquad &\mbox{strongly in }
\Lt^{r}(\Omega), \label{usconL2}
\end{eqnarray}
\end{subequations}
where $r \in [1,\infty)$ if $d=2$ and $r \in [1,6)$ if $d=3$;
and
\begin{subequations}
\begin{alignat}{2}
\!\!\!\!M^{\frac{1}{2}}\,\hpsiaed^{n} &\rightarrow
M^{\frac{1}{2}}\,
\hpsiae^{n} &&\quad \mbox{weakly in }
L^2(\Omega\times D),~~~ \label{psiwconL2}\\
\bet
M^{\frac{1}{2}}\,\nabq \hpsiaed^{n}
&\rightarrow M^{\frac{1}{2}}\,\nabq \hpsiae^{n}
&&\quad \mbox{weakly in }
\Lt^2(\Omega\times D), \label{psiwconH1}\\
\bet
M^{\frac{1}{2}}\,\nabx \hpsiaed^{n}
&\rightarrow M^{\frac{1}{2}}\,\nabx \hpsiae^{n}
&&\quad \mbox{weakly in }
\Lt^2(\Omega\times D), \label{psiwconH1x}\\
\bet
M^{\frac{1}{2}}\,\hpsiaed^{n} &\rightarrow
M^{\frac{1}{2}}\,\hpsiae^{n}
&&\quad \mbox{strongly in }
L^{2}(\Omega\times D),\label{psisconL2}
\\
\bet
M^{\frac{1}{2}}\,\beta_\delta^L(\hpsiaed^{n}) &\rightarrow
M^{\frac{1}{2}}\,\beta^L(\hpsiae^{n})
&&\quad \mbox{strongly in }
L^s(\Omega\times D),\label{betaLdsconL2}
\end{alignat}
for all $s \in [2,\infty)$ and, for $i=1, \dots,  K$,
\begin{alignat}{1}
&\hspace{-4mm}\Ctt_i(M\,\hpsiaed^{n}) \rightarrow \Ctt_i(M\,\hpsiae^{n})
\qquad \!\!\!\!\!\mbox{ strongly in }
\Ltt^{2}(\Omega).\label{CwconL2}
\end{alignat}
\end{subequations}
Further, $(\utae^n, \hpsiae^n)$ solves (\ref{Gequn},b) for $n=1,\dots, N$. Hence there
exists a solution $\{(\utae^n,\hpsiae^n)\}_{n=1}^N$ to
{\em ({\rm P}$^{\Delta t}_{\epsilon,L}$)}, with $\utae^n \in \Vt$ and $\hpsiae^n \in
\hat X \cap \hat{Z}_2$ for all $n=1,\dots, N$.
\end{lemma}

%~\vspace{-8mm}

\begin{proof}
The weak convergence results (\ref{uwconH1}), (\ref{psiwconL2}) and that $\hpsiae^n \ge 0$ a.e.\ on $\Omega \times D$ follow immediately from the first two bounds on the left-hand side of
(\ref{E1}), on noting (\ref{cFbelow}). The strong convergence
(\ref{usconL2}) for $\utaed^{n}$ follows from (\ref{uwconH1}),
on noting that $\Vt \subset \Ht^{1}_0(\Omega)$ is compactly embedded in $\Lt^r(\Omega)$ for
the stated values of $r$.

It follows immediately from the bound on the fifth term on the left-hand side of (\ref{E1}) that (\ref{psiwconH1})
holds for some limit $\gt \in \Lt^2(\Omega \times D)$, which we need to identify. However, for any $\etat \in
\Ct^{1}_0(\Omega\times D)$, it follows from (\ref{eqM}) and the compact support of $\etat$ on $D$ that
$ [\nabq \cdot (M^{\frac{1}{2}}\,\etat)\,]/M^{\frac{1}{2}} \in L^2(\Omega \times D)$,
and hence the above convergence implies, noting (\ref{psiwconL2}), that
\begin{align}
\hspace{-0.2cm}\int_{\Omega \times D} \gt \cdot
\etat \dq \dx
&\leftarrow
- \int_{\Omega \times D}M^{\frac{1}{2}}\, \hpsiaed^{n}\,
\frac{\nabq \cdot (M^{\frac{1}{2}}\,\etat)}{M^{\frac{1}{2}}}
\dq \dx
\nonumber
\\
\bet
& \hspace{-2.2cm}\rightarrow
- \int_{\Omega \times D} M^{\frac{1}{2}}\,\hpsiae^n\,
\frac{\nabq \cdot (M^{\frac{1}{2}}\,\etat)}{M^{\frac{1}{2}}}
\dq \dx =
- \int_{\Omega \times D} \hpsiae^n\,
\nabq \cdot (M^{\frac{1}{2}}\,\etat)
\dq \dx
\label{derivid}
\end{align}
as $\delta  \rightarrow 0_+$.
Equivalently, on dividing and multiplying by $M^{\frac{1}{2}}$ under the integral sign in the left-most term
in \eqref{derivid}, we have that
\[ \int_{\Omega \times D} M^{-\frac{1}{2}}\gt \cdot M^{\frac{1}{2}}\etat \dq \dx=  -  \int_{\Omega \times D} \hat{\psi}^n_{\varepsilon, L} \, \nabq \cdot (M^{\frac{1}{2}}\etat) \dq \dx\qquad \forall \etat \in \undertilde{C}^1_0(\Omega \times D).
\]
Observe that $\etat \in \undertilde{C}^1_0(\Omega \times D) \mapsto M^{\frac{1}{2}} \etat \in \undertilde{C}^1_0(\Omega \times D)$ is a bijection of $\undertilde{C}^1_0(\Omega \times D)$ onto itself; thus, the equality above is
equivalent to
\[ \int_{\Omega \times D} M^{-\frac{1}{2}}\gt \cdot \chit \dq \dx=  -  \int_{\Omega \times D} \hat{\psi}^n_{\varepsilon, L} \, (\nabq \cdot \chit) \dq \dx\qquad \forall \chit
\in \undertilde{C}^1_0(\Omega \times D).
\]
Since $\Ct^\infty_0(\Omega \times D) \subset \Ct^1_0(\Omega \times D)$, the last identity also holds for all
$\etat \in \Ct^\infty_0(\Omega \times D)$.
As $M^{\frac{1}{2}} \in L^\infty(D)$ and $M^{-\frac{1}{2}} \in L^\infty_{\rm loc}(D)$,
it follows that
$M^{-\frac{1}{2}}\gt
\in \Lt^2_{\rm loc}(\Omega \times D)$
and $\hat{\psi}^n_{\varepsilon, L} \in L^2_{\rm loc}(\Omega \times D)$.
By identification of a locally integrable function with a
distribution we deduce that $M^{-\frac{1}{2}} \gt$ is the distributional gradient of
$\hat{\psi}^n_{\varepsilon, L}$ w.r.t. $\qt$:
\[ M^{-\frac{1}{2}}\gt = \nabq \hat{\psi}^n_{\varepsilon, L} \qquad \mbox{in $\undertilde{\mathcal{D}}'(\Omega \times D)$.}\]
As $M^{-\frac{1}{2}} \gt \in \undertilde{L}^2_{\rm loc}(\Omega \times D)$, whereby also
$\nabq \hat{\psi}^n_{\varepsilon, L} \in \undertilde{L}^2_{\rm loc}(\Omega \times D)$,
it follows that
\[ \gt = M^{\frac{1}{2}} \nabq \hat{\psi}^n_{\varepsilon, L} \in \undertilde{L}^2_{\rm loc}(\Omega \times D).\]
However, the left-hand side belongs to $\undertilde{L}^2(\Omega \times D)$, which then implies that the right-hand side also belongs to $\undertilde{L}^2(\Omega \times D)$. Thus we have shown that
\begin{equation}\label{last}
\gt = M^{\frac{1}{2}} \nabq \hat{\psi}^n_{\varepsilon, L} \in \undertilde{L}^2(\Omega \times D),
\end{equation}
and hence the desired result (\ref{psiwconH1}), as required.
A similar argument proves (\ref{psiwconH1x}) on noting (\ref{psiwconL2}), and the fourth bound in (\ref{E1}).

The strong convergence result (\ref{psisconL2}) for $\hpsiaed^{n}$
follows directly from (\ref{psiwconL2}--c) and (\ref{wcomp2}).
Finally, %the desired results
(\ref{betaLdsconL2},f) follow from (\ref{psisconL2}), (\ref{betaLd}), (\ref{eqCtt}) and (\ref{eqCttbd}).

It follows from (\ref{uwconH1},b), (\ref{psiwconH1}--f), (\ref{bgen},b), (\ref{agen},b), (\ref{lgenbd}) and
(\ref{cal K}) that we may pass to the limit $\delta \rightarrow 0_+$ in
(\ref{bLMd},b) to obtain that $(\utae^n,\hpsiae^n) \in \Vt \times \hat X$
with $\hpsiae^n \geq 0$ a.e.\ on $\Omega \times D$ solves (\ref{bLM}), (\ref{genLM}); i.e. it solves (\ref{Gequn},b).

Next we prove the integral constraint on $\hpsiae^n$. First, for $m=n-1,\,n$, let
\begin{align}
\rho^m_{\epsilon,L}(\xt) := \int_D M(\qt)\, \hpsiae^m(\xt,\qt) \dq,\qquad
\xt \in \Omega.
\label{zetan}
\end{align}
For $m=n-1,\,n$, as $\hpsiae^m \in \hat X$,
we deduce from the Cauchy--Schwarz inequality and Fubini's theorem that
$\rho^n_{\epsilon,L} \in H^1(\Omega)$ and $\rho^{n-1}_{\epsilon,L} \in L^2(\Omega)$.
We introduce
also the following closed linear subspace of $\hat X = H^1_M(\Omega \times D)$:
\begin{align}
\!\!\!\!H^1(\Omega)\otimes 1(D)
:= \left\{\hat\varphi \in H^1_M(\Omega \times D)\,:\, \hat\varphi(\cdot,\qt^*)
= \hat\varphi(\cdot,\qt^{**})\quad\mbox{$\forall \qt^*, \qt^{**} \in D$}\right\}.
\label{H101}
\end{align}
Then, on choosing $\hat \varphi = \varphi \in H^1(\Omega)\otimes 1(D)$
in (\ref{psiG}), we deduce from (\ref{zetan}) and Fubini's theorem that,
for all $\varphi \in H^1(\Omega)$,
\begin{align}
\label{psiGI}
&\int_{\Omega}
\frac{\rho^n_{\epsilon,L}-\rho^{n-1}_{\epsilon,L}}{\Delta t}
\,\varphi \dx
+ \int_{\Omega} \left[
\epsilon\,\nabx \rho^n_{\epsilon,L} -
\utae^{n-1}\,\rho^n_{\epsilon,L} \right]\cdot\, \nabx
\varphi \dx
= 0,
\end{align}
with
\begin{align}
\label{psiGIb}
\!\!\!0 \leq \rho^0_{\epsilon,L} :=  \int_D M\, \hat\psi^0_{\epsilon, L} \dq  =  \int_D M\, \beta^L(\hat\psi^0)\dq  \leq  \int_D M\, \hat\psi^0 \dq \leq 1, \quad \!\!\mbox{a.e. on $\Omega$.}
\end{align}

By introducing the function $z^m_{\epsilon,L}:= 1 - \rho^m_{\epsilon,L}$, $m=n-1,\,n$,
and noting that $z^n_{\epsilon,L} \in H^1(\Omega)$ and $z^{n-1}_{\epsilon,L} \in L^2(\Omega)$,
we deduce from (\ref{psiGI}), and as $\utae^{n-1}$ is divergence-free on $\Omega$
with zero trace on $\partial \Omega$, that
\begin{align}
\label{psiGIz}
&\int_{\Omega}
\frac{z^n_{\epsilon,L}-z^{n-1}_{\epsilon,L}}{\Delta t}
\,\varphi \dx
+ \int_{\Omega} \left[
\epsilon\,\nabx z^n_{\epsilon,L} -
\utae^{n-1}\,z^n_{\epsilon,L} \right]\cdot\, \nabx
\varphi \dx
= 0,
\end{align}
for all $\varphi \in H^1(\Omega)$. Let us now define by
$[x]_{\pm}:={\textstyle\frac{1}{2}}\left(x\pm |x|\right)$
the positive and negative parts, $[x]_+$ and $[x]_{-}$, of a real number $x$,
respectively.
As $\hpsiae^{n-1} \in \hat Z_2$,
we then have that $[z^{n-1}_{\epsilon,L}]_{-} =0$ a.e. on $\Omega$.
Taking $\varphi = [z^n_{\epsilon,L}]_{-}$ as a test function in  \eqref{psiGIz}, noting
that this is a legitimate choice since $[z^n_{\epsilon,L}]_{-} \in H^1(\Omega)$, decomposing
$z^m_{\epsilon,L}$, $m=n-1,n$, into their positive and negative parts,
and noting that $\utae^{n-1}$ is divergence-free on $\Omega$ and has zero trace
on $\partial \Omega$, we deduce that
\[
\|[z^n_{\epsilon,L}]_{-}\|^2 + \Delta t \, \varepsilon
\|\nabx [z^n_{\epsilon,L}]_{-}\|^2 = 0,
\]
where $\|\cdot\|$ denotes the $L^2(\Omega)$ norm.
Hence, $[z^n_{\epsilon,L}]_{-} = 0$ a.e. on $\Omega$. In other words,
$z^n_{\epsilon,L} \geq 0$ a.e. on
$\Omega$, which then gives that $\rho^n_{\epsilon,L} \leq 1$ a.e. on $\Omega$,
i.e. $\hpsiae^n \in \hat Z_2$
as required. As $(\utae^0,\hpsiae^0) \in \Vt \times \hat Z_2$, performing
the above existence proof at each time level $t_n$, $n=1,\ldots,N$,
yields a solution  $\{(\utae^n,\hpsiae^n)\}_{n=1}^N$  to (P$^{\Delta t}_{\epsilon,L}$).
\end{proof}

\section{Entropy estimates}
\label{sec:entropy}
\setcounter{equation}{0}

Next, we derive bounds on the solution of $({\rm P}^{\Delta t}_{\epsilon, L})$, independent of $L$.
Our starting point is Lemma \ref{conv}, concerning the existence of a solution to the problem  $({\rm P}^{\Delta t}_{\varepsilon,L})$. The model $({\rm P}^{\Delta t}_{\varepsilon,L})$ includes `microscopic cut-off' in the drag term of the Fokker--Planck equation,  where $L>1$ is a (fixed, but otherwise arbitrary,) cut-off parameter.
Our ultimate objective is to pass to the limits $L \rightarrow \infty$ and $\Delta t \rightarrow 0_+$ in the model $({\rm P}^{\Delta t}_{\varepsilon,L})$, with $L$ and $\Delta t$ linked by the condition $\Delta t = o(L^{-1})$,
as $L \rightarrow \infty$.
To that end, we need to develop various bounds on sequences of weak solutions of $({\rm P}^{\Delta t}_{\varepsilon,L})$ that are uniform in the cut-off parameter
$L$ and thus permit the extraction of weakly convergent subsequences, as $L \rightarrow \infty$, through the use of a weak compactness argument. The derivation of such bounds, based on the use of the relative entropy associated with the Maxwellian $M$, is our
main task in this section.

Let us introduce the following definitions, in line with (\ref{fn}):
\begin{subequations}
\begin{alignat}{1}
\hspace{-4mm}\utaeD(\cdot,t)&:=\,\frac{t-t_{n-1}}{\Delta t}\,
\utae^n(\cdot)+
\frac{t_n-t}{\Delta t}\,\utae^{n-1}(\cdot),
\;\, t\in [t_{n-1},t_n], \;\, n=1,\dots,N, \label{ulin}\\
\hspace{-2.5mm}\utaeDp(\cdot,t)&:=\ut^n(\cdot),\;\;
\utaeDm(\cdot,t):=\ut^{n-1}(\cdot),
\;\; t\in(t_{n-1},t_n], \;\; n=1,\dots,N. \label{upm}
\end{alignat}
\end{subequations}
We shall adopt $\uta^{\Delta t (,\pm)}$ as a collective symbol for $\uta^{\Delta t}$, $\uta^{\Delta t,\pm}$. The corresponding notations $\psia^{\Delta t (,\pm)}$, $\psia^{\Delta t}$, and $\psia^{\Delta t,\pm}$ are defined analogously; recall \eqref{idatabd} and \eqref{inidata-1}.

We note for future reference that
\begin{equation}
\utaeD-\utae^{\Delta t,\pm}= (t-t_{n}^{\pm})
\,\frac{\partial \utaeD}{\partial t},
\quad t \in (t_{n-1},t_{n}), \quad n=1,\dots,N, \label{eqtime+}
\end{equation}
where $t_{n}^{+} := t_{n}$ and $t_{n}^{-} := t_{n-1}$, with an analogous relationship in the case of $\psia^{\Delta t}$.

Using the above notation,
(\ref{Gequn}) summed for $n=1, \dots,  N$ can be restated  in a form
that is reminiscent of a weak formulation of  (\ref{ns1a}--d):
find $\ut^{\Delta t(,\pm)}_{\varepsilon, L} \in \Vt$, $t \in (0,T]$, such that
\begin{align}\label{equncon}
&\!\displaystyle\int_{0}^{T}\!\! \int_\Omega  \frac{\partial \utaeD}{\partial t}\cdot
\wt \dx \dt
+ \int_{0}^T\!\! \int_{\Omega}
\left[ \left[ (\utaeDm \cdot \nabx) \utaeDp \right]\cdot\,\wt
+ \nu \,\nabxtt \utaeDp
:
\wnabtt \right]\! \dx \dt
\nonumber
\\
&\hspace{1cm} =\int_{0}^T
\left[ \langle \ft^{\Delta t,+}, \wt\rangle_{H^1_0(\Omega)}
- k\,\sum_{i=1}^K \int_{\Omega}
\Ctt_i(M\,\hpsiae^{\Delta t,+}): \nabxtt
\wt \dx \right]\! \dt,
\end{align}
for all $\wt \in L^1(0,T;\Vt )$, subject to the initial condition $\utaeD(\cdot,0)= \ut^0 \in \Vt$.

Analogously, (after a minor re-ordering of terms on the left-hand side
for presentational reasons,) (\ref{psiG}) summed through $n=1, \dots,  N$
can be restated as follows: find $\hat\psi^{\Delta t(,\pm)}_{\varepsilon, L}(t) \in \hat Z_2$, $t \in (0,T]$, such that
\begin{align}\label{eqpsincon}
&\int_{0}^T \int_{\Omega \times D}
M\,\frac{ \partial \hpsiae^{\Delta t}}{\partial t}\,
\hat \varphi \dq \dx \dt
\nonumber\\& \hspace{0.5cm}
+ \int_{0}^T \int_{\Omega \times D} M\,\left[
\epsilon\, \nabx \hpsiae^{\Delta
t,+} - \utae^{\Delta t,-}\,\hpsiae^{\Delta t,+} \right]\cdot\, \nabx
\hat \varphi
\,\dq \dx \dt
\nonumber \\
& \hspace{0.5cm} +
\frac{1}{2\,\lambda}
\int_{0}^T \int_{\Omega \times D} M\,\sum_{i=1}^K
 \,\sum_{j=1}^K A_{ij}\,\nabqj \hpsiae^{\Delta t,+}
\cdot\, \nabqi
\hat \varphi
\,\dq \dx \dt
\nonumber \\
&
\hspace{0.5cm}
-
\int_{0}^T \int_{\Omega \times D} M\,\sum_{i=1}^K
\left[\sigtt(\utae^{\Delta t,+})
\,\qt_i\right]\,\beta^L(\hpsiae^{\Delta t,+}) \,\cdot\, \nabqi
\hat \varphi
\,\dq \dx \dt = 0,
\end{align}
for all $\hat \varphi \in L^1(0,T;\hat X)$, subject to the initial condition $\hpsiae^{\Delta t}(\cdot,\cdot,0) = \beta^L(\hat \psi^0(\cdot,\cdot)) \in \hat Z_2$.
We emphasize that \eqref{equncon} and \eqref{eqpsincon} are an equivalent restatement of problem (${\rm P}^{\Delta t}_{\epsilon,L}$), for which existence of a solution has been established (cf. Lemma \ref{conv}).

Similarly, with analogous notation for $\{\rho_{\epsilon,L}^n\}_{n=0}^N$,
(\ref{psiGI}) summed for $n=1,\dots,N$ can be restated as follows:
Given $\ut_{\epsilon,L}^{\Delta t(,\pm)}(t) \in \Vt$, $t \in (0,T]$,  solving (\ref{equncon}),
find $\rho_{\epsilon,L}^{\Delta t(,\pm)}(t)\in \mathcal{K} := \{ \eta \in H^1(\Omega) : \eta \in [0,1] \mbox{ a.e.\ on } \Omega \}$,
$t \in (0,T]$, such that
\begin{align}
&\int_0^T \int_\Omega
\frac{\partial \rho_{\epsilon,L}^{\Delta t}}{\partial t}\, \varphi \dx \dt
+ \int_0^T \int_{\Omega} \left[ \epsilon\, \nabx  \rho_{\epsilon,L}^{\Delta t, +} -
\ut_{\epsilon,L}^{\Delta t,-}  \,\rho_{\epsilon,L}^{\Delta t, +} \right]
\cdot \nabx \varphi \dx \dt =0
\nonumber \\
& \hspace{2.5in}
\qquad \forall \varphi \in L^1(0,T;H^1(\Omega)),
\label{zetacon}
\end{align}
subject to the initial condition $\rho_{\epsilon,L}^{\Delta t}(\cdot,0) = \int_D M(\qt)
\beta^L(\hat \psi^0(\cdot,\qt)) \dq \in \mathcal{K}$; cf.\ (\ref{zetan}) and recall that $
\hat \psi_{\epsilon,L}^0 = \beta^L(\hat \psi^0)$.
Once again, on noting (\ref{zetan}) and (\ref{psiGI}),
we have established the existence of a solution to
(\ref{zetacon}) and that
\begin{align}
\rho_{\epsilon,L}^{\Delta t(,\pm)}(\xt,t) = \int_D M(\qt)\, \hat \psi_{\epsilon,L}^{\Delta t(,\pm)}
(\xt,\qt,t) \dq \qquad \mbox{for a.e. } (\xt,t) \in \Omega \times (0,T).
\label{zetancon}
\end{align}

In conjunction with $\beta^L$, defined by \eqref{betaLa},
we consider the following cut-off version $\mathcal{F}^L$
of the entropy function $\mathcal{F}\,: s \in \mathbb{R}_{\geq 0} \mapsto \mathcal{F}(s) = s (\log s - 1) + 1
\in \mathbb{R}_{\geq 0}$:
\begin{equation}\label{eq:FL}
\mathcal{F}^L(s):= \left\{\begin{array}{ll}
s(\log s - 1) + 1,   &  ~~0 \leq s \leq L,\\
\frac{s^2 - L^2}{2L} + s(\log L - 1) + 1,  &  ~~L \leq s.
\end{array} \right.
\end{equation}
Note that
\begin{equation}\label{eq:FL1}
(\mathcal{F}^L)'(s) = \left\{\begin{array}{ll}
\log s,   &  ~~0 < s \leq L,\\
\frac{s}{L} + \log L - 1,  &  ~~L \leq s,
\end{array} \right.
\end{equation}
and
\begin{equation}\label{eq:FL2}
(\mathcal{F}^L)''(s) = \left\{\begin{array}{ll}
\frac{1}{s},   &  ~~0 < s \leq L,\\
\frac{1}{L},  &  ~~L \leq s.
\end{array} \right.
\end{equation}
Hence,
\begin{equation}\label{eq:FL2a}
\beta^L(s) = \min(s,L) = [(\mathcal{F}^L)''(s)]^{-1},\quad s \in \mathbb{R}_{\geq 0},
\end{equation}
with the convention $1/\infty:=0$ when $s=0$, and
\begin{equation}\label{eq:FL2b}
(\mathcal{F}^L)''(s) \geq \mathcal{F}''(s) = s^{-1},\quad s \in \mathbb{R}_{> 0}.
\end{equation}
We shall also require the following inequality, relating $\mathcal{F}^L$ to $\mathcal{F}$:
\begin{equation}\label{eq:FL2c}
\mathcal{F}^L(s) \geq \mathcal{F}(s),\quad s \in \mathbb{R}_{\geq 0}.
\end{equation}
For $0\leq s \leq 1$, \eqref{eq:FL2c} trivially holds, with equality. For $s\geq 1$, it
follows from \eqref{eq:FL2b}, with $s$ replaced by a dummy variable $\sigma$, after integrating
twice over $\sigma \in [1,s]$, and noting that $(\mathcal{F}^L)'(1)= \mathcal{F}'(1)$ and $(\mathcal{F}^L)(1)=\mathcal{F}(1)$.

\subsection{$L$-independent bounds on the spatial derivatives}
\label{Lindep-space}

We are now ready to embark on the derivation of the required bounds, uniform in the cut-off parameter $L$,
on norms of $\uta^{\Delta t,+}$ and $\psia^{\Delta t,+}$.
As far as $\uta^{\Delta t,+}$ is concerned, this is a relatively straightforward exercise. We select $\wt = \chi_{[0,t]}\,\uta^{\Delta t,+}$ as test function in \eqref{equncon}, with $t$ chosen as $t_n$, $n \in \{1,\dots,N\}$, and $\chi_{[0,t]}$ denoting the
characteristic function of the interval $[0,t]$. We then deduce, with $t = t_n$ and noting \eqref{idatabd}, that
\begin{eqnarray}
&&\|\uta^{\Delta t,+}(t)\|^2 + \frac{1}{\Delta t}\int_0^t  \|\uta^{\Delta t,+}(s) - \uta^{\Delta t,-}(s)\|^2 \dd s
+ \nu \int_0^t \|\nabxtt \uta^{\Delta t,+}(s)\|^2 \dd s \nonumber\\
&&\quad\leq \|\ut_0\|^2 + \frac{1}{\nu}\int_0^t\|\ft^{\Delta t,+}(s)\|^2_{(H^1_0(\Omega))'} \dd s \nonumber\\
&& \quad \quad \hspace{0.3mm} -2k \int_0^t \int_{\Omega \times D} M(\qt)\,\sum_{i=1}^K \qt_i \qt_i^{\rm T} \,U_i'\big({\textstyle \frac{1}{2}}|\qt_i|^2\big)\,\psia^{\Delta t,+} :
\nabxtt \uta^{\Delta t,+} \dq \dx \dd s,~~~~ \label{eq:energy-u}
\end{eqnarray}
where, again, $\|\cdot\|$ denotes the $L^2$ norm over $\Omega$.
%and we have used \eqref{idatabd}.
We intentionally {\em did not} bound the final term on
the right-hand side of \eqref{eq:energy-u}.
As we shall see in what follows,
this simple trick will prove helpful: our bounds on $\psia^{\Delta t, +}$ below will furnish
an identical term with the opposite sign, so then by combining the bounds on $\uta^{\Delta t, +}$ and $\psia^{\Delta t, +}$
this pair of, otherwise dangerous, terms will be removed. This fortuitous cancellation reflects
the balance of total energy in the system.

Having dealt with $\uta^{\Delta t,+}$, we now embark on the less straightforward task of deriving bounds on
norms of $\psia^{\Delta t,+}$ that are uniform in the cut-off parameter $L$.
The appropriate choice of test function in \eqref{eqpsincon} for this purpose
is $\hat\varphi = \chi_{[0,t]}\,(\mathcal{F}^L)'(\psia^{\Delta t,+})$ with $t=t_n$, $n \in \{1,\dots,N\}$;
this can be seen by noting that with such a $\hat\varphi$, at
least formally, the final term on the left-hand side of
\eqref{eqpsincon} can be manipulated to become identical to
the final term in \eqref{eq:energy-u}, but with the opposite sign.
While Lemma \ref{conv} guarantees that $\psia^{\Delta t,+}(\cdot,\cdot,t)$ belongs to $\hat{Z}_2$ for all $t \in [0,T]$, and is therefore nonnegative a.e. on $\Omega \times D \times [0,T]$, there is unfortunately
no reason why $\psia^{\Delta t,+}$ should be strictly positive on $\Omega\times D \times [0,T]$,
and therefore the
expression $(\mathcal{F}^L)'(\psia^{\Delta t,+})$ may in general
be undefined; the same is true of $(\mathcal{F}^L)''(\psia^{\Delta t,+})$,
which also appears in the algebraic manipulations. We shall circumvent this problem by working
with $(\mathcal{F}^L)'(\psia^{\Delta t,+} + \alpha)$ instead of $(\mathcal{F}^L)'(\psia^{\Delta t,+})$, where $\alpha>0$; since
$\psia^{\Delta t,+}$ is known to be nonnegative from Lemma \ref{conv}, $(\mathcal{F}^L)'(\psia^{\Delta t,+} + \alpha)$ and $(\mathcal{F}^L)''(\psia^{\Delta t,+} + \alpha)$
are well-defined.
After deriving the relevant bounds, which will involve $\mathcal{F}^L(\psia^{\Delta t,+} + \alpha)$ only, we shall pass to the limit $\alpha \rightarrow 0_{+}$, noting that, unlike
$(\mathcal{F}^L)'(\psia^{\Delta t,+})$ and $(\mathcal{F}^L)''(\psia^{\Delta t,+})$, the function
$(\mathcal{F}^L)(\psia^{\Delta t,+})$ is well-defined for any {\em nonnegative} $\psia^{\Delta t,+}$.
Thus, we take any $\alpha \in (0,1)$, whereby $0 < \alpha < 1 < L$, and we choose
$\hat\varphi = \chi_{[0,t]}\,(\mathcal{F}^L)'(\psia^{\Delta t,+} + \alpha)$,
with $t = t_n$, $\;n \in \{1,\dots,N\}$,
as test function in \eqref{eqpsincon}. As the calculations are quite involved, for the sake of clarity of exposition we shall manipulate the terms in \eqref{eqpsincon} one at a time and will then merge the resulting bounds on the individual terms with \eqref{equncon}
to obtain a single energy inequality for the pair $(\uta^{\Delta t,+},\psia^{\Delta t,+})$.
%For the sake of brevity, some of the more elementary transitions are omitted; we refer the reader to our extended paper \cite{BS2010} for %details.

We start by considering the first term in \eqref{eqpsincon}. Clearly $\mathcal{F}^L(\cdot + \alpha)$ is twice continuously differentiable on the interval $(-\alpha,\infty)$ for any $\alpha>0$. Thus, by Taylor series expansion with remainder of the function
$$s \in [0,\infty) \mapsto \mathcal{F}^L(s +\alpha) \in [0,\infty),$$
we have, for any $c \in [0,\infty)$, that
\[ (s-c)\, (\mathcal{F}^L)'(s+\alpha) = \mathcal{F}^L(s+\alpha) - \mathcal{F}^L(c+\alpha) + \frac{1}{2}(s-c)^2\,(\mathcal{F}^L)''(\theta s + (1-\theta)c+\alpha),\]
with $\theta \in (0,1)$. Hence, on noting that $t\in [0,T]\mapsto\hat\psi^{\Delta t}_{\epsilon,L}(\cdot,\cdot,t)\in \hat{X}$ is
piecewise linear relative to the partition $\{0=t_0, t_1,\dots, t_N=T\}$ of the interval $[0,T]$,
\begin{eqnarray*}
{\rm T}_1&\!:=&\int_0^T \int_{\Omega \times D} M\,\frac{\partial \psia^{\Delta t}}{\partial s} \, \chi_{[0,t]}\,(\mathcal{F}^L)'(\psia^{\Delta t,+} + \alpha) \dq \dx \dd s \nonumber \\
&=&\int_0^t \int_{\Omega \times D} M \frac{\partial}{\partial s} (\psia^{\Delta t} + \alpha)\,  (\mathcal{F}^L)'(\psia^{\Delta t,+} + \alpha) \dq \dx \dd s
\nonumber \\
&=& \int_{\Omega \times D} M\mathcal{F}^L(\psia^{\Delta t,+}(t) + \alpha) \dq \dx - \int_{\Omega \times D} M\mathcal{F}^L(\beta^L(\hat\psi^0) + \alpha) \dq \dx
\nonumber \\
&&%\hspace{-10mm}
+ \frac{1}{2 \Delta t}\int_0^t\int_{\Omega \times D}\!\!\!  M(\mathcal{F}^L)''(\theta\psia^{\Delta t,+}
+ (1-\theta)\psia^{\Delta t, -} + \alpha)\,(\psia^{\Delta t,+} - \psia^{\Delta t,-})^2 \dq \dx \dd s\nonumber.
\end{eqnarray*}
Noting from \eqref{eq:FL2} that $(\mathcal{F}^L)''(s + \alpha) \geq 1/L$ for all $s \in [0,\infty)$ and all $\alpha>0$, this then implies, with $t=t_n$, $n \in \{1,\dots,N\}$, that
\begin{eqnarray}\label{eq:energy-psi1}
{\rm T}_1&\geq &\int_{\Omega \times D} M\,\mathcal{F}^L(\psia^{\Delta t,+}(t) + \alpha) \dq \dx - \int_{\Omega \times D} M\,\mathcal{F}^L(\beta^L(\hat\psi^0) + \alpha) \dq \dx
\nonumber \\
&&+ \frac{1}{2 \Delta t\, L }\int_0^t\int_{\Omega \times D} M\,(\psia^{\Delta t,+} - \psia^{\Delta t,-})^2 \dq \dx \dd s.
\end{eqnarray}
The denominator in the prefactor of the last integral motivates us to link $\Delta t$ to $L$ so that $\Delta t\, L = o(1)$
as $\Delta t \!\rightarrow\! 0_{+}$ (or, equivalently, $\Delta t = o(L^{-1})$ as $L \rightarrow \infty$), in order to drive the integral multiplied by the prefactor to $0$ in the limit of $L \rightarrow \infty$,
once the product of the two has been bounded above by a constant, independent of $L$.

Next we consider the second term in \eqref{eqpsincon}, using repeatedly that
$\nabx \cdot ~\uta^{\Delta t,-} = 0$ and that $\uta^{\Delta t,-}$ has zero trace on $\partial\Omega$:
\begin{eqnarray*}
&&\hspace{-2mm}{\rm T}_2:=\int_{0}^T \int_{\Omega \times D}\!\! M\!\left[ \epsilon\,\nabx \psia^{\Delta t,+} - \uta^{\Delta t,-}\,\psia^{\Delta t,+}\right]\cdot\, \nabx \chi_{[0,t]}\,(\mathcal{F}^L)'(\psia^{\Delta t,+} + \alpha)  \,\dq \dx \dd s
\nonumber\\
&&\qquad\!\hspace{-2.1mm} = \varepsilon \int_0^t \int_{\Omega \times D} M\, \nabx (\psia^{\Delta t,+} + \alpha) \cdot \nabx (\mathcal{F}^L)'(\psia^{\Delta t,+} + \alpha) \dq \dx \dd s\\
&&\qquad\quad - \int_0^t  \int_{\Omega \times D} M\, \uta^{\Delta t,-} (\psia^{\Delta t,+} + \alpha) \cdot \nabx (\mathcal F^L)'(\psia^{\Delta t,+} + \alpha) \dq \dx \dd s,
\end{eqnarray*}
where in the last line we subtracted $0$ in the form of
\[ \alpha \int_0^t \int_{\Omega \times D} M\, \uta^{\Delta t,-} \cdot \nabx (\mathcal{F}^L)'(\psia^{\Delta t,+} + \alpha) \dq \dx \dd s = 0.\]
Hence, similarly to \eqref{convGdL},
\begin{eqnarray*}
%&&\hspace{-2mm}
{\rm T}_2&:=& \varepsilon \int_0^t\!\! \int_{\Omega \times D} M (\mathcal{F}^L)''(\psia^{\Delta t,+} + \alpha)
|\nabx (\psia^{\Delta t,+} + \alpha)|^2 \dq \dx \dd s\nonumber\\
&&- \int_0^t\!\!  \int_{\Omega \times D} M \uta^{\Delta t,-} (\psia^{\Delta t,+} + \alpha) \cdot [(\mathcal F^L)''(\psia^{\Delta t,+} + \alpha)\nabx(\psia^{\Delta t,+} + \alpha)] \dq \dx \dd s\nonumber\\
&=& \varepsilon \int_0^t\!\! \int_{\Omega \times D} M (\mathcal{F}^L)''(\psia^{\Delta t,+} + \alpha)
|\nabx \psia^{\Delta t,+} |^2 \dq \dx \dd s\nonumber\\
&&- \int_0^t\!\!  \int_{\Omega \times D}\!\! M \uta^{\Delta t,-} (\psia^{\Delta t,+} + \alpha) \cdot \nabx \psia^{\Delta t,+}\! \\
&&\hspace{5cm}  \times
\left\{\!\begin{array}{ll} 1/(\psia^{\Delta t,+} + \alpha) &  \mbox{if $\psia^{\Delta t,+} + \alpha \leq L$} \\
                         {1}/{L}              &  \mbox{if $\psia^{\Delta t,+} + \alpha \geq L$}
                        \end{array}\! \right\}\dq \dx \dd s\nonumber\\
&=& \varepsilon \int_0^t\!\! \int_{\Omega \times D} M (\mathcal{F}^L)''(\psia^{\Delta t,+} + \alpha)
|\nabx \psia^{\Delta t,+} |^2 \dq \dx \dd s\nonumber\\
&&- \int_0^t\!\!  \int_{\Omega \times D}\! M \uta^{\Delta t,-} \cdot \nabx [G^L(\psia^{\Delta t,+} + \alpha)] \dq \dx \dd s,
%\nonumber\\
\end{eqnarray*}
where $G^L$ denotes the (locally Lipschitz continuous) function defined on $\mathbb{R}$ by
$s - \frac{L}{2}$ if $s \leq L$ and $\frac{1}{2L}s^2$ otherwise.
On noting that the integral involving $G^L$ vanishes, \eqref{eq:FL2b} yields the lower bound
\begin{eqnarray}\label{eq:energy-psi2}
&&{\rm T}_2 \geq  \varepsilon \int_0^t \int_{\Omega \times D} M\, (\psia^{\Delta t,+} + \alpha)^{-1}\,
|\nabx (\psia^{\Delta t,+} + \alpha) |^2 \dq \dx \dd s.
\end{eqnarray}

Next, we consider the third term in \eqref{eqpsincon}. Thanks to \eqref{A} we have, again with $t=t_n$
and $n \in \{1,\dots, N\}$:
\begin{eqnarray}\label{eq:energy-psi3}
&&\hspace{-2.3mm}{\rm T}_3 := \frac{1}{2\,\lambda} \int_{0}^T \int_{\Omega \times D}M\, \sum_{i=1}^K \sum_{j=1}^K A_{ij} \,\nabqj \psia^{\Delta t,+}
\cdot\, \nabqi \chi_{[0,t]}(\mathcal{F}^L)'(\psia^{\Delta t,+} + \alpha)  \,\dq \dx \dd s\nonumber\\
&&\quad = \frac{1}{2\,\lambda} \int_0^t \int_{\Omega \times D}M\,(\mathcal{F}^L)''(\psia^{\Delta t,+} + \alpha)\,\sum_{i=1}^K \sum_{j=1}^K A_{ij} \,\nabqj \psia^{\Delta t,+} \cdot\, \nabqi \psia^{\Delta t,+} \dq \dx \dd s\nonumber\\
&&\quad \geq  \frac{a_0}{2\,\lambda}  \int_0^t \int_{\Omega \times D}M\,(\mathcal{F}^L)''(\psia^{\Delta t,+} + \alpha)\,\sum_{i=1}^K |\nabqi \psia^{\Delta t,+} |^2 \,\dq \dx \dd s\nonumber\\
%&&\quad \geq \frac{a_0}{2\,\lambda}  \int_0^t \int_{\Omega \times D}M\, (\psia + \alpha)^{-1}\,\sum_{i=1}^K |\nabqi \psia |^2 \,\dq \dx \dd s\nonumber\\
%&&\quad = \frac{2a_0}{\lambda}  \int_0^t \int_{\Omega \times D}M\, \sum_{i=1}^K \left|\nabqi \sqrt{\psia + \alpha} \right|^2 \,\dq \dx \dd s \nonumber \\
%&&\quad = \frac{2a_0}{\lambda}  \int_0^t \int_{\Omega \times D}M\, \left|\nabq \sqrt{\psia + \alpha} \right|^2 \,\dq \dx \dd s.
&&\quad =  \frac{a_0}{2\,\lambda}  \int_0^t \int_{\Omega \times D}M\,(\mathcal{F}^L)''(\psia^{\Delta t,+} + \alpha)\,|\nabq \psia^{\Delta t,+} |^2 \,\dq \dx \dd s.
\end{eqnarray}

We are now ready to consider the final term in \eqref{eqpsincon}, with $t=t_n$, $n \in \{1,\dots, N\}$:
\begin{eqnarray}\label{eq:energy-psi4}
&&\hspace{-2.4mm}{\rm T}_4:= - \int_{0}^T\!\! \int_{\Omega \times D}M\, \sum_{i=1}^K [\, \sigtt(\uta^{\Delta t,+}) \,\qt_i\,]\,\beta^L(\psia^{\Delta t,+})\nonumber\\
&&\hspace{6cm}\,\cdot\, \nabqi \chi_{[0,t]}\,({\mathcal F}^L)'(\psia^{\Delta t,+} + \alpha) \,\dq \dx \dd s\nonumber\\
&&\quad = - \int_0^t\!\! \int_{\Omega \times D} M\, \sum_{i=1}^K [\,(\nabxtt \uta^{\Delta t,+})\,\qt_i\,]\, \beta^L(\psia^{\Delta t,+})\nonumber\\
&&\hspace{6cm}\, \cdot\, (\mathcal{F}_L)''(\psia^{\Delta t,+} + \alpha)\, \nabqi \psia^{\Delta t,+}\, \dq \dx \dd s\nonumber \\
&&\quad = - \int_0^t\!\! \int_{\Omega \times D} M\, \sum_{i=1}^K [\,(\nabxtt \uta^{\Delta t,+})\,\qt_i\,] \,\frac{\beta^L(\psia^{\Delta t,+})}{\beta^L(\psia^{\Delta t,+} + \alpha)}\cdot \nabqi \psia^{\Delta t,+}\, \dq \dx \dd s\nonumber\\
&&\quad = - \int_0^t\!\! \int_{\Omega \times D} M\, \sum_{i=1}^K [\,(\nabxtt \uta^{\Delta t,+})\,\qt_i\,]
\cdot \nabqi \psia^{\Delta t,+}\, \dq \dx \dd s\nonumber\\
&&\qquad\, + \int_0^t\!\! \int_{\Omega \times D} M\, \sum_{i=1}^K [\,(\nabxtt \uta^{\Delta t,+})\,\qt_i\,] \left[1 -\frac{\beta^L(\psia^{\Delta t,+})}{\beta^L(\psia^{\Delta t,+} + \alpha)}\right] \cdot \nabqi \psia^{\Delta t,+}\, \dq \dx \dd s\nonumber\\
&&\quad = - \int_0^t\!\! \int_{\Omega \times D} M\,\sum_{i=1}^K \qt_i\,\qt_i^{\rm T}\,U_i'(\textstyle{\frac{1}{2}|\qt|^2})\,\psia^{\Delta t,+} : \nabxtt \uta^{\Delta t,+}  \dq \dx \dd s\nonumber\\
&&\qquad\, + \int_0^t\!\! \int_{\Omega \times D} M\, \sum_{i=1}^K [\,(\nabxtt \uta^{\Delta t,+})\,\qt_i\,] \left[1 -\frac{\beta^L(\psia^{\Delta t,+})}{\beta^L(\psia^{\Delta t,+} + \alpha)}\right] \cdot  \nabqi \psia^{\Delta t,+}\, \dq \dx \dd s,\nonumber\\
\end{eqnarray}
where in the transition to the final equality we applied \eqref{intbyparts} with $B:= \nabxtt \ut^{\Delta t,+}_{\epsilon,L}$ (on account of it being independent of the variable $\qt$), together with the fact that
$\mathfrak{tr}\,(\nabxtt \uta^{\Delta t,+}) = \nabx \cdot\, \uta^{\Delta t,+} = 0$.
Summing \eqref{eq:energy-psi1}--\eqref{eq:energy-psi3} and \eqref{eq:energy-psi4} yields,
with $t=t_n$ and $n \in \{1,\dots, N\}$, the following inequality:
\begin{eqnarray}\label{eq:energy-psi-summ1}
&&\hspace{-3mm}\int_{\Omega \times D} M\, \mathcal{F}^L(\psia^{\Delta t,+}(t) + \alpha) \dq \dx
+\,\frac{1}{2 \Delta t\, L}\int_0^t\int_{\Omega \times D}  M\,(\psia^{\Delta t,+} - \psia^{\Delta t,-})^2 \dq \dx \dd s \nonumber\\
&&\hspace{0mm}+\, \varepsilon \int_0^t \int_{\Omega \times D} M\,
\frac{|\nabx \psia^{\Delta t,+}|^2}{\psia^{\Delta t,+} + \alpha} \dq \dx \dd s\nonumber\\
&&\hspace{0mm} +\,\frac{a_0}{2\,\lambda}  \int_0^t \int_{\Omega \times D}M\,(\mathcal{F}^L)''(\psia^{\Delta t,+} + \alpha)\,|\nabq \psia^{\Delta t,+} |^2 \,\dq \dx \dd s\nonumber\\
&&\hspace{-3mm}\leq \int_{\Omega \times D} M\, \mathcal{F}^L(\beta^L(\hat\psi^0) + \alpha) \dq \dx\nonumber\\
&&\hspace{0mm}+ \int_0^t \int_{\Omega \times D} M\,\sum_{i=1}^K \qt_i\,\qt_i^{\rm T}\,U_i'(\textstyle{\frac{1}{2}|\qt|^2})\,\psia^{\Delta t,+} : \nabxtt \uta^{\Delta t,+}  \dq \dx \dd s
\nonumber\\
&&\hspace{0mm}- \int_0^t \int_{\Omega \times D} M\, \sum_{i=1}^K \left[(\nabxtt \uta^{\Delta t,+})\,\qt_i\right] \left[1 -\frac{\beta^L(\psia^{\Delta t,+})}{\beta^L(\psia^{\Delta t,+} + \alpha)}\right]\cdot \nabqi \psia^{\Delta t,+}\, \dq \dx \dd s.
\nonumber\\
\end{eqnarray}
Comparing \eqref{eq:energy-psi-summ1} with \eqref{eq:energy-u} we see that after multiplying \eqref{eq:energy-psi-summ1} by $2k$ and adding the resulting inequality to \eqref{eq:energy-u}
the final term in \eqref{eq:energy-u} is cancelled by $2k$ times the
second term on the right-hand side of \eqref{eq:energy-psi-summ1}. Hence, for any $t=t_n$, with
$n \in \{1,\dots,N\}$, we deduce that
\begin{eqnarray}\label{eq:energy-u+psi}
&&\hspace{-2mm}\|\uta^{\Delta t, +}(t)\|^2 + \frac{1}{\Delta t} \int_0^t \|\uta^{\Delta t, +} - \uta^{\Delta t,-}\|^2
\dd s + \nu \int_0^t \|\nabxtt \uta^{\Delta t, +}(s)\|^2 \dd s\nonumber\\
&&\hspace{-1mm}+\,2k\int_{\Omega \times D}\!\! M\, \mathcal{F}^L(\psia^{\Delta t, +}(t) + \alpha) \dq \dx + \frac{k}{\Delta t\, L}
\int_0^t \int_{\Omega \times D}\!\! M\, (\psia^{\Delta t, +} - \psia^{\Delta t, -})^2 \dq \dx \dd s
\nonumber \\
&&\hspace{7mm}\quad +\, 2k\,\varepsilon \int_0^t \int_{\Omega \times D} M\,
\frac{|\nabx \psia^{\Delta t, +} |^2}{\psia^{\Delta t, +} + \alpha} \dq \dx \dd s\nonumber\\
&&\hspace{14mm}\quad\quad + \frac{a_0 k}{\lambda}  \int_0^t \int_{\Omega \times D}M\,(\mathcal{F}^L)''(\psia^{\Delta t, +} + \alpha)\,|\nabq \psia^{\Delta t, +} |^2 \,\dq \dx \dd s\nonumber
\\
&&\hspace{-2mm}\leq \|\ut_0\|^2 + \frac{1}{\nu}\int_0^t\|\ft^{\Delta t,+}(s)\|^2_{(H^1_0(\Omega))'} \dd s + 2k \int_{\Omega \times D} M\, \mathcal{F}^L(\beta^L(\hat\psi^0) + \alpha) \dq \dx \nonumber
\\
&&\hspace{0mm} -\, 2k \int_0^t \int_{\Omega \times D} M\, \sum_{i=1}^K \left[(\nabxtt \uta^{\Delta t, +})\,\qt_i\right] \left[1 -\frac{\beta^L(\psia^{\Delta t, +})}{\beta^L(\psia^{\Delta t, +} + \alpha)}\right]\cdot \nabqi \psia^{\Delta t, +}\, \dq \dx \dd s.\nonumber\\
\end{eqnarray}
It remains to bound the last term on the right-hand side of \eqref{eq:energy-u+psi}. Noting that
$\beta^L$ is Lipschitz continuous, with Lipschitz constant equal to $1$, and $\beta^L(s+\alpha) \geq \alpha$
for $s \geq 0$, we have that
\begin{eqnarray*}
0 &\leq& \left(1 - \frac{\beta^L(\psia^{\Delta t, +})}{\beta^L(\psia^{\Delta t, +} + \alpha)}\right) \frac{1}{\sqrt{({\mathcal F}^L)''(\psi^{\Delta t, +} + \alpha)}} = \frac{\beta^L(\psia^{\Delta t, +} + \alpha) - \beta^L(\psia^{\Delta t, +})}{\sqrt{\beta^L(\psia^{\Delta t, +} + \alpha)}}\\
&\leq & \frac{\beta^L(\psia^{\Delta t, +} + \alpha) - \beta^L(\psia^{\Delta t, +})}{\sqrt{\alpha}} \leq
\left\{\begin{array}{cl}
\sqrt{\alpha}   &    \,\mbox{when $\psia^{\Delta t, +} \leq L$},\\
0               &    \,\mbox{when $\psia^{\Delta t, +} \geq L$}.
\end{array} \right.
\end{eqnarray*}
With this bound we now focus our attention on the last term in the inequality \eqref{eq:energy-u+psi}.
Let $\bt: = (b_1,\dots, b_K)$ and $b:=|\bt|_1:= b_1 + \cdots + b_K$; as $|q_i| \leq \sqrt{b_i}$, $i=1,\dots, K$, we have that $|\qt| \leq \sqrt{b}$ for all $\qt \in D$. For $t=t_n$, $n \in \{1,\dots,N\}$, we then have
that
\begin{eqnarray*}
&&\left|- 2k \int_0^t \int_{\Omega \times D} M\, \sum_{i=1}^K [\,(\nabxtt \uta^{\Delta t, +})\,\qt_i\,] \left[1 -\frac{\beta^L(\psia^{\Delta t, +})}{\beta^L(\psia^{\Delta t, +} + \alpha)}\right]\cdot \nabqi \psia^{\Delta t, +}\, \dq \dx \dd s\right|\nonumber\\
&&\leq 2k \int_0^t \int_{\Omega \times D} M\,  |\nabxtt \uta^{\Delta t, +}|\,|\qt|\, \left[1 -\frac{\beta^L(\psia^{\Delta t, +})}{\beta^L(\psia^{\Delta t, +} + \alpha)}\right]\, |\nabq \psia^{\Delta t, +}|\, \dq \dx \dd s \nonumber \\
&&\leq 2k \int_0^t \int_{\Omega \times D} M\,  |\nabxtt \uta^{\Delta t, +}|\,|\qt|\,
\left\{\begin{array}{cl}
\sqrt{\alpha}   &    \,\mbox{when $\psia^{\Delta t, +} \leq L$},\\
0               &    \,\mbox{when $\psia^{\Delta t, +} \geq L$}
\end{array} \right.\nonumber\\
&&\hspace{4.5cm}\times
\sqrt{(\mathcal{F}^L)''(\psia^{\Delta t, +} + \alpha)}
\, |\nabq \psia^{\Delta t, +}|\, \dq \dx \dd s \nonumber\\
&&\leq 2k \sqrt{\alpha\,b} \int_0^t \int_{\Omega \times D} M\,  |\nabxtt \uta^{\Delta t, +}|\,
\sqrt{(\mathcal{F}^L)''(\psia^{\Delta t, +} + \alpha)}
\, |\nabq \psia^{\Delta t, +}|\, \dq \dx \dd s \nonumber\\
&& = 2k \sqrt{\alpha\,b} \int_0^t \int_{\Omega}  |\nabxtt \uta^{\Delta t, +}| \left(\int_D M
\sqrt{(\mathcal{F}^L)''(\psia^{\Delta t, +} + \alpha)}
\, |\nabq \psia^{\Delta t, +}|\, \dq \right)\dx \dd s\nonumber\\
&& \leq 2k \sqrt{\alpha\,b} \int_0^t \int_{\Omega}  |\nabxtt \uta^{\Delta t, +}|
\left(\int_D M
(\mathcal{F}^L)''(\psia^{\Delta t, +} + \alpha)
\, |\nabq \psia^{\Delta t, +}|^2\, \dq \right)^{\frac{1}{2}}\dx \dd s\nonumber%\\
\end{eqnarray*}
\begin{eqnarray}\label{eq:lastterm}
&& \leq 2k \sqrt{\alpha\,b} \left(\int_0^t \int_{\Omega}  |\nabxtt \uta^{\Delta t, +}|^2 \dx \dd s\right)^{\frac{1}{2}}\nonumber\\
&&\qqquad\times \left(\int_0^t \int_{\Omega \times D} M (\mathcal{F}^L)''(\psia^{\Delta t, +} + \alpha)
\, |\nabq \psia^{\Delta t, +}|^2\, \dq \dx \dd s\right)^{\frac{1}{2}}\nonumber\\
&&\leq \frac{a_0\, k}{2 \lambda} \left(\int_0^t \int_{\Omega \times D} M (\mathcal{F}^L)''(\psia^{\Delta t, +} + \alpha)
\, |\nabq \psia^{\Delta t, +}|^2\, \dq \dx \dd s\right)\nonumber\\
&&\qqquad +\, \alpha\,\frac{2\lambda\, b\, k}{a_0}\left(\int_0^t \int_{\Omega}  |\nabxtt \uta^{\Delta t, +}|^2 \dx \dd s\right).
\end{eqnarray}
Substitution of \eqref{eq:lastterm} into \eqref{eq:energy-u+psi} and use of \eqref{eq:FL2b} to bound $(\mathcal{F}^L)''(\psia^{\Delta t, +} + \alpha)$ from below by $\mathcal{F}''(\psia^{\Delta t, +} + \alpha)= (\psia^{\Delta t, +} + \alpha)^{-1}$ and
\eqref{eq:FL2c} to bound $\mathcal{F}^L(\psia^{\Delta t, +} + \alpha)$ by $\mathcal{F}(\psia^{\Delta t, +}+\alpha)$ from below finally yields, for all $t=t_n$, $n \in \{1,\dots, N\}$, that
=

\begin{eqnarray}\label{eq:energy-u+psi1}
&&\hspace{-2mm}\|\uta^{\Delta t, +}(t)\|^2 + \frac{1}{\Delta t} \int_0^t \|\uta^{\Delta t, +} - \uta^{\Delta t,-}\|^2
\dd s + \nu \int_0^t \|\nabxtt \uta^{\Delta t, +}(s)\|^2 \dd s\nonumber\\
&&\hspace{-2mm}+ \,2k\int_{\Omega \times D}\!\! M\, \mathcal{F}(\psia^{\Delta t, +}(t) + \alpha) \dq \dx + \frac{k}{\Delta t\, L}
\int_0^t \int_{\Omega \times D}\!\! M\, (\psia^{\Delta t, +} - \psia^{\Delta t, -})^2 \dq \dx \dd s
\nonumber \\
&&\hspace{-2mm}+\, 2k\,\varepsilon \int_0^t \int_{\Omega \times D} M\,
\frac{|\nabx \psia^{\Delta t, +} |^2}{\psia^{\Delta t, +} + \alpha} \dq \dx \dd s
+\, \frac{a_0 k}{2\,\lambda}  \int_0^t \int_{\Omega \times D}M\, \frac{|\nabq \psia^{\Delta t, +}|^2}{\psia^{\Delta t, +} + \alpha} \,\dq \dx \dd s\nonumber\\
&&\hspace{4mm}\leq \|\ut_0\|^2 + \frac{1}{\nu}\int_0^t\|\ft^{\Delta t,+}(s)\|^2_{(H^1_0(\Omega))'} \dd s + 2k \int_{\Omega \times D} M\, \mathcal{F}^L(\beta^L(\hat\psi^0) + \alpha) \dq \dx\nonumber\\
&&\hspace{10mm} +\, \alpha\,\frac{2\lambda\, b\, k}{a_0} \int_0^t \|\nabxtt \uta^{\Delta t, +}(s)\|^2 \dd s.
\end{eqnarray}
The only restriction we have imposed on $\alpha$ so far is that it belongs to the open interval $(0,1)$; let us now restrict the range of $\alpha$ further by demanding that, in fact,
\begin{equation}\label{alphacond}
0 < \alpha < \min \left(1 , \frac{a_0\,\nu}{\,2\lambda\,b\,k}\right).
\end{equation}
Then, the last term on the right-hand side of \eqref{eq:energy-u+psi1} can be absorbed into the third term
on the left-hand side, giving, for $t=t_n$ and $n \in \{1,\dots, N\}$,
\begin{eqnarray}\label{eq:energy-u+psi2}
&&\hspace{-2mm}\|\uta^{\Delta t, +}(t)\|^2 + \frac{1}{\Delta t} \int_0^t \|\uta^{\Delta t, +} - \uta^{\Delta t,-}\|^2
\dd s + \left(\nu - \alpha\,\frac{2\lambda\, b\, k}{a_0}\right) \int_0^t \|\nabxtt \uta^{\Delta t, +}(s)\|^2 \dd s\nonumber\\
&&\hspace{-2mm}+ \,2k\int_{\Omega \times D}\!\! M\, \mathcal{F}(\psia^{\Delta t, +}(t) + \alpha) \dq \dx + \frac{k}{\Delta t\, L}
\int_0^t \int_{\Omega \times D}\!\! M\, (\psia^{\Delta t, +} - \psia^{\Delta t, -})^2 \dq \dx \dd s
\nonumber \\
&&\hspace{-2mm}\quad +\, 2k\,\varepsilon \int_0^t \int_{\Omega \times D} M\,
\frac{|\nabx \psia^{\Delta t, +} |^2}{\psia^{\Delta t, +} + \alpha} \dq \dx \dd s
+\, \frac{a_0 k}{2\lambda}  \int_0^t \int_{\Omega \times D}M\,\frac{|\nabq \psia^{\Delta t, +}|^2}{\psia^{\Delta t, +} + \alpha} \,\dq \dx \dd s\nonumber\\
&&\hspace{-1mm}\leq \|\ut_0\|^2 + \frac{1}{\nu}\int_0^t\|\ft^{\Delta t,+}(s)\|^2_{(H^1_0(\Omega))'} \dd s + 2k \int_{\Omega \times D} M\, \mathcal{F}^L(\beta^L(\hat\psi^0) + \alpha) \dq \dx.
\end{eqnarray}
We now focus our attention on the final integral on the right-hand side of \eqref{eq:energy-u+psi2}:
\begin{eqnarray*}
{\rm T}_5(\alpha) &:=&\int_{\Omega \times D}\!\! M\, \mathcal{F}^L(\beta^L(\hat\psi^0) + \alpha) \dq \dx
= \int_{\mathfrak{A}_{L,\alpha} \cup \mathfrak{B}_{L,\alpha}} M\, \mathcal{F}^L(\beta^L(\hat\psi^0) + \alpha) \dq \dx,
\end{eqnarray*}
where
\begin{eqnarray*}
\mathfrak{A}_{L,\alpha}&:=& \{(\xt,\qt) \in \Omega \times D\,:\, 0 \leq \beta^L(\hat\psi^0(\xt,\qt)) \leq L - \alpha\},\\
\mathfrak{B}_{L,\alpha}&:=& \{(\xt,\qt) \in \Omega \times D\,:\, L-\alpha < \beta^L(\hat\psi^0(\xt,\qt)) \leq L\}.
\end{eqnarray*}
We begin by noting that
\[ \int_{\mathfrak{A}_{L,\alpha}} M\, \mathcal{F}^L(\beta^L(\hat\psi^0) + \alpha) \dq \dx = \int_{\mathfrak{A}_{L,\alpha}} M\, \mathcal{F}(\beta^L(\hat\psi^0) + \alpha) \dq \dx.\]
For the integral over $\mathfrak{B}_{L,\alpha}$ we have
\begin{eqnarray*}
&&\int_{\mathfrak{B}_{L,\alpha}} M \mathcal{F}^L(\beta^L(\hat\psi^0) + \alpha) \dq \dx\\
&&= \int_{\mathfrak{B}_{L,\alpha}} M \left[\frac{(\beta^L(\hat\psi^0) + \alpha)^2 - L^2}{2L} +
(\beta^L(\hat\psi^0) + \alpha)(\log L - 1) + 1 \right] \dq \dx\\
&&\leq  \int_{\mathfrak{B}_{L,\alpha}} M \left[\frac{(L + \alpha)^2 - L^2}{2L} +
(\beta^L(\hat\psi^0) + \alpha)(\log (\beta^L(\hat\psi^0) +\alpha) - 1) + 1 \right] \dq \dx\\
&&= \alpha\left(1 + \frac{\alpha}{2L}\right) \int_{\mathfrak{B}_{L,\alpha}} M \dq \dx +
\int_{\mathfrak{B}_{L,\alpha}} M \mathcal{F}(\beta^L(\hat\psi^0) + \alpha) \dq \dx\\
&& \leq \frac{3}{2}\alpha |\Omega| + \int_{\mathfrak{B}_{L,\alpha}} M \mathcal{F}(\beta^L(\hat\psi^0) + \alpha) \dq \dx.
\end{eqnarray*}
Thus we have shown that
\begin{equation}\label{before-two}
{\rm T}_5(\alpha) \leq \frac{3}{2}\alpha |\Omega| + \int_{\Omega \times D} M\,\mathcal{F}(\beta^L(\hat\psi^0) + \alpha) \dq \dx.
\end{equation}
Now, there are two possibilities:
\begin{itemize}
\item[\textit{Case 1.}] If $\beta^L(\hat\psi^0) + \alpha \leq 1$, then $0 \leq \beta^L(\hat\psi^0) \leq 1 - \alpha$. Since $L>1$ it follows that $0 \leq \beta^L(s) \leq 1$ if, and only if, $\beta^L(s)=s$. Thus we deduce that in this case $\beta^L(\hat\psi^0) = \hat\psi^0$, and therefore $0 \leq \mathcal{F}(\beta^L(\hat\psi^0) + \alpha) = \mathcal{F}(\hat\psi^0 + \alpha)$.
\item[\textit{Case 2.}] Alternatively, if $1<\beta^L(\hat\psi^0) + \alpha$, then, on noting that
$\beta^L(s) \leq s$ for all $s \in [0,\infty)$, it follows that $1 < \beta^L(\hat\psi^0) + \alpha \leq
\hat\psi^0 + \alpha$. However the function $\mathcal{F}$ is strictly monotonic increasing on the interval $[1,\infty)$, which then implies that $0 = \mathcal{F}(1) < \mathcal{F}(\beta^L(\hat\psi^0) + \alpha) \leq \mathcal{F}(\hat\psi^0 + \alpha)$.
\end{itemize}
The conclusion we draw is that, either way, $0 \leq \mathcal{F}(\beta^L(\hat\psi^0) + \alpha) \leq \mathcal{F}(\hat\psi^0 + \alpha)$. Hence,
\begin{equation}\label{bound-on-t5}
{\rm T}_5(\alpha) \leq \frac{3}{2}\alpha |\Omega| + \int_{\Omega \times D} M\, \mathcal{F}(\hat\psi^0 + \alpha) \dq \dx.
\end{equation}
Substituting \eqref{bound-on-t5} into \eqref{eq:energy-u+psi2} thus yields, for $t=t_n$ and $n \in \{1,\dots, N\}$,
\begin{eqnarray}\label{eq:energy-u+psi3}
&&\|\uta^{\Delta t, +}(t)\|^2 + \frac{1}{\Delta t} \int_0^t \|\uta^{\Delta t, +} - \uta^{\Delta t,-}\|^2
\dd s + \left(\nu - \alpha\,\frac{2\lambda\, b\, k}{a_0}\right)  \int_0^t \|\nabxtt \uta^{\Delta t, +}(s)\|^2 \dd s\nonumber\\%[2ex]
&&
%\hspace{12mm}
+ \,2k\int_{\Omega \times D}\!\! M\, \mathcal{F}(\psia^{\Delta t, +}(t) + \alpha) \dq \dx
%\nonumber\\[2ex]
%&&\hspace{12mm}
+ \frac{k}{\Delta t\, L}
\int_0^t \int_{\Omega \times D}\!\! M\, (\psia^{\Delta t, +} - \psia^{\Delta t, -})^2 \dq \dx \dd s
\nonumber \\%[2ex]
&&\hspace{12mm} +\, 2k\,\varepsilon \int_0^t \int_{\Omega \times D} M\,
\frac{|\nabx \psia^{\Delta t, +} |^2}{\psia^{\Delta t, +} + \alpha} \dq \dx \dd s \nonumber\\%[3ex]
&&\hspace{12mm} +\, \frac{a_0 k}{2\lambda}  \int_0^t \int_{\Omega \times D}M\,\frac{|\nabq \psia^{\Delta t, +}|^2}{\psia^{\Delta t, +} + \alpha} \,\dq \dx \dd s\nonumber\\%[2ex]
&&\leq \|\ut_0\|^2 + \frac{1}{\nu}\!\int_0^t\!\!\|\ft^{\Delta t,+}(s)\|^2_{(H^1_0(\Omega))'}\! \dd s \nonumber\\%[2ex]
&&\hspace{12mm} +~ 3\alpha k  |\Omega| + 2k\! \int_{\Omega \times D}\!\! M\, \mathcal{F}(\hat\psi^0 + \alpha) \dq \dx.
\end{eqnarray}
The key observation at this point is that the right-hand side of \eqref{eq:energy-u+psi3} is
completely independent of the cut-off parameter $L$.

We shall tidy up the bound \eqref{eq:energy-u+psi3} by passing to the limit $\alpha \rightarrow 0_+$.
The first $\alpha$-dependent term on the right-hand side of \eqref{eq:energy-u+psi3} trivially converges to $0$ as $\alpha \rightarrow 0_+$; concerning the second $\alpha$-dependent term,
Lebesgue's dominated convergence theorem implies that
\[ \lim_{\alpha \rightarrow 0_+} \int_{\Omega \times D} M\, \mathcal{F}(\hat\psi^0 + \alpha) \dq \dx = \int_{\Omega \times D} M\, \mathcal{F}(\hat\psi^0) \dq \dx.
\]
Similarly, we can easily pass to the limit on the left-hand side of \eqref{eq:energy-u+psi3}.
By applying Fatou's lemma to the fourth, sixth and seventh term on the left-hand side of \eqref{eq:energy-u+psi3} we get, for $t = t_n$, $n \in \{1,\dots,N\}$, that
\begin{eqnarray*}
&&\hspace{-6.15cm} \mbox{lim inf}_{\alpha \rightarrow 0_+}\int_{\Omega \times D} M\, \mathcal{F}(\psia^{\Delta t, +} (t) + \alpha) \geq \int_{\Omega \times D} M\, \mathcal{F}(\psia^{\Delta t, +} (t)) \dq \dx,
\\~\\%~\\
\mbox{lim inf}_{\alpha \rightarrow 0_+} \int_0^t\! \int_{\Omega \times D}\!\! M\,
\frac{|\nabx \psia^{\Delta t, +} |^2}{\psia^{\Delta t, +}  + \alpha} \dq \dx \dd s &\geq & \int_0^t \!\int_{\Omega \times D} M\,
\frac{|\nabx\psia^{\Delta t, +}  |^2}{\psia^{\Delta t, +}} \dq \dx \dd s\\
&=& 4\int_0^t \!\int_{\Omega \times D}\!\! M\,
\big|\nabx \sqrt{\psia^{\Delta t, +} }\big|^2 \dq \dx \dd s,~~~~~
\\~\\%~\\
\mbox{lim inf}_{\alpha \rightarrow 0_+} \int_0^t \!\int_{\Omega \times D}\!\! M\,
\frac{|\nabq \psia^{\Delta t, +} |^2}{\psia^{\Delta t, +}  + \alpha} \dq \dx \dd s &\geq& \int_0^t \!\int_{\Omega \times D} M\,
\frac{|\nabq \psia^{\Delta t, +}|^2}{\psia^{\Delta t, +} } \dq \dx \dd s\\
&=& 4\int_0^t \!\int_{\Omega \times D} \!\!M\, \big|\nabq \sqrt{\psia^{\Delta t, +} }\big|^2 \dq \dx \dd s.~~~~~
\end{eqnarray*}
Thus, after passage to the limit $\alpha\rightarrow 0_+$, on recalling \eqref{inidata-1}, we have, for all $t=t_n$, $n \in \{1,\dots, N\}$, that
\begin{eqnarray}
&&\hspace{-2mm}\|\uta^{\Delta t, +}(t)\|^2 + \frac{1}{\Delta t} \int_0^t \|\uta^{\Delta t, +} - \uta^{\Delta t,-}\|^2
\dd s + \nu \int_0^t \|\nabxtt \uta^{\Delta t, +}(s)\|^2 \dd s\nonumber\\
&&\hspace{4mm}+ \,2k\int_{\Omega \times D}\!\! M\, \mathcal{F}(\psia^{\Delta t, +}(t)) \dq \dx + \frac{k}{\Delta t\, L}
\int_0^t \int_{\Omega \times D}\!\! M\, (\psia^{\Delta t, +} - \psia^{\Delta t, -})^2 \dq \dx \dd s
\nonumber \\
&&\hspace{6mm}\quad +\, 8k\,\varepsilon \int_0^t \int_{\Omega \times D} M\,
|\nabx \sqrt{\psia^{\Delta t, +}} |^2 \dq \dx \dd s\nonumber\\
&&\hspace{8mm}\quad\quad +\, \frac{2a_0 k}{\lambda}  \int_0^t \int_{\Omega \times D}M\, |\nabq \sqrt{\psia^{\Delta t, +}}|^2 \,\dq \dx \dd s\nonumber\\
&&\hspace{0mm}\leq \|\ut_0\|^2 + \frac{1}{\nu}\int_0^t\|\ft^{\Delta t,+}(s)\|^2_{(H^1_0(\Omega))'} \dd s + 2k \int_{\Omega \times D} M\, \mathcal{F}(\hat\psi^0) \dq \dx\label{penultimate-line}\\
&&\hspace{0mm} \leq \|\ut_0\|^2 + \frac{1}{\nu}\int_0^T\|\ft(s)\|^2_{(H^1_0(\Omega))'} \dd s + 2k \int_{\Omega \times D} M\, \mathcal{F}(\hat\psi_0) \dq \dx =:[{\sf B}(\ut_0,\ft, \hat\psi_0)]^2, \nonumber\\
%~~~~~~~~
\label{eq:energy-u+psi-final2}
\end{eqnarray}
where, in the last line, we used \eqref{inidata-1} to bound the third term in \eqref{penultimate-line},
and that $t \in [0,T]$ together with the definition \eqref{fn} of $\ft^{\Delta t,+}$ to bound the second term.

We select $\varphi = \chi_{[0,t]}\,\rho_{\epsilon,L}^{\Delta t,+}$ as test function in \eqref{zetacon}, with $t$ chosen as $t_n$ and $n \in \{1,\dots,N\}$.
Then, similarly to (\ref{eq:energy-u}), we deduce, with $t = t_n$, that
\begin{align}
&\|\rho_{\epsilon,L}^{\Delta t,+}(t)\|^2 + \frac{1}{\Delta t}\int_0^t  \|\rho_{\epsilon,L}^{\Delta t,+}(s) - \rho_{\epsilon,L}^{\Delta t,-}(s)\|^2 \dd s
+ 2\epsilon \int_0^t \|\nabx \rho_{\epsilon,L}^{\Delta t,+}(s)\|^2 \dd s
\nonumber \\
& \hspace{2.5in}
\leq \left\|\int_D \beta^L(\hat \psi^0) \dq \right\|^2
\leq |\Omega|,
\label{eq:energy-zeta}
\end{align}
where we have noted (\ref{simpid}), (\ref{tripid}) and that
$\beta^L(\hat \psi^0) \in \hat Z_2$.

\subsection{$L$-independent bounds on the time-derivatives}
\label{Lindep-time}

Next, we derive $L$-independent bounds on the time-derivatives of the functions $\uta^{\Delta t}$, $\psia^{\Delta t}$ and $\rho_{\epsilon,L}^{\Delta t}$.
We begin by bounding the time-derivative of $\psia^{\Delta t}$ using \eqref{eq:energy-u+psi-final2}; we shall then bound the time-derivatives of
$\rho_{\epsilon,L}^{\Delta t}$ and $\uta^{\Delta t}$ in a similar manner.

\subsubsection{$L$-independent bound on the time-derivative of $\psia^{\Delta t}$}
\label{sec:time-psia}

It follows from \eqref{eqpsincon} that
\begin{align}
\label{eq:weaka2bound}
&\left|\,\int_{0}^T\int_{\Omega \times D} M\, \frac{\partial \psia^{\Delta t}}{\partial t}\, \hat \varphi \dq \dx \dd t\,\right|
\leq  \left|\,\varepsilon \int_{0}^T \int_{\Omega \times D} M\,
\nabx \psia^{\Delta t,+} \cdot \nabx \hat\varphi \dq \dx \dt \,\right|\nonumber\\
&\quad\qquad + \left|\,\int_{0}^T \int_{\Omega \times D} M\,
\uta^{\Delta t,-}\,\psia^{\Delta t,+} \cdot\, \nabx \hat \varphi \,\dq \dx \dt\,\right| \nonumber\\
&\quad\qquad + \left|\,\frac{1}{2\,\lambda}\int_{0}^T \int_{\Omega \times D}M\,
\sum_{i=1}^K \sum_{j=1}^K A_{ij} \,\nabqj \psia^{\Delta t,+}
\cdot\, \nabqi \hat \varphi  \,\dq \dx \dt\,\right| \nonumber \\
&\quad\qquad  + \left|\, \int_{0}^T \int_{\Omega \times D}M\, \sum_{i=1}^K \left[
\sigtt(\uta^{\Delta t,+}) \,\qt_i\right]\,\beta^L(\psia^{\Delta t,+})\,\cdot\, \nabqi
\hat \varphi \,\dq \dx \dt \,\right|
\nonumber\\
&\quad =: {\rm S}_1 + {\rm S_2} + {\rm S_3} + {\rm S_4}
\qqquad \forall \hat \varphi \in L^1(0,T; \hat{X}).
\end{align}

We proceed to bound each of the terms ${\rm S}_1, \dots, {\rm S}_4$, bearing in mind (cf. the last sentence in the statement of Lemma \ref{conv}) that
\begin{subequations}
\begin{align}\label{properties-a}
&\psia^{\Delta t,+} \geq 0 \quad \mbox{a.e. on $\Omega \times D \times [0,T]$}, \qquad \qquad  \int_D M(\qt) \dq = 1,&\\
& 0 \leq \int_D M(\qt) \psia^{\Delta t,+}(\xt,\qt, t)\dq \leq 1\qquad\mbox{~for a.e. $(x,t) \in \Omega \times D$}.& \label{properties-b}
\end{align}
\end{subequations}
We shall use throughout the rest of this section test functions $\hat\varphi$ such that
\begin{equation}
\hat\varphi \in L^2(0,T; H^1(\Omega;L^\infty(D))\cap L^2(\Omega; W^{1,\infty}(D))).
\end{equation}

We begin by considering ${\rm S}_1$; noting (\ref{properties-a},b)
and \eqref{eq:energy-u+psi-final2}, we have that
\begin{eqnarray}
{\rm S}_1 &=& 2 \varepsilon \left|\int_{0}^T \int_{\Omega \times D} M\,\sqrt{\psia^{\Delta t,+}}\,\,\nabx \sqrt{\psia^{\Delta t,+}} \cdot \nabx \hat\varphi \dq \dx \dt \right|\nonumber\\
&\leq& 2 \varepsilon\int_0^T\!\!\! \int_\Omega\! \left[\left(\int_D M\psia^{\Delta t,+} \dq\right)^{\!\!\frac{1}{2}}\!
\left(\int_D M \bigg|\nabx \sqrt{\psia^{\Delta t,+}}\bigg|^2
\dq \right)^{\!\!\frac{1}{2}}\! \|\nabx \hat\varphi \|_{L^\infty(D)}\!\right]\!\!\dx \dd t \nonumber\\
&\leq& \sqrt{\frac{\varepsilon}{2k}}\,\left(\!8 k \varepsilon \!\int_0^T\!\!\! \int_{\Omega\times D}\!\! M  \bigg|\nabx \sqrt{\psia^{\Delta t,+}}\bigg|^2\! \dq \dx \dt \right)^{\!\!\frac{1}{2}}\! \left(\int_0^T\!\!\! \int_\Omega \|\nabx\hat\varphi \|^2_{L^\infty(D)}\! \dx \dd t\!\right)^{\!\!\frac{1}{2}}. \nonumber
\end{eqnarray}
Hence, by \eqref{eq:energy-u+psi-final2} with $t=t_N=T$,
\begin{equation}\label{eq:T1bound}
{\rm S}_1 \leq \sqrt{\frac{\varepsilon}{2k}}\, {\sf B}(\ut_0,\ft, \hat\psi_0)\, \left(\int_0^T \int_\Omega \|\nabx\hat\varphi \|^2_{L^\infty(D)} \dx \dd t\right)^{\!\!\frac{1}{2}}.
\end{equation}

Next, we consider the term ${\rm S}_2$:
\begin{eqnarray}
{\rm S}_2 &\leq& \int_0^T \int_\Omega |\uta^{\Delta t,-}| \left(\int_D M \psia^{\Delta t,+} \dq\right) \|\nabx\hat\varphi\|_{L^\infty(D)} \dx \dd t\nonumber\\
&\leq& \left(\int_0^T \int_\Omega |\uta^{\Delta t,-}|^2 \dx \dt \right)^{\!\!\frac{1}{2}} \left(\int_0^T\int_\Omega \|\nabx
\hat\varphi\|^2_{L^\infty(D)} \dx \dd t \right)^{\!\!\frac{1}{2}}\nonumber\\
& \leq & C_{{\sf P}}(\Omega) \left(\int_0^T \int_\Omega |\nabxtt\uta^{\Delta t,-}|^2 \dx \dt \right)^{\!\!\frac{1}{2}}\! \left(\int_0^T\int_\Omega \|\nabx
\hat\varphi\|^2_{L^\infty(D)} \dx \dd t \right)^{\!\!\frac{1}{2}}\!\!,~~~~~~~~~\label{Poinc}
\end{eqnarray}
where $C_{{\sf P}}(\Omega)$ denotes the (positive) constant appearing in the  Poincar\'{e} inequality $\|\uta^{\Delta t,-}\| \leq C_{{\sf P}}(\Omega)\, \|\nabxtt \uta^{\Delta t,-}\|$ on $\Omega$, with $\uta^{\Delta t,-} \in \Vt\subset \Ht^1_0(\Omega)$.
On recalling \eqref{eq:energy-u+psi-final2}, the definitions of $\uta^{\Delta t,+}$ and $\uta^{\Delta t,-}$ from \eqref{upm}, and noting
\eqref{idatabd}, we have that
\begin{eqnarray}\label{bound-a}
\int_0^T\!\!\! \int_\Omega |\nabxtt\uta^{\Delta t,-}|^2 \dx \dt  &=& \int_0^T \|\nabxtt\uta^{\Delta t,-}\|^2 \dd t
\nonumber\\ &=& \Delta t \|\nabxtt\ut^0\|^2 + \int_0^{T-\Delta t} \|\nabxtt \uta^{\Delta t, +} \|^2 \dt\nonumber\\
& \leq &\|\ut_0\|^2 + \int_0^T\|\nabxtt \uta^{\Delta t, +} \|^2 \dt \nonumber\\
& \leq & \left(1 + {\textstyle \frac{1}{\nu}}\right) [{\sf B}(\ut_0, \ft, \hat\psi_0)]^2.~~~~~~~~~
\end{eqnarray}
Therefore,
\begin{eqnarray}\label{extraS2}
{\rm S}_2 \leq C_{{\sf P}}(\Omega)\left(1 + {\textstyle \frac{1}{\nu}}\right)^{\frac{1}{2}} {\sf B}(\ut_0, \ft, \hat\psi_0) \left(\int_0^T\int_\Omega \|\nabx
\hat\varphi\|^2_{L^\infty(D)} \dx \dd t \right)^{\frac{1}{2}}.
\end{eqnarray}
Alternatively, directly from the second line of \eqref{Poinc}, we have that
\begin{eqnarray}\label{without}
{\rm S}_2 &\leq& \sqrt{T} \left(\mbox{ess.sup}_{t \in [0,T]} \int_\Omega |\uta^{\Delta t,-}|^2 \dx \right)^{\frac{1}{2}} \left(\int_0^T\int_\Omega \|\nabx
\hat\varphi\|^2_{L^\infty(D)} \dx \dd t \right)^{\frac{1}{2}}\!\!.~~~~~~~
\end{eqnarray}
Similarly as above,
\begin{eqnarray}\label{bound-b}
&&\mbox{ess.sup}_{t \in [0,T]}\! \int_\Omega |\uta^{\Delta t,-}|^2 \dx  = \mbox{ess.sup}_{t \in [0,T]} \|\uta^{\Delta t,-}(t)\|^2\nonumber\\
&& \hspace{2cm} = \max\left(\|\ut_0\|^2 , \mbox{ess.sup}_{t \in (0,T-\Delta t]}\|\uta^{\Delta t,+}(t)\|^2\right)\nonumber\\
&& \hspace{2cm} \leq  \max\left(\|\ut_0\|^2 , \mbox{ess.sup}_{t \in [0,T]}\|\uta^{\Delta t,+}(t)\|^2\right)\nonumber\\
&& \hspace{2cm} \leq [{\sf B}(\ut_0, \ft, \hat\psi_0)]^2.~~~~~~~~~
 \end{eqnarray}
Combining \eqref{extraS2}, \eqref{without} and \eqref{bound-b},
we have that
\begin{eqnarray}\label{eq:T2bound}
\!\!\!\!\!\!\!\!\!\!\!{\rm S}_2 \leq \min\left(\!C_{{\sf P}}(\Omega)\, \left(1 + {\textstyle{\frac{1}{\nu}}}\right)^{\frac{1}{2}} ,
\sqrt{T}\right) {\sf B}(\ut_0,\ft, \hat\psi_0)\!
\left(\int_0^T\!\!\int_\Omega \|\nabx
\hat\varphi\|^2_{L^\infty(D)} \dx \dd t \right)^{\frac{1}{2}}\!\!\!.
\end{eqnarray}

We are ready to consider ${\rm S}_3$; we have that
\begin{eqnarray*}
{\rm S}_3 & = & \frac{1}{2\,\lambda} \left|\int_{0}^T \int_{\Omega \times D} M\,2 \sqrt{\psia^{\Delta t, +}}\, \sum_{i=1}^K \sum_{j=1}^K A_{ij} \nabqj \sqrt{\psia^{\Delta t, +}} \cdot \nabqi \hat\varphi \,\dq \dx \dt\right|\\
&\leq& \frac{1}{\lambda}\!\left(\sum_{i,j=1}^K A_{ij}^2\right)^{\!\!\frac{1}{2}}\!\int_0^T\!\! \int_{\Omega \times D}
M \sqrt{\psia^{\Delta t, +}} \left(\sum_{i,j=1}^K \bigg|\nabqj \sqrt{\psia^{\Delta t, +}}\bigg|^2\, |\nabqi \hat\varphi|^2 \right)^{\!\!\frac{1}{2}}
\!\!\dq \dx \dd t\\
&=& \frac{1}{\lambda} |A|\, \int_0^T \int_{\Omega \times D}
M \sqrt{\psia^{\Delta t, +}} \left(\sum_{j=1}^K \bigg|\nabqj \sqrt{\psia^{\Delta t, +}}\bigg|^2\right)^{\!\!\frac{1}{2}} \left(\sum_{i=1}^K |\nabqi \hat\varphi|^2 \right)^{\!\!\frac{1}{2}}
\dq \dx \dd t\\
& = & \frac{1}{\lambda} |A|\, \int_0^T \int_{\Omega \times D}
M \sqrt{\psia^{\Delta t, +}}\, \bigg|\nabq \sqrt{\psia^{\Delta t, +}}\bigg|\, |\nabq \hat\varphi| \dq \dx \dd t\\
& \leq & \frac{1}{\lambda} |A|\, \int_0^T\!\! \int_\Omega \left(\int_D\! M \psia^{\Delta t, +} \dq\right)^{\!\!\frac{1}{2}}
\left(\int_D\! M\, \bigg |\nabq \sqrt{\psia^{\Delta t, +}}\bigg|^2 \dq \right)^{\!\!\frac{1}{2}} \|\nabq \hat\varphi\|_{L^\infty(D)}
\dx \dd t\\
& \leq & \frac{|A|}{\sqrt{2a_0 k \lambda}} \left(\!\frac{2a_0 k}{\lambda}\int_0^T\!\!\! \int_{\Omega \times D}
\!\!M\, \bigg|\nabq \sqrt{\psia^{\Delta t, +}} \bigg |^2\! \dq \dx \dd t\!\right)^{\!\!\frac{1}{2}} \!\left(\int_0^T \!\!\!\int_\Omega \|\nabq
\hat\varphi\|^2_{L^\infty(D)}\! \dx \dd t\!\right)^{\!\!\frac{1}{2}}\!\!.
\end{eqnarray*}
Thus, by \eqref{eq:energy-u+psi-final2},
\begin{eqnarray}\label{eq:T3bound}
{\rm S}_3 \leq \frac{|A|}{\sqrt{2a_0 k \lambda}}\,  {\sf B}(\ut_0,\ft, \hat\psi_0)
\left(\int_0^T\int_\Omega \|\nabq \hat\varphi\|^2_{L^\infty(D)} \dx \dd t \right)^{\!\!\frac{1}{2}}.
\end{eqnarray}

Finally, for term ${\rm S}_4$, recalling the notation $b:=|\bt|_1$ (cf. the paragraph before
\eqref{eq:lastterm}) together with the inequality $\beta^L(s) \leq s$ for $s \in \mathbb{R}_{\geq 0}$, we have that
\begin{eqnarray*}
{\rm S}_4 &\leq& \int_{0}^T \int_{\Omega \times D} M\,|\sigtt(\uta^{\Delta t, +})|\, \beta^L(\psia^{\Delta t, +})\, \sum_{i=1}^K \,|\qt_i|\, |\nabqi \hat \varphi| \,\dq \dx \dt\\
& \leq & \sqrt{b}\, \int_{0}^T \int_{\Omega \times D} M\,|\sigtt(\uta^{\Delta t, +})|\, \beta^L(\psia^{\Delta t, +})\, |\nabq \hat \varphi| \,\dq \dx \dt\\
& \leq & \sqrt{b}\, \int_{0}^T \int_{\Omega} \,|\sigtt(\uta^{\Delta t, +})|\, \left(\int_D M\beta^L(\psia^{\Delta t, +}) \dq\right) \, \|\nabq \hat \varphi\|_{L^\infty(D)} \dx \dt\\
& \leq & \sqrt{b}\, \int_{0}^T \int_{\Omega} \,|\sigtt(\uta^{\Delta t, +})|\, \left(\int_D M \psia^{\Delta t, +} \dq\right) \, \|\nabq \hat \varphi\|_{L^\infty(D)} \dx \dt\\
%\end{eqnarray*}
%\begin{eqnarray*}
& \leq & \sqrt{b}\, \int_{0}^T \int_{\Omega} \,|\sigtt(\uta^{\Delta t, +})| \, \|\nabq \hat \varphi\|_{L^\infty(D)} \dx \dt\\
& \leq & \sqrt{b}\, \left(\int_{0}^T \int_{\Omega} \,|\sigtt(\uta^{\Delta t, +})|^2 \dx \dd t\right)^{\frac{1}{2}} \left(\int_0^T \int_{\Omega} \|\nabq \hat\varphi\|^2_{L^\infty(D)} \dx \dd t\right)^{\frac{1}{2}}\\
&=& \sqrt{\frac{b}{\nu}}\, \left(\nu \int_{0}^T \int_{\Omega} \,|\nabxtt \uta^{\Delta t, +}|^2 \dx \dd t\right)^{\frac{1}{2}} \left(\int_0^T \int_{\Omega} \|\nabq \hat\varphi\|^2_{L^\infty(D)} \dx \dd t\right)^{\frac{1}{2}}.
\end{eqnarray*}
Hence, by \eqref{eq:energy-u+psi-final2},
\begin{eqnarray}\label{eq:T4bound}
{\rm S}_4 \leq  \sqrt{\frac{b}{\nu}}
\, {\sf B}(\ut_0,\ft, \hat\psi_0)\, \left(\int_0^T \int_{\Omega} \|\nabq \hat\varphi\|^2_{L^\infty(D)} \dx \dd t\right)^{\frac{1}{2}}.
\end{eqnarray}

Upon substituting the bounds on the terms ${\rm S}_1$ to ${\rm S}_4$ into \eqref{eq:weaka2bound},
with $\hat\varphi \in L^2(0,T; H^1(\Omega;L^\infty(D))\cap L^2(\Omega; W^{1,\infty}(D)))$, and
noting that the latter space is contained in $L^1(0,T;\hat{X})$ we deduce from \eqref{eq:weaka2bound}
that
\begin{eqnarray}\label{psi-time-bound}
&&\hspace{-0.5cm}\left|\,\int_{0}^T\int_{\Omega \times D} M\, \frac{\partial \psia^{\Delta t}}{\partial t}\,\hat \varphi
\dq \dx \dt\,\right|\nonumber\\
&& \leq {C}_{\ast}\, {\sf B}(\ut_0,\ft, \hat\psi_0)\,
\left(\int_0^T \int_{\Omega} \left[\|\nabx \hat\varphi\|^2_{L^\infty(D)} + \|\nabq \hat\varphi\|^2_{L^\infty(D)}\right] \dx \dd t\right)^{\frac{1}{2}}\!\!,
\end{eqnarray}
for any $\hat\varphi \in L^2(0,T; H^1(\Omega;L^\infty(D))\cap L^2(\Omega; W^{1,\infty}(D)))$, where $C_\ast$
denotes a positive constant (that can be computed by tracking the constants in \eqref{eq:T1bound}--\eqref{eq:T4bound}), which depends solely on $\epsilon$, $\nu$, $C_{\sf P}(\Omega)$, $T$, $|A|$, $a_0$, $k$, $K$, $\lambda$, $K$ and $b$.

We now consider the time-derivative of $\rho_{\epsilon,L}^{\Delta t}$. It follows from
(\ref{zetacon}), (\ref{eq:energy-zeta}), (\ref{properties-b}) and (\ref{bound-b}) that
\begin{align}
&\left|\,\int_0^T \int_\Omega
\frac{\partial \rho_{\epsilon,L}^{\Delta t}}{\partial t}\, \varphi \dx \dt\,\right|
\leq
\int_0^T \int_{\Omega} \left| \left[ \epsilon\, \nabx  \rho_{\epsilon,L}^{\Delta t, +} -
\ut_{\epsilon,L}^{\Delta t,-}  \,\rho_{\epsilon,L}^{\Delta t, +} \right]
\cdot \nabx \varphi \right| \dx  \dt
\nonumber \\
& \quad \leq \left[ \epsilon \left( \int_0^T \|\nabx \rho_{\epsilon,L}^{\Delta t, +}\|^2 \dt \right)^{\frac{1}{2}}
+ \mbox{ess.sup}_{t \in [0,T]} \|\rho_{\epsilon,L}^{\Delta t, +}\|_{L^\infty(\Omega)}
\left( \int_0^T \|\ut_{\epsilon,L}^{\Delta t, -}\|^2 \dt \right)^{\frac{1}{2}}
\right]
\nonumber \\
& \hspace{3in} \times
\left( \int_0^T \|\nabx \varphi\|^2 \dt \right)^{\frac{1}{2}}
\nonumber \\
& \quad \leq \left[ \epsilon \left(\frac{|\Omega|}{2}\right)^{\frac{1}{2}}
+ {\sf B}(\ut_0, \ft, \hat\psi_0)
\right]
\left( \int_0^T \|\nabx \varphi\|^2 \dt \right)^{\frac{1}{2}}
\qquad \forall \varphi \in L^2(0,T;H^1(\Omega)).
\label{zetacondtbd}
\end{align}

\subsubsection{$L$-independent bound on the time-derivative of $\uta^{\Delta t}$}
\label{sec:time-uta}

In this section we shall derive an $L$-independent bound on the time-derivative of $\uta^{\Delta t}$. Our starting point is \eqref{equncon}, from which we deduce that

\begin{align}\label{equncon1}
&\displaystyle
\left|\,\int_{0}^{T}\!\! \int_\Omega  \frac{\partial \utaeD}{\partial t}\cdot
\wt \dx \dt\,\right|
\nonumber
\\
&
\hspace{0.5cm} \leq \left|\,\int_{0}^T\!\! \int_{\Omega}
 \left[ (\utaeDm \cdot \nabx) \utaeDp \right]\,\cdot\,\wt \dx \dt \,\right|
+ \nu \left|\,\int_0^T\!\! \int_\Omega\,\nabxtt \utaeDp
: \wnabtt \dx \dt \,\right|
\nonumber
\\
&\hspace{1cm}+\left|\,\int_{0}^T  \langle \ft^{\Delta t,+}, \wt\rangle_{H^1_0(\Omega)} \dd t \,\right|
+ k\,\left|\,\sum_{i=1}^K \int_{0}^T\!\! \int_{\Omega}
\Ctt_i(M\,\hpsiae^{\Delta t,+}): \nabxtt
\wt \dx \dt \,\right|
\nonumber
\\
& \hspace{0.5cm} =: {\rm U_1} + {\rm U}_2 + {\rm U}_3 + {\rm U}_4 \hspace{1in}
\qquad \forall \wt \in L^1(0,T;\Vt).
\end{align}

On recalling from the discussion following \eqref{eqvn2} the definition of
$\Vt_\sigma$, we shall assume henceforth that
$$ \wt \in L^2(0,T;\Vt_\sigma),\ \sigma \geq \textstyle{\frac{1}{2}}d,\ \sigma>1.$$
Clearly, $L^2(0,T;\Vt_\sigma) \subset L^1(0,T;\Vt)$.
By Lemma 4.1 in Ch.\ 3 of Temam \cite{Temam}
%\footnote{Lemma 4.1 in Temam \cite{Temam} and its proof, which demand that $\sigma \geq d/2$
%($s\geq n/2$ in terms of the notation there) are not valid for $d=2$ and $\sigma = 1$, ($n=2$ and $s=1$ in terms of the notation there); this though is easily %fixed by demanding that $\sigma\geq d/2$ and $\sigma>1$ (viz. $s \geq n/2$ and $s>1$).},
and using \eqref{bound-b} and \eqref{eq:energy-u+psi-final2}, we have, with $\sigma \geq \textstyle{\frac{1}{2}}d$, $\sigma>1$, and $c(\Omega,d)$ a constant that only depends on $\Omega$ and $d$, that
\begin{eqnarray}\label{time-u-t1}
{\rm U}_1 & \leq & c(\Omega,d)\,\mbox{ess.sup}_{t \in [0,T]} \|\utaeDm\| \left(\int_0^T \|\nabx \utaeDp\|^2 \dd t\right)^{\!\!\frac{1}{2}} \left(\int_0^T\|\wt\|^2_{V_\sigma}\dd t\right)^{\!\!\frac{1}{2}}\nonumber\\
& \leq & c(\Omega,d)\,\sqrt{\frac{1}{\nu}} \, [{\sf B}(\ut_0, \ft, \hat\psi_0)]^2 \,
\left(\int_0^T\|\wt\|^2_{V_\sigma}\dd t\right)^{\!\!\frac{1}{2}}.
\end{eqnarray}

For term ${\rm U}_2$ we have,
\begin{eqnarray}\label{time-u-t2}
{\rm U}_2 & \leq & \sqrt{\nu} \left(\nu\int_0^T\!\!\|\nabxtt \utaeDp\|^2 \dt \right)^{\!\!\frac{1}{2}}
\left(\int_0^T \|\wnabtt\|^2\dt \right)^{\!\!\frac{1}{2}}\nonumber\\
& \leq & \sqrt{\nu} \, {\sf B}(\ut_0, \ft, \hat\psi_0) \, \left(\int_0^T \|\wnabtt\|^2 \dt \right)^{\!\!\frac{1}{2}}.
\end{eqnarray}

Concerning the term ${\rm U}_3$, on noting the definition of the norm $\|\cdot\|_{(H^1_0(\Omega))'}$ and
that thanks to \eqref{fn} we have
$\|\ft^{\Delta t,+}\|_{L^2(0,T;(H^1_0(\Omega))')} \leq \|\ft\|_{L^2(0,T;(H^1_0(\Omega))')},$
it follows that
\begin{eqnarray}\label{time-u-t3}
{\rm U}_3 &\leq&
\sqrt{\nu}\, {\sf B}(\ut_0, \ft, \hat\psi_0) \, \left(\int_0^T \|\wnabtt\|^2 \dt \right)^{\!\!\frac{1}{2}}.
\end{eqnarray}

Before we embark on the estimation of the term ${\rm U}_4$ we observe that
\begin{eqnarray}\label{U4a}
{\rm U}_4 & = & k \left|\,\int_0^T \int_\Omega  \left[\int_{D} M\,
\hpsiae^{\Delta t,+} \sum_{i=1}^K \qt_i\, \qt_i^{\rm T}\, U_i'\left(\textstyle{\frac{1}{2}}|\qt_i|^2\right)  : \nabxtt \wt \dq \right]\!\! \dx \dt\,\right|\nonumber\\
&=& k \left|\,\int_{0}^T \int_{\Omega}\left[\int_{D} M \sum_{i=1}^K (\nabxtt \wt)\qt_i \cdot
\nabqi \hpsiae^{\Delta t,+} \dq \right]\!\!\dx \dt\, \right|,
\end{eqnarray}
where we used the integration-by-parts formula \eqref{intbyparts} to transform the expression in the
square brackets in the first line into the expression in the square brackets in the second line.
Thus we have that
\begin{eqnarray}\label{U4b}
{\rm U}_4 & \leq & k \int_{0}^T \int_{\Omega\times D} M\,  |\nabxtt \wt| \left(\sum_{i=1}^K |\qt_i| \,
|\nabqi \hpsiae^{\Delta t,+} |\right) \dq \dx \dt\nonumber\\
& \leq &  k \int_{0}^T \int_{\Omega\times D} M\, |\nabxtt \wt|\, |\qt|\,
|\nabq \hpsiae^{\Delta t,+} | \dq \dx \dt\nonumber\\
& \leq &  k\,\sqrt{b} \int_{0}^T \int_{\Omega} |\nabxtt \wt|\,\left(
\int_D M\, |\nabq \hpsiae^{\Delta t,+} | \dq \right)\dx \dt\nonumber\\
&=& 2k\,\sqrt{b} \int_{0}^T \int_{\Omega} |\nabxtt \wt|\,\left(
\int_D M\, \sqrt{\psiae^{\Delta t,+}}\, \big|\nabq \sqrt{\hpsiae^{\Delta t,+} }\big| \dq\right) \dx \dt
\nonumber\\
&\leq& 2k\,\sqrt{b} \int_{0}^T \int_{\Omega} |\nabxtt \wt|\,\left(
\int_D M\, \big|\nabq \sqrt{\hpsiae^{\Delta t,+}}\big|^2 \dq\right)^{\!\!\frac{1}{2}} \dx \dt,
\end{eqnarray}
where in the transition to the last line we used the Cauchy--Schwarz inequality and
\eqref{properties-b}. Hence, by \eqref{eq:energy-u+psi-final2},
\begin{eqnarray}\label{time-u-t4}
{\rm U}_4 &\leq& 2k\,\sqrt{b}\left(\int_0^T\!\!
\int_{\Omega \times D} \!\!M\,  \big|\nabq \sqrt{\hpsiae^{\Delta t,+}}\big|^2 \dq \dx \dt\!\right)^{\!\!\frac{1}{2}}\! \left(\int_{0}^T\!\! \int_{\Omega} |\nabxtt \wt|^2 \dx \dt\! \right)^{\!\!\frac{1}{2}}\!\!\nonumber\\
&=&  \sqrt{\frac{2\lambda\, b\, k}{a_0}}\left(\frac{2 a_0 k}{\lambda} \int_0^T\!\!
\int_{\Omega \times D} \!\!M \big|\nabq \sqrt{\hpsiae^{\Delta t,+}}\big|^2 \dq \dx \dt\!\right)^{\!\!\frac{1}{2}}\! \left(\int_{0}^T \|\nabxtt \wt\|^2 \dt\! \right)^{\!\!\frac{1}{2}}
\nonumber\\
&\leq&  \sqrt{\frac{2\lambda\,b\,k}{a_0}}\,\,  {\sf B}(\ut_0,\ft, \hat\psi_0) \left(\int_{0}^T \|\nabxtt \wt\|^2 \dt\! \right)^{\!\!\frac{1}{2}}.
\end{eqnarray}
Collecting the bounds on the terms ${\rm U}_1$ to ${\rm U}_4$ and inserting them into \eqref{equncon1} yields
\begin{align}
&\displaystyle
\left|\,\int_{0}^{T}\!\! \int_\Omega  \frac{\partial \utaeD}{\partial t}\cdot
\wt \dx \dt\,\right|\leq C_{\ast \ast}\, \max\left([{\sf B}(\ut_0,\ft, \hat\psi_0)]^2, {\sf B}(\ut_0,\ft, \hat\psi_0)\right)\!
\left(\int_{0}^T\! \|\wt\|^2_{V_\sigma} \dt\! \right)^{\!\!\frac{1}{2}}\!\!,\nonumber\\
\label{u-time-bound}
\end{align}
for any $\wt \in L^2(0,T; \Vt_\sigma)$, $\sigma \geq \textstyle{\frac{1}{2}}d$, $\sigma>1$,
where $C_{\ast\ast}$ denotes a positive constant (that can be computed by tracking the constants in \eqref{time-u-t1}--\eqref{time-u-t4}), which depends solely on $\Omega$, $d$, $\nu$, $k$, $K$, $\lambda$, $a_0$ and $b$.

\section{Dubinski{\u\i}'s compactness theorem}
\label{sec:dubinskii}
\setcounter{equation}{0}

Having developed a collection of $L$-independent bounds in Sections \ref{Lindep-space} and \ref{Lindep-time},
we now describe the theoretical tool that will be used to set up a weak compactness argument using these
bounds: Dubinski{\u\i}'s compactness theorem in seminormed sets.

Let $\mathcal{A}$ be a linear space over the field $\mathbb{R}$ of real numbers, and suppose that $\mathcal{M}$ is a subset of $\mathcal{A}$ such that
\begin{equation}
\label{eq:property}
(\forall \varphi \in \mathcal{M})\; (\forall c \in \mathbb{R}_{\geq 0})\;\;\; c\, \varphi \in \mathcal{M}.
\end{equation}
In other words, whenever $\varphi$ is contained in $\mathcal{M}$, the ray through $\varphi$ from the origin of the linear space $\mathcal{A}$ is also contained in $\mathcal{M}$.  Note in particular that while any set $\mathcal{M}$ with property \eqref{eq:property} must contain the zero element of the linear space $\mathcal{A}$, the set
$\mathcal{M}$ need not be closed under summation. The linear space $\mathcal{A}$ will be referred to as the {\em ambient space} for $\mathcal{M}$.

Suppose further that each element $\varphi$ of a set $\mathcal{M}$ with property \eqref{eq:property}
is assigned a certain real number, denoted $[\varphi]_{\mathcal M}$, such that:
\begin{enumerate}
\item[(i)] $[\varphi]_{\mathcal M} \geq 0$; and $[\varphi]_{\mathcal M} = 0$ if, and only if, $\varphi=0$; and
\item[(ii)] $(\forall c \in \mathbb{R}_{\geq 0})\; [c\, \varphi]_{\mathcal M} = c\,[\varphi]_{\mathcal M}$.
\end{enumerate}
We shall then say that $\mathcal{M}$ is a {\em seminormed set}.

A subset $\mathcal{B}$ of a seminormed
set $\mathcal{M}$ is said to be {\em bounded} if there exists a positive constant $K_0$ such that $[\varphi]_{\mathcal M} \leq K_0$ for all $\varphi \in \mathcal{B}$.

A seminormed set $\mathcal{M}$ contained in a normed linear space $\mathcal{A}$ with norm $\|\cdot\|_{\mathcal A}$ is said to be {\em embedded in $\mathcal{A}$}, and we write $\mathcal{M} \hookrightarrow \mathcal{A}$, if
the inclusion map $i : \varphi \in \mathcal{M} \mapsto i(\varphi)=\varphi \in \mathcal{A}$ (which is, by definition, injective and positively 1-homogeneous, i.e. $i(c\, \varphi) = c\, i(\varphi)$ for all $c \in \mathbb{R}_{\geq 0}$ and all $\varphi \in \mathcal{M}$) is a bounded operator, i.e.
\[ (\exists K_0 \in \mathbb{R}_{>0})\; (\forall \varphi \in \mathcal{M})\;\;\; \|i(\varphi) \|_{\mathcal A} \leq K_0 [\varphi]_{\mathcal M}.
\]
The symbol $i(~)$ is usually omitted from the notation $i(\varphi)$, and $\varphi \in \mathcal{M}$ is simply identified with $\varphi \in \mathcal{A}$. Thus, a bounded subset of a seminormed set is also a bounded subset of the ambient normed linear space the seminormed set is embedded in.

The embedding of a seminormed set $\mathcal{M}$ into a normed linear space $\mathcal{A}$ is said to be {\em compact} if from any infinite, bounded set of elements of $\mathcal{M}$ one can extract a subsequence that converges in $\mathcal{A}$; we shall write $\mathcal{M} \hookrightarrow\!\!\!\rightarrow \mathcal{A}$ to denote that $\mathcal{M}$ is compactly embedded in $\mathcal{A}$.

Suppose that $T$ is a positive real number,
$\varphi$ maps the nonempty closed interval $[0,T]$ into a seminormed set $\mathcal{M}$,
and $p\in \mathbb{R}$, $p \geq 1$. We denote by $L^p(0,T; \mathcal{M})$ the set of all functions $\varphi\,:\, t \in [0,T] \mapsto \varphi(t) \in \mathcal{M}$ such that
\[ \mbox{$\left(\int_0^T [\varphi(t)]^p_{\mathcal M} \,{\rm d}t\right)^{1/p} < \infty$;}\]
$L^p(0,T; \mathcal{M})$ is then a seminormed set in the ambient linear space $L^p(0,T; \mathcal{A})$, with
\[\mbox{$[\varphi]_{L^p(0,T;\mathcal{M})} := \left( \int_0^T [\varphi(t)]^p_{\mathcal M} \,{\rm d}t\right)^{1/p}.  $}\]
We denote by $L^\infty(0,T; \mathcal{M})$ and $[\varphi]_{L^\infty(0,T;\mathcal{M})}$ the usual modifications of these definitions when $p=\infty$.

For two normed linear
spaces, $\mathcal{A}_0$ and $\mathcal{A}_1$,  we shall continue to denote by $\mathcal{A}_0 \hookrightarrow \mathcal{A}_1$ that $\mathcal{A}_0$ is (continuously) embedded in $\mathcal{A}_1$. The following theorem is due to Dubinski{\u\i} \cite{DUB}; see also \cite{BS-DUB}.

\begin{theorem}[Dubinski{\u\i} \cite{DUB}] \label{thm:Dubinski}
Suppose that $\mathcal{A}_0$ and $\mathcal{A}_1$ are Banach spaces, $\mathcal{A}_0 \hookrightarrow \mathcal{A}_1$, and $\mathcal{M}$ is a seminormed subset of $\mathcal{A}_0$ such that $\mathcal{M} \hookrightarrow\!\!\!\rightarrow
\mathcal{A}_0$. Consider the set
\[ \mathcal{Y}:= \left\{\varphi\,:\,[0,T] \rightarrow \mathcal{M}\,:\,
[\varphi]_{L^p(0,T;\mathcal M)} + \left\|\frac{{\rm d}\varphi}{{\rm d}t} \right\|_{L^{p_1}(0,T;\mathcal{A}_1)}
< \infty   \right\},
\]
where $1 \leq p \leq \infty$, $1 < p_1 \leq \infty$, $\|\cdot\|_{\mathcal{A}_1}$ is the norm of $\mathcal{A}_1$, and ${\rm d}\varphi/{\rm d}t$ is understood in the sense of $\mathcal{A}_1$-valued distributions
on the open interval $(0,T)$.
Then, $\mathcal{Y}$, with % is a seminormed set with seminorm
\[ [\varphi]_{\mathcal Y}:= [\varphi]_{L^p(0,T;\mathcal M)} + \left\|\frac{{\rm d}\varphi}{{\rm d}t} \right\|_{L^{p_1}(0,T;\mathcal{A}_1)},\]
is a seminormed set in $L^p(0,T;\mathcal{A}_0)\cap W^{1,p_1}(0,T;\mathcal{A}_1)$, and $\mathcal{Y}
\hookrightarrow\!\!\!\rightarrow
 L^p(0,T; \mathcal{A}_0)$.
\end{theorem}

We note that
in Dubinski{\u\i} \cite{DUB} the author writes $\mathbb{R}$
instead of our $\mathbb{R}_{\geq 0}$ in
\eqref{eq:property} and property (ii). The proof of Thm.\ 1 in Dubinski{\u\i}'s
work, stated as Theorem \ref{thm:Dubinski}
above, reveals however that the result remains valid with our weaker
notion of {\em seminormed set}, as (5.1) and
property (ii) are only ever used in the proof with $c \geq 0$; we refer to our paper \cite{BS-DUB} for details.
In the next section, we shall apply Dubinski{\u\i}'s theorem by selecting
\begin{eqnarray*}
&&\mathcal{A}_0 = L^1_M(\Omega \times D)\quad \mbox{with norm}\quad
\|\hat\varphi\|_{\mathcal{A}_0} := \int_{\Omega \times D} M(\qt)\, |\hat\varphi (\xt, \qt)| \dd \xt \dd \qt% < \infty \bigg\}
\end{eqnarray*}
and
\begin{eqnarray*}
&&\!\!\!\mathcal{M} = \bigg\{ \hat\varphi \in \mathcal{A}_0\,: \hat\varphi \geq 0 \quad \mbox{with} \\
&&\qquad  \int_{\Omega \times D} M(\qt)\left( \left|\nabx \sqrt{\hat\varphi(\xt,\qt)}\right|^2 + \left|\nabq \sqrt{\hat\varphi(\xt,\qt)}\right|^2\right)\dd \xt \dd \qt < \infty \bigg\},
\end{eqnarray*}
and, for $\hat\varphi \in \mathcal{M}$,  we define
\[ [\hat\varphi]_{\mathcal M}:= \|\hat\varphi\|_{\mathcal A_0} + \int_{\Omega \times D} M(\qt)\left( \left|\nabx \sqrt{\hat\varphi(\xt,\qt)}\right|^2 + \left|\nabq \sqrt{\hat\varphi(\xt,\qt)}\right|^2\right)\dd \xt \dd \qt.
\]
Note that $\mathcal{M}$ is a seminormed subset of the ambient space $\mathcal{A}_0$. Finally, we put
\[ \mathcal{A}_1 := M^{-1} (H^{s}(\Omega \times D))' := \{\hat\varphi\,:\,M \hat\varphi \in (H^{s}(\Omega \times D))'\},\]
equipped with the norm $\|\hat\varphi\|_{\mathcal A_1} := \|M \hat \varphi\|_{(H^{s}(\Omega \times D))'}$,
and take $s>1+\frac{1}{2}(K+1)d$. Our choice of $\mathcal{A}_1$ is motivated by the fact that, thanks to the Sobolev embedding theorem on $\Omega \times D
\subset \mathbb{R}^{d \times Kd} \cong \mathbb{R}^{(K+1)d}$, the final factor
on the right-hand side of \eqref{psi-time-bound} can be further bounded from above by
a constant multiple of $\|\hat\varphi\|_{L^2(0,T; H^s(\Omega \times D))}$, with $s>1+\frac{1}{2}(K+1)d$.
For such $s$ it then follows, again from the Sobolev embedding theorem that, for any $\hat\varphi \in \mathcal{A}_0$,
\[\|\hat\varphi\|_{\mathcal A_1} = \sup_{\!\!\!\chi \in H^s(\Omega \times D)}\! \frac{|(M\hat\varphi, \chi)|}{\|\chi\|_{H^s(\Omega \times D)}}
\leq \sup_{\!\!\!\chi \in H^s(\Omega \times D)}\! \frac{\|\hat\varphi\|_{\mathcal A_0} \|\chi\|_{L^\infty(\Omega \times D)}}{\|\chi\|_{H^s(\Omega \times D)}} \leq K_0 \|\hat\varphi\|_{\mathcal{A}_0},
\]
where $K_0$ is any positive constant that is greater than or equal to the constant $K_s$, the norm of the
continuous linear operator corresponding to the Sobolev embedding $(H^{s}(\Omega \times D)
\hookrightarrow )
  H^{s-1}(\Omega \times D) \hookrightarrow
  L^\infty(\Omega \times D)$, $s>1+ \frac{1}{2}(K+1)d$. Hence,
  %on noting that $L^1_M(\Omega\times D) \equiv M^{-1} \,L^1(\Omega \times D)
%\hookrightarrow M^{-1} \,(L^\infty(\Omega \times D))'$,
%we have that
  $\mathcal{A}_0
  \hookrightarrow
   \mathcal{A}_1$.

Trivially, $\mathcal{M} \hookrightarrow \mathcal{A}_0$.
We shall show that in fact $\mathcal{M} \hookrightarrow\!\!\!\rightarrow \mathcal{A}_0$. Suppose, to this end, that $\mathcal{B}$ is
an infinite, bounded subset of $\mathcal{M}$. We can assume without loss of generality that $\mathcal{B}$ is the
infinite sequence $\{\hat\varphi_n\}_{n \geq 1}
\subset
\mathcal{M}$ with $[\hat\varphi_n]_{\mathcal M}
\leq K_0$ for all $n \geq 1$, where $K_0$ is a fixed positive constant. We define $\hat\rho_n:=
\sqrt{\hat\varphi_n}$ and note that $\hat\rho_n \geq 0$ and $\hat\rho_n \in H^1_M(\Omega \times D)$
for all $n\geq 1$, with
\[ \|\hat\rho_n\|^2_{H^1_M(\Omega \times D)} = [\hat\varphi_n]_{\mathcal M} \leq K_0 \qquad
\forall n \geq 1.\]
Since $H^1_M(\Omega \times D)$ is compactly embedded in $L^2_M(\Omega \times D)$
(see Appendix D for a proof of this),
%in the extended version of this paper \cite{BS2010} for a proof of this),
we deduce that
the sequence $\{\hat\rho_n\}_{n \geq 1}$ has a subsequence $\{\hat\rho_{n_k}\}_{k \geq 1}$ that is
convergent in $L^2_M(\Omega \times D)$; denote the limit of this subsequence by $\hat\rho$;
$\hat \rho \in L^2_M(\Omega \times D)$. Then, since a subsequence of
the sequence $\{\hat\rho_{n_k}\}_{k \geq 1}$
also converges to $\hat\rho$ a.e. on $\Omega \times D$ and each $\hat\rho_{n_k}$ is nonnegative
on $\Omega \times D$, the same is true of $\hat\rho$. Now, define $\hat\varphi:= \hat\rho^{~\!2}$,
and note that $\hat\varphi \in L^1_M(\Omega \times D)$. Clearly,
\begin{eqnarray*}
&&\| \hat\varphi_{n_k} - \hat\varphi \|_{L^1_M(\Omega \times D)} = \int_{\Omega \times D}
M\, (\,\hat\rho_{n_k} + \hat\rho\,)\, |\,\hat\rho_{n_k} - \hat\rho\,| \dd \xt \dd \qt\\
&&\qquad \leq \|\, \hat\rho_{n_k} + \hat\rho \,\|_{L^2_M(\Omega \times D)}\, \|\, \hat\rho_{n_k}
- \hat\rho \,\|_{L^2_M(\Omega \times D)}\\
&& \qquad \leq \left(\|\, \hat\rho_{n_k}\,\|_{L^2_M(\Omega \times D)} + \|\,\hat\rho \,\|_{L^2_M(\Omega \times D)}\right) \|\, \hat\rho_{n_k}
- \hat\rho \,\|_{L^2_M(\Omega \times D)}.
\end{eqnarray*}
As $\{\hat\rho_{n_k}\}_{k \geq 1}$ converges to $\hat\rho$ in $L^2_M(\Omega \times D)$, and is therefore
also a bounded sequence in $L^2_M(\Omega \times D)$, it follows from the last inequality that $\{\hat\varphi_{n_k}\}_{k \geq 1}$ converges to $\hat\varphi$ in $L^1_M(\Omega \times D) = \mathcal{A}_0$. This implies that
$\mathcal{M}$ is compactly embedded in $\mathcal{A}_0$;
hence the triple $\mathcal{M} \hookrightarrow\!\!\!\rightarrow \mathcal{A}_0 \hookrightarrow \mathcal{A}_1$ satisfies the conditions of Theorem \ref{thm:Dubinski}.

\begin{remark}\label{rem5.1}
In fact, there is a deep connection between $\mathcal{M}$ and the set of functions with finite relative
entropy on $D$; this can be seen by noting the {\em logarithmic Sobolev inequality}:
\begin{equation}\label{eq:logs0}
\hspace{-0mm}\int_{D} M(\qt)\, |\hat\rho(\qt)|^2 \log \frac{|\hat\rho(\qt)|^2}{\|\hat\rho\|^2_{L^2_M(D)}} \dd \qt
\leq \frac{2}{\kappa} \int_{D} M(\qt)\, \big|\nabq \hat\rho(\qt)\big|^2 \dd \qt \quad \forall \hat\rho \in H^1_M(D),
\end{equation}
with a constant $\kappa>0$; the inequality \eqref{eq:logs0} is known to hold whenever $M$ satisfies the
{\em Bakry--\'Emery condition}: ${\sf Hess}(- \log M(\qt)) \geq \kappa\, {{\sf Id}}$ on $D$,
asserting the logarithmic concavity of the Maxwellian on $D$, with the last inequality understood in the sense of symmetric
$Kd \times Kd$ matrices.
The inequality \eqref{eq:logs0} follows from inequality (1.3) in Arnold et al. \cite{ABD}, with the Maxwellian
$M$ extended by $0$ to the whole of $\mathbb{R}^{Kd}$ to define a probability measure on $\mathbb{R}^{Kd}$
supported on $D=D_1 \times \cdots \times D_K$.

The validity of the Bakry--\'Emery condition for the FENE Maxwellian, for example, is an easy consequence of the fact that
\begin{eqnarray}\label{bakry}
&&{\sf Hess}(- \log M(\qt)) = {\sf Hess}\left(\sum_{i=1}^K U_i\left({\textstyle \frac{1}{2}}|\qt_i|^2\right)\right)\nonumber\\
&&\qquad\qquad =  \mbox{\sf diag}\left({\sf Hess}\left(U_1({\textstyle \frac{1}{2}}|\qt_1|^2)\right), \dots, {\sf Hess}\left(U_K({\textstyle \frac{1}{2}}|\qt_K|^2)\right)\right),
\end{eqnarray}
for all $\qt := (\qt_1^{\rm T},\dots, \qt_K^{\rm T})^{\rm T} \in D_1 \times \cdots \times D_K = D$,
and the following lower bounds (cf. Knezevic \& S\"uli \cite{KS1}, Sec. 2.1) on the $d\times d$ Hessian matrices that are the diagonal blocks of
the $Kd \times Kd$ Hessian matrix ${\sf Hess}(- \log M(\qt))$:
\[ \xit^{\rm T}_i \left({\sf Hess}\left(U_i({\textstyle \frac{1}{2}}|\qt_i|^2)\right)\right) \xit_i \geq (1-|\qt_i|^2/b)^{-1}|\xit_i|^2 \geq |\xit_i|^2,\]
for all $\qt_i \in D_i$ and all $\xit_i \in \mathbb{R}^d$, $i = 1,\dots, K$. Hence,
\[\xit^{\rm T} {\sf Hess}(- \log M(\qt)) \xit \geq |\xit|^2\]
for all $\qt \in D$ and all $\xit \in \mathbb{R}^{Kd}$, yielding
\[{\sf Hess}(- \log M(\qt)) \geq {{\sf Id}}\qquad \forall \qt \in D;\]
i.e. $\kappa=1$.

More generally, we see from \eqref{bakry} that if $\qt_i\in D_i \mapsto U_i(\frac{1}{2}|\qt_i|^2)$ is strongly convex on $D_i$ for each $i=1,\dots,K$, then $M$ satisfies the Bakry--\'{E}mery condition on $D$.

On writing $\hat\varphi(\qt):= |\hat\rho(\qt)|^2\, (\,\geq 0)$ in \eqref{eq:logs0}, we then have that
\begin{equation} \label{eq:logs1}
\int_{D} M(\qt)\, \hat\varphi(\qt) \log \frac{\hat\varphi(\qt)}{\|\hat\varphi\|_{L^1_M(D)}} \dd \qt
\leq \frac{2}{\kappa} \int_{D} M(\qt) \left|\nabq \sqrt{\hat\varphi(\qt)}\right|^2 \dd \qt,
\end{equation}
for all $\hat\varphi$ such that $\hat\varphi \geq 0$ on $D$ and $\sqrt{\hat\varphi} \in H^1_M(D)$. Taking $\hat\varphi = \varphi/M$ where $\varphi$ is a probability density function on $D$, we have that $\|\hat\varphi\|_{L^1_M(D)} = \|\varphi\|_{L^1(D)} = 1$; thus,
on denoting by $\mu$ the Gibbs measure defined by $\dd \mu = M(\qt) \dq$, the left-hand side of \eqref{eq:logs1} becomes
\[ S(\varphi | M) := \int_D \frac{\varphi}{M} \left(\log \frac{\varphi}{M}\right)\!\dd\mu,\]
referred to as the {\em relative entropy} of $\varphi$ with respect to $M$. The expression appearing on the right-hand side of \eqref{eq:logs1} is $1/(2\kappa)$ times the {\em Fisher information}, $I(\hat\varphi)$, of $\hat\varphi$:
\[I(\hat\varphi):= \mathbb{E}\left[\left|\nabq \log {\hat\varphi(\qt)}\right|^2\right] =
\int_{D} \left|\nabq \log {\hat\varphi(\qt)}\right|^2 \hat{\varphi}(\qt) \dd \mu = 4 \int_{D} \left|\nabq \sqrt{\hat\varphi(\qt)}\right|^2 \dd \mu,\]
where, $\mathbb{E}$ is the expectation with respect to the Gibbs measure $\mu$ defined above.
$\quad\diamond$
\end{remark}

\begin{lemma}\label{le:supplementary}
Suppose that a sequence $\{\hat\varphi_n\}_{n=1}^\infty$ converges in
$L^1(0,T; L^1_M(\Omega \times D))$ to a function $\hat\varphi \in L^1(0,T; L^1_M(\Omega \times D))$,
and is bounded in $L^\infty(0,T; L^1_M(\Omega \times D))$, i.e. there exists $K_0>0$ such that
$\|\varphi_n\|_{L^\infty(0,T; L^1_M(\Omega \times D))} \leq K_0$ for all $n \geq 1$.
Then, $\hat\varphi \in L^p(0,T; L^1_M(\Omega \times D))$ for all $p \in [1,\infty)$,
and the sequence $\{\hat\varphi_n\}_{n \geq 1}$ converges to $\hat\varphi$ in $L^p(0,T; L^1_M(\Omega \times D))$ for all $p \in [1,\infty)$.
\end{lemma}

\begin{proof}
Since $\{\hat\varphi_n\}_{n \geq 1}$ converges in $L^1(0,T; L^1_M(\Omega \times D))$, it follows that
it is a Cauchy sequence in $L^1(0,T; L^1_M(\Omega \times D))$; thus, for any $p \in [1,\infty)$,
there exists $n_0=n_0(\varepsilon,p) \in \mathbb{N}$ such that for all $m, n \geq n_0(\varepsilon,p)$ we have
\[ \int_0^T \|\hat\varphi_n - \hat\varphi_m\|_{L^1_M(\Omega \times D)} \dd t < \frac{\varepsilon^p}{(2K_0)^{p-1}}.\]
Hence, for all $m, n \geq n_0(\varepsilon,p)$,
\begin{eqnarray*}
&&\hspace{-2.5mm}\left(\int_0^T \|\hat\varphi_n - \hat\varphi_m\|^p_{L^1_M(\Omega \times D)} \dd t\right)^{1/p}\\
&&\,\leq
\mbox{ess.sup}_{t \in [0,T]} \|\hat\varphi_n - \hat\varphi_m\|^{1-(1/p)}_{L^1_M(\Omega \times D)} \left(\int_0^T \|\hat\varphi_n - \hat\varphi_m\|_{L^1_M(\Omega \times D)}\right)^{1/p}\\
&&\, < \varepsilon.~~~~
\end{eqnarray*}
This in turn implies that $\{\hat\varphi_n\}_{n \geq 1}$ is a Cauchy sequence in the function space $L^p(0,T; L^1_M(\Omega \times D))$, for each $p \in [1,\infty)$.

Since $L^p(0,T; L^1_M(\Omega \times D))$ is complete,
$\{\hat\varphi_n\}_{n \geq 1}$ converges in $L^p(0,T; L^1_M(\Omega \times D))$ to a limit, which we denote by $\hat\varphi_{(p)}$, say. As,
by assumption, $\{\hat\varphi_n\}_{n \geq 1}$ converges in $L^1(0,T; L^1_M(\Omega \times D))$, and
\[L^p(0,T; L^1_M(\Omega \times D)) \subset L^1(0,T; L^1_M(\Omega \times D))\]
for each $p \in [1,\infty)$,
it follows by uniqueness of the limit that $\hat\varphi_{(p)} = \hat\varphi$ for all $p \in [1,\infty)$.
Hence, also, $\hat\varphi \in L^p(0,T; L^1_M(\Omega \times D))$ for all $p \in [1,\infty)$.
That completes the proof.
\end{proof}

\section{Passage to the limit $L \rightarrow \infty$: existence of weak solutions to the FENE chain model with centre-of-mass diffusion}
\label{sec:passage.to.limit}
\setcounter{equation}{0}

The bounds \eqref{eq:energy-u+psi-final2}, \eqref{psi-time-bound} and \eqref{u-time-bound} imply the existence
of a constant ${C}_{\star}>0$, depending only on ${\sf B}(\ut_0,\ft,\hat\psi_0)$ and the constants $C_{\ast}$ and $C_{\ast\ast}$, which in turn depend only on
$\epsilon$, $\nu$, $C_{\sf P}(\Omega)$, $T$, $|A|$, $a_0$, $k$, $K$, $\lambda$, $\Omega$, $d$ and $b$, but {\em not} on $L$ or $\Delta t$, such that:
\begin{eqnarray}\label{eq:energy-u+psi-final5}
&&\mbox{ess.sup}_{t \in [0,T]}\|\uta^{\Delta t, +}(t)\|^2 + \frac{1}{\Delta t} \int_0^T \|\uta^{\Delta t, +} - \uta^{\Delta t,-}\|^2
\dd s + \,\int_0^T \|\nabxtt \uta^{\Delta t, +}(s)\|^2 \dd s\nonumber\\[2ex]
&&\qquad +\, \mbox{ess.sup}_{t \in [0,T]}
\int_{\Omega \times D}\!\! M\, \mathcal{F}(\psia^{\Delta t, +}(t)) \dq \dx\nonumber\\[2ex]
&&\qquad\qquad +\, \frac{1}{\Delta t\, L}
\int_0^T\!\! \int_{\Omega \times D}\!\! M\, (\psia^{\Delta t, +} - \psia^{\Delta t, -})^2 \dq \dx \dd s
\nonumber \\[2ex]
&&\qquad +\, \int_0^T\!\! \int_{\Omega \times D} M\,
\big|\nabx \sqrt{\psia^{\Delta t, +}} \big|^2 \dq \dx \dd s
+\, \int_0^T\!\! \int_{\Omega \times D}M\,\big|\nabq \sqrt{\psia^{\Delta t, +}}\big|^2 \,\dq \dx \dd s\nonumber\\[2ex]
&&\qquad\qquad +\, \int_0^T\left\|\frac{\partial \uta^{\Delta t}}{\partial t}\right\|^2_{V_\sigma'}\dt + \int_0^T\left\|M\,\frac{\partial \psia^{\Delta t}}{\partial t}\right\|^2_{(H^s(\Omega \times D))'}\dt \leq C_\star,
\end{eqnarray}
where $\|\cdot\|_{V_\sigma'}$ denotes the norm of the dual space $\Vt_\sigma'$ of $\Vt_\sigma$
with $\sigma\geq \frac{1}{2}d$, $\sigma>1$ (cf. the paragraph following \eqref{equncon1});
and $\|\cdot\|_{(H^s(\Omega \times D))'}$ is the
norm of the dual space $(H^s(\Omega \times D))'$ of $H^s(\Omega \times D)$, with $s>1+\frac{1}{2}(K+1)d$.
The bounds in \eqref{eq:energy-u+psi-final5} on the time-derivatives follow from \eqref{u-time-bound}, and from \eqref{psi-time-bound} using the Sobolev embedding theorem.

By virtue of \eqref{bound-b}, \eqref{bound-a}, the definitions (\ref{ulin},b), %and \eqref{upm},
and with an argument completely analogous to \eqref{bound-a} on noting \eqref{inidata} in the case of the fourth term in \eqref{eq:energy-u+psi-final5},  and using
\eqref{eq:FL2a}, \eqref{inidata-1} and recalling that $L>1$ %\eqref{inidata},
we have (with a possible adjustment of the constant $C_\star$, if necessary,) that
\begin{eqnarray}\label{eq:energy-u+psi-final6}
&&\mbox{ess.sup}_{t \in [0,T]}\|\uta^{\Delta t(,\pm)}(t)\|^2 + \frac{1}{\Delta t} \int_0^T \|\uta^{\Delta t,+} - \uta^{\Delta t,-}\|^2
\dd s \nonumber\\
&&\qquad +\,\int_0^T \|\nabxtt \uta^{\Delta t(,\pm)}(s)\|^2 \dd s +  \mbox{ess.sup}_{t \in [0,T]}
\int_{\Omega \times D}\!\! M\, \mathcal{F}(\psia^{\Delta t(, \pm)}(t)) \dq \dx\nonumber\\
&&\qquad\qquad +\, \frac{1}{\Delta t\, L}
\int_0^T\!\! \int_{\Omega \times D}\!\! M\, (\psia^{\Delta t, +} - \psia^{\Delta t, -})^2 \dq \dx \dd s
\nonumber \\
&&\qquad +\, \int_0^T\!\! \int_{\Omega \times D} M\,
\big|\nabx \sqrt{\psia^{\Delta t, +}} \big|^2 \dq \dx \dd s
+\, \int_0^T\!\! \int_{\Omega \times D}M\,\big|\nabq \sqrt{\psia^{\Delta t, +}}\big|^2 \,\dq \dx \dd s\nonumber\\
&&\qquad\qquad +\, \int_0^T\left\|\frac{\partial \uta^{\Delta t}}{\partial t}\right\|^2_{V_\sigma'}\dt + \int_0^T\left\|M\frac{\partial \psia^{\Delta t}}{\partial t}\right\|^2_{(H^s(\Omega \times D))'}\dt \leq C_\star.
\end{eqnarray}
On noting (\ref{properties-a},b), (\ref{ulin},b), \eqref{inidata-1} and \eqref{inidata},
we also have that
\begin{equation}\label{additional1}
\psia^{\Delta t(,\pm)} \geq 0\qquad \mbox{a.e. on $\Omega \times D \times [0,T]$}
\end{equation}
and
\begin{equation}\label{additional2}
\int_D M(\qt)\, \psia^{\Delta t(,\pm)}(\xt,\qt,t)\dq \leq 1\quad \mbox{for a.e. $(\xt,t) \in \Omega \times [0,T]$.}
\end{equation}

Henceforth, we shall assume that
\begin{equation}\label{LT}
\Delta t = o(L^{-1})\qquad \mbox{as $L \rightarrow \infty$}.
 \end{equation}
Requiring, for example, that $0<\Delta t \leq C_0/(L\,\log L)$, $L > 1$, with an arbitrary (but fixed)
constant $C_0$ will suffice to ensure that \eqref{LT} holds. The sequences
$\{\uta^{\Delta t(,\pm)}\}_{L>1}$ and $\{\psia^{\Delta t(,\pm)}\}_{L>1}$
as well as all sequences of spatial and temporal derivatives of the entries of these two sequences
will thus be, indirectly, indexed by $L$ alone, although for reasons of consistency with our previous notation
we shall not introduce new, compressed, notation with $\Delta t$ omitted from the superscripts. Instead, whenever $L\rightarrow \infty$ in the rest of this section, it will be understood that $\Delta t$ tends to $0$ according to \eqref{LT}.
We are now almost ready to pass to the limit with $L\rightarrow \infty$. Before doing so, however, we first need to state the definition of the function $\hat\psi^0$ that obeys \eqref{inidata-1}, for a given $\hat\psi_0$ satisfying \eqref{inidata}.
We emphasize that up to this point we simply accepted without proof the existence of a function $\hat\psi^0$ obeying \eqref{inidata-1}
for a given $\hat\psi_0$.
The reason we have been evading to state the precise choice of $\hat\psi^0$ was for the sake
of clarity of exposition. The definition of $\hat\psi^0$ and the verification of the properties listed under \eqref{inidata-1}
rely on mathematical tools that were not in place at the start of Section \ref{sec:existence-cut-off} where the notation $\hat\psi^0$ was introduced, but were developed later, in the last two sections. The details of `lifting' $\hat\psi_0$ into a `smoother' function $\hat\psi^0$ are technical; they are discussed in the next subsection.

A second remark is in order.
One might wonder whether one could simply choose $\hat\psi^0$ as $\hat\psi_0$; indeed, with such
a choice all of the properties listed in \eqref{inidata} would be automatically satisfied,
bar one: there is no guarantee that $[\hat\psi_0]^{1/2} \in H^1_M(\Omega \times D)$.
Although the property $[\hat\psi^0]^{1/2} \in H^1_M(\Omega \times D)$ has not yet been used, it will play a crucial role
in our passage to the limit with $L\rightarrow \infty$ in Section \ref{passage}.
In fact, in the light of the logarithmic Sobolev inequality
\eqref{eq:logs1}, on comparing the requirements on $\hat\psi_0$ in \eqref{inidata} with those on $\hat\psi^0$ in
\eqref{inidata-1}, one can clearly see that the role of the condition $[\hat\psi^0]^{1/2} \in H^1_M(\Omega \times D)$ in \eqref{inidata-1} is to `lift' the initial datum $\hat\psi_0$
with finite relative entropy into a `smoother' initial datum $\hat\psi^0$ that
also has finite Fisher information, in analogy with the process of `lifting' the initial velocity
$\ut_0$ from $\Ht$ into $\ut^0$ in $\Vt$.
That the choice of $\hat\psi_0$ as $\hat\psi^0$ is not a good one can be
seen by noting the mismatch %in the amount of information contained in
between the third term
in \eqref{eq:energy-u+psi-final6} arising from the Navier--Stokes equation on the one hand, and the sixth
and seventh term in \eqref{eq:energy-u+psi-final6} that stem from the Fokker--Planck equation. The
absence of bounds at this stage on $\hat\psi^{\Delta t,-}$ and $\hat\psi^{\Delta t}$ in those terms
is entirely due to the fact that, to derive \eqref{eq:energy-u+psi-final6}, we did not
use that $[\hat\psi^0]^{1/2} \in H^1_M(\Omega \times D)$. This shortcoming of
\eqref{eq:energy-u+psi-final6} will be rectified as soon as we have defined
$\hat\psi^0$ and shown that it possesses {\em all} of the properties listed in \eqref{inidata-1}.

\subsection{The definition of $\hat\psi^0$}

Given $\hat\psi_0$ satisfying the conditions in \eqref{inidata} and $\Lambda >1$,
we consider the following discrete-in-time problem in weak form: find
$\hat\zeta^{\Lambda,1} \in H^1_M(\Omega \times D)$ such that
\begin{align}\label{zeta-eq}
&\int_{\Omega \times D} M\, \frac{\hat\zeta^{\Lambda,1} - \hat\zeta^{\Lambda,0}}{\Delta t}\,\hat\varphi \dq \dx
+ \int_{\Omega \times D} M \left[\nabx \hat\zeta^{\Lambda,1} \cdot \nabx \hat\varphi
+ \nabq \hat\zeta^{\Lambda,1} \cdot \nabq \hat\varphi \right] \dq \dx = 0
\end{align}
for all $\hat\varphi \in H^1_M(\Omega \times D)$, with $\hat\zeta^{\Lambda,0} := \beta^\Lambda(\hat\psi_0) \in L^2_M(\Omega \times D)$.
Here $\beta^\Lambda$ is defined by \eqref{betaLa}, with $L$ replaced by $\Lambda$. The function $\mathcal{F}^\Lambda$, which we shall
encounter below, is defined by \eqref{eq:FL}, with $L$ replaced by $\Lambda$.

The existence of a unique solution $\hat\zeta^{\Lambda,1} \in H^1_M(\Omega \times D)$ to \eqref{zeta-eq}, for each $\Delta t>0$ and $\Lambda>1$, follows
immediately by applying the Lax--Milgram theorem. The parameter $\Lambda$ plays an analogous role to the cut-off parameter $L$; however since we shall
let $\Lambda \rightarrow \infty$ in this subsection while, for the moment at least, the parameter $L$ is kept fixed,
we had to use a symbol other than $L$ in \eqref{zeta-eq} in order to avoid confusion; we chose the letter $\Lambda$ for this purpose
in order to emphasize the connection with $L$.
\begin{lemma}\label{one-bound}
Let $\hat\zeta^{\Lambda,1}$ be defined by \eqref{zeta-eq}, and consider $\gamma^{\Lambda,n}$ defined by
\begin{equation}\label{gamma-def}
\gamma^{\Lambda,n}(\xt):= \int_D M(\qt)\,\hat\zeta^{\Lambda,n}(\xt,\qt)\dq,\qquad n=0,1.
\end{equation}
Then, $\hat\zeta^{\Lambda,1}$ is nonnegative a.e. on $\Omega \times D$, and $0 \leq \gamma^{\Lambda,1} \leq 1$
a.e. on $\Omega$.
\end{lemma}
\begin{proof}
The proof of nonnegativity of $\hat\zeta^{\Lambda,1}$ is straightforward (cf. the discussion following \eqref{psiGIz}). Indeed, we have that
$[{\hat{\zeta}}^{\Lambda,0}]_{-} =0$ a.e. on $\Omega\times D$, thanks to \eqref{inidata} and the definition
of $\beta^\Lambda$;  we then take $\varphi = [{\hat{\zeta}}^{\Lambda,1}]_{-}$ as a test function in  \eqref{zeta-eq}, noting that this is a legitimate choice
since ${\hat{\zeta}}^{\Lambda,1} \in H^1_M(\Omega\times D)$ and therefore $[{\hat{\zeta}}^{\Lambda,1}]_{-} \in H^1_M(\Omega\times D)$ also (cf. Lemma 3.3 in
Barrett, Schwab \& S\"uli \cite{BSS}). On decomposing ${\hat{\zeta}}^{\Lambda,1} = [{\hat{\zeta}}^{\Lambda,1}]_{+} + [{\hat{\zeta}}^{\Lambda,1}]_{-}$, and using
that $[{\hat{\zeta}}^{\Lambda,1}]_{+} \,[\hat\zeta^{\Lambda,1}]_{-}=0$, $\nabx [{\hat{\zeta}}^{\Lambda,1}]_{+}\cdot \nabx [{\hat{\zeta}}^{\Lambda,1}]_{-} = 0$
and $\nabq [{\hat{\zeta}}^{\Lambda,1}]_{+}\cdot \nabq [{\hat{\zeta}}^{\Lambda,1}]_{-} = 0$
a.e. on $\Omega\times D$, we deduce that
\begin{eqnarray*}
&&\frac{1}{\Delta t}\|M^{\frac{1}{2}}\,[{\hat{\zeta}}^{\Lambda,1}]_{-}\|^2 + \|M^{\frac{1}{2}}\, \nabx [{\hat{\zeta}}^{\Lambda,1}]_{-}\|^2
+ \|M^{\frac{1}{2}}\,\nabq [{\hat{\zeta}}^{\Lambda,1}]_{-}\|^2\\
&&\qqqquad = \frac{1}{\Delta t}\int_{\Omega \times D} M\,\hat\zeta^{\Lambda,0}\,[\hat\zeta^{\Lambda,1}]_{-}
\dq \dx = \frac{1}{\Delta t}\int_{\Omega \times D} M\,[\hat\zeta^{\Lambda,0}]_{+}\,[\hat\zeta^{\Lambda,1}]_{-}
\dq \dx \leq 0,
\end{eqnarray*}
where $\|\cdot\|$ denotes the $L^2(\Omega\times D)$ norm. This then implies that
$\|M^{\frac{1}{2}}\,[{\hat{\zeta}}^{\Lambda,1}]_{-}\|^2  \leq  0$.
Hence, $[{\hat{\zeta}}^{\Lambda,1}]_{-} = 0$ a.e. on $\Omega \times D$. In other words, ${\hat{\zeta}}^{\Lambda,1} \geq 0$ a.e.
on $\Omega \times D$, as claimed.

In order to prove the upper bound in the statement of the lemma, we proceed as follows.
With $\gamma^{\Lambda,n}$ as defined in \eqref{gamma-def}, we deduce from the definition of ${\hat{\zeta}}^{\Lambda,1}$
and Fubini's theorem that $\gamma^{\Lambda,1} \in H^1(\Omega)$. Furthermore, on selecting
$\hat\varphi = \varphi \in H^1(\Omega)\otimes 1(D)$ in \eqref{zeta-eq},
recall (\ref{H101}),
we have that
\begin{align}
\label{eq:weaka2aa}
&\int_{\Omega}
\frac{\gamma^{\Lambda,1}- \gamma^{\Lambda,0}}{\Delta t}\,\varphi\dx
+  \int_{\Omega}
\nabx {\gamma}^{\Lambda,1} \cdot\, \nabx
\varphi \dx = 0
\qquad \forall
\varphi \in
H^1(\Omega).
\end{align}
As ${\hat{\zeta}}^{\Lambda,0} = \beta^\Lambda(\hat\psi_0)$, and $0 \leq \beta^\Lambda(s) \leq s$ for all
$s \in \mathbb{R}_{\geq 0}$, we also have by \eqref{inidata} that
\begin{equation}\label{z-bound}
 0 \leq \gamma^{\Lambda,0} =\int_D M \beta^{\Lambda}(\hat\psi_0) \dq \leq \int_D M \hat\psi_0 \dq  = 1
\quad \mbox{on $\Omega$}.
\end{equation}

Consider $z^{\Lambda,n}:= 1 - \gamma^{\Lambda,n}$, $n=0,1$.
On substituting $\gamma^{\Lambda,n}= 1 - z^{\Lambda,n}$, $n=0,1$, into \eqref{eq:weaka2aa}, we have that
\begin{align}
\label{eq:weaka2aaa}
&\int_{\Omega}
\frac{z^{\Lambda,1}- z^{\Lambda,0}}{\Delta t}\,\varphi\dx
+  \int_{\Omega}
\nabx z^{\Lambda,1} \cdot\, \nabx
\varphi \dx = 0
\qquad \forall
\varphi \in
H^1(\Omega).
\end{align}
Also, by \eqref{z-bound}, we have that $0 \leq z^{\Lambda,0} \leq  1$. By using an identical procedure to
the one in the first part of the proof, we then deduce that $[z^{\Lambda,1}]_{-} = 0$ a.e. on $\Omega$. Thus,
$z^{\Lambda,1} \geq 0$ a.e. on $\Omega$, which then implies that $\gamma^{\Lambda,1} \leq 1$ a.e. on $\Omega$, as claimed.
\end{proof}

Next, we shall pass to the limit $\Lambda \rightarrow \infty$; as we shall see in the final part of Lemma \ref{psi0properties}
below, this will require the use of smoother test functions in problem \eqref{zeta-eq}, as otherwise the term involving
$\hat\zeta^{\Lambda,0}=\beta^\Lambda(\hat \psi_0)$ is not defined in the limit. In any case, our objective is to use the
limit of the sequence $\{\zeta^{\Lambda,1}\}_{\Lambda>1}$, once it has been shown to exist, as our definition of the function $\hat\psi^0$. We shall then show that $\hat\psi^0$ thus defined has all the
properties listed in \eqref{inidata-1}.

To this end, we need to derive $\Lambda$-independent
bounds on norms of ${\hat{\zeta}}^{\Lambda,1}$, very similar to the $L$-independent bounds discussed in Section
\ref{sec:entropy}. Since the argument is almost identical to (but simpler than) the one there
(viz. \eqref{zeta-eq} can be viewed as a special case of \eqref{psiG}, with $\ft^n$, $\uta^{n-1}$ and $\uta^n$ taken to be identically zero, $\lambda=\frac{1}{2}$, $\epsilon=1$, $N=1$, and $A$ chosen as the $K\times K$ identity matrix), we shall not include the details
here. It suffices to say that, on testing \eqref{zeta-eq} with $\mathcal{F}'(\hat\zeta^{\Lambda,1} + \alpha)$ and passing to the limit
$\alpha \rightarrow 0_{+}$, analogously as in the proof of \eqref{penultimate-line} in Section \ref{sec:entropy}, we obtain that
\begin{eqnarray}\label{eq:zetabd}
&&\int_{\Omega \times D} M \mathcal{F}(\hat\zeta^{\Lambda,1}) \dq \dx
%+ \frac{1}{2 \Lambda} \int_{\Omega \times D} M (\hat\zeta^{\Lambda,1} - \hat\zeta^{\Lambda,0})^2 \dq \dx
+\, 4\,\Delta t \int_{\Omega \times D} M
\big|\nabx \sqrt{\hat\zeta^{\Lambda,1}} \big|^2 \dq \dx\nonumber\\
&&\qquad +\, 4\,
\Delta t \int_{\Omega \times D}M\, \,\big|\nabq \sqrt{\hat\zeta^{\Lambda,1}}\big|^2 \,\dq \dx%\nonumber\\
%&&
\leq \int_{\Omega \times D} M \mathcal{F}(\hat\psi_0) \dq \dx.
\end{eqnarray}
Our passage to the limit $\Lambda \rightarrow \infty$ in \eqref{zeta-eq} is based on a weak compactness argument, using
\eqref{eq:zetabd}, and is discussed below.

We have from Lemma \ref{one-bound} that $\{[\hat\zeta^{\Lambda,1}]^{\frac{1}{2}}\}_{\Lambda>1}$
is a bounded sequence in $L^2_M(\Omega \times D)$.
Using this in conjunction with the second and third bound in \eqref{eq:zetabd} we deduce that,
for $\Delta t>0$ fixed, $\{[\hat\zeta^{\Lambda,1}]^{\frac{1}{2}}\}_{\Lambda>1}$ is a bounded sequence in
$H^1_M(\Omega \times D)$. Thanks to the compact embedding of $H^1_M(\Omega \times D)$ into
$L^2_M(\Omega \times D)$ (cf. Appendix D), we deduce that
$\{[\hat\zeta^{\Lambda,1}]^{\frac{1}{2}}\}_{\Lambda>1}$
has a strongly convergent subsequence in $L^2_M(\Omega \times D)$, whose limit we label by
$\mathsf{Z}$, and we then let $\hat\zeta^{~\!\!1}:=\mathsf{Z}^2$.
For future reference we note that, upon extraction
of a subsequence (not indicated), $\hat\zeta^{\Lambda,1}$ then converges to $\hat{\zeta}^{\!\!~1}$ a.e. on
$\Omega \times D$; and $\hat\zeta^{\Lambda,1}(\xt,\cdot)$ converges to $\hat{\zeta}^{\!\!~1}(\xt,\cdot)$ a.e. on
$D$, for a.e. $\xt \in \Omega$.

By definition, we have that $\hat{\zeta}^{\!\!~1} \geq 0$; furthermore, thanks to the upper bound on $\gamma^{\Lambda,1}$ stated in
Lemma \ref{one-bound}, the remark in the last sentence of the
previous paragraph, and Fatou's lemma, we also have that
\begin{equation}\label{upperbound}
\int_D M(\qt)\,\hat{\zeta}^{\!\!~1}(\xt,\qt) \dq \leq 1\qquad \mbox{for a.e. $\xt \in \Omega$.}
\end{equation}
Further, again as a direct consequence of the definition of $\hat{\zeta}^{\!\!~1}$, we have that
\begin{equation}\label{sqrtpsi-zeta}
\sqrt{ {\hat{\zeta}}^{\Lambda,1}} \rightarrow \sqrt{{\hat{\zeta}^{~\!1}}}\qquad \mbox{strongly in $L^2_M(\Omega \times D)$}.
\end{equation}
Application of the factorization $c_1-c_2 = (\sqrt{c_1} - \sqrt{c_2})\,(\sqrt{c_1} + \sqrt{c_2})$
with $c_1, c_2 \in \mathbb{R}_{\geq 0}$, the Cauchy--Schwarz inequality
and \eqref{sqrtpsi-zeta} then yields that
\begin{equation}\label{sqrtpsi-zetaaa}
{\hat{\zeta}}^{\Lambda,1} \rightarrow \hat{\zeta}^{~\!\!1}\qquad \mbox{strongly in $L^1_M(\Omega \times D)$}.
\end{equation}

Finally, we define
\begin{equation}\label{psizerodef}
\hat\psi^0:= \hat{\zeta}^{\!\!~1}.
\end{equation}
It follows from the nonnegativity of $\hat{\zeta}^{\!\!~1}$ and \eqref{upperbound} that
\begin{equation}\label{nonnegativity}
\mbox{$\hat\psi^0\geq 0\quad$a.e. on $\Omega \times D\qquad$and}\qquad 0 \leq \int_D M(\qt)\, \hat\psi^0(\xt, \qt) \dq \leq 1\quad \mbox{for a.e. $\xt \in
\Omega$.}
\end{equation}
Further, from the bound on the first term in \eqref{eq:zetabd} and Fatou's lemma, together with the
fact that, thanks to the continuity of $\mathcal{F}$,
(a subsequence, not indicated, of) $\{\mathcal{F}({\hat{\zeta}}^{\Lambda,1})\}_{\Lambda>0}$ converges
to $\mathcal{F}(\hat{\zeta}^{\!\!~1}) = \mathcal{F}(\hat\psi^0)$ a.e. on $\Omega \times D$, we also have that
\begin{equation}\label{F-stability}
\int_{\Omega \times D} M\, \mathcal{F}(\hat\psi^0) \dq \dx
\leq \int_{\Omega \times D} M\, \mathcal{F}(\hat\psi_0) \dq \dx.
\end{equation}

Next, we note that from \eqref{sqrtpsi-zeta} we have that, as $\Lambda \rightarrow \infty$,
\begin{equation}\label{strongpsi-zeta}
M^{\frac{1}{2}}\, \sqrt{\hat{\zeta}^{\Lambda,1}} \rightarrow M^{\frac{1}{2}}\,\sqrt{\hat{\zeta}^{~\!\!1}}\qquad \mbox{strongly in $L^2(\Omega \times D)$}.
\end{equation}

We shall use \eqref{strongpsi-zeta} to deduce weak convergence of the sequences of $\xt$ and $\qt$
gradients of $\hat\zeta^{\Lambda,1}$. We proceed as in the proof of
Lemma \ref{conv}. The bound on the third term on the left-hand side
of \eqref{eq:zetabd} implies the existence of a subsequence (not indicated) and an element
$\gt \in \Lt^2(\Omega \times D)$, such that
\begin{equation}\label{square}
M^{\frac{1}{2}}\, \nabq \sqrt{{\hat{\zeta}}^{\Lambda,1}}
\rightarrow \gt \qquad
\mbox{weakly in $\Lt^2(\Omega \times D)$.}
\end{equation}
Proceeding as in (\ref{derivid})--(\ref{last}) in the proof of Lemma \ref{conv}
with $\hat{\psi}^n_{\epsilon,L,\delta}$, $\hat{\psi}^n_{\epsilon,L}$ and $\delta
\rightarrow 0_+$ replaced
by $\sqrt{\hat {\zeta}^{\Lambda,1}}$, $\sqrt{\hat{\zeta}^{1}}$ and
$\Lambda \rightarrow \infty$, respectively,
we obtain the weak convergence result:
\begin{subequations}
\begin{align}\label{xlimitL2}
M^{\frac{1}{2}}\, \nabq\sqrt{\hat{\zeta}^{\Lambda,1}} \rightarrow  M^{\frac{1}{2}}\, \nabq\sqrt{\hat{\zeta}^{\!\!~1}}\quad \mbox{weakly in $L^2(\Omega \times
D)$},
\end{align}
and similarly for the $\xt$ gradient
\begin{align}\label{qlimitL2}
M^{\frac{1}{2}}\, \nabx\sqrt{\hat{\zeta}^{\Lambda,1}}\rightarrow  M^{\frac{1}{2}}\, \nabx\sqrt{\hat{\zeta}^{\!\!~1}}\quad \mbox{weakly in $L^2(\Omega \times
D)$,}
\end{align}
\end{subequations}
as $\Lambda \rightarrow \infty$. Then, inequality
(\ref{eq:zetabd}), (\ref{xlimitL2},b) and the weak lower-semicontinuity of the $L^2(\Omega \times D)$ norm imply that
\begin{align}\label{zeta-space-bound-d}
&4\,\Delta t \int_{\Omega \times D} M \left[\big|\nabx \sqrt{\hat{\zeta}^{\!\!~1}}\big|^2
+
\big|\nabq \sqrt{\hat{\zeta}^{\!\!~1}}\big|^2 \right]\dq \dx
\leq \int_{\Omega \times D} M\, \mathcal{F}(\hat\psi_0) \dq \dx.
\end{align}

After these preparations, we are now ready to state the central result of this subsection. Before we
do so, a comment is in order. Strictly speaking, we should have written $\hat\psi^0_{\Delta t}$ instead of
$\hat\psi^0$ in our definition \eqref{psizerodef}, as $\hat\psi^0$ depends on the choice of $\Delta t$.
For notational simplicity, we prefer the more compact notation, $\hat\psi^0$, with the dependence of
$\hat\psi^0$ on $\Delta t$ implicitly understood; we shall only write $\hat\psi^0_{\Delta t}$, when
it is necessary to emphasize the dependence of $\Delta t$. Of course, $\hat\psi_0$ is independent of $\Delta t$.

We shall show that, with our definition of $\hat\psi^0$, the properties under \eqref{inidata-1}
hold, together with some additional properties that we extract from \eqref{psizerodef}.

\begin{lemma}\label{psi0properties}
The function $\hat\psi^0=\hat\psi^0_{\Delta t}$ defined by \eqref{psizerodef} has the following properties:
\begin{itemize}
\item[\ding{202}]  $~~~\hat\psi^0 \in \hat{Z}_1$;
\item[\ding{203}]  $\displaystyle{~~~\int_{\Omega \times D} M\, \mathcal{F}(\hat\psi^0)\dq \dx \leq \int_{\Omega \times D} M\, \mathcal{F}(\hat\psi_0)\dq
    \dx}$;
\item[\ding{204}]  $\displaystyle{~~~4\,\Delta t \,\int_{\Omega \times D} M\left[ |\nabx \sqrt{{\hat{\psi}^0}}|^2 + |\nabq
    \sqrt{{\hat{\psi}^0}}|^2\right] \dq \dx \leq \int_{\Omega \times D} M\, \mathcal{F}(\hat\psi_0) \dq \dx}$;
\item[\ding{205}]
    $~~~\lim_{\Delta t \rightarrow 0_+} \hat\psi^0 = \hat\psi_0$, weakly in $L^1_M(\Omega \times D)$;
\item[\ding{206}]
    $~~~\lim_{\Delta t \rightarrow 0_+} \beta^L(\hat\psi^0) = \hat\psi_0$, weakly in
    $L^1_{M}(\Omega \times D)$.
\end{itemize}
\end{lemma}

\begin{proof}
~

\begin{itemize}
\item[\ding{202}] This is an immediate consequence of \eqref{nonnegativity} and the definition (\ref{hatZ})
of $\hat Z_1$.
\item[\ding{203}] This property was established in \eqref{F-stability} above.
\item[\ding{204}]
The inequality follows by using \eqref{psizerodef}
in the left-hand side of \eqref{zeta-space-bound-d}.

\item[\ding{205}] We begin by noting that an argument, completely analogous to (but simpler than) the one in
Section \ref{sec:time-psia} that resulted in \eqref{psi-time-bound}, applied to (\ref{zeta-eq})
now, yields
\begin{align*}
&\!\!\!\left|\,\int_{\Omega \times D}\!\! M\, \frac{\hat\zeta^{\Lambda,1} - \hat\zeta^{\Lambda,0}}{\Delta t}\, \hat\varphi
\dq \dx \,\right|
\leq
2 \left(\int_{\Omega \times D}\!\! M\! \left[|\nabx \sqrt{\hat\zeta^{\Lambda,1}}|^2 + |\nabq \sqrt{\hat\zeta^{\Lambda,1}}|^2\right] \!\!\dq \dx
\right)^{\frac{1}{2}}\nonumber\\
&\qquad\qquad\times
\left(\int_\Omega \left[\|\nabx \hat\varphi\|^2_{L^\infty(D)} + \|\nabq \hat\varphi\|^2_{L^\infty(D)}\right]  \dx \right)^{\frac{1}{2}},
\end{align*}
for all $\hat\varphi \in H^1(\Omega;L^\infty(D))\cap L^2(\Omega;W^{1,\infty}(D))$. By noting
\eqref{eq:zetabd} we deduce that
\begin{align}\label{firstbound}
&\left|\int_{\Omega \times D} M \,(\hat\zeta^{\Lambda,1} - \hat\zeta^{\Lambda,0})\, \hat\varphi \dq \dx \right|
\leq
(\Delta t)^{\frac{1}{2}} \left(\int_{\Omega \times D} M\, \mathcal{F}(\hat\psi_0) \dq \dx \right)^{\frac{1}{2}}\nonumber\\
&\qquad\qquad \times
\left(\int_\Omega \left[\|\nabx \hat\varphi\|^2_{L^\infty(D)} + \|\nabq \hat\varphi\|^2_{L^\infty(D)}\right]  \dx \right)^{\frac{1}{2}}
\end{align}
for all $\hat\varphi \in H^1(\Omega;L^\infty(D))\cap L^2(\Omega;W^{1,\infty}(D))$.
As the right-hand side of \eqref{firstbound} is independent of $\Lambda$, we can pass
to the limit $\Lambda \rightarrow \infty$ on both sides of \eqref{firstbound}, using the strong
convergence of $\hat\zeta^{\Lambda,1}$ to $\hat\psi^0$ in $L^1_M(\Omega \times D)$ as $\Lambda \rightarrow \infty$
(see \eqref{sqrtpsi-zetaaa} and the definition of \eqref{psizerodef}) together with the strong
convergence of $\hat\zeta^{\Lambda,0}=\beta^\Lambda(\hat\psi_0)$ to $\hat\psi_0$ in $L^1_M(\Omega \times D)$ as $\Lambda \rightarrow \infty$,
with $\Delta t$ kept fixed. We deduce that
\begin{align}\label{secondbound}
&\hspace{-1cm}\left|\,\int_{\Omega \times D} M \,(\hat\psi^0 - \hat\psi_0)\, \hat\varphi \dq \dx \,\right|
\leq
(\Delta t)^{\frac{1}{2}} \left(\int_{\Omega \times D} M \mathcal{F}(\hat\psi_0) \dq \dx \right)^{\!\frac{1}{2}}
\nonumber\\
&\qquad\qquad\qquad\!\!\!\!\times \left(\int_\Omega \left[\|\nabx \hat\varphi\|^2_{L^\infty(D)} + \|\nabq \hat\varphi\|^2_{L^\infty(D)}\right]  \dx
\right)^{\!\frac{1}{2}}
\end{align}
for all $\hat\varphi \in H^1(\Omega;L^\infty(D))\cap L^2(\Omega;W^{1,\infty}(D))$ and therefore in particular
for all $\hat\varphi \in H^s(\Omega \times D)$ with $s > 1+ \frac{1}{2}(K+1)d$.

As the last two factors on the right-hand side of \eqref{secondbound} are independent of $\Delta t$,
we can pass to the limit $\Delta t \rightarrow 0_+$ on both sides of \eqref{secondbound}
to deduce that $\hat\psi^0= \hat\psi^0_{\Delta t}$ converges to $\hat\psi_0$ weakly
in $M^{-1}(H^s(\Omega \times D))'$, $s > 1+ \frac{1}{2}(K+1)d$, as $\Delta t \rightarrow 0_+$.

Noting \eqref{F-stability}
and the fact that $\mathcal{F}(r)/r \rightarrow \infty$ as $r\rightarrow \infty$, we deduce from de la Vall\'ee-Poussin's
theorem that the family $\{\hat\psi^0_{\Delta t}\}_{\Delta t>0}$ is uniformly integrable in $L^1_M(\Omega \times D)$.
Hence, by the Dunford--Pettis theorem, the family $\{\hat\psi^0_{\Delta t}\}_{\Delta t>0}$ is weakly
relatively compact in $L^1_M(\Omega \times D)$. Consequently, one can extract a subsequence
$\{\hat\psi^0_{\Delta t_k}\}_{k=1}^\infty$ that converges
weakly in $L^1_M(\Omega \times D)$; however the uniqueness of the weak limit
together with the weak convergence of the (entire) sequence $\hat\psi^0= \hat\psi^0_{\Delta t}$ to $\hat\psi_0$ in $M^{-1}(H^s(\Omega \times D))'$, $s > 1+
\frac{1}{2}(K+1)d$, as $\Delta t \rightarrow 0_+$, established in the
previous paragraph, then implies that the (entire)
sequence $\hat\psi^0= \hat\psi^0_{\Delta t}$ converges to $\hat\psi_0$ weakly in $L^1_M(\Omega \times D)$,
as $\Delta t \rightarrow 0_+$, on noting that $L^1_M(\Omega\times D)$
is (continuously) embedded in $M^{-1}(H^s(\Omega \times D))'$ for $s > 1+ \frac{1}{2}(K+1)d$ (cf. the discussion
following Theorem \ref{thm:Dubinski}).

\item[\ding{206}]
It follows from $\hat \psi^0 \in \hat Z_1$ and (\ref{betaLa}) that
\begin{align}
0 \leq
\int_{\hat \psi^0 \geq L} M\,L \dq \dx \leq
\int_{\Omega \times D} M \, \beta^L(\hat \psi^0) \dq  \dx \leq
\int_{\Omega \times D} M \, \hat \psi^0 \dq  \dx \leq
|\Omega|.
\label{betapsi0bd}
\end{align}
On noting that $\mathcal{F}$ is nonnegative and monotonically increasing on $[1,\infty)$,
and that $\mathcal{F}(s) \in [0,1]$ for $s \in [0,1]$, we deduce that
\begin{align}
&\hspace{0cm}\int_{\Omega \times D} M\,\mathcal{F}([\hat \psi^0 -L]_+) \dq \dx
\nonumber \\
& \hspace{0.3in} = \int_{\hat \psi^0 \in [0,L+1)}
M\,\mathcal{F}([\hat \psi^0 -L]_+) \dq \dx
+\int_{\hat \psi^0 \geq L+1} M\,\mathcal{F}([\hat \psi^0 -L]_+) \dq \dx
\nonumber \\
& \hspace{0.3in} \leq \int_{\Omega \times D}
M \dq \dx
+\int_{\Omega \times D} M\,\mathcal{F}(\hat \psi^0) \dq \dx
\leq C.
\label{FpsiLbd}
\end{align}

Let us recall the logarithmic Young's inequality
\begin{align}
r\,s \leq r\,\log r -r + {\rm e}^s
\qquad \mbox{for all } r,\,s \in {\mathbb R}_{\geq 0}.
\label{logY}
\end{align}
This follows from the Fenchel--Young inequality:
\[ r \, s \leq g^\ast(r) + g(s) \qquad \mbox{for all $r, s \in
\mathbb{R}$},\]
involving the convex function $g\,:\, s \in \mathbb{R} \mapsto
g(s) \in (-\infty,+\infty]$ and its convex conjugate $g^\ast$,
with $g(s) = {\rm e}^s$ and
\[ g^\ast(r) = \left\{ \begin{array}{cl}
  + \infty & \mbox{if $r < 0$,}\\
  0        & \mbox{if $r = 0$,}\\
  r\,(\log r - 1) & \mbox{if $r > 0$;}
                \end{array}
                \right.
\]
with the resulting inequality then restricted to $\mathbb{R}_{\geq 0}$.
It immediately follows from (\ref{logY}) that
$r\,s
\leq \mathcal{F}(r) +{\rm e}^s$
for all $r,\,s \in {\mathbb R}_{\geq 0}$.

Applying the last inequality with $r = [\hat \psi^0 -L]_+$ and $s = \log L$,
we have that
\begin{align}
[\hat \psi^0 -L]_+\, (\log L)\leq \mathcal{F}([\hat \psi^0 -L]_+) + L.
\label{logYL}
\end{align}
The bounds (\ref{betapsi0bd}), (\ref{FpsiLbd}) (noting that the integrand of the left-most
integral in (\ref{FpsiLbd}) is nonnegative) and (\ref{logYL}) then imply
\begin{align}
&\hspace{-0.3cm}\int_{\Omega \times D} M\,[\hat \psi^0 -L]_+ \dq \dx
=
\int_{\hat \psi^0 \geq L} M\,[\hat \psi^0 -L]_+ \dq \dx
\nonumber \\
&\hspace{0.1cm}\leq \frac{1}{\log L}
\left[\int_{\hat \psi^0 \geq L}  M\,\mathcal{F}([\hat \psi^0 -L]_+) \dq \dx +
\int_{\hat \psi^0 \geq L}  M\,L \dq \dx \right]
%\nonumber \\
%&\hspace{-1cm}
\leq  \frac{C}{\log L}.
\label{Lplusbd}
\end{align}

Hence, for any $\hat \varphi \in L^\infty(\Omega \times D)$ we have from
(\ref{Lplusbd}) on recalling the relationship $\Delta t =o(L^{-1})$ that
$\hat \psi^0 = \hat \psi^0_{\Delta t}$ satisfies
\begin{align}
&\lim_{\Delta t \rightarrow 0_+} \left| \int_{\Omega \times D} M\,
(\hat \psi^0-\beta^L(\hat \psi^0)) \,\hat \varphi \dq \dx \right|
=
\lim_{\Delta t \rightarrow 0_+} \left| \int_{\Omega \times D} M\,
[\hat \psi^0-L]_+ \,\hat \varphi \dq \dx \right|
\nonumber \\
&\hspace{0.3cm}\leq \left(\lim_{\Delta t \rightarrow 0_+} \int_{\Omega \times D} M\,
[\hat \psi^0-L]_+ \dq \dx \right)
\|\hat \varphi\|_{L^\infty(\Omega \times D)}
%\nonumber \\
%&\hspace{-2cm}
=0.
\label{psioLconv}
\end{align}
Therefore, similarly to (\ref{secondbound}), we have that
the sequence $\{\hat \psi^0_{\Delta t}
-\beta^L(\hat \psi^0_{\Delta t})\}_{\Delta t >0}$
converges to zero weakly in
$M^{-1}(H^s(\Omega \times D))'$ for $s > \frac{1}{2}(K+1)d$, as $\Delta t \rightarrow 0_+$.

Noting \eqref{FpsiLbd}
and the fact that $\mathcal{F}(r)/r \rightarrow \infty$ as $r\rightarrow \infty$, we deduce from de le Vall\'ee Poussin's
theorem that the family $$\{\hat\psi^0_{\Delta t}-\beta^L(\hat\psi^0_{\Delta t})\}_{\Delta t>0} \equiv \{[\hat\psi^0_{\Delta t}-L]_+\}_{\Delta t>0}$$
is uniformly integrable in $L^1_M(\Omega \times D)$.
Hence, we can proceed as for the sequence $\{\hat\psi^0_{\Delta t}\}_{\Delta t>0}$
in the proof of \ding{205} to show that
the (entire)
sequence
\[\mbox{$\hat\psi^0-\beta^L(\hat\psi^0)= \hat\psi^0_{\Delta t}
- \beta^L(\hat\psi^0_{\Delta t}) \rightarrow 0~~~~~$ weakly in $L^1_M(\Omega \times D)$,
as $\Delta t \rightarrow 0_+$},\]
on noting that $L^1_M(\Omega\times D)$
is (continuously) embedded in $M^{-1}(H^s(\Omega \times D))'$ for $s > \frac{1}{2}(K+1)d$ (cf.\ the discussion
following Theorem \ref{thm:Dubinski}).
Hence, we have proved the desired result.
\end{itemize}
\end{proof}

Noting item 3 in Lemma \ref{psi0properties}, we can now return to
the inequality \eqref{eq:energy-u+psi-final6}, and supplement it with additional
bounds, in the sixth and seventh term on the left-hand side. The first additional bound can
be seen as the analogue of \eqref{bound-a}:
\begin{eqnarray}\label{bound-aaa}
&&\hspace{-1cm}4\int_0^T \int_{\Omega\times D} M\left[|\nabx\sqrt{\psia^{\Delta t,-}}|^2 +
|\nabq\sqrt{\psia^{\Delta t,-}}|^2\right] \dq \dx \dt  \nonumber\\
&&= 4\Delta t  \int_{\Omega\times D} M\left[|\nabx\sqrt{\beta^L(\hat\psi^0)}|^2 +
|\nabq\sqrt{\beta^L(\hat\psi^0)}|^2\right] \dq \dx
\nonumber\\
&&\qquad +\, 4\int_0^{T-\Delta t} \int_{\Omega\times D} M\left[|\nabx\sqrt{\psia^{\Delta t,+}}|^2 +
|\nabq\sqrt{\psia^{\Delta t,+}}|^2\right] \dq \dx \dt\nonumber\\
&& \leq 4\Delta t  \int_{\Omega\times D} M\left[|\nabx\sqrt{\hat\psi^0}|^2 +
|\nabq\sqrt{\hat\psi^0}|^2\right] \dq \dx
\nonumber\\
&&\qquad +\, 4\int_0^{T-\Delta t} \int_{\Omega\times D} M\left[|\nabx\sqrt{\psia^{\Delta t,+}}|^2 +
|\nabq\sqrt{\psia^{\Delta t,+}}|^2\right] \dq \dx \dt\nonumber\\
&& \leq \int_{\Omega \times D}M\, \mathcal{F}(\hat\psi_0)\dq \dx\nonumber\\
&&\qquad +\, 4 \int_0^{T} \int_{\Omega\times D} M\left[|\nabx\sqrt{\psia^{\Delta t,+}}|^2 +
|\nabq\sqrt{\psia^{\Delta t,+}}|^2\right] \dq \dx \dt
\leq  C_\star,~
\end{eqnarray}
where in the last inequality we used \eqref{inidata-1} and the bounds on the sixth and seventh
term in \eqref{eq:energy-u+psi-final6}; here and henceforth $C_\star$ signifies a generic positive constant,
independent of $L$ and $\Delta t$. On combining \eqref{bound-aaa} with our previous
bounds on the sixth and seventh term in \eqref{eq:energy-u+psi-final6}, we deduce that
\begin{equation}\label{bound-bbb}
\hspace{-1cm}4\int_0^T \int_{\Omega\times D} M\left[|\nabxtt\sqrt{\psia^{\Delta t,\pm}}|^2 +
|\nabq\sqrt{\psia^{\Delta t,\pm}}|^2\right] \dq \dx \dt  \leq C_\star.
\end{equation}
%
%%A simple calculation \cite{BS2010} then shows that these imply an analogous inequality for $\psia^{\Delta t}$:
%%%
%%\begin{equation}\label{bound-ccc}
%%\hspace{-1cm}4\int_0^T \int_{\Omega\times D} M\left[|\nabx\sqrt{\psia^{\Delta t}}|^2 +
%%|\nabq\sqrt{\psia^{\Delta t}}|^2\right] \dq \dx \dt  \leq C_\star,
%%\end{equation}
%%%
%%%where, again, $C_\ast$ denotes a generic positive constant independent of $L$ and $\Delta t$.
%
It remains to derive an analogous bound on $\psia^{\Delta t}$. To this end, let $n \in \{1,\dots, N\}$
and consider $t \in (t_{n-1},t_n)$; we recall that
\begin{equation}
\psia(\cdot,\cdot, t)=\,\frac{t-t_{n-1}}{\Delta t}\,
\psia^{\Delta t,+}(\cdot,\cdot,t) + \frac{t_n-t}{\Delta t}\,\psia^{\Delta t, -}(\cdot,\cdot,t). \label{psilin}
\end{equation}
For ease of exposition we shall write
\[\gamma_+:= \frac{t-t_{n-1}}{\Delta t} \qquad \mbox{and}\qquad  \gamma_{-}:=\frac{t_n-t}{\Delta t}\]
in the argument that follows, noting that $\gamma_{+} + \gamma_{-} = 1$ and both $\gamma_{+}$ and
$\gamma_{-}$ are positive. The functions $t \in (t_{n-1},t_n) \mapsto \psia^{\Delta t,\pm}(\cdot,\cdot,t)$ are constant in time and
$\psia^{\Delta t,(\pm)}(\xt,\qt,t)\geq 0$ over the set $\Omega \times D \times (t_{n-1},t_n)$, $n \in \{1,\dots, N\}$. For any $\alpha \in (0,1)$ we have that
\begin{eqnarray*}
\frac{|\nabx \psia^{\Delta t}|^2}{\psia^{\Delta t}+\alpha}
&=& \frac{|\gamma_+\,\nabx \psia^{\Delta t,+} + \gamma_{-}\,\nabx \psia^{\Delta t,-}|^2}{\gamma_+\,(\psia^{\Delta t,+}+\alpha) + \gamma_{-}(\,\psia^{\Delta t,-}+\alpha)}
\nonumber\\
&\leq&  2\,\frac{\gamma_+^2 |\nabx \psia^{\Delta t,+}|^2 + \gamma_{-}^2 |\nabx \psia^{\Delta t,-}|^2}{\gamma_+\,(\psia^{\Delta t,+}+\alpha) + \gamma_{-}\,(\psia^{\Delta t,-}+\alpha)}
\nonumber\\
&\leq&  2\,\frac{\gamma_{+} |\nabx \psia^{\Delta t,+}|^2}{\psia^{\Delta t,+}+\alpha}
+ 2\,\frac{\gamma_{-} |\nabx \psia^{\Delta t,-}|^2}{\psia^{\Delta t,-}+\alpha}.\nonumber
\end{eqnarray*}
Hence, on bounding $\gamma^{\pm}$ by 1, we deduce that
\begin{equation}\label{last-psia}
\frac{|\nabx \psia^{\Delta t}|^2}{\psia^{\Delta t}+\alpha}
\leq 2\,\frac{|\nabx \psia^{\Delta t, +}|^2}{\psia^{\Delta t, +}+\alpha} + 2\,\frac{|\nabx \psia^{\Delta t, -}|^2}{\psia^{\Delta t, -}+\alpha},
\end{equation}
for all $(\xt,\qt,t) \in \Omega \times D\times (t_{n-1},t_n)$, $n=1,\dots,N$, and all $\alpha \in (0,1)$. On multiplying
\eqref{last-psia} by $M$, integrating over $\Omega \times D \times (t_{n-1},t_n)$, summing over $n=1,\dots, N$,
and passing to the limit $\alpha \rightarrow 0_+$ using the monotone convergence theorem, we deduce that
\begin{eqnarray*}
4\int_0^T\int_{\Omega \times D} M\big|\nabx \sqrt{\psia^{\Delta t}}\big|^2\dq\dx \dt
&\leq& 2 \left[4\int_0^T \int_{\Omega \times D} M\,|\nabx\sqrt{\psia^{\Delta t,+}}|^2 \dq \dx \dt\right.\\
&&+  \left.4\int_0^T \int_{\Omega \times D} M\,|\nabx\sqrt{\psia^{\Delta t,-}}|^2\dq \dx \dt\right].
\end{eqnarray*}
Analogously,
\begin{eqnarray*}
4\int_0^T\int_{\Omega \times D} M\big|\nabq \sqrt{\psia^{\Delta t}}\big|^2\dq\dx \dt
&\leq& 2 \left[4\int_0^T \int_{\Omega \times D} M\,|\nabq\sqrt{\psia^{\Delta t,+}}|^2 \dq \dx \dt\right.\\
&&+  \left.4\int_0^T \int_{\Omega \times D} M\,|\nabq\sqrt{\psia^{\Delta t,-}}|^2\dq \dx \dt\right].
\end{eqnarray*}
%%%%%%%%%%%%%%%%%%%%%%%%%%%%%%%%%%%%%%%%%%%%%%%%%%%%%%%%%%%%%%%%%%%%%
Summing the last two inequalities and recalling \eqref{bound-bbb}, we then deduce that
\begin{equation}\label{bound-ccc}
\hspace{-1cm}4\int_0^T \int_{\Omega\times D} M\left[|\nabx\sqrt{\psia^{\Delta t}}|^2 +
|\nabq\sqrt{\psia^{\Delta t}}|^2\right] \dq \dx \dt  \leq C_\star,
\end{equation}
where, again, $C_\ast$ denotes a generic positive constant independent of $L$ and $\Delta t$.

Finally, on combining \eqref{bound-bbb} and \eqref{bound-ccc} with
\eqref{eq:energy-u+psi-final6} we arrive at the following bound, which represents the starting
point for the convergence analysis that will be developed in the next subsection.

With $\sigma \geq \frac{1}{2}d$, $\sigma>1$ and
$s > 1 + \frac{1}{2}(K+1)d$, we have that:
\begin{eqnarray}
&&\mbox{ess.sup}_{t \in [0,T]}\|\uta^{\Delta t(,\pm)}(t)\|^2 + \frac{1}{\Delta t} \int_0^T \|\uta^{\Delta t,+}(t) - \uta^{\Delta t,-}(t)\|^2
\dd t \nonumber\\
&&+\,\int_0^T \|\nabxtt \uta^{\Delta t(,\pm)}(t)\|^2 \dd t +  \mbox{ess.sup}_{t \in [0,T]}
\int_{\Omega \times D}\!\! M\, \mathcal{F}(\psia^{\Delta t(, \pm)}(t)) \dq \dx\nonumber\\
&&\qquad\qquad +\, \frac{1}{\Delta t\, L}
\int_0^T\!\! \int_{\Omega \times D}\!\! M\, (\psia^{\Delta t,+} - \psia^{\Delta t,-})^2 \dq \dx \dd t
\nonumber \\
&&+\, \int_0^T\!\! \int_{\Omega \times D} M\,
\big|\nabx \sqrt{\psia^{\Delta t(,\pm)}} \big|^2 \dq \dx \dd t
+\, \int_0^T\!\! \int_{\Omega \times D}M\,\big|\nabq \sqrt{\psia^{\Delta t(,\pm)}}\big|^2 \,\dq \dx \dd t\nonumber\\
&&\qquad\qquad +\, \int_0^T\left\|\frac{\partial \uta^{\Delta t}}{\partial t}(t)\right\|^2_{V_\sigma'}\dt + \int_0^T\left\|M\,\frac{\partial \psia^{\Delta
t}}{\partial t}(t)\right\|^2_{(H^s(\Omega \times D))'}\dt \leq C_\star.\label{eq:energy-u+psi-final6a}
\end{eqnarray}
Similarly,
\begin{eqnarray}\label{eq:energy-u+psi-final6a-zeta}
&&\mbox{ess.sup}_{t \in [0,T]}\|\rho^{\Delta t(,\pm)}_{\epsilon,L}(t)\|^2_{L^\infty(\Omega)}
+ \frac{1}{\Delta t} \int_0^T \|\rho^{\Delta t,+}_{\epsilon,L}(t) - \rho^{\Delta t,-}_{\epsilon,L}(t)\|^2
\dd t \nonumber\\
&&\qquad \quad +\, \int_0^T
\|\nabx \rho^{\Delta t(,\pm)}_{\epsilon,L}(t) \|^2\dd t + \int_0^T\left
\|\frac{\partial \rho_{\epsilon,L}^{\Delta t}}{\partial t}(t)\right\|^2_{(H^1(\Omega))'}\dt\leq
C_\star.
\end{eqnarray}
Here, the bound on the first term on the left-hand side follows from \eqref{zetancon}, \eqref{additional1} and \eqref{additional2};
the bound on the second term comes from \eqref{eq:energy-zeta}, and the bound on the last term from \eqref{zetacondtbd}. The bound on the
third term on the left-hand side of
\eqref{eq:energy-u+psi-final6a-zeta} is obtained by applying $\nabx$ to both sides of \eqref{zetancon} with the integrand
$$M(\qt)\hat\psi^{\Delta t (,\pm)}_{\epsilon,L}$$
rewritten as
$$M(\qt)\left[ \sqrt{\hat\psi^{\Delta t (,\pm)}_{\epsilon,L}}\,\right]^2,$$
then exchanging the order of $\nabx$ and the integral over $D$ on the right-hand side of
the resulting identity, applying $\nabx$ to the integrand using the chain rule,
followed by taking the modulus on both sides and applying the Cauchy--Schwarz inequality to the integral over $D$ on the right,
integrating the square of the resulting inequality over $[0,T] \times \Omega$
and, finally, recalling again the definition \eqref{zetancon} and using the bound on the first term in \eqref{eq:energy-u+psi-final6a-zeta}
and the bound on the sixth term in \eqref{eq:energy-u+psi-final6a}. In fact, in the
case of $\rho^{\Delta t,+}_{\epsilon,L}$ the stated bound on the third term on the left-hand side of
\eqref{eq:energy-u+psi-final6a-zeta} follows directly from \eqref{eq:energy-zeta}.

\subsection{Passage to the limit $L\rightarrow \infty$}
\label{passage}

We are now ready to prove the central result of the paper.
%In what follows,
%$\langle \cdot , \cdot \rangle_{H^s(\Omega)}$ denotes the duality pairing between $(H^s(\Omega))'$ %and
%$H^s(\Omega)$ relative to the pivot space $L^2(\Omega)$ with inner product $(\cdot,\cdot)$.

\begin{theorem}
\label{convfinal} Suppose that the assumptions \eqref{inidata} and the condition \eqref{LT},
relating $\Delta t$ to $L$, hold. Then,
there exists a subsequence of $\{(\utae^{\Delta t}, \hpsiae^{\Delta t})\}_{L >1}$ (not indicated)
with $\Delta t = o(L^{-1})$, and a pair of functions $(\ute, \hat\psi_\epsilon)$ such that
\[\ute \in L^{\infty}(0,T;\Lt^2(\Omega))\cap L^{2}(0,T;\Vt) \cap H^1(0,T;\Vt'_\sigma),\quad \sigma\geq \textstyle{\frac{1}{2}}d,\; \sigma>1,\]
and
\[\hat\psi_\epsilon \in L^1(0,T;L^1_M(\Omega \times D))\cap
H^1(0,T; M^{-1}(H^s(\Omega \times D))'),
\quad s>1 + \textstyle{\frac{1}{2}}(K+1)d,\]
with $\hat\psi_\epsilon \geq 0$ a.e. on $\Omega \times D \times [0,T]$,
\begin{equation}\label{mass-conserved}
\rho_\epsilon(\xt,t):=\int_D M(\qt)\,\hat\psi_\epsilon(\xt,\qt,t) \dq = 1\quad \mbox{for a.e. $(x,t) \in \Omega \times [0,T]$},
\end{equation}
whereby $\hat\psi_\epsilon \in L^\infty(0,T; L^1_M(\Omega \times D))$;
and finite relative entropy and Fisher information, with
\begin{equation}\label{relent-fisher}
\mathcal{F}(\hat\psi_\epsilon) \in L^\infty(0,T;L^1_M(\Omega\times D))\quad \mbox{and}\quad \sqrt{\hat\psi_\epsilon} \in L^{2}(0,T;H^1_M(\Omega \times D)),
\end{equation}
such that, as $L\rightarrow \infty$ (and thereby $\Delta t \rightarrow 0_+$),
\begin{subequations}
\begin{alignat}{2}
&\utae^{\Delta t (,\pm)} \rightarrow \ute \qquad &&\mbox{weak$^\star$ in }
L^{\infty}(0,T;{\Lt}^2(\Omega)), \label{uwconL2a}\\
%\bet
&\utae^{\Delta t (,\pm)} \rightarrow \ute \qquad &&\mbox{weakly in }
L^{2}(0,T;\Vt), \label{uwconH1a}\\
%\bet
&\utae^{\Delta t (,\pm)} \rightarrow \ute \qquad &&\mbox{strongly in }
L^{2}(0,T;\Lt^{r}(\Omega)), \label{usconL2a}\\
%\bet
&~\frac{\partial \utae^{\Delta t}}{\partial t} \rightarrow  \frac{\partial \ute}{\partial t}
\qquad &&\mbox{weakly in }
L^2(0,T;\Vt_\sigma'), \label{utwconL2a}
\end{alignat}
\end{subequations}
where $r \in [1,\infty)$ if $d=2$ and $r \in [1,6)$ if $d=3$;
and
\begin{subequations}
\begin{alignat}{2}
%\bet
&M^{\frac{1}{2}}\,\nabx \sqrt{\hpsiae^{\Delta t(,\pm)}} \rightarrow M^{\frac{1}{2}}\,\nabx \sqrt{\hat\psi_\epsilon}
&&\quad \mbox{weakly in } L^{2}(0,T;\Lt^2(\Omega\times D)), \label{psiwconH1a}\\
%\bet
&M^{\frac{1}{2}}\,\nabq \sqrt{\hpsiae^{\Delta t(,\pm)}} \rightarrow M^{\frac{1}{2}}\,\nabq \sqrt{\hat\psi_\epsilon}
&&\quad \mbox{weakly in } L^{2}(0,T;\Lt^2(\Omega\times D)), \label{psiwconH1xa}\\
%\bet
&~~~~~~~~~M\,\frac{\partial \hpsiae^{\Delta t}} {\partial t} \rightarrow
M\,\frac{\partial \hat\psi_\epsilon}{\partial t}
&&\quad \mbox{weakly in }
L^2(0,T;(H^s(\Omega\times D))'), \label{psitwconL2a}\\
%\bet
&~~~~~~~~~~~\hpsiae^{\Delta t (,\pm)} \rightarrow \hat\psi_\epsilon
&&\quad \mbox{strongly in }
L^{p}(0,T;L^{1}_M(\Omega\times D)),\label{psisconL2a}
\end{alignat}
for all $p \in [1,\infty)$; and,
\begin{alignat}{2}
%\bet
&\!\!\!\!\nabx\cdot\sum_{i=1}^K\Ctt_i(M\,\hpsiae^{\Delta t ,+}) \rightarrow \nabx \cdot\sum_{i=1}^K
\Ctt_i(M\,\hat\psi_\epsilon)
&&\quad\!\!\!\! \mbox{weakly in }
L^{2}(0,T;\Vt_{\sigma}').\!\label{CwconL2a}
\end{alignat}
\end{subequations}
The pair $(\ut_\epsilon,\hat\psi_\epsilon)$ is a global weak solution to problem (P$_\epsilon$), in the sense that
\begin{align}\label{equnconP}
&\displaystyle-\int_{0}^{T}
\int_{\Omega} \ut_\epsilon \cdot \frac{\partial \wt}{\partial t}
\dx \dt
+ \int_{0}^T \int_{\Omega}
\left[ \left[ (\ut_\epsilon \cdot \nabx) \ut_\epsilon \right]\,\cdot\,\wt
+ \nu \,\nabxtt \ut_\epsilon
:
\wnabtt \right] \dx \dt
\nonumber
\\
&\quad  = \int_\Omega \ut_0(\xt) \cdot \wt(\xt ,0) \dx  + \int_{0}^T
\left[ \langle \ft, \wt\rangle_{H^1_0(\Omega)}
- k\,\sum_{i=1}^K \int_{\Omega}
\Ctt_i(M\,\hat\psi_\epsilon): \nabxtt
\wt \dx \right] \dt
\nonumber
\\
& \hspace{1.8in}
\qquad \forall \wt \in W^{1,1}(0,T;\Vt_\sigma) \mbox{ s.t.\ $\wt(\cdot,T)=0$},
\end{align}
and
\begin{align}\label{eqpsinconP}
&-\int_{0}^T
\int_{\Omega \times D} M\,\hat\psi_\epsilon\, \frac{\partial \hat\varphi}{\partial t}
\dq \dx \dt
+ \int_{0}^T \int_{\Omega \times D} M\,\left[
\epsilon\, \nabx \hat\psi_\epsilon - \ut_\epsilon \,\hat\psi_\epsilon \right]\cdot\, \nabx
\hat \varphi
\,\dq \dx \dt
\nonumber \\
&
% \hspace{0.5cm}
+\frac{1}{2\,\lambda}
\int_{0}^T \int_{\Omega \times D} M\,\sum_{i=1}^K
 \,\sum_{j=1}^K A_{ij}\,\nabqj \hat\psi_\epsilon
\cdot\, \nabqi
\hat \varphi
\,\dq \dx \dt
\nonumber \\
&
%\hspace{0.5cm}
-
\int_{0}^T \int_{\Omega \times D}\! M\,\sum_{i=1}^K\,
\left[\sigtt(\ut_\epsilon)
\,\qt_i\right] \hat\psi_\epsilon \,\cdot\, \nabqi
\hat \varphi
\,\dq \dx \dt = \int_{\Omega \times D}\! \hat\psi_0(\xt,\qt)\,\hat\varphi(\xt,\qt,0)
\dq \dx
\nonumber \\
& \hspace{2.9cm}
\qquad \forall \hat \varphi \in W^{1,1}(0,T;H^s(\Omega\times D))
\mbox{ s.t.\ $\hat\varphi(\cdot,\cdot,T)=0$}.
\end{align}
%
%The initial condition $\ut_\epsilon(\cdot,0) = \ut_0(\cdot)$
%is satisfied in the sense of weakly continuous functions, in the space $C_w([0,T];\Ht)$;
%the initial condition $\hat\psi_\epsilon(\cdot,\cdot,0) = \hat\psi_0(\cdot,\cdot)$
%is satisfied in the sense of continuous functions, in $C([0,T];M^{-1}(H^s(\Omega \times %D))')$.
In addition, the function $\ut_\epsilon$ is weakly continuous as a mapping from $[0,T]$ to $\Ht$, and $\hat\psi_\epsilon$ is weakly continuous as a mapping from $[0,T]$ to $L^1_M(\Omega \times D)$.
%\smallskip
The weak solution $(\ut_\epsilon,\hat\psi_\epsilon)$ satisfies the following energy inequality for a.e. $t \in [0,T]$~\!\!:
\begin{eqnarray}\label{eq:energyest}
&&\!\!\|\ut_\epsilon(t)\|^2 + \nu \int_0^t \|\nabxtt \ut_\epsilon(s)\|^2 \dd s
+ \,k\int_{\Omega \times D}\!\! M\, \mathcal{F}(\hat\psi_\epsilon(t)) \dq \dx
\nonumber \\
&&\quad +\, 4\,k\,\varepsilon \int_0^t \int_{\Omega \times D} M\,
|\nabx \sqrt{\hat\psi_\epsilon} |^2 \dq \dx \dd s
+\, \frac{a_0 k}{\lambda}  \int_0^t \int_{\Omega \times D}M\,|\nabq \sqrt{\hat\psi_\epsilon}|^2 \,\dq \dx \dd s\nonumber\\
&&\hspace{-1mm}\leq \|\ut_0\|^2 + \frac{1}{\nu}\!\int_0^t\|\ft(s)\|^2_{(H^1_0(\Omega))'} \dd s + k\! \int_{\Omega \times D} M\, \mathcal{F}(\hat\psi_0) \dq \dx \leq [{\sf B}(\ut_0,\ft, \hat\psi_0)]^2,~~~~~~~
\end{eqnarray}
with $\mathcal{F}(s)= s(\log s - 1) + 1$, $s \geq 0$,
and $[{\sf B}(\ut_0,\ft, \hat\psi_0)]^2$ as defined in
{\rm (\ref{eq:energy-u+psi-final2})}.
\end{theorem}
\begin{proof} Since the proof is long, we have broken it up into a number of steps.

\smallskip

\textit{Step 1.} On recalling the weak$^\star$ compactness
of bounded balls in the Banach space $L^\infty(0,T;\Lt^2(\Omega))$ and noting the bound on
the first term on the left-hand side of \eqref{eq:energy-u+psi-final6a}, upon three successive
extractions of subsequences we deduce the existence of an unbounded
index set $\mathcal{L} \subset (1,\infty)$
such that each of the three sequences $\{\uta^{\Delta t(,\pm)}\}_{L\in \mathcal{L}}$
converges to its respective weak$^\star$ limit in $L^\infty(0,T;\Lt^2(\Omega))$ as $L \rightarrow \infty$
with $L \in \mathcal{L}$. Thanks to (\ref{ulin},b),
\begin{equation}\label{connection}
 \int_0^T \|\uta^{\Delta t}(s) - \uta^{\Delta t,+}(s)\|^2 \dd s = {\textstyle{\frac{1}{3}}}
\int_0^T \|\uta^{\Delta t,+}(s) - \uta^{\Delta t,-}(s)\|^2 \dd s  \leq \textstyle{\frac{1}{3}}
C_\star \Delta t,
\end{equation}
where the last inequality is a consequence of the second bound in \eqref{eq:energy-u+psi-final6a}.
On passing to the limit $L \rightarrow \infty$ with $L \in \mathcal{L}$ and using \eqref{LT}
we thus deduce that
the weak$^\star$ limits of the sequences $\{\uta^{\Delta t(,\pm)}\}_{L \in \mathcal{L}}$
coincide. We label this common limit by $\ute$; by construction then,
$\ute \in L^\infty(0,T;\Lt^2(\Omega))$. Thus we have shown \eqref{uwconL2a}.

Upon further successive extraction of subsequences from $\{\uta^{\Delta t(,\pm)}\}_{L\in \mathcal{L}}$
and noting the bounds on the third and eighth term on the left-hand side of
\eqref{eq:energy-u+psi-final6a} the limits (\ref{uwconH1a},d) follow
directly from the weak compactness of bounded balls in the Hilbert spaces $L^2(0,T;\Vt)$ and $L^2(0,T;\Vt'_\sigma)$ and \eqref{uwconL2a} thanks to the uniqueness of limits of
sequences in the weak topology of $L^2(0,T;\Vt)$ and $L^2(0,T;\Vt'_\sigma)$, respectively.

By the Aubin--Lions--Simon compactness theorem (cf. \eqref{compact1}), we then deduce
\eqref{usconL2a} in the case of $\uta^{\Delta t}$ on noting the compact embedding of
$\Vt$ into $\Lt^r(\Omega)\cap \Ht$, with the values of $r$ as in the statement of the theorem.
In particular, with $r=2$, $\{\uta^{\Delta t}\}_{L \in \mathcal{L}}$ converges to $\ute$,
strongly in $L^2(0,T;\Lt^2(\Omega))$ as $L\rightarrow \infty$.
Then, by the bound on the left-most term in \eqref{connection}, we deduce that
$\{\uta^{\Delta t,+}\}_{L \in \mathcal{L}}$ also converges to $\ute$, strongly in $L^2(0,T;\Lt^2(\Omega))$ as
$L \rightarrow \infty$ (and thereby $\Delta t \rightarrow 0_+$). Further,
by the bound on the middle term in \eqref{connection} we have that the same is true
of $\{\uta^{\Delta t,-}\}_{L \in \mathcal{L}}$. Thus we have shown that the three sequences $\{\uta^{\Delta t(,\pm)}\}_{L \in \mathcal{L}}$
all converge to $\ute$, strongly in $L^2(0,T;\Lt^2(\Omega))$.
Since the sequences $\{\uta^{\Delta t(,\pm)}\}_{L \in \mathcal{L}}$ are
bounded in $L^2(0,T;\Ht^1(\Omega))$ (cf. the bound on the third term in \eqref{eq:energy-u+psi-final6a})
and strongly convergent in $L^2(0,T;\Lt^2(\Omega))$, we deduce from  \eqref{eqinterp} that
\eqref{usconL2a} holds, with the values of $r$ as in the statement of the theorem.
Thus we have proved (\ref{uwconL2a}--d).

\smallskip

\textit{Step 2.}
Dubinski{\u\i}'s theorem, with $\mathcal{A}_0$, $\mathcal{A}_1$ and $\mathcal{M}$ as
in the discussion following the statement of Theorem \ref{thm:Dubinski} and selecting $p=1$
and $p_1=2$, %in Theorem \ref{thm:Dubinski},
implies that
\begin{align*}
&\left\{\varphi\,:\,[0,T] \rightarrow \mathcal{M}\,:\,
[\varphi]_{L^1(0,T;\mathcal M)} + \left\|\frac{{\rm d}\varphi}{{\rm d}t} \right\|_{L^{2}(0,T;\mathcal{A}_1)}
< \infty   \right\}
\\ & \hspace{2.3in}
\hookrightarrow\!\!\!\rightarrow L^1(0,T;\mathcal{A}_0)
= L^1(0,T;L^1_M(\Omega\times D)).
\end{align*}
Using this compact embedding, together with the bounds on the sixth, the seventh and the last
term on the left-hand side of \eqref{eq:energy-u+psi-final6a}, in conjunction with
\eqref{additional1} and \eqref{additional2}, we deduce (upon extraction of a subsequence)
strong convergence of $\{\psia^{\Delta t}\}_{L>1}$ in $L^1(0,T; L^1_M(\Omega \times D))$ to an element
$\hat\psi_\epsilon \in L^1(0,T; L^1_M(\Omega \times D))$, as $L \rightarrow \infty$.

Thanks to the bound on the fifth term in \eqref{eq:energy-u+psi-final6a}, by the
Cauchy--Schwarz inequality and an argument identical to the one in \eqref{connection},
we have that
\begin{align}
\left(\int_0^T\!\!\! \int_{\Omega \times D}\!\! M\, |\psia^{\Delta t} - \psia^{\Delta t,+}|\dq \dx \dt
\right)^2 \!&\leq \frac{T\, |\Omega|}{3} \!\! \displaystyle \int_0^T\!\!\!
\displaystyle \int_{\Omega \times D}\!\! M\, (\psia^{\Delta t,+} - \psia^{\Delta t,-})^2
\dq \dx \dt\nonumber\\
\label{difference}
&\leq \textstyle \frac{1}{3}C_\star T\, |\Omega|\,\Delta t\, L.
\end{align}
On recalling \eqref{LT}, and using the triangle inequality in the $L^1(0,T;L^1_M(\Omega \times D))$ norm, together
with \eqref{difference} and the strong  convergence of $\{\psia^{\Delta t}\}_{L>1}$ to
$\hat\psi_\epsilon$ in $L^1(0,T; L^1_M(\Omega \times D))$,
we deduce, as $L \rightarrow \infty$,
strong convergence of $\{\psia^{\Delta t,+}\}_{L>1}$ in $L^1(0,T; L^1_M(\Omega \times D))$ to the same element $\hat\psi_\epsilon \in L^1(0,T; L^1_M(\Omega \times D))$. The inequality \eqref{difference} then implies strong convergence of $\{\psia^{\Delta t,-}\}_{L>1}$ to $\hat\psi_\epsilon$, also. This completes the proof of \eqref{psisconL2a} for $p=1$.

From \eqref{additional1} and \eqref{additional2} we have that
\begin{equation}\label{psi-bound-6.1}
\|\psia^{\Delta t (,\pm)}(t)\|_{L^1_M(\Omega \times D)} \leq |\Omega|
\end{equation}
for a.e. $t$ in $[0,T]$ and all $L>1$.
In other words, the sequences $\{\psia^{\Delta t (,\pm)}\}_{L>1}$ are bounded in $L^\infty(0,T;L^1_M(\Omega \times D))$. By  Lemma \ref{le:supplementary},  the strong convergence of these
to $\hat\psi_\epsilon$ in $L^1(0,T;L^1_M(\Omega \times D))$, shown above, then implies
strong convergence in $L^p(0,T;L^1_M(\Omega \times D))$ to the same limit for all values of $p \in [1,\infty)$. That now completes the proof of \eqref{psisconL2a}.

Since strong convergence in $L^p(0,T;L^1_M(\Omega \times D))$, $p \geq 1$, implies convergence almost everywhere on $\Omega \times D \times [0,T]$ of a subsequence, it follows from \eqref{additional1} that $\hat\psi_\epsilon
\geq 0$ on $\Omega \times D \times [0,T]$.
Furthermore, by Fubini's theorem, strong convergence  of $\{\psia^{\Delta t(,\pm)}\}_{L>1}$ to $\hat\psi_\epsilon$ in $L^1(0,T;L^1_M(\Omega \times D))$ implies that
\[ \int_D M(\qt)\,|\psia^{\Delta t(,\pm)}(\xt,\qt,t) -  \hat\psi_\epsilon(\xt,\qt,t)|\dq \rightarrow 0 \qquad
\mbox{as $L \rightarrow \infty$}\]
for a.e. $(\xt,t) \in \Omega \times [0,T]$.

Hence we have that $\int_D M(\qt)\,\psia^{\Delta t(,\pm)}(\xt,\qt,t)\dq$
converges to $\int_D M(\qt)\,\hat\psi_\epsilon(\xt,\qt,t)\dq$, as $L \rightarrow \infty$,
for a.e. $(\xt,t) \in \Omega \times [0,T]$, and then \eqref{additional2} implies that
\begin{equation}\label{1boundonpsi}
\int_D M(\qt)\,\hat\psi_\epsilon(\xt,\qt,t)\dq \leq 1 \qquad \mbox{for a.e. $(x,t) \in \Omega \times [0,T]$.}
\end{equation}
We will show later that the inequality here can in fact be sharpened to an equality.

As the sequences $\{\psia^{\Delta t(,\pm)}\}_{L>1}$ converge to $\hat\psi_\epsilon$ strongly in
$L^1(0,T; L^1_M(\Omega \times D))$, it follows (upon extraction of suitable subsequences) that
they converge to $\hat\psi_\epsilon$ a.e. on
$\Omega \times D \times [0,T]$. This then, in turn, implies that the sequences
$\{\mathcal{F}(\psia^{\Delta t(,\pm)})\}_{L>1}$
converge to $\mathcal{F}(\hat\psi_\epsilon)$ a.e. on $\Omega \times D \times [0,T]$; in particular,
for a.e. $t \in [0,T]$, the sequences $\{\mathcal{F}(\psia^{\Delta t(,\pm)}(\cdot,\cdot,t))\}_{L>1}$
converge to $\mathcal{F}(\hat\psi_\epsilon(\cdot,\cdot,t))$ a.e. on $\Omega \times D$.
Since $\mathcal{F}$ is nonnegative, Fatou's lemma then implies that, for a.e. $t \in [0,T]$,
\begin{eqnarray}\label{fatou-app}
&&\int_{\Omega \times D} M(\qt)\, \mathcal{F}(\hat\psi_\epsilon(\xt,\qt,t))\dx \dq\nonumber\\
&&\qqquad \leq \mbox{lim inf}_{L \rightarrow \infty}
\int_{\Omega \times D} M(\qt)\, \mathcal{F}(\psia^{\Delta t(,\pm)}(\xt,\qt,t)) \dx \dq \leq C_\ast,
\end{eqnarray}
where the second inequality in \eqref{fatou-app} stems from
the bound on the fourth term on the left-hand side of \eqref{eq:energy-u+psi-final6a}. As the
integrand in the expression on the left-hand side of \eqref{fatou-app} is nonnegative, we deduce that $\mathcal{F}(\hat\psi_\epsilon)$
belongs to $L^\infty(0,T;L^1_M(\Omega \times D))$, as asserted in the statement of the theorem.

We observe in passing that since $|\sqrt{c_1} - \sqrt{c_2}\,|\leq \sqrt{|c_1-c_2 |}$ for any two nonnegative real
numbers $c_1$ and $c_2$, the strong convergence \eqref{psisconL2a} directly implies that, as $L \rightarrow \infty$ (and thereby $\Delta t \rightarrow 0_+$),
\begin{equation}\label{sqrtpsi}
\sqrt{ \psia^{\Delta t(,\pm)}} \rightarrow \sqrt{\hat\psi_\epsilon}\qquad \mbox{strongly in $L^p(0,T;L^2_M(\Omega \times D))$}\quad \forall p \in [1,\infty),
\end{equation}
and therefore, as $L \rightarrow \infty$ (and $\Delta t \rightarrow 0_+$),
\begin{equation}\label{strongpsi}
\!M^{\frac{1}{2}}\, \sqrt{\psia^{\Delta t(,\pm)}} \rightarrow M^{\frac{1}{2}}\,\sqrt{\hat\psi_\epsilon}\quad \mbox{strongly in $L^p(0,T;L^2(\Omega \times D))$}\quad \forall p \in [1,\infty).
\end{equation}
By proceeding in exactly the same way as in the previous subsection, between equations \eqref{strongpsi-zeta}
and \eqref{qlimitL2}, with $\hat\zeta^{\Lambda,1}$ and $\hat\zeta^{\!\!~1}=\hat\psi^0$ replaced by
$\psia^{\Delta t(,\pm)}$ and $\hat\psi_\epsilon$, respectively, but now using the sixth and the seventh bound in \eqref{eq:energy-u+psi-final6a}, and \eqref{additional2}, we deduce that (\ref{psiwconH1a},b) hold.

The convergence result \eqref{psitwconL2a} follows from the bound on the last term on the left-hand
side of \eqref{eq:energy-u+psi-final6a} and the weak compactness of bounded balls in the Hilbert
space $L^2(0,T; (H^s(\Omega \times D))')$, $s>1 + \textstyle{\frac{1}{2}}(K+1)d$.

The proof of \eqref{CwconL2a} is considerably more complicated, and will be given below.

\smallskip
After all these technical preparations we are now ready to return to \eqref{equncon} and
\eqref{eqpsincon} and pass to the limit $L\rightarrow \infty$ (and thereby also $\Delta t \rightarrow
0_+$); we shall also prove \eqref{CwconL2a}.
Since there are quite a few terms to deal with,
we shall discuss them one at a time, starting with equation
\eqref{eqpsincon}, and followed by equation \eqref{equncon}.

\smallskip

\textit{Step 3.}
We begin by passing to the limit $L \rightarrow \infty$ (and $\Delta t \rightarrow 0_+$) on equation
\eqref{eqpsincon}.
In what follows, we shall take test functions
$\hat\varphi \in C^1([0,T];C^\infty(\overline{\Omega\times D}))$ such that
$\hat\varphi(\cdot,\cdot,T)=0$. Note
that, for any $s \geq 0$, the set of all such test functions $\hat\varphi$ is a dense linear subspace of the linear space of functions in $W^{1,1}(0,T;H^s(\Omega \times D))$ vanishing at $t=T$.
As each of the terms in \eqref{eqpsincon} has been shown to be a continuous linear
functional with respect to $\hat\varphi$ on $L^2(0,T;H^s(\Omega \times D))$ for $s>1 + \frac{1}{2}(K+1)d$, and therefore also on $W^{1,1}(0,T;H^s(\Omega \times D))$ for
$s> 1 + \frac{1}{2}(K+1)d$, which is (continuously) embedded in $L^2(0,T;H^s(\Omega \times D))$ for $s>1 + \frac{1}{2}(K+1)d$, the use of such test functions for the purposes of the argument below is fully justified.

\smallskip

\textit{Step 3.1.} Integration by parts with respect to $t$ in the first term in \eqref{eqpsincon} gives
\begin{align}
\int_{0}^T \int_{\Omega \times D} M\,\frac{ \partial \hpsiae^{\Delta t}}{\partial t}\,
\hat \varphi \dq \dx \dt
= &- \int_0^T \int_{\Omega \times D} M\, \hpsiae^{\Delta t}\, \frac{\partial \hat\varphi}{\partial t} \dq \dx \dt \nonumber\\
&- \int_{\Omega \times D} M(\qt) \,\beta^L(\hat\psi^0(\xt,\qt))\, \hat\varphi(\xt, \qt, 0) \dq \dx \label{partialint}
\end{align}
for all $\hat \varphi \in C^1([0,T];C^\infty(\overline{\Omega \times D}))$ such that
$\hat\varphi(\cdot,\cdot,T)=0$. Using \eqref{psisconL2a} and noting point \ding{206} of Lemma \ref{psi0properties}, we immediately have that, as $L \rightarrow \infty$ (and $\Delta t \rightarrow 0_+$), the first term on the right-hand
side of \eqref{partialint} converges to the first term on the left-hand side of \eqref{eqpsinconP} and the second term on the right-hand side of \eqref{partialint} converges to $-\int_{\Omega \times D} \hat\psi_0 (\xt, \qt)\, \hat\varphi(\xt,\qt,0) \dq \dx$, resulting
in the first term on the right-hand side of \eqref{eqpsinconP}. That completes Step 3.1.

\smallskip

\textit{Step 3.2.} The second term in \eqref{eqpsincon} will be dealt with by decomposing it into two further terms, the first of
which tends to $0$, while the second converges to the expected limiting value. We proceed as follows:
\begin{align*}
&\epsilon \int_{0}^T \int_{\Omega \times D} M\, \nabx \hpsiae^{\Delta
t,+} \cdot\, \nabx \hat \varphi \,\dq \dx \dt \nonumber\\
&\qqquad= 2\epsilon \int_{0}^T \int_{\Omega \times D} M
\left(\sqrt{\psia^{\Delta t,+}} - \sqrt{\hat\psi_\epsilon}\,\right)\, \nabx \sqrt{\psiae^{\Delta t,+}} \cdot\, \nabx \hat \varphi \,\dq \dx \dt\nonumber\\
&\qqquad\quad+2\epsilon \int_{0}^T \int_{\Omega \times D} M\,
\sqrt{\hat\psi_\epsilon}\, \nabx \sqrt{\psia^{\Delta t,+}} \cdot\, \nabx \hat \varphi \,\dq \dx \dt\nonumber\\
&\qqquad =: {\rm V}_1 + {\rm V_2}.
\end{align*}
We shall show that ${\rm V}_1$ converges to $0$ and that ${\rm V}_2$ converges to the expected limit. First, note that
\begin{align}
&|{\rm V}_1|  \leq  2 \epsilon \int_0^T \int_{\Omega}\left(\int_D M|\sqrt{\psia^{\Delta t,+}}
- \sqrt{\hat\psi_\epsilon}|^2 \dq\right)^{\frac{1}{2}}\nonumber\\
&\qqqquad\times \left(\int_D M\,|\nabx\sqrt{\psia^{\Delta t,+}}|^2 \dq\right)^{\frac{1}{2}}\|\nabx \hat\varphi\|_{L^\infty(D)} \dx \dt
\nonumber\\
&\qquad \leq 2 \epsilon\int_0^T \left(\int_{\Omega \times D} M|\sqrt{\psia^{\Delta t,+}}
- \sqrt{\hat\psi_\epsilon}|^2 \dq\dx\right)^{\frac{1}{2}} \nonumber\\
&\qqqquad \times \left(\int_{\Omega\times D}M\,|\nabx \sqrt{\psia^{\Delta t,+}}|^2\dq \dx \right)^{\frac{1}{2}}\|\nabx\hat\varphi\|_{L^\infty(\Omega \times D)}\dt\nonumber\\
%\end{align}
%\begin{align}
&\qquad =  2 \epsilon\int_0^T \| \sqrt{M\,\psia^{\Delta t,+}}
- \sqrt{M\,\hat\psi_\epsilon}\|_{L^2(\Omega \times D)}\nonumber\\
&\qqqquad \times \|M^{\frac{1}{2}}\,\nabx \sqrt{\psia^{\Delta t,+}}\|_{L^2(\Omega \times D)}\,\|\nabx\hat\varphi\|_{L^\infty(\Omega \times D)}\dt \nonumber\\
&\qquad \leq 2 \epsilon\left(\int_0^T \|M^{\frac{1}{2}}\,\nabx \sqrt{\psia^{\Delta t,+}}\|^2_{L^2(\Omega \times D)}\dt\right)^{\frac{1}{2}}\nonumber\\
&\qqqquad \times \left(\int_0^T\!\! \| \sqrt{M\,\psia^{\Delta t,+}}
- \sqrt{M\,\hat\psi_\epsilon}\|_{L^2(\Omega \times D)}^r \dt \right)^{\!\frac{1}{r}}\!\!\left(\int_0^T\!\! \|\nabx \hat\varphi\|_{L^\infty(\Omega \times D)}^{\frac{2r}{r-2}} \dt\right)^{\!\frac{r-2}{2r}}\!\!\!,
\nonumber
\end{align}
were $r \in (2,\infty)$. Using the bound on the sixth term in \eqref{eq:energy-u+psi-final6} together with the
Sobolev embedding theorem, we then have (with $C_\ast$ now denoting a possibly different constant than in
\eqref{eq:energy-u+psi-final6}, but one that is still independent of $L$ and $\Delta t$) that
\begin{align}
&|{\rm V}_1| \leq 2 C_\ast^{\frac{1}{2}} \epsilon \| \sqrt{M\,\psia^{\Delta t,+}} - \sqrt{M\,\hat\psi_\epsilon}\,\|_{L^r(0,T;L^2(\Omega \times D))}\, \|\nabx\hat\varphi\|_{L^{\frac{2r}{r-2}}(0,T;L^\infty(\Omega \times D))}\nonumber\\
&\qquad \!\leq 2 C_\ast^{\frac{1}{2}} \epsilon\, \|\psia^{\Delta t,+} - \hat\psi_\epsilon\|_{L^{\frac{r}{2}}(0,T;L^1_M(\Omega \times D))}^{\frac{1}{2}}\, \|\nabx\hat\varphi\|_{L^{\frac{2r}{r-2}}(0,T;L^\infty(\Omega \times D))} ,\nonumber
\end{align}
where we also used the elementary inequality $|\sqrt{c_1} - \sqrt{c_2}| \leq \sqrt{|c_1-c_2|}$ with $c_1, c_2 \in \mathbb{R}_{\geq 0}$.
The norm of the difference in the last displayed line is known to converge to $0$ as $L \rightarrow \infty$
(and $\Delta t \rightarrow 0_+$) by \eqref{psisconL2a}. This then implies that
the term ${\rm V}_1$ converges to $0$ as  $L \rightarrow \infty$
(and $\Delta t \rightarrow 0_+$).

Concerning the term ${\rm V}_2$, we have that
\begin{align*}
&{\rm V}_2 = 2\epsilon \int_{0}^T \int_{\Omega \times D} M^{\frac{1}{2}}\,
\nabx \sqrt{\psia^{\Delta t,+}} \cdot\, \sqrt{M\,\hat\psi_\epsilon}\, \nabx \hat \varphi \,\dq \dx \dt.
\end{align*}
Once we have verified
that $\sqrt{M\,\hat\psi_\epsilon}\, \nabx \hat \varphi$ belongs to $L^2(0,T;\Lt^2(\Omega\times D))$,
the weak convergence result \eqref{psiwconH1a} will imply that
\begin{align*}
&{\rm V}_2 \rightarrow 2\epsilon \int_{0}^T \int_{\Omega \times D} M^{\frac{1}{2}}\,
\nabx \sqrt{\hat\psi_\epsilon} \cdot\, \sqrt{M\,\hat\psi_\epsilon}\, \nabx \hat \varphi \,\dq \dx \dt\\
&\qqqquad = \epsilon \int_{0}^T \int_{\Omega \times D} M\,
\nabx \hat\psi_\epsilon \cdot\ \nabx \hat \varphi \,\dq \dx \dt
\end{align*}
as $L\rightarrow \infty$ (and $\Delta t \rightarrow 0_+$), and we will have completed Step 3.2. Let us therefore show that $\sqrt{M\,\hat\psi_\epsilon}\,\nabx\hat\varphi$
belongs to $L^2(0,T;\Lt^2(\Omega \times D))$; the justification is quite straightforward: using \eqref{1boundonpsi} we have that
\begin{align*}
&\int_0^T\int_{\Omega \times D} |\sqrt{M\,\hat\psi_\epsilon}\,\nabx\hat\varphi|^2 \dx \dt
= \int_0^T \int_{\Omega \times D} M\,\hat\psi_\epsilon\,|\nabx\hat\varphi|^2 \dq \dx \dt\nonumber\\
&\qquad \leq \int_0^T \int_\Omega \|\nabx\hat\varphi\|^2_{L^\infty(D)}\,\left(\int_D M\,\hat\psi_\epsilon \dq\right) \dx \dt\nonumber\\
&\qquad \leq \int_0^T \int_{\Omega} \|\nabx\hat\varphi\|^2_{L^\infty(D)} \dx \dt\\
&\qquad = \|\nabx \hat\varphi\|^2_{L^2(0,T;L^2(\Omega;L^\infty(D)))} < \infty.
\end{align*}
That now completes Step 3.2.

\smallskip

\textit{Step 3.3.} The third term in \eqref{eqpsincon} is dealt with as follows:
\begin{align*}
& - \int_{0}^T \int_{\Omega \times D} M\, \utae^{\Delta t,-}\,\hpsiae^{\Delta t,+} \cdot\, \nabx
\hat \varphi \,\dq \dx \dt = - \int_{0}^T \int_{\Omega \times D} M\, \ut_\epsilon \,\hat\psi_\epsilon \cdot\, \nabx \hat \varphi \,\dq \dx \dt\\
&\qqqquad + \int_{0}^T \int_{\Omega \times D} M\, (\ut_\epsilon- \utae^{\Delta t,-})\,\psia^{\Delta t,+} \cdot\, \nabx \hat \varphi \,\dq \dx \dt\\
&\qqqquad +  \int_{0}^T \int_{\Omega \times D} M\, \ut_\epsilon\,(\hat\psi_\epsilon - \hpsiae^{\Delta t,+}) \cdot\, \nabx \hat \varphi \,\dq \dx \dt.
\end{align*}
We label the last two terms by ${\rm V}_3$ and ${\rm V}_4$ and we show that each of them converges to
$0$ as $L \rightarrow 0$ (and $\Delta t \rightarrow 0_+$). We start with term ${\rm V}_3$; below, we
apply H\"older's inequality with $r \in (1,\infty)$ in the case of $d=2$ and with $r \in (1,6)$ when  $d=3$:
\begin{align*}
&~\!|{\rm V}_3| = \left|\int_{0}^T \int_{\Omega} (\ut_\epsilon- \utae^{\Delta t,-})\cdot\left[\int_D
M\psia^{\Delta t,+} \left(\nabx \hat \varphi\right) \dq\right] \dx \dt\right|\\
&\qquad \leq \int_0^T \int_\Omega |\ut_\epsilon- \utae^{\Delta t,-}| \left[\int_D
M\psia^{\Delta t,+}\dq\right] \|\nabx \hat \varphi\|_{L^\infty(D)}
\dx \dt\\
&\qquad \leq \int_0^T \int_\Omega |\ut_\epsilon- \utae^{\Delta t,-}| \, \|\nabx \hat \varphi\|_{L^\infty(D)} \dx \dt\\
&\qquad \leq \int_0^T\left(\int_\Omega |\ut_\epsilon- \utae^{\Delta t,-}|^r\dx\right)^{\frac{1}{r}} \left(\int_\Omega\|\nabx \hat \varphi\|^{\frac{r}{r-1}}_{L^\infty(D)}\dx \right)^{\frac{r-1}{r}}\dt\\
&\qquad \leq \int_0^T \|\ut_\epsilon- \utae^{\Delta t,-}\|_{L^r(\Omega)}\,\|\nabx \hat \varphi\|_{L^{\frac{r}{r-1}}(\Omega; L^\infty(D))}\dt\\
&\qquad \leq \|\ut_\epsilon- \utae^{\Delta t,-}\|_{L^2(0,T;L^r(\Omega))}\,\|\nabx \hat \varphi\|_{L^2(0,T;L^{\frac{r}{r-1}}(\Omega; L^\infty(D)))},
\end{align*}
where in the transition from the second line to the third line we made use of \eqref{properties-b}.
Thanks to \eqref{usconL2a} the first factor in the last line
converges to $0$, and hence ${\rm V}_3$ converges to $0$
also, as $L\rightarrow \infty$ (and $\Delta t \rightarrow 0_+$).

For ${\rm V}_4$, we have, by using Fubini's theorem, the factorization
\begin{equation}\label{psi-fact}
M\left(\hat\psi_\epsilon - \hpsiae^{\Delta t,+}\right) = M^{\frac{1}{2}}\left(\sqrt{\hat\psi_\epsilon} - \sqrt{\hpsiae^{\Delta t,+}}\right)\,  M^{\frac{1}{2}}\left(\sqrt{\hat\psi_\epsilon} + \sqrt{\hpsiae^{\Delta t,+}}\right),
\end{equation}
together with the Cauchy--Schwarz inequality, \eqref{properties-b}, \eqref{1boundonpsi} and the %elementary
inequality
$|\sqrt{c_1} - \sqrt{c_2}~\!| \leq \sqrt{|c_1-c_2|}$ with $c_1, c_2 \in \mathbb{R}_{\geq 0}$, that
\begin{align*}
&~|{\rm V}_4| \leq \int_{0}^T \int_{\Omega \times D} M\, |\ut_\epsilon|\,|\hat\psi_\epsilon - \hpsiae^{\Delta t,+}|\,|\nabx \hat \varphi| \,\dq \dx \dt\\
&\qquad \leq \int_{0}^T \int_{\Omega} |\ut_\epsilon|\,\left(\int_D M\, |\hat\psi_\epsilon - \hpsiae^{\Delta t,+}| \dq\right) \|\nabx \hat \varphi\|_{L^\infty(D)}\dx \dt\\
&\qquad \leq {2}\,\int_{0}^T \int_{\Omega} |\ut_\epsilon|\,\left(\int_D M\, |\sqrt{\hat\psi_\epsilon} - \sqrt{\hpsiae^{\Delta t,+}}|^2 \dq\right)^{\frac{1}{2}} \|\nabx \hat \varphi\|_{L^\infty(D)}\dx \dt\\
&\qquad \leq {2}\,\int_{0}^T \left(\int_{\Omega} |\ut_\epsilon|^2\dx\right)^{\frac{1}{2}} \,\left(\int_{\Omega \times D} M\, |\sqrt{\hat\psi_\epsilon} - \sqrt{\hpsiae^{\Delta t,+}}|^2\dq \dx \right)^{\frac{1}{2}}\\
&\qqqquad\hspace{6.6cm} \times  \|\nabx \hat \varphi\|_{L^\infty(\Omega \times D)}\dt\\
&\qquad \leq {2}\,\int_{0}^T \|\ut_\epsilon\|_{L^2(\Omega)} \,\left(\int_{\Omega \times D} M\, |\hat\psi_\epsilon - \hpsiae^{\Delta t,+}|\dq \dx \right)^{\frac{1}{2}}\\
&\qqqquad\hspace{6.6cm} \times  \|\nabx \hat \varphi\|_{L^\infty(\Omega \times D)}\dt\\
&\qquad \leq {2}\, \|\ut_\epsilon\|_{L^\infty(0,T;L^2(\Omega))} \|\hat\psi_\epsilon - \hpsiae^{\Delta t,+}\|^{\frac{1}{2}}_{L^1(0,T;L^1_M(\Omega \times D))}  \|\nabx \hat \varphi\|_{L^2(0,T;L^\infty(\Omega \times D))}.
\end{align*}
By \eqref{uwconL2a} the first factor in the last line is finite while, according to \eqref{psisconL2a} (with $p=1$),
the middle factor converges to $0$ as $L \rightarrow \infty$ (and $\Delta t \rightarrow 0_+$). This
proves that ${\rm V}_4$ converges to $0$ as $L \rightarrow \infty$ (and $\Delta t \rightarrow 0_+$),
also. That completes Step 3.3.

\smallskip

\textit{Step 3.4.} Thanks to \eqref{psiwconH1xa}, as $L\rightarrow \infty$ (and $\Delta t \rightarrow 0_+$),
\begin{align*}
&M^{\frac{1}{2}}\,\nabq \sqrt{\hpsiae^{\Delta t(,\pm)}} \rightarrow M^{\frac{1}{2}}\,\nabq \sqrt{\hat\psi_\epsilon}
&&\quad \mbox{weakly in } L^{2}(0,T;\Lt^2(\Omega\times D)).
\end{align*}
This, in turn, implies that, componentwise, as $L\rightarrow \infty$ (and $\Delta t \rightarrow 0_+$),
\begin{align*}
&M^{\frac{1}{2}}\,\nabqj \sqrt{\hpsiae^{\Delta t(,\pm)}} \rightarrow M^{\frac{1}{2}}\,\nabqj \sqrt{\hat\psi_\epsilon}
&&\quad \mbox{weakly in } L^{2}(0,T;\Lt^2(\Omega\times D)),
\end{align*}
for each $j=1,\dots, K$, whereby also,
\begin{align*}
&M^{\frac{1}{2}}\, \sum_{j=1}^K A_{ij}\nabqj \sqrt{\hpsiae^{\Delta t(,\pm)}} \rightarrow M^{\frac{1}{2}}\, \sum_{j=1}^KA_{ij}\nabqj \sqrt{\hat\psi_\epsilon}
&&\quad \mbox{weakly in } L^{2}(0,T;\Lt^2(\Omega\times D)),
\end{align*}
for each $i=1,\dots,K$. That places us in a very similar position as in the case of Step 3.2, and we can argue in an
identical manner as there to show that
\begin{align*}
&\frac{1}{2\,\lambda}
\int_{0}^T \int_{\Omega \times D} M\,\sum_{i=1}^K
 \,\sum_{j=1}^K A_{ij}\,\nabqj \hpsiae^{\Delta t,+}
\cdot\, \nabqi
\hat \varphi
\,\dq \dx \dt \nonumber\\
&\qquad\rightarrow \frac{1}{2\,\lambda}
\int_{0}^T \int_{\Omega \times D} M\,\sum_{i=1}^K
 \,\sum_{j=1}^K A_{ij}\,\nabqj \hat\psi_\epsilon
\cdot\, \nabqi \hat \varphi
\,\dq \dx \dt
\end{align*}
as $L \rightarrow \infty$ and $\Delta t \rightarrow 0_+$, for all $\hat\varphi \in L^{\frac{2r}{r-2}}(0,T;W^{1,\infty}(\Omega \times D))$, $r \in (2,\infty)$, and in particular for all $\hat\varphi \in C^1([0,T];C^\infty(\overline{\Omega \times D}))$. That completes Step 3.4.

\smallskip

\textit{Step 3.5.} The final term in \eqref{eqpsincon}, the drag term, is the one in the
equation that is the most difficult to deal with.
We shall break it up into four subterms, three of which will be shown to converge to
$0$ in the limit of $L\rightarrow \infty$ (and $\Delta t\rightarrow 0_{+}$), leaving the
fourth term as the (expected) limiting value:
\begin{align}\label{bound-3.5}
& -
\int_{0}^T \int_{\Omega \times D} M\,\sum_{i=1}^K
[\sigtt(\utae^{\Delta t,+})
\,\qt_i]\,\beta^L(\hpsiae^{\Delta t,+}) \,\cdot\, \nabqi
\hat \varphi
\,\dq \dx \dt \nonumber\\
&= - \int_{0}^T \int_{\Omega \times D} M\,\sum_{i=1}^K
\left[\left(\nabxtt\utae^{\Delta t,+}\right) \,\qt_i\right]\,\beta^L(\hpsiae^{\Delta t,+}) \,\cdot\, \nabqi
\hat \varphi
\,\dq \dx \dt\nonumber\\
&= - \int_{0}^T \int_{\Omega \times D} M\,\sum_{i=1}^K
\left[\left(\nabxtt\utae^{\Delta t,+}\right) \,\qt_i\right]\,\left(\beta^L(\hpsiae^{\Delta t,+})
- \beta^L(\hat\psi_\epsilon)\right)\,\cdot\, \nabqi
\hat \varphi
\,\dq \dx \dt\nonumber\\
&\,\quad- \int_{0}^T \int_{\Omega \times D} M\,\sum_{i=1}^K
\left[\left(\nabxtt\utae^{\Delta t,+}\right) \,\qt_i\right]\,\left(\beta^L(\hat\psi_\epsilon)
- \hat\psi_\epsilon\right)\,\cdot\, \nabqi
\hat \varphi
\,\dq \dx \dt\nonumber\\
&\,\quad- \int_{0}^T \int_{\Omega \times D} M\,\sum_{i=1}^K
\left[\left(\nabxtt\utae^{\Delta t,+}-\nabxtt\ut_\epsilon\right) \,\qt_i\right]\, \hat\psi_\epsilon\,\cdot\, \nabqi
\hat \varphi
\,\dq \dx \dt\nonumber\\
&\,\quad- \int_{0}^T \int_{\Omega \times D} M\,\sum_{i=1}^K
\left[\left(\nabxtt\ut_\epsilon\right) \,\qt_i\right]\, \hat\psi_\epsilon\,\cdot\, \nabqi
\hat \varphi
\,\dq \dx \dt.
\end{align}
Strictly speaking, we should have written ``$L \rightarrow \infty$, with $L \in \mathcal{L}$,'' as in Step 1 above, instead of ``$L \rightarrow \infty$''; for the sake of brevity we chose to use the latter, compressed notation. The same notational convention applies below.

We label the first three terms on the right-hand side by ${\rm V}_5$, ${\rm V}_6$, ${\rm V}_7$, respectively.
We shall show that each of the three terms converges to $0$, leaving the fourth
term as the limit of the left-most expression in the chain, as $L\rightarrow \infty$ (and $\Delta t
\rightarrow 0_+$).

We begin by bounding the term ${\rm V}_5$, noting that $\beta^L$ is Lipschitz
continuous, with Lipschitz constant $1$, writing, as before $b:=|\bt|_1$, and using the
factorization \eqref{psi-fact} together with \eqref{properties-b} and
(\ref{1boundonpsi}), and then proceeding as
in the case of term ${\rm V}_4$ in Step 3.3:
\begin{align*}
&|{\rm V}_5| \leq \sqrt{b}\,\int_{0}^T \int_{\Omega \times D} M\, |\nabxtt\utae^{\Delta t,+}|\, |\hpsiae^{\Delta t,+} - \hat\psi_\epsilon|\,|\nabq \hat \varphi|\,\dq \dx \dt\\
&\qquad\leq \,\sqrt{b}\,\int_{0}^T \left[\int_\Omega |\nabxtt\utae^{\Delta t,+}|\,\left(\int_{D} M\, \, |\hpsiae^{\Delta t,+} - \hat\psi_\epsilon|\,\dq\right)\dx \right]\|\nabq \hat \varphi\|_{L^\infty(\Omega \times D)} \dt\\
&\qquad\leq \,2\sqrt{b}\,\|\nabxtt\utae^{\Delta t,+}\|_{L^2(0,T;L^2(\Omega))}\,\|\hpsiae^{\Delta t,+} - \hat\psi_\epsilon\|^{\frac{1}{2}}_{L^1(0,T;L^1_M(\Omega\times D))}\\
&\qqqquad\hspace{6.6cm} \times   \|\nabq \hat \varphi\|_{L^\infty((0,T)\times \Omega \times D))}.
\end{align*}
On noting the bound on the third term on the left-hand side of \eqref{eq:energy-u+psi-final6a}
and the convergence result \eqref{psisconL2a} that was proved in Step 2, we deduce that term
${\rm V}_5$ converges to $0$ as $L\rightarrow \infty$ (and $\Delta t \rightarrow 0_+$).

We move on to term ${\rm V}_6$, using an identical argument as in the case of term ${\rm V}_5$:
\begin{align*}
&|{\rm V}_6| \leq \sqrt{b}\,\int_{0}^T \int_{\Omega \times D} M\,|\nabxtt\utae^{\Delta t,+}| \,|\beta^L(\hat\psi_\epsilon)
- \hat\psi_\epsilon|\, |\nabq \hat \varphi| \dq \dx \dt\\
&\qquad\!\leq \sqrt{b}\,\int_{0}^T \left[\int_{\Omega} |\nabxtt\utae^{\Delta t,+}| \left(\int_D M\,|\beta^L(\hat\psi_\epsilon)
- \hat\psi_\epsilon|\,\dq\right) \dx \right] \|\nabq \hat \varphi\|_{L^\infty(\Omega \times D)} \dt\\
&\qquad\leq \,2\sqrt{b}\,\|\nabxtt\utae^{\Delta t,+}\|_{L^2(0,T;L^2(\Omega))}\,\|\beta^{L}(\hat\psi_\epsilon) - \hat\psi_\epsilon\|^{\frac{1}{2}}_{L^1(0,T;L^1_M(\Omega\times D))}\\
&\qqqquad\hspace{6.6cm} \times   \|\nabq \hat \varphi\|_{L^\infty((0,T)\times \Omega \times D))}.
\end{align*}
Observe that $0 \leq \hat\psi_\epsilon - \beta^{L}(\hat\psi_\epsilon) \leq \hat\psi_\epsilon$
and that $\hat\psi_\epsilon - \beta^{L}(\hat\psi_\epsilon)$ converges to $0$ almost everywhere on
$\Omega \times D \times (0,T)$ as $L \rightarrow \infty$.
Note further that, thanks to \eqref{psisconL2a} with $p=1$,
$\hat\psi_\epsilon \in L^1(0,T;L^1_M(\Omega \times D))$. Thus, Lebesgue's dominated convergence
theorem implies that, as $L \rightarrow \infty$, the middle factor in the last displayed
line converges to $0$. Hence, recalling the bound on the third term on the left-hand side of
\eqref{eq:energy-u+psi-final6a}, we thus deduce that ${\rm V}_6$ converges to $0$ as
$L\rightarrow \infty$ (and $\Delta t \rightarrow 0_+$).

Finally, we consider the term ${\rm V}_7$:
\begin{align*}
&\,\quad {\rm V}_7:=- \int_{0}^T \int_{\Omega \times D} M\,\sum_{i=1}^K
\left[\left(\nabxtt\utae^{\Delta t,+}-\nabxtt\ut_\epsilon\right) \,\qt_i\right]\, \hat\psi_\epsilon\,\cdot\, \nabqi \hat \varphi \,\dq \dx \dt.
\end{align*}
We observe that, before starting to bound ${\rm V}_7$, we should perform an integration by
parts in order to transfer the $x$-gradients from the difference $\nabxtt\utae^{\Delta t,+}-\nabxtt\ut_\epsilon$ onto the other factors
under the integral sign, as we only have weak, but not strong, convergence of $\nabxtt\utae^{\Delta t,+}-\nabxtt\ut_\epsilon$ to $0$, (cf. \eqref{uwconH1a}) whereas the difference $\utae^{\Delta t,+}-\ut_\epsilon$ converges to $0$
strongly by virtue of \eqref{usconL2a}.

We note is this respect
that the function $\xt \in \Omega \mapsto \hat\psi_\varepsilon(\xt,\qt,t) \in \mathbb{R}_{\geq 0}$
has a well-defined trace
on $\partial \Omega$ for a.e. $(\qt,t) \in D \times (0,T)$, since, thanks to \eqref{psiwconH1a},
\[
\mbox{$\sqrt{\hat\psi_\epsilon(\cdot,\qt,t)} \in H^1(\Omega),\qquad$ and therefore $\qquad\left.\sqrt{\hat\psi_\epsilon(\cdot,\qt,t)}\right|_{\partial \Omega} \in H^{1/2}(\partial
\Omega)$,}\]
for a.e. $(\qt,t) \in D \times (0,T)$, implying that $\sqrt{\hat\psi_\epsilon(\cdot,\qt,t)}\,|_{\partial\Omega} \in L^{2p}(\partial \Omega)$ for a.e. $(\qt,t) \in D \times (0,T)$, with $2p \in [1,\infty)$ when $d=2$ and $2p \in [1,4]$ when $d=3$, whereby $\hat\psi_\epsilon|_{\partial \Omega} \in L^{p}(\partial \Omega)$ for a.e. $(\qt,t) \in D \times (0,T)$, with $p \in [1,\infty)$ when $d=2$ and $p \in [1,2]$ when $d=3$.
As the functions $\ut_\epsilon$ and $\ut_{\epsilon,L}^{\Delta t, +}$ have zero trace
on $\partial\Omega$, the boundary integral that arises in the course of integration by parts is correctly defined and,
in fact, vanishes. With these preliminary remarks in mind, we first write
\begin{align*}
&{\rm V}_7=- \!\int_{0}^T\!\!\! \int_{\Omega \times D}\!\!\! M\sum_{i=1}^K \sum_{m,n=1}^d\!\frac{\partial}{\partial  x_m} \left[\left((\utae^{\Delta t,+})_n-(\ut_\epsilon)_n\right) (\qt_i)_m\right]\!\hat\psi_\epsilon \, (\nabqi \hat \varphi)_n \dq \dx \dt.
\end{align*}
Here, $(\utae^{\Delta t,+})_n$ and $(\ut_\epsilon)_n$ denote the $n$th among the $d$ components of the
vectors $\utae^{\Delta t,+}$ and $\ut_\epsilon$, $1 \leq n \leq d$, respectively,
and $(\nabqi \hat \varphi)_n$ denotes the $n$th among the $d$ components of the vector
$\nabqi \hat \varphi$, $1 \leq n \leq d$, for each $i \in \{1,\dots, K\}$. Similarly,
$(\qt_i)_m$ denotes the $m$th component, $1 \leq m \leq d$, of the $d$-component vector
$\qt_i$ for $i \in \{1,\dots, K\}$. Now, on integrating by parts w.r.t. $x_m$
and cancelling the boundary integral terms, with the justification given above, we have that

\begin{align*}
&{\rm V}_7= \int_{0}^T\!\!\! \int_{\Omega\times D}\!\!M\sum_{i=1}^K\sum_{m,n=1}^d
\left[\left((\utae^{\Delta t,+})_n-(\ut_\epsilon)_n\right) \,(\qt_i)_m\right] \frac{\partial}{\partial  x_m}\left(\hat\psi_\epsilon\, (\nabqi \hat \varphi)_n\right)\!\dq \dx \dt\\
&\quad\,\,= \int_{0}^T\!\!\! \int_{\Omega\times D}\!\!M\sum_{i=1}^K\sum_{m,n=1}^d
\left[\left((\utae^{\Delta t,+})_n-(\ut_\epsilon)_n\right) \,(\qt_i)_m\right] \frac{\partial \hat\psi_\epsilon}{\partial  x_m}\, (\nabqi \hat \varphi)_n\dq \dx \dt\\
&\quad\quad\,\,+ \int_{0}^T\!\!\! \int_{\Omega\times D}\!\!M\sum_{i=1}^K\sum_{m,n=1}^d
\left[\left((\utae^{\Delta t,+})_n-(\ut_\epsilon)_n\right) \,(\qt_i)_m\right] \left(\hat\psi_\epsilon\,\frac{\partial}{\partial  x_m} (\nabqi \hat \varphi)_n\right)\!\dq \dx \dt\\
&\quad\,=:{\rm V}_{7,1} + {\rm V}_{7,2}.
\end{align*}
For the term ${\rm V}_{7,1}$, we have that
\begin{align*}
&|{\rm V}_{7,1}| \leq \int_{0}^T\!\!\! \int_{\Omega\times D}\!\!M\,\sum_{i=1}^K\sum_{m,n=1}^d
\left|(\utae^{\Delta t,+})_n-(\ut_\epsilon)_n\right|\,|(\qt_i)_m|\, \left|\frac{\partial \hat\psi_\epsilon}{\partial  x_m}\right|\, \left|(\nabqi \hat \varphi)_n\right|\,\dq \dx \dt\\
&\quad\,\,\leq \int_{0}^T\!\!\! \int_{\Omega\times D}\!\!M\,\left(\sum_{i=1}^K\sum_{m,n=1}^d
\left|(\utae^{\Delta t,+})_n-(\ut_\epsilon)_n\right|^2\,|(\qt_i)_m|^2\right)^{\frac{1}{2}}\\
&\hspace{5.6cm}\times
\left(\sum_{i=1}^K\sum_{m,n=1}^d \left|\frac{\partial \hat\psi_\epsilon}{\partial  x_m}\right|^2\, \left|(\nabqi \hat \varphi)_n\right|^2\right)^{\frac{1}{2}}\,\dq \dx \dt\\
&\quad\,\,= \int_{0}^T \int_{\Omega\times D}\!\!M\,|\utae^{\Delta t,+}- \ut_\epsilon|\,|\qt|
|\nabx\hat\psi_\epsilon|\, |\nabq \hat \varphi\,| \dq \dx \dt\\
&\quad\,\,\leq \sqrt{b} \int_{0}^T \left(\int_{\Omega\times D} M\,|\utae^{\Delta t,+}- \ut_\epsilon|\, |\nabx\hat\psi_\epsilon| \dq \dx\right)\|\nabq \hat \varphi\,\|_{L^\infty(\Omega \times D)}\dt\\
&\quad\,\,= 2\sqrt{b} \int_{0}^T \left[\int_{\Omega} |\utae^{\Delta t,+}- \ut_\epsilon|\, \left(\int_D M\, \sqrt{\hat\psi_\epsilon}\,|\nabx\sqrt{\hat\psi_\epsilon}| \dq\right) \dx\right]\|\nabq \hat \varphi\,\|_{L^\infty(\Omega \times D)}\dt\\
&\quad\,\,\leq 2\sqrt{b} \int_{0}^T \left[\int_{\Omega} |\utae^{\Delta t,+}- \ut_\epsilon|\, \left(\int_D M\, |\nabx\sqrt{\hat\psi_\epsilon}|^2 \dq\right)^{\frac{1}{2}} \dx\right]\|\nabq \hat \varphi\,\|_{L^\infty(\Omega \times D)}\dt,
\end{align*}
where in the transition to the last line we used the Cauchy--Schwarz inequality in conjunction with
the upper bound \eqref{1boundonpsi}. Hence,
\begin{align*}
&|{\rm V}_{7,1}| \leq
2\sqrt{b}\, \|\utae^{\Delta t,+}- \ut_\epsilon\|_{L^2(0,T;L^2(\Omega))}\, \|\nabx\sqrt{\hat\psi_\epsilon}\|_{L^2(0,T;L^2_M(\Omega \times D))}\\
&\hspace{6.5cm}\times \|\nabq \hat \varphi\,\|_{L^\infty(0,T;L^\infty(\Omega \times D))}.
\end{align*}
Thanks to \eqref{usconL2a} with $r=2$ and  \eqref{psiwconH1a},
${\rm V}_{7,1}$ tends to $0$ as
$L \rightarrow 0$ (and $\Delta t \rightarrow 0_+$).

Let us now consider the term ${\rm V}_{7,2}$. Proceeding similarly as in the case of the term ${\rm V}_{7,1}$, using \eqref{1boundonpsi}, yields
\begin{align*}
&|{\rm V}_{7,2}| \leq
\sqrt{b}\,\|\utae^{\Delta t,+}-\ut_\epsilon\|_{L^2(0,T;L^2(\Omega))}\,\|\nabxtt \nabq \hat \varphi\|_{L^2(0,T;L^2(\Omega;L^\infty(D)))}.
\end{align*}
Noting \eqref{usconL2a} with $r=2$, we deduce that ${\rm V}_{7,2}$ converges to $0$ as
$L \rightarrow 0$ (and $\Delta t \rightarrow 0_+$).
Having shown that both ${\rm V}_{7,1}$ and  ${\rm V}_{7,2}$ converge to $0$ as
$L \rightarrow 0$ (and $\Delta t \rightarrow 0_+$), it follows that the same is
true of ${\rm V}_7 = {\rm V}_{7,1} + {\rm V}_{7,2}$. We have already shown that
${\rm V}_5$ and ${\rm V}_{6}$ converge to $0$ as $L \rightarrow 0$ (and $\Delta t
\rightarrow 0_+$). Since the sum of the first three terms on the left-hand side of
\eqref{bound-3.5} converges to $0$, it follows that the left-most expression in
the chain \eqref{bound-3.5} converges to the right-most term, in the
limit of $L \rightarrow \infty$ (and $\Delta t \rightarrow 0_+$).
That completes Step 3.5.

Having dealt with \eqref{eqpsincon}, we now turn to \eqref{equncon},
with the aim to pass to the limit with $L$ (and $\Delta t$). In Steps 3.6 and 3.7
below we shall choose as our test function
\[ \wt \in C^1([0,T];\Ct^\infty_0(\Omega)) \quad \mbox{with $\wt(\cdot,T)=0$, and $\nabx \cdot \wt = 0$ on $\Omega$
for all $t \in [0,T]$}.\]
Clearly, any such $\wt$ belongs to $L^1(0,T;\Vt)$ and is therefore a legitimate choice of
test function in \eqref{equncon}. Furthermore, for any $\sigma\geq 1$, the set of such smooth test
functions $\wt$ is dense in the space of all functions in $W^{1,1}(0,T;\Vt_\sigma)$ that vanish at $t=T$. As each term in
\eqref{equncon} has been shown before to be a continuous linear functional on $L^2(0,T;\Vt_\sigma)$,
$\sigma \geq \frac{1}{2}d$, $\sigma>1$ and $W^{1,1}(0,T;\Vt_\sigma)$ is
(continuously) embedded in $L^2(0,T;\Vt_\sigma)$,
$\sigma \geq \frac{1}{2}d$, $\sigma>1$, the use of such smooth test functions for the purposes of the argument below is fully justified.

\smallskip

\textit{Step 3.6.} The terms on the left-hand side of \eqref{equncon} are handled routinely,
using \eqref{eq:energy-u+psi-final6a} and, respectively, integration by parts in
time in conjunction with
%\eqref{uwconL2a},
%
%\eqref{utwconL2a},
\eqref{usconL2a} with $r=2$, \eqref{uwconH1a}
and recalling that $\ut^0 \rightarrow \ut_0$ weakly in $\Ht$.
In particular, the second (nonlinear) term on the left-hand
side of \eqref{equncon} is quite simple to deal with on rewriting it as
$- \int_0^T(\uta^{\Delta t,+} \otimes \uta^{\Delta t,-}, \nabxtt\wt)\dt$,
and then considering the difference
$\int_0^T(\ut_\epsilon \otimes \ut_\epsilon - \uta^{\Delta t,+} \otimes \uta^{\Delta t,-}, \nabxtt\wt)\dt$, which is bounded by
\[ \left(\int_0^T\|\ut_\epsilon \otimes \ut_\epsilon - \uta^{\Delta t,+} \otimes \uta^{\Delta t,-}\|_{L^1(\Omega)}\dt\right)  \|\nabxtt\wt\|_{L^\infty(0,T;L^\infty(\Omega))}.\]
By adding and subtracting $\ut_\epsilon \otimes \uta^{\Delta t,-}$ inside the first norm sign, using
the triangle inequality, followed by the Cauchy--Schwarz inequality in each of the resulting terms,
and then applying the first bound in \eqref{eq:energy-u+psi-final6a}, and \eqref{usconL2a} with $r=2$,
we deduce that the above expression converges to $0$ as $L\rightarrow \infty$ (and $\Delta t\rightarrow
0_+$).  The convergence of the first term on the right-hand side of \eqref{equncon} to the correct limit,
as $L \rightarrow \infty$ (and $\Delta t \rightarrow 0_+$), is an immediate consequence of \eqref{fncon}.
We refer the reader for a similar argument
%in the case of the incompressible Navier--Stokes equations,
to Ch.\ 3, Sec.\ 4 of Temam \cite{Temam}. That completes Step 3.6.

\smallskip

\textit{Step 3.7.} The extra-stress tensor appearing on the right-hand side of \eqref{equncon} is
dealt with as follows. First, by using \eqref{intbyparts}
and noting that $\wt$ is, by assumption, divergence-free,
%and performing identical transformations
and proceeding in exactly the same manner as in
%the argument leading from the first line of equation
\eqref{U4a}, %to the third line of \eqref{U4b},
but with $\psia^{\Delta t,+}$
now replaced by $\psia^{\Delta t,+} - \hat\psi_\epsilon$, we have that
\begin{align*}
&{\rm V}_8:=\left|\,k\int_{0}^T\!\sum_{i=1}^K \int_{\Omega}
\Ctt_i(M\,\psia^{\Delta t,+}): \nabxtt \wt \dx \dt - k\int_{0}^T\!\sum_{i=1}^K \int_{\Omega}
\Ctt_i(M\,\hat\psi_\epsilon): \nabxtt \wt \dx \dt\,\right|\\
%&\quad\,\,
%= k \left|\,\int_{0}^T\!\sum_{i=1}^K \int_{\Omega}
%\Ctt_i(M\,(\psiae^{\Delta t,+}-\hat\psi_\epsilon): \nabxtt \wt \dx \dt\,\right| \\
&\quad\,\, = k \left|\,\int_{0}^T \int_{\Omega}\left[\int_{D} M \sum_{i=1}^K (\nabxtt \wt)\qt_i \cdot
\nabqi (\hpsiae^{\Delta t,+}-\hat\psi_\epsilon) \dq \right] \dx \dt \,\right|.
\end{align*}
%
%where in the transition to the last line we applied the integration-by-parts formula
%\eqref{intbyparts} noting that $\wt$ is, by assumption, divergence-free.
We rewrite the second factor in the integrand of the last integral as follows:
\begin{align*}
 M\left(\nabqi\psia^{\Delta t,+} - \nabqi \hat\psi_\epsilon\right)
 %&= 2\,M \left[\sqrt{\psia^{\Delta t,+}}\, \nabqi\! \sqrt{\psia^{\Delta t,+}} - \sqrt{\hat\psi_\epsilon}\, \nabqi\! %\sqrt{\hat\psi_\epsilon}\,\right]\\
 &=
 2 \left(\sqrt{M\,\psia^{\Delta t,+}} -  \sqrt{M\,\hat\psi_\epsilon}\,\right)\left[M^{\frac{1}{2}}\,\nabqi\! \sqrt{\psia^{\Delta t,+}}\right] \\
 &\qquad +  2\,\sqrt{M\,\hat\psi_\epsilon} \left(M^{\frac{1}{2}}\,\nabqi\! \sqrt{\psia^{\Delta t,+}}- M^{\frac{1}{2}}\,\nabqi \! \sqrt{\hat\psi_\epsilon}\,\right).
\end{align*}
Hence we obtain the following inequality:
\begin{align*}
&{\rm V}_8\leq 2k\,\sqrt{b}\int_0^T \int_\Omega\,|\nabxtt\wt| \left[\int_D \left|\sqrt{M\,\psia^{\Delta t,+}} -  \sqrt{M\,\hat\psi_\epsilon}\right|\,\left|M^{\frac{1}{2}}\,\nabq \sqrt{\psia^{\Delta t,+}}\right|
\dq\right]\!\dx \dt\\
&\qquad \, + 2k \left|\int_{0}^T\!\!\!\int_{\Omega}\!\left[\int_{D} \sum_{i=1}^K (\nabxtt \wt)\qt_i \cdot
\sqrt{M\,\hat\psi_\epsilon} \left(\!M^{\frac{1}{2}}\,\nabqi \sqrt{\psia^{\Delta t,+}}-M^{\frac{1}{2}}\,\nabqi \sqrt{\hat\psi_\epsilon}\,\right)\!\!\dq \right]\!\!\dx \dt \right|\\
&\quad =: {\rm V}_{8,1} + {\rm V}_{8,2}.
\end{align*}
We emphasize at this point that in the term ${\rm V}_{8,2}$ we intentionally did not move
the modulus sign under the integral: we shall be relying on the weak convergence result \eqref{psiwconH1xa}
to drive the term ${\rm V}_{8,2}$ to $0$ in the limit of $L\rightarrow \infty$ (and $\Delta t \rightarrow 0_+)$,
so it is essential that the modulus sign is kept outside the integral.

For ${\rm V}_{8,1}$, we have, by using that
$|\sqrt{c_1}- \sqrt{c_2}| \leq \sqrt{|c_1-c_2|}$ for any $c_1, c_2 \in \mathbb{R}_{\geq 0}$:
\begin{align*}
&{\rm V}_{8,1} \leq 2k\,\sqrt{b} \int_0^T \|\nabxtt\wt\|_{L^\infty(\Omega)} \|\psia^{\Delta t,+} -  \hat\psi_\epsilon \|^{\frac{1}{2}}_{L^1_M(\Omega \times D)}\,\|\nabq \sqrt{\psia^{\Delta t,+}}\|_{L^2_M(\Omega \times D)} \dt\nonumber\\
&\qquad \!\leq 2k\,\sqrt{b}\, \|\nabxtt \wt\|_{L^\infty(0,T;L^\infty(\Omega))}\nonumber\\
&\hspace{3cm}\times \|\psia^{\Delta t,+} -  \hat\psi_\epsilon \|^{\frac{1}{2}}_{L^1(0,T;L^1_M(\Omega \times D))}\,\|\nabq \sqrt{\psia^{\Delta t,+}}\|_{L^2(0,T;L^2_M(\Omega \times D))}.
\end{align*}
By noting \eqref{psisconL2a} with $p=1$ and the bound
on the seventh term in \eqref{eq:energy-u+psi-final6} we deduce that the term ${\rm V}_{8,1}$
converges to $0$ in the limit of $L \rightarrow \infty$ (and $\Delta t \rightarrow 0_+$).

Finally, for ${\rm V}_{8,2}$, we first define the $(Kd)$-component column-vector function
$\undertilde{\Xi} := [\undertilde{\Xi}_1^{\rm T}, \dots, \undertilde{\Xi}^{\rm T}_K]^{\rm T}$, where $\undertilde{\Xi}_i := \sqrt{M\,\hat\psi_\epsilon}~ (\nabxtt \wt)\qt_i$, $i=1,\dots,K$, and note that
\begin{align*}
&{\rm V}_{8,2} = 2k \left|\int_{0}^T\!\! \int_{\Omega}\left[\int_{D} \sum_{i=1}^K \undertilde{\Xi}_i\cdot
\left(M^{\frac{1}{2}}\,\nabqi \sqrt{\psia^{\Delta t,+}}-M^{\frac{1}{2}}\,\nabqi \sqrt{\hat\psi_\epsilon}\,\right)\!\dq \right]\!\dx \dt \right|\\
&\qquad =  2k \left|\int_{0}^T\!\! \int_{\Omega \times D}
\left(M^{\frac{1}{2}}\,\nabq \sqrt{\psia^{\Delta t,+}}-M^{\frac{1}{2}}\,\nabq \sqrt{\hat\psi_\epsilon}\,\right)\cdot \,\undertilde{\Xi} \dq\dx\dt \right|.
\end{align*}
The convergence of ${\rm V}_{8,2}$ to $0$ will directly follow from \eqref{psiwconH1xa} once we have shown that
$\undertilde{\Xi} \in L^2(0,T;\Lt^2(\Omega \times D))$. The latter is straightforward to verify, using \eqref{1boundonpsi}:
\begin{align*}
\int_0^T \int_{\Omega \times D} |\undertilde{\Xi}|^2 \dq \dx \dt &= \sum_{i=1}^K \int_0^T \int_{\Omega \times D} |\undertilde{\Xi_i}|^2 \dq \dx \dt
\\
&= \sum_{i=1}^K \int_0^T \int_{\Omega \times D} \bigg|\sqrt{M\,\hat\psi_\epsilon}\,(\nabxtt \wt)\qt_i
\bigg|^2 \dq \dx \dt\\
& \leq \sum_{i=1}^K \int_0^T \int_{\Omega \times D} |\nabxtt \wt|^2\, |\qt_i|^2 M\,\hat\psi_\epsilon \dq \dx \dt\\
&  =  b\, \int_0^T \int_\Omega |\nabxtt \wt|^2\, \int_D M\,\hat\psi_\epsilon \dq \dx \dt\\
& \leq b\, \|\nabxtt \wt\|^2_{L^2(0,T;L^2(\Omega))} < \infty,
\end{align*}
where in the transition to the last line we used \eqref{1boundonpsi}.

Thus we deduce from \eqref{psiwconH1xa} that ${\rm V}_{8,2}$ converges to $0$ as $L\rightarrow \infty$
(and $\Delta t \rightarrow 0_+$). As both ${\rm V}_{8,1}$ and ${\rm V}_{8,2}$ converge to $0$, the same is
true of $V_8$, which then implies \eqref{CwconL2a}, thanks
to the denseness of the set of divergence-free functions contained in $C^1([0,T];\Ct^\infty_0(\Omega))$ and
vanishing at $t=T$ in the function space $L^2(0,T;\Vt_\sigma)$, $\sigma\geq \frac{1}{2}d$, $\sigma>1$.
That completes Step 3.7, and the proof of \eqref{CwconL2a}.

\smallskip

\textit{Step 3.8.} Steps 3.1--3.7 enable us to pass to the limits $L \rightarrow \infty$, $\Delta t \rightarrow 0_+$, with $\Delta t = o(L^{-1})$ as $L \rightarrow \infty$,
to deduce the existence of a pair
$(\ut_\epsilon,\hat\psi_\epsilon)$ satisfying \eqref{equnconP}, \eqref{eqpsinconP}
for smooth test functions $\hat\varphi$ and $\wt$, as above. The denseness of
the set of divergence-free functions contained in $C^1([0,T];\Ct^\infty_0(\Omega))$
and vanishing at $t=T$ in the set of all functions in
$W^{1,1}(0,T;\Vt_\sigma)$ and vanishing at $t=T$, $\sigma\geq \frac{1}{2}d$, $\sigma>1$,
and the denseness of the set of functions contained in $C^1([0,T];C^\infty(\overline{\Omega \times D}))$ and
vanishing at $t=T$ in the set of all functions in $W^{1,1}(0,T;H^s(\Omega \times D))$ and
vanishing at $t=T$, $s > 1 + \frac{1}{2}(K+1)d$, yield \eqref{equnconP} and \eqref{eqpsinconP}. That completes Step 3.8.

\smallskip

\textit{Step 3.9.}
Let $X$ be a Banach space. We shall denote by $C_{w}([0,T];X)$ the set of
all functions $x\in L^\infty(0,T;X)$ such that $t \in [0,T] \mapsto \langle x', x(t) \rangle_X \in \mathbb{R}$ is continuous on $[0,T]$ for all $x' \in X'$, the dual space of $X$. Whenever $X$ has a predual, $E$, say, (viz. $E'=X$),
we shall denote by $C_{w\ast}([0,T];X)$ the set of
all functions $x \in L^\infty(0,T;X)$ such that $t \in [0,T] \mapsto \langle x(t), e \rangle_E \in \mathbb{R}$ is continuous on $[0,T]$ for all $e \in E$.

%~\vspace{-6.5mm}

The next result will play an important role in what follows.

\begin{lemma}\label{lemma-strauss}
Let $X$ and $Y$ be Banach spaces.

%~\vspace{-6.5mm}

\begin{itemize}
\item[(a)] If the space $X$ is reflexive and is continuously embedded in the space $Y$, then
$L^\infty(0,T; X) \cap C_{w}([0,T]; Y) = C_{w}([0,T];X)$.
\item[(b)] If $X$ has separable predual $E$ and $Y$ has predual $F$ such that
$F$ is continuously embedded in $E$, then
$L^\infty(0,T;X) \cap C_{w\ast}([0,T]; Y) = C_{w\ast}([0,T];X)$.
\end{itemize}
\end{lemma}

%~\vspace{-6.5mm}

Part (a) is due to Strauss \cite{Strauss} (cf. Lions \& Magenes \cite{Lions-Magenes}, Lemma 8.1, Ch. 3, Sec. 8.4); part (b) is proved analogously, {\em via}
the sequential Banach--Alaoglu theorem.
That $\ut_\epsilon\in C_w([0,T];\Ht)$
then follows from $\ut_\epsilon \in L^\infty(0,T;\Ht) \cap H^1(0,T; \Vt'_\sigma)$
by Lemma \ref{lemma-strauss}(a), with $X:=\Ht$, $Y:=\Vt_\sigma'$, $\sigma \geq \frac{1}{2}d$, $\sigma>1$.
That $\hat\psi_\epsilon\in C_{w}([0,T];L^1_M(\Omega \times D))$ follows from $\mathcal{F}(\hat\psi_\epsilon) \in L^\infty(0,T; L^1_M(\Omega
\times D))$ and $\hat\psi_\epsilon \in H^1(0,T; M^{-1}(H^s(\Omega \times D))')$
by Lemma \ref{lemma-strauss}(b) on taking
$X:=L^\Phi_M(\Omega \times D)$, the Maxwellian weighted Orlicz space with Young's function $\Phi(r) = \mathcal{F}(1+|r|)$ (cf. Kufner, John \& Fu\v{c}ik \cite{KJF},
Sec. 3.18.2) whose separable
predual $E:=E^\Psi_M(\Omega \times D)$ has Young's function $\Psi(r) = \exp|r| - |r| - 1$, and $Y := M^{-1}(H^s(\Omega \times D))'$ whose predual w.r.t. the duality pairing $\langle M \cdot , \cdot \rangle_{H^s(\Omega \times D)}$ is $F:=H^s(\Omega \times D)$, $s> 1 + \frac{1}{2}(K+1)d$, and noting that $C_{w\ast}([0;T]; L^\Phi_M(\Omega \times D)) \hookrightarrow C_{w}([0,T]; L^1_M(\Omega \times D))$. The last embedding and that $F \hookrightarrow E$ are proved by adapting
Def. 3.6.1. and Thm. 3.2.3 in Kufner, John \& Fu\v{c}ik \cite{KJF} to the measure $M(\qt)\dq \dx$ to show that $L^\infty(\Omega \times D) \hookrightarrow L^\Xi_M(\Omega \times D)$ for any Young's function $\Xi$, and then adapting Theorem 3.17.7 {\em ibid.} to deduce that $F \hookrightarrow L^\infty(\Omega \times D) \hookrightarrow E^\Psi_M(\Omega \times D)= E$.
%Hence $\hat\psi_\epsilon \in C_{w\ast}([0;T]; L^\Phi_M(\Omega \times D))$. Noting again %that $L^\infty(\Omega \times D) \hookrightarrow E^\Psi_M(\Omega \times D)$ implies that %; hence, $\hat\psi_\epsilon \in
%C_{w}([0,T]; L^1_M(\Omega \times D))$.
%
%
%The latter is proved by adapting Def. 3.6.1. and Thm. 3.2.3 in Kufner, John \& %Fu\v{c}ik \cite{KJF}
%to the measure $M(\qt)\dq \dx$ to show that
%$L^\infty(\Omega \times D) \hookrightarrow L^\Xi_M(\Omega \times D)$ for any Young's %function $\Xi$, and then adapting Theorem 3.17.7 {\em ibid.} to finally deduce that
%$F \hookrightarrow L^\infty(\Omega \times D) \hookrightarrow E^\Psi_M(\Omega \times D)
%= E$.
%%$L^\infty(\Omega \times D) \hookrightarrow E^\Psi_M(\Omega \times D)$.
[The abstract framework in Temam \cite{Temam}, Ch.\ 3, Sec.\ 4 then implies that $\ut_\epsilon$ and $\hat\psi_\epsilon$ satisfy
%the initial conditions
$\ut_\epsilon(\cdot,0)=\ut_0(\cdot)$ and $\hat\psi_\epsilon(\cdot,\cdot,0) = \hat\psi(\cdot,\cdot)$ in the sense of $C_w([0,T];\Ht)$ and $C_w([0,T]; L^1_M(\Omega \times D))$, respectively.]
\begin{comment}
The weak continuity of the function $\ut_\epsilon\,:\,[0,T] \rightarrow \Ht$
follows from \eqref{uwconL2a} using Theorem 2.1 in Strauss \cite{Strauss}
with $X=\Ht$ and $Y=\Vt_\sigma'$, $\sigma \geq \frac{1}{2}d$, $\sigma>1$.
 We then deduce (cf. Temam \cite{Temam}, Ch.\ 3, Sec.\ 4) that $\ut_\epsilon$ satisfies the initial condition $\ut_\epsilon(\cdot,0)=\ut_0(\cdot)$ in the sense of weakly continuous functions from $[0,T]$ to $\Ht$.
Since $\hat\psi_\epsilon \in H^1(0,T;M^{-1}(H^s(\Omega \times D))')$, $s> 1 + \frac{1}{2}(K+1)d$, we have by the Sobolev embedding theorem that $\hat\psi_\epsilon \in C([0,T];M^{-1}(H^s(\Omega \times D))')$.
Following, again, the abstract framework in Temam \cite{Temam}, Ch.\ 3, Sec.\ 4, we deduce that $\hat\psi_\varepsilon$ satisfies the initial condition in the sense of weakly continuous functions from the interval $[0,T]$ to
$M^{-1}(H^s(\Omega \times D))'$. That completes Step 3.9.
\end{comment}

%\begin{lemma}\label{lem-zeta}
%Let $X$ and $Y$ be Banach spaces, with $X$ reflexive, and continuously
%embedded in $Y$. If $\varphi \in L^\infty(0,T;X)$ and is weakly continuous as a
%function with values in $Y$, then $\varphi$ is weakly continuous as a function with values in $X$.
%\end{lemma}
%\noindent

\smallskip

\textit{Step 3.10.}
The energy inequality \eqref{eq:energyest} is a direct consequence of (\ref{uwconL2a}-c)
and (\ref{psiwconH1a},b,d), on noting the (weak) lower-semicontinuity of the terms on the left-hand side
of \eqref{eq:energy-u+psi-final2} and \eqref{fatou-app}.
%, and that, thanks to \eqref{mass-conserved}, the function
%$\mathcal{H}$ on both sides of the inequality can be replaced by $\mathcal{F}$ thanks to the
%cancellation of $\int_{\Omega\times D}\hat\psi(\xt,\qt,t)\dq \dx = \int_{\Omega \times D}\hat\psi_0(\xt,\qt)\dq \dx$, $t \in [0,T]$ on the two sides of \eqref{eq:energy-u+psi-final2}.
That completes Step 3.10.

\smallskip

\textit{Step 3.11.} It remains to prove \eqref{mass-conserved}.
The bounds on the first and third term on the left-hand side of \eqref{eq:energy-u+psi-final6a-zeta} imply
that the sequences $\{\rho_{\epsilon,L}^{\Delta t(,\pm)}\}_{L>1}$ are bounded in $L^\infty(0,T; L^\infty(\Omega))\cap
L^2(0,T;H^1(\Omega))$; the bound on the fourth term in \eqref{eq:energy-u+psi-final6a-zeta} yields that
$\{\rho_{\epsilon,L}^{\Delta t}\}_{L>1}$ is bounded in $H^1((0,T); (H^1(\Omega))')$.
In fact, by noting \eqref{additional1} and \eqref{additional2}, we have that $\{\rho_{\epsilon,L}^{\Delta t(,\pm)}\}_{L>1}$
is a bounded sequence in $L^2(0,T;\mathcal{K})$. Thus,
there exist subsequences of $\{\rho_{\epsilon,L}^{\Delta t(,\pm)}\}_{L>1}$
(not indicated) with $\Delta t =o(L^{-1})$ and, thanks to the uniform bound on the second term on the left-hand side of
\eqref{eq:energy-u+psi-final6a-zeta}, a common limiting function
$\rho_\epsilon \in L^{\infty}(0,T;L^\infty(\Omega))\cap
L^2(0,T;\mathcal{K})\cap H^1(0,T;(H^1(\Omega))')$
such that
\begin{subequations}
\begin{alignat}{2}
\rho_{\epsilon,L}^{\Delta t(,\pm)} &\rightarrow \rho_\epsilon \qquad&&
\mbox{weak$^\star$ in } L^\infty(0,T;L^\infty(\Omega)),
\label{zwconL2} \\
\rho_{\epsilon,L}^{\Delta t(,\pm)} &\rightarrow \rho_\epsilon \qquad&&
\mbox{weakly in } L^2(0,T;H^1(\Omega)),
\label{zwconH1}\\
\frac{\partial\rho_{\epsilon,L}^{\Delta t}}{\partial t}
&\rightarrow \frac{\partial \rho_\epsilon}
{\partial t} \qquad&&
\mbox{weakly in } L^2(0,T;(H^1(\Omega))'),
\label{ztwconH-1}
\end{alignat}
\end{subequations}
as $L \rightarrow \infty$ (and thereby $\Delta t \rightarrow 0_+$),
It follows from (\ref{zwconH1},c) and Lemma 1.2 in Ch.\ 3  of Temam \cite{Temam} that $\rho_\epsilon$ belongs to $C([0,T]; L^2(\Omega))$, in fact.
Using (\ref{zwconL2}--c) and (\ref{usconL2a}), we can pass to the limit
as $L \rightarrow \infty$ (and $\Delta t \rightarrow 0_+$) in (\ref{zetacon}) to obtain that,
for all $\varphi \in L^2(0,T;H^1(\Omega))$,
\begin{align}
&\int_0^T \left\langle
\frac{\partial \rho_{\epsilon}}{\partial t}, \varphi \right\rangle_{\!\!H^1(\Omega)} \dt
+ \int_0^T \int_{\Omega} \left[ \epsilon\, \nabx  \rho_{\epsilon} -
\ut_{\epsilon}  \,\rho_{\epsilon} \right]
\cdot \nabx \varphi \dx \dt =0.
\label{zetaconlim}
\end{align}
We note also that Fubini's theorem, (\ref{zetancon}) and (\ref{psisconL2a}) yield that
\begin{align}
\hspace{-2mm}\int_0^T \int_\Omega \left|\rho_{\epsilon,L}^{\Delta t} - \int_D M \,\hat \psi_{\epsilon} \dq\right|
\dx \dt &= \int_0^T \int_\Omega \left|\int_D M \,(\hat \psi_{\epsilon,L}^{\Delta t}-
\hat \psi_{\epsilon}) \dq\right|
\dx \dt \nonumber \\
&\hspace{-1.8in} \leq \int_0^T \int_{\Omega\times D} M \,|\hat \psi_{\epsilon,L}^{\Delta t}-
\hat \psi_{\epsilon}|
\dq \dx \dt \rightarrow 0 \quad \mbox{as~~} L \rightarrow \infty \mbox{~~(and $\Delta t \rightarrow 0_+$)}.
\label{equivcon}
\end{align}
Thus, $\rho_{\epsilon,L}^{\Delta t} \rightarrow \int_D M \,\hat \psi_{\epsilon} \dq$
strongly in $L^1(0,T;L^1(\Omega))$ as $L \rightarrow \infty$ (and $\Delta t \rightarrow 0_+$).
Comparing this with (\ref{zwconL2}) implies that
\begin{align}
\rho_\epsilon(\xt,t) = \int_D M(\qt) \,\hat \psi_{\epsilon}(\xt,\qt,t) \dq
\qquad \mbox{for a.e. } (\xt,t) \in \Omega \times (0,T).
\label{zetaequiv}
\end{align}
It follows from Step 3.9 that, for
$s > 1 + \textstyle{\frac{1}{2}}(K+1)d$, we have
\[ \lim_{t \rightarrow 0_+} \int_{\Omega \times D} M(\qt)\, (\hat\psi_\epsilon (\xt,\qt,t) - \hat\psi_0(\xt,\qt))\,  \hat\varphi(\xt, \qt) \dq \dx = 0\qquad
\forall \hat\varphi \in H^s(\Omega \times D).\]
Consequently, using %Fubini's theorem,
\eqref{zetaequiv} and \eqref{inidata}  we then deduce by selecting
any $\hat\varphi = \varphi \in C^\infty(\overline{\Omega}) ~(\cong C^\infty(\overline{\Omega})\otimes 1(D) \subset H^s(\Omega \times D))$ that
\begin{align}
\lim_{t \rightarrow 0_+} \int_{\Omega} \rho_{\epsilon}(\xt,t)\, \varphi(\xt) \dx
&= \lim_{t \rightarrow 0_+} \int_{\Omega} \left(\int_D M(\qt) \,\hat \psi_\epsilon(\xt,\qt,t) \dq \right) \varphi(\xt) \dx
\nonumber \\
&\hspace{-1in}= \lim_{t \rightarrow 0_+}
\int_{\Omega\times D} M(\qt) \,\hat \psi_\epsilon(\xt,\qt,t) \,\varphi(\xt) \dq \dx
%\nonumber \\
%&\hspace{-1in}
=
\int_{\Omega\times D} M(\qt) \,\hat \psi_ 0(\xt,\qt) \,\varphi(\xt) \dq \dx
\nonumber \\
&\hspace{-1in} = \int_{\Omega} \left(\int_D M(\qt)\,\hat \psi_ 0(\xt,\qt)) \dq \right)
\varphi(\xt) \dx =
\int_{\Omega} \varphi(\xt) \dx.
\label{zeta0conv}
\end{align}
As $\rho_\epsilon \in C([0,T]; L^2(\Omega))$, it follows from (\ref{zeta0conv}) that
$\rho_\epsilon(\xt,0) = 1$ for a.e.\ $\xt \in \Omega$.

Clearly the linear parabolic problem \eqref{zetaconlim}
with initial datum $\rho_\epsilon(\xt,0) = 1$ for a.e.\ $\xt \in \Omega$
has the unique solution $\rho_\epsilon \equiv 1$ on $\Omega \times [0,T]$. Using this in \eqref{zetaequiv} implies \eqref{mass-conserved}, and completes Step 3.11 and the proof.
\end{proof}

%%%%%%%%%%%%%%%%%%%%%%%%%%%%%%%%%%%%%%%%%%%%%%%%%%%%%%%%%%%%%%%%%%%%%%%%%%%%%%%%%%%%%%%%%%%%%%%%%%

\section{Exponential decay to the equilibrium solution}
\label{sec:decay}
\setcounter{equation}{0}

We shall show that, in the absence of a body force (i.e. with $\ft \equiv \zerot$),
weak solutions $(\ut_\epsilon, \hat\psi_\epsilon)$ to (P$_\epsilon$), whose existence we have proved {\em via} our limiting procedure in the
previous section, decay exponentially in time to the trivial solution of the steady counterpart of problem
(P$_\epsilon$) at a rate that is independent of the specific choice of the initial data for the Navier--Stokes
and Fokker--Planck equations.
Our result is similar to the one derived by Jourdain, Leli\`evre, Le Bris \& Otto \cite{JLLO}, except that
the arguments there were partially formal in the sense that the existence of a unique global-in-time solution,
which was required to be regular enough, was assumed; in fact, the probability
density function was supposed there to be a classical solution to the Fokker--Planck equation;
$\hat\psi_0$ was required to belong to $L^\infty(\Omega \times D)$
and to be strictly positive, and $\ut$ was assumed to be in
$L^\infty(0,\infty;\Wt^{1,\infty}(\Omega))$ (cf. p.105, (B.128), (B.129) therein; as well as the recent paper
of Arnold, Carrillo and Manzini \cite{ACM} for refinements and extensions).
In contrast, we require no additional regularity hypotheses here.

\begin{theorem} Suppose the assumptions of Theorem \ref{convfinal} hold and $M$ satisfies the
Bakry--\'{E}mery condition (cf. Remark \ref{rem5.1}) with $\kappa>0$; then, for all $T>0$,
\begin{align}\label{eq:bd55a}
&\|\ut_\epsilon(T)\|^2 + \frac{k}{|\Omega|} \|\hat\psi_\epsilon(T) - 1\|^2_{L^1_M(\Omega \times D)}   \nonumber\\
&\qquad\leq {\rm e}^{-\gamma_0 T}\left[ \|\ut_0\|^2 + 2k\int_{\Omega \times D} M\, \mathcal{F}(\hat\psi_0) \dq \dx \right]
+ \frac{1}{\nu}\int_{0}^{T} \|\ft\|^2_{(H^1_0(\Omega))'} \dd s,
\end{align}
where $\gamma_0 := \min\left(\frac{\nu}{C_{\sf P}^2}\,,\,\frac{\kappa\,a_0}{2\lambda}\right)$. In particular if $\ft \equiv 0$,
the following inequality holds:
\begin{align}\label{eq:bd55aa}
&\|\ut_\epsilon(T)\|^2 + \frac{k}{|\Omega|} \|\hat\psi_\epsilon(T) - 1\|^2_{L^1_M(\Omega \times D)}   %\nonumber\\&\qqqquad
\leq {\rm e}^{-\gamma_0 T}\left[ \|\ut_0\|^2 + 2k\int_{\Omega \times D} M\, \mathcal{F}(\hat\psi_0) \dq \dx \right].
\end{align}
\end{theorem}

\begin{proof}
We take $t = t_1 = \Delta t$ and write $0=t_0$ in \eqref{eq:energy-u+psi2}, and we replace the function $\mathcal{F}$ on the left-hand side of \eqref{eq:energy-u+psi2}
by $\mathcal{F}^L$, noting that, prior to \eqref{eq:energy-u+psi2}, in \eqref{eq:energy-u+psi} we in fact had $\mathcal{F}^L$ on the left-hand side of the inequality, and $\mathcal{F}^L$ was subsequently bounded below by $\mathcal{F}$; thus
we reinstate the $\mathcal{F}^L$ we previously had. We
recall that $\ut^0 = \uta^{\Delta t,-}(t_1)$ and $\beta^L(\hat\psi^0) = \psia^{\Delta t, -}(t_1)$ and adopt the notational convention $t_{-1}:=-\infty$ (say), which allows us to write $\uta^{\Delta t,+}(t_0)$ instead of $\uta^{\Delta t,-}(t_1)$
and $\psia^{\Delta t,+}(t_0)$ instead of $\psia^{\Delta t,-}(t_1)$.
Hence we have that
\begin{eqnarray}
&&\hspace{-2mm}\|\uta^{\Delta t, +}(t_1)\|^2
+ \frac{1}{\Delta t} \int_{t_0}^{t_1}\! \|\uta^{\Delta t, +} - \uta^{\Delta t,-}\|^2 \dd s +
\left(\nu - \alpha \frac{2\lambda\,b\,k}{a_0}\right)\! \int_{t_0}^{t_1}\! \|\nabxtt \uta^{\Delta t, +}(s)\|^2 \dd s\nonumber\\
&&\hspace{-1mm}
+ \,2k\int_{\Omega \times D}\!\!\! M \mathcal{F}^L(\psia^{\Delta t, +}(t_1) + \alpha) \dq \dx
%&&
%\qquad
+ \frac{k}{\Delta t\, L} \int_{t_0}^{t_1}\! \int_{\Omega \times D}\!\!\! M (\psia^{\Delta t, +} - \psia^{\Delta t, -})^2 \dq \dx \dd s
\nonumber\\
%\end{eqnarray}
%\begin{eqnarray}
&&\hspace{-1mm}
+\, 2k\,\varepsilon \int_{t_0}^{t_1} \int_{\Omega \times D} M\,
\frac{|\nabx \psia^{\Delta t, +} |^2}{\psia^{\Delta t, +} + \alpha} \dq \dx \dd s
%&&
%\qquad
+\, \frac{a_0 k}{2\,\lambda}  \int_{t_0}^{t_1}
\int_{\Omega \times D}M\, \frac{|\nabq \psia^{\Delta t, +}|^2}{\psia^{\Delta t, +} + \alpha} \,\dq \dx \dd s\nonumber\\
&&%\hspace{-2mm}
\qqqquad\leq \|\uta^{\Delta t,+}(t_0)\|^2 + \frac{1}{\nu}\int_{t_0}^{t_1}\!\!\|\ft^{\Delta t,+}(s)\|^2_{(H^1_0(\Omega))'} \dd s
\nonumber\\
&&\qqqquad\qqqquad
+ 2k \int_{\Omega \times D}\!\!\! M\, \mathcal{F}^L(\psia^{\Delta t,+}(t_0) + \alpha)\dq \dx.
\label{eq:energy-u+psi2-local1}
\end{eqnarray}
Closer inspection of the procedure that resulted in inequality \eqref{eq:energy-u+psi2} reveals that
\eqref{eq:energy-u+psi2} could have been equivalently arrived at by repeating the argument that gave
us \eqref{eq:energy-u+psi2-local1} on each time interval $[t_{n-1},t_n]$, $n=1,\dots,N$; viz.,
\begin{eqnarray}\label{eq:energy-u+psi2-local2}
&&\hspace{-2mm}\|\uta^{\Delta t, +}(t_n)\|^2
+ \frac{1}{\Delta t}\! \int_{t_{n-1}}^{t_n} \|\uta^{\Delta t, +} - \uta^{\Delta t,-}\|^2 \dd s
\nonumber\\
&&\hspace{-2mm}
%\quad\quad\quad\quad
+
\left(\nu - \alpha \frac{2\lambda\,b\,k}{a_0}\right)\! \int_{t_{n-1}}^{t_n} \|\nabxtt \uta^{\Delta t, +}(s)\|^2 \dd s
%\nonumber\\
%&&\hspace{-2mm}\quad\quad\quad\quad
+ \,2k\int_{\Omega \times D}\!\! M\, \mathcal{F}^{L}(\psia^{\Delta t, +}(t_n) + \alpha) \dq \dx\nonumber\\
&&\hspace{-2mm}\quad\quad\quad\quad\quad\quad
+\, \frac{k}{\Delta t\, L} \int_{t_{n-1}}^{t_n}
\int_{\Omega \times D}\!\! M\, (\psia^{\Delta t, +} - \psia^{\Delta t, -})^2 \dq \dx \dd s
\nonumber \\
&&\hspace{-2mm}
+\, 2k\,\varepsilon \int_{t_{n-1}}^{t_n} \int_{\Omega \times D}\! M\,
\frac{|\nabx \psia^{\Delta t, +} |^2}{\psia^{\Delta t, +} + \alpha} \dq \dx \dd s
+\, \frac{a_0 k}{2\,\lambda}  \int_{t_{n-1}}^{t_n}
\int_{\Omega \times D}\!M\, \frac{|\nabq \psia^{\Delta t, +}|^2}{\psia^{\Delta t, +} + \alpha} \,\dq \dx \dd s% \quad\qquad\qquad
\nonumber\\
&&\hspace{-2mm}\quad\quad\quad\leq \|\uta^{\Delta t,+}(t_{n-1})\|^2 + \frac{1}{\nu}\int_{t_{n-1}}^{t_n}\!\!\|\ft^{\Delta t,+}(s)\|^2_{(H^1_0(\Omega))'} \dd s
\nonumber\\
&&\hspace{-2mm}\quad\quad\quad\quad\quad\quad + 2k \int_{\Omega \times D}\!\! M\, \mathcal{F}^L(\psia^{\Delta t,+}(t_{n-1}) + \alpha) \dq \dx,
\qquad n=1,\dots, N,
\end{eqnarray}
summing these through $n$ and then bounding $\mathcal{F}^L$ on the left-hand side below by $\mathcal{F}$.

Here we proceed differently: we shall retain $\mathcal{F}^L$ on both sides of \eqref{eq:energy-u+psi2-local2},
and omit the second, fifth and sixth term from the left-hand side of \eqref{eq:energy-u+psi2-local2}. Thus we
have that
\begin{eqnarray}\label{eq:energy-u+psi2-local3}
&&\hspace{-2mm}\|\uta^{\Delta t, +}(t_n)\|^2
+ \left(\nu - \alpha \frac{2\lambda\,b\,k}{a_0}\right) \int_{t_{n-1}}^{t_n} \|\nabxtt \uta^{\Delta t, +}(s)\|^2 \dd s
\nonumber\\
&&\hspace{-2mm}\qquad+ \,2k\int_{\Omega \times D}\!\! M\, \mathcal{F}^L(\psia^{\Delta t, +}(t_n) + \alpha) \dq \dx
\nonumber \\
&&\hspace{-2mm}\qquad\qquad +\, \frac{2 a_0 k}{\lambda}  \int_{t_{n-1}}^{t_n}
\int_{\Omega \times D}M\, |\nabq \sqrt{\psia^{\Delta t, +}+\alpha}|^2 \,\dq \dx \dd s\nonumber\\
&&\hspace{-2mm}\leq \|\uta^{\Delta t,+}(t_{n-1})\|^2 + \frac{1}{\nu}\int_{t_{n-1}}^{t_n}\!\!\|\ft^{\Delta t,+}(s)\|^2_{(H^1_0(\Omega))'} \dd s\nonumber\\
&&\hspace{-2mm}\quad\quad +\, 2k \int_{\Omega \times D}\!\! M\, \mathcal{F}^L(\psia^{\Delta t,+}(t_{n-1})+\alpha) \dq \dx,
\qquad n=1,\dots, N.
\end{eqnarray}
Thanks to Poincar\'e's inequality, recall (\ref{Poinc}), there exists a positive constant
$C_{\sf P}=C_{\sf P}(\Omega)$, such that
\begin{align}\label{bd1}
\| \uta^{\Delta t, +}(\cdot,s) \| \leq C_{{\sf P}}(\Omega)\, \|\nabxtt \uta^{\Delta t, +}
(\cdot,s)\|
\end{align}
for $s \in (t_{n-1},t_n]$; $n=1,\dots, N$.
Also, by the logarithmic Sobolev inequality \eqref{eq:logs1}, we have
for a.e.\ $\xt \in \Omega$ that
\begin{align*}
\!\!\int_D M(\qt)\, [\psia^{\Delta t, +}(\xt,\qt,s)+\alpha]\,
\log \frac{\psia^{\Delta t, +}(\xt,\qt,s)+\alpha}{\int_D M(\qt)\,[\psia^{\Delta t, +}(\xt,\qt,s) +\alpha]\dq} \dq&\\
\leq \frac{2}{\kappa} \int_D M(\qt)\,
\left|\nabq \sqrt{\psia^{\Delta t, +}(\xt,\qt,s)+\alpha}\right|^2 \dq&,
\end{align*}
for $s \in (t_{n-1},t_n]$; $n=1,\dots, N$. Hence, for a.e.\ $\xt \in \Omega$,
\begin{align}\label{last-bb}
&\hspace{-2mm}\int_{D} M(\qt)\, [\psia^{\Delta t, +}(\xt,\qt,s) +\alpha] \log [\psia^{\Delta t, +}(\xt,\qt,s)+\alpha]\dq\nonumber\\
&\hspace{-2mm}\qqqquad\leq \frac{2}{\kappa} \int_D M(\qt) \left|\nabq \sqrt{\psia^{\Delta t, +}(\xt,\qt,s)+\alpha}
\right|^2 \dq\nonumber\\
&\hspace{-2mm}+ \left(\int_D M(\qt)\, [\psia^{\Delta t, +}(\xt,\qt,s)+\alpha]\dq\right)\,\log\left( \int_D M(\qt)\, [\psia^{\Delta t, +}(\xt,\qt,s)+\alpha]\dq\right),
\end{align}
for $s \in (t_{n-1},t_n]$, $n=1,\dots, N$. Note that, thanks to \eqref{properties-b}
and the monotonicity of the mapping $s\in \mathbb{R}_{>0} \mapsto \log s\in \mathbb{R}$,
the second factor in the second term on the right-hand
side of \eqref{last-bb} is $\leq \log(1+\alpha)$. Since $\alpha \in (0,1)$, we have $\log(1+\alpha)>0$;
also, the first factor in the
second term on the right-hand side of \eqref{last-bb} is positive thanks to \eqref{properties-a} and
by \eqref{properties-b} it is bounded above by $(1+\alpha)$. Hence the
second term on the right-hand side of \eqref{last-bb} is bounded above by
the product $(1+\alpha) \log(1+\alpha)$. We integrate the resulting inequality over $\Omega$ to deduce that
\begin{align*}
&\int_{\Omega \times D}\!\!\! M(\qt)\, [\psia^{\Delta t, +}(\xt,\qt,s)+\alpha] \log [\psia^{\Delta t, +}(\xt,\qt,s)+\alpha]\dq\dx \\
&\qquad \leq \frac{2}{\kappa} \int_{\Omega \times D}\!\!\! M(\qt)\, |\nabq \sqrt{\psia^{\Delta t, +}(\xt,\qt,s)+\alpha}|^2 \dq \dx + |\Omega|\,(1+\alpha) \log(1+\alpha),
\end{align*}
for $s \in (t_{n-1},t_n]$, $n=1,\dots, N$. Equivalently, on noting that
$s \log s = \mathcal{F}(s) - (1 - s)$,
we can rewrite the last inequality in the following form:
\begin{align}\label{penultimate}
&\int_{\Omega \times D} M(\qt)\, \mathcal{F}(\psia^{\Delta t, +}(\xt,\qt,s)+\alpha) \dq\dx \nonumber\\
&\quad \leq \frac{2}{\kappa} \int_{\Omega \times D} M(\qt)\, |\nabq \sqrt{\psia^{\Delta t, +}(\xt,\qt,s)+\alpha}|^2 \dq \dx\nonumber\\
&\quad\quad  + \int_{\Omega \times D} M(\qt)\,(1 - \psia^{\Delta t, +}(\xt,\qt,s)-\alpha) \dq \dx
+ |\Omega|\,(1+\alpha) \log(1+\alpha),
\end{align}
for $s \in (t_{n-1},t_n]$, $n=1,\dots, N$.
This then in turn implies, thanks to the fact that $\psia^{\Delta t, +}(\xt,\qt,\cdot)$ is constant on the interval $(t_{n-1},t_n]$ for
all $(\xt,\qt) \in \Omega \times D$, that
\begin{align}
&\frac{\kappa\, a_0\, k}{\lambda}\, \Delta t \,\int_{\Omega \times D} M(\qt)\, \mathcal{F}(\psia^{\Delta t, +}(t_n)+\alpha) \dq\dx \nonumber\\
&
\quad\leq \frac{2a_0\,k}{\lambda} \int_{t_{n-1}}^{t_n} \int_{\Omega \times D}  M\, |\nabq \sqrt{\psia^{\Delta t, +}+\alpha}|^2 \dq \dx \,{\rm d}s \nonumber\\
&
\qquad+\frac{\kappa\, a_0\, k}{\lambda}\left[ \int_{t_{n-1}}^{t_n}\int_{\Omega \times D}\!\! M\, (1 - \psia^{\Delta t, +} -\alpha) \dq \dx \dd s
+ \Delta t\,|\Omega|\,(1+\alpha) \log(1+\alpha)\right]\!,\nonumber
\end{align}
for $n=1,\dots, N$. Using this and \eqref{bd1} in \eqref{eq:energy-u+psi2-local3} then yields
\begin{eqnarray}\label{eq:energy-u+psi2-local4}
&&\left(1+ \frac{\Delta t}{C_{\sf P}^2}\left(\nu - \alpha \frac{2\lambda\,b\,k}{a_0}\right)\right)\|\uta^{\Delta t, +}(t_n)\|^2
\nonumber\\
&&\quad+ \left(1+ \frac{\kappa\, a_0}{2\lambda}\Delta t\right) 2k\,\int_{\Omega \times D}\!\! M\, \mathcal{F}(\psia^{\Delta t, +}(t_n)+\alpha) \dq \dx
\nonumber \\
&&\quad \quad +\, 2k \int_{\Omega \times D}M\,[\mathcal{F}^L(\psia^{\Delta t,+}(t_n) + \alpha) - \mathcal{F}(\psia^{\Delta t,+}(t_n)+\alpha)]\dq \dx\qquad \nonumber\\
&&\leq \|\uta^{\Delta t,+}(t_{n-1})\|^2 + 2k \int_{\Omega \times D}\!\! M\, \mathcal{F}(\psia^{\Delta t,+}(t_{n-1})+\alpha) \dq \dx\nonumber\\
&&\quad +\,
 2k \int_{\Omega \times D}M\,[\mathcal{F}^L(\psia^{\Delta t,+}(t_{n-1}) + \alpha) - \mathcal{F}(\psia^{\Delta t,+}(t_{n-1})+\alpha)]\dq \dx
\nonumber\\
&&\qquad+\, \frac{\kappa\,a_0\, k}{\lambda}\, \int_{t_{n-1}}^{t_n}\int_{\Omega \times D} M\, (1 - \psia^{\Delta t, +}-\alpha) \dq \dx \dt\nonumber\\
&&\qqquad +\,\frac{\kappa\,a_0\, k}{\lambda}\, \Delta t\,|\Omega|\,(1+\alpha) \log(1+\alpha)\,  +\, \frac{1}{\nu}\int_{t_{n-1}}^{t_n}\!\!\|\ft^{\Delta t,+}\|^2_{(H^1_0(\Omega))'} \dd s,
\end{eqnarray}
for $n=1,\dots, N$. We now introduce, for $n=1,\dots,N$, the following notation:
\begin{align*}
\gamma(\alpha)&:= \min\left(\frac{1}{C_{\sf P}^2}\left(\nu - \alpha \frac{2\lambda\,b\,k}{a_0}\right),\frac{\kappa\,a_0}{2\lambda}\right),\\
A_n(\alpha)   &:= \|\uta^{\Delta t, +}(t_n)\|^2 + 2k\,\int_{\Omega \times D}\!\! M\, \mathcal{F}(\psia^{\Delta t, +}(t_n)+\alpha) \dq \dx,\\
B_n(\alpha)   &:= 2k \int_{\Omega \times D}M\,[\mathcal{F}^L(\psia^{\Delta t,+}(t_n) + \alpha) - \mathcal{F}(\psia^{\Delta t,+}(t_n)+\alpha)]\dq \dx,\\
C_n(\alpha)   &:= \frac{\kappa\,a_0\,k}{\lambda}\, \int_{t_{n-1}}^{t_n}\int_{\Omega \times D} M\, (1 - \psia^{\Delta t, +}(\xt,\qt,s)) \dq \dx \dd s\\
&\hspace{0.7cm} +\,\frac{\kappa\,a_0\, k}{\lambda}\, \Delta t\,|\Omega| \left[(1+\alpha) \log(1+\alpha)-\alpha\right]  + \frac{1}{\nu}\int_{t_{n-1}}^{t_n}\!\!\|\ft^{\Delta t,+}\|^2_{(H^1_0(\Omega))'} \dd s.
\end{align*}
We shall assume henceforth that $\alpha$ is sufficiently small in the sense that \eqref{alphacond} holds. For all such $\alpha$,
$\gamma(\alpha)>0$; further, trivially, $A_n(\alpha)$ is nonnegative; by \eqref{eq:FL2c},  we have that
$B_n(\alpha)$ is nonnegative, and by \eqref{additional2} and since $\mathcal{F}(1+\alpha)\geq 0$ for all $\alpha \geq 0$, $C_n(\alpha)$ is also nonnegative. In terms of this notation \eqref{eq:energy-u+psi2-local4} can be rewritten as follows:
\[ \left(1 + \gamma(\alpha) \, \Delta t\right) A_n(\alpha) + B_n(\alpha) \leq A_{n-1}(\alpha) + B_{n-1}(\alpha) + C_n(\alpha),\qquad n=1,\dots, N.\]
It then follows by induction that
\[ A_n(\alpha) \leq (1 + \gamma(\alpha) \,\Delta t)^{-n} A_0(\alpha) + \sum_{j=1}^n D_j(\alpha),\qquad n = 1,\dots, N.\]
That is,
\[ A_n(\alpha) + B_n(\alpha) \leq (1 + \gamma(\alpha) \,\Delta t)^{-n} A_0(\alpha) + \left\{B_0(\alpha) + \sum_{j=1}^n C_j(\alpha)\right\},\qquad n = 1,\dots, N.\]
%
%
%
%
%Hence a straightforward argument based on induction \cite{BS2010} implies that:
%%
%\[ A_n(\alpha) + B_n(\alpha) \leq (1 + \gamma(\alpha) \,\Delta t)^{-n} A_0(\alpha) + B_0(\alpha) + \sum_{j=1}^n C_j(\alpha),\qquad n = %1,\dots, N.\]
%%
In particular, with $n=N$, by omitting the nonnegative term $B_N(\alpha)$ from the left-hand side of the resulting inequality, and recalling that $T = t_N = N \Delta t$, we get
\begin{align}\label{bd3}
&\|\uta^{\Delta t, +}(T)\|^2 + 2k\,\int_{\Omega \times D}\!\! M\, \mathcal{F}(\psia^{\Delta t, +}(T)+\alpha) \dq \dx \nonumber\\
&\quad \leq \left(1 + \frac{\gamma(\alpha)\, T}{N}\right)^{-N}\left[ \|\uta^{\Delta t, +}(0)\|^2 + 2k\,\int_{\Omega \times D}\!\! M\, \mathcal{F}(\psia^{\Delta t, +}(0)+\alpha) \dq \dx \right]\nonumber\\
&\quad\quad +\, 2k\, \int_{\Omega \times D} M \left[\mathcal{F}^L(\psia^{\Delta t, +}(0)+\alpha)
- \mathcal{F}(\psia^{\Delta t, +}(0)+\alpha)\right] \dq\dx\nonumber\\
&\quad\quad + \frac{\kappa\,a_0\, k}{\lambda}\, \int_{0}^{T}\int_{\Omega \times D} M\, (1 - \psia^{\Delta t, +}(\xt,\qt,s)) \dq \dx \dd s
\nonumber\\
&\quad\quad +\,\frac{\kappa\, a_0\, k}{\lambda}\, T\,|\Omega|\,\mathcal{F}(1+\alpha) + \frac{1}{\nu}\int_{0}^{T}\!\!\|\ft^{\Delta t,+}\|^2_{(H^1_0(\Omega))'} \dd s.
\end{align}
Using that $\uta^{\Delta t, +}(0) = \ut^0$ and $\psia^{\Delta t,+} = \beta^L(\hat\psi^0)$,
we then obtain from \eqref{bd3} that
\begin{align}\label{bd3a}
&\|\uta^{\Delta t, +}(T)\|^2 + 2k\,\int_{\Omega \times D}\!\! M\, \mathcal{F}(\psia^{\Delta t, +}(T)+\alpha) \dq \dx \nonumber\\
&\quad \leq \left(1 + \frac{\gamma(\alpha)\, T}{N}\right)^{-N}\left[ \|\ut^0\|^2 + 2k\,\int_{\Omega \times D}\!\! M\, \mathcal{F}(\beta^L(\hat\psi^0)+\alpha) \dq \dx \right]\nonumber\\
&\quad\quad +\, 2k\, \int_{\Omega \times D} M \left[\mathcal{F}^L(\beta^L(\hat\psi^0)+\alpha)
- \mathcal{F}(\beta^L(\hat\psi^0)+\alpha)\right]\!\dq\dx\nonumber\\
&\quad\quad + \frac{\kappa\,a_0\, k}{\lambda}\, \int_{0}^{T}\int_{\Omega \times D} M\, (1 - \psia^{\Delta t, +}(\xt,\qt,s)) \dq \dx \dd s
\nonumber\\
&\quad\quad +\,\frac{\kappa\,a_0\, k}{\lambda}\, T\,|\Omega|\, \mathcal{F}(1+\alpha) + \frac{1}{\nu}\int_{0}^{T}\!\!\|\ft^{\Delta t,+}\|^2_{(H^1_0(\Omega))'} \dd s.
\end{align}
Applying \eqref{eq:FL2c} and \eqref{bound-on-t5} in the second factor in the first term on the right-hand side  of \eqref{bd3a}
and using \eqref{before-two} in the square brackets in the second term on the right-hand side, we have that
\begin{align}\label{bd4}
&\|\uta^{\Delta t, +}(T)\|^2 + 2k\,\int_{\Omega \times D}\!\! M\, \mathcal{F}(\psia^{\Delta t, +}(T)+\alpha) \dq \dx \nonumber\\
&\quad \leq \left(1 + \frac{\gamma(\alpha)\, T}{N}\right)^{-N}\left[ \|\ut^0\|^2 + 3\alpha\,k\, |\Omega| + 2k\int_{\Omega \times D} M\, \mathcal{F}(\hat\psi^0 + \alpha) \dq \dx \right]
\nonumber\\
&\quad\quad+ 3 \alpha\,k\,|\Omega|\,+ \frac{\kappa\,a_0\, k}{\lambda}\, \int_{0}^{T}\int_{\Omega \times D} M\, (1 - \psia^{\Delta t, +}(\xt,\qt,s)) \dq \dx \dd s
\nonumber
\\
&\quad\quad +\,\frac{\kappa\,a_0\, k}{\lambda}\, T \,|\Omega|\,\mathcal{F}(1+\alpha) + \frac{1}{\nu}\int_{0}^{T}\!\!\|\ft^{\Delta t,+}\|^2_{(H^1_0(\Omega))'} \dd s.
\end{align}
We now pass to the limit $\alpha \rightarrow 0_+$, with $L$ and $\Delta t$ fixed, in much the same way as in Section \ref{Lindep-space}. Noting that $\lim_{\alpha \rightarrow 0_+}\gamma(\alpha)= \gamma_0$,  we thus obtain from \eqref{bd4}, \eqref{idatabd}
and \eqref{inidata-1}, that
\begin{align}\label{bd5}
&\|\uta^{\Delta t, +}(T)\|^2 + 2k\,\int_{\Omega \times D}\!\! M\, \mathcal{F}(\psia^{\Delta t, +}(T)) \dq \dx \nonumber\\
&\quad \leq \left(1 + \frac{\gamma_0\, T}{N}\right)^{-N}\left[ \|\ut_0\|^2 + 2k\int_{\Omega \times D} M\, \mathcal{F}(\hat\psi_0) \dq \dx \right]\nonumber\\
&\qquad+ \frac{\kappa\,a_0\, k}{\lambda}\, \int_{0}^{T}\int_{\Omega \times D} M\, (1 - \psia^{\Delta t, +}) \dq \dx \dd s
+ \frac{1}{\nu}\int_{0}^{T}\!\!\|\ft^{\Delta t,+}\|^2_{(H^1_0(\Omega))'} \dd s.
\end{align}
In order to pass to the limits $L \rightarrow \infty$ and $\Delta t \rightarrow 0_+$
(with $\Delta t = o(L^{-1})$)
in the first two terms on the left-hand side of \eqref{bd5} we require
additional considerations.

Noting \eqref{usconL2a} for the sequence $\{\ut_{\epsilon,L}^{\Delta t}\}_{L>1}$ of Theorem \ref{convfinal}, and passing to a subsequence (not indicated),
as $L \rightarrow \infty$ and $\Delta t \rightarrow 0_+$ (with $\Delta t = o(L^{-1})$) we have that $\|\ut_{\epsilon,L}^{\Delta t}(t) - \ut_\epsilon(t)\|$ converges to $0$ for a.e. $t \in (0,T)$; let $t_\ast$ be one such $t$ in $(0,T)$. It then follows from (\ref{utwconL2a}) that, for any $\vt \in \Vt_\sigma \subset \Ht$,
\begin{align}
&\hspace{-0.8mm}|\langle \ut_\epsilon(T)- \ut_{\epsilon,L}^{\Delta t}(T),\vt\rangle_{V_\sigma}
|
\leq \left|\int_{t_\ast}^T\!\! \left \langle \frac{\partial (\ut_\epsilon-
\ut_{\epsilon,L}^{\Delta t})}{\partial t}(t),\vt \right \rangle_{\!\!V_\sigma}
\!\!\!\dt \,\right|
+ \|\ut_{\epsilon,L}^{\Delta t}(t_\ast) - \ut_\epsilon(t_\ast)\|\,\|\vt\| \nonumber \\
& \hspace{3.55cm}\rightarrow 0 \mbox{ as } L \rightarrow \infty \mbox{ and } \Delta t
\rightarrow 0_+ \mbox{ with } \Delta t=o(L^{-1}).\!\!\!\!
\label{uTV}
\end{align}
Since $\ut_\epsilon\,: [0,T] \rightarrow \Ht$ is weakly continuous,
we have that $\ut_\epsilon(T) \in \Ht$.
It follows from the bound on the first term in \eqref{eq:energy-u+psi-final6a},
as $t \in [0,T]\mapsto \|\ut_{\epsilon,L}^{\Delta t}(t)\|\in \mathbb{R}_{\geq 0}$
is a {\em continuous} (piecewise linear) function, that,
for any $\vt_0 \in \Ht$ and any $\vt \in \Vt_\sigma$,
\begin{align*}
|(\ut_\epsilon(T)- \ut_{\epsilon,L}^{\Delta t}(T),\vt_0)
| &\leq |\langle \ut_\epsilon(T)- \ut_{\epsilon,L}^{\Delta t}(T),\vt\rangle_{V_\sigma}
| + \left[ \|\ut_\epsilon(T)\| + C_\ast^{\frac{1}{2}}\right]\, \|\vt_0 -\vt\|.
\end{align*}
Recalling (\ref{uTV}), it follows from the last inequality that
\[
\mbox{lim}~\!\mbox{sup}_{L \rightarrow \infty}|(\ut_\epsilon(T)-
\ut^{\Delta t}_{\epsilon, L}(T), \vt_0)| \leq C\,\|\vt_0-\vt\|\qquad
\forall \vt_0 \in \Ht,\quad\forall \vt \in \Vt_\sigma.
\]
As $\Vt_\sigma$ is dense in $\Ht$, we thus deduce that $\{\ut_{\epsilon,L}^{\Delta t}(T)\}_{L>1}$ converges to $\ut_\epsilon(T)$
weakly in $\Ht$ as $L \rightarrow \infty$ and $\Delta t \rightarrow 0_+$,
with $\Delta t = o(L^{-1})$.
Hence, by the weak lower-semicontinuity of the $L^2(\Omega)$ norm
and noting that
$\ut^{\Delta t}_{\epsilon, L}(T) = \ut^{\Delta t,+}_{\epsilon, L}(T)$, we have
\begin{equation}\label{lowersc1}
\|\ut_\epsilon(T)\| \leq
\mbox{lim~\!inf}_{L \rightarrow \infty} \|\ut^{\Delta t,+}_{\epsilon, L}(T)\|.
\end{equation}

Analogously to \eqref{uTV},  noting \eqref{psisconL2a} for the sequence
$\{\hat \psi_{\epsilon,L}^{\Delta t}\}_{L>1}$ of Theorem \ref{convfinal}, we have
(on passing to a subsequence, not indicated,) that
$\hat \psi_{\epsilon,L}^{\Delta t}(T)$ converges weakly to $\hat \psi_\epsilon(T)$
in $M^{-1}(H^s(\Omega \times D))'$ as $L \rightarrow \infty$ and $\Delta t \rightarrow 0_+$, with $\Delta t = o(L^{-1})$.
Thanks to Theorem \ref{convfinal}, $\hat \psi_\epsilon\,:\, [0,T] \rightarrow L^1_M(\Omega \times D)$ is weakly
continuous; hence we have that $\hat\psi_\epsilon(T) \in L^1_M(\Omega \times D)$.
%Thanks to Step 3.11 in the proof of Theorem \ref{convfinal}, $\rho_\epsilon \in %C([0,T];L^2(\Omega))$;
%whereby $\rho_\epsilon (T) \in L^1(\Omega)$;
%equivalently, $\hat \psi_\epsilon(T) \in L^1_M(\Omega \times D)$.
Similarly to the argument in the proof of \ding{205} of Lemma \ref{psi0properties},
it follows from the bound on the fourth term in \eqref{eq:energy-u+psi-final6a}, on noting that $\mathcal{F}(r)/r \rightarrow \infty$ as $r\rightarrow \infty$, together with the de la Vall\'ee-Poussin and Dunford--Pettis theorems, that, upon extraction of a further subsequence (not indicated), $\hat\psi^{\Delta t}_{\epsilon, L}(T)$ converges weakly in $L^1_M(\Omega \times D)$ to some limit $A$, as $L \rightarrow \infty$
and $\Delta t \rightarrow 0_+$, with $\Delta t = o(L^{-1})$.
The fact that $A= \psi_\epsilon(T)$ follows %as $H^s(\Omega \times D)
%\subset L^\infty(\Omega \times D)$ and
from the weak convergence of
$\hat \psi_{\epsilon,L}^{\Delta t}(T)$ to $\hat \psi_\epsilon(T)$
in $M^{-1} (H^s(\Omega \times D))' ~(\hookleftarrow  L^1_M(\Omega \times D))$.
Finally, since $r \in [0,\infty) \mapsto \mathcal{F}(r) \in \mathbb{R}_{\geq 0}$ is continuous and convex, on applying Tonelli's weak lower semicontinuity theorem in $L^1_M(\Omega \times D)$ (cf. Theorem 3.20
in Dacorogna \cite{Dacorogna}),
%or Theorem 6.54 in Fonseca \& Leoni \cite{Fonseca_Leoni}.)
\begin{comment}
%
\begin{equation}\label{lowersc2a}
\hspace{-0.4mm}\int_{\Omega \times D}\! M \mathcal{F}(\hat\psi_\epsilon(t)) \dq \dx \leq
\mbox{lim}~\!\mbox{inf}_{L \rightarrow \infty}
 \int_{\Omega \times D}\! M \mathcal{F}(\hat\psi^{\Delta t}_{\epsilon,L}(t)) \dq \dx
 \quad \forall t \in [0,T],\!\!
\end{equation}
%
with $\Delta t \rightarrow 0_+$ and  $\Delta t = o(L^{-1})$. Taking $t=T$ in
\eqref{lowersc2a} and
\end{comment}
%
\begin{equation}\label{lowersc2}
\int_{\Omega \times D} M \mathcal{F}(\hat\psi_\epsilon(T)) \dq \dx \leq
\mbox{lim}~\!\mbox{inf}_{L \rightarrow \infty}
 \int_{\Omega \times D} M \mathcal{F}(\hat\psi^{\Delta t,+}_{\epsilon,L}(T)) \dq \dx,
\end{equation}
where we have noted that
$\hat\psi^{\Delta t}_{\epsilon, L}(T) = \hat\psi^{\Delta t,+}_{\epsilon, L}(T)$.

We are now ready to pass to the limit in \eqref{bd5}. Using \eqref{lowersc1} and \eqref{lowersc2}, \eqref{psisconL2a}, \eqref{mass-conserved} and \eqref{fncon},
and letting $L \rightarrow \infty$ (whereby $\Delta t \rightarrow 0_+$ according to $\Delta t = o(L^{-1})$ and therefore $N=T/\Delta t \rightarrow \infty$), we deduce from \eqref{bd5} that
\begin{align}\label{bd7}
&\|\ut_\epsilon(T)\|^2 + 2k\,\int_{\Omega \times D}\!\! M\, \mathcal{F}(\hat\psi_\epsilon(T)) \dq \dx \nonumber\\
&\qquad \leq {\rm e}^{-\gamma_0 T}\left[ \|\ut_0\|^2 + 2k\int_{\Omega \times D} M\, \mathcal{F}(\hat\psi_0) \dq \dx \right]
+ \frac{1}{\nu}\int_{0}^{T}\!\!\|\ft(s)\|^2_{(H^1_0(\Omega))'} \dd s.
\end{align}

The Csisz\'ar--Kullback inequality (cf., for example, (1.1) and (1.2) in the work of
Unterreiter et al. \cite{UAMT}) with respect to the Gibbs measure $\undertilde{\mu}$
defined by $\dd \undertilde{\mu}= M(\qt)\dq$  yields, on noting \eqref{mass-conserved}, for a.e.\ $\xt \in \Omega$, that
\[ \|\hat\psi_\epsilon(\xt,\cdot,T) - 1\|_{L^1_M(D)} \leq \left[2 \int_{D} M\, \mathcal{F}(\hat\psi_\epsilon(\xt,\qt,T))
\dq\right]^{\frac{1}{2}},\]
which, after integration over $\Omega$ implies, by the Cauchy--Schwarz inequality, that
%and squaring both sides, that
%
\[ \|\hat\psi_\epsilon(T) - 1\|^2_{L^1_M(\Omega \times D)} \leq 2|\Omega| \int_{\Omega \times D} M\, \mathcal{F}(\hat\psi_\epsilon(T))
\dq\dx.\]
Combining this with \eqref{bd7} yields (\ref{eq:bd55a}).
Taking $\ft\equiv 0$, the stated exponential decay in time of $(\ut_\epsilon,\hat\psi_\epsilon)$ to $(\zerot,1)$
in the $\Lt^2(\Omega) \times L^1_M(\Omega \times D)$ norm follows from \eqref{eq:bd55a}.
\end{proof}

\begin{remark}
By introducing the free energy
as the sum of the kinetic energy and the relative entropy:
\[\mbox{$\mathfrak{E}(t):= \frac{1}{2}\|\ut_\epsilon(t)\|^2 + k\,\int_{\Omega \times D} M\, \mathcal{F}(\hat\psi_\epsilon(t)) \dq \dx$,}\]
we deduce from \eqref{bd7} that, for any $T>0$,
\[ \mbox{$\mathfrak{E}(T) \leq {\rm e}^{-\gamma_0 T} \mathfrak{E}(0) + \frac{1}{2\nu}\int_{0}^{T} \|\ft(s)\|^2_{(H^1_0(\Omega))'} \dd s$}.\]
Thus in particular when $\ft=\zerot$, the free energy decays to $0$ as a function of time
from any initial datum $(\ut_0, \psi_0)$ with initial velocity $\ut_0 \in \Ht$
and initial probability density function $\psi_0$ that has finite relative entropy with respect to the log-concave Maxwellian $M$.

It is interesting to note the dependence of $\gamma_0 = \min\left(\frac{\nu}{C_{\sf P}^2}\,,\,\frac{\kappa\,a_0}{2\lambda}\right)$,
the rate at which the fluid relaxes to equilibrium,
on the dimensionless viscosity coefficient $\nu$ of the solvent, the minimum eigenvalue $a_0$
of the Rouse matrix $A$, the geometry of the flow domain encoded in the Poincar\'e constant $C_{\sf P}(\Omega)$, the Weissenberg number $\lambda$, and the Bakry--\'{E}mery constant $\kappa$ for the Maxwellian $M$ of the model.
We also observe that the right-hand side of the energy inequality \eqref{eq:energyest} and  $\gamma_0$ are independent of
the centre-of-mass diffusion coefficient $\varepsilon$ appearing in the equation \eqref{fp0}.$\quad\diamond$
\end{remark}

\noindent
\textbf{Acknowledgement}
ES was supported by the EPSRC Science and Innovation award to the Oxford Centre
for Nonlinear PDE (EP/E035027/1).

%~\vspace{-6.5mm}

\bibliographystyle{siam}

\bibliography{polyjwbesrefs}

\appendix

\section{Cartesian products of Lipschitz domains}
\label{AppendixA}
\setcounter{equation}{0}

Let us suppose that $D := D_1 \times \cdots \times D_K$, where $D_i$, $i=1,\dots, K$, are
bounded open balls in $\mathbb{R}^d$ centred at $\zerot \in \mathbb{R}^d$. Then, $D$ is a bounded open
Lipschitz domain (cf. \ref{AppendixC} below for a precise definition of a Lipschitz domain).
We present two different proofs of this statement.

The first argument is based on the observation that the Cartesian product of $K$ bounded open Lipschitz
domains $D_i \subset \mathbb{R}^{n_i}$, $i=1, \dots, K$, with $n_i \geq 1$, is a bounded open Lipschitz domain in $\mathbb{R}^{n_1 +\dots + n_K}$. First, note that as a Cartesian product of a finite collection of
open sets, $D$ itself is open. That $D$ is a Lipschitz domain follows by combining Theorem 3.1 in the Ph.D. Thesis of Hochmuth \cite{RH}, which implies that the Cartesian product of a finite number of bounded domains
$D_i \subset \mathbb{R}^{n_i}$, $i=1, \dots, K$, with $n_i \geq 1$,  each satisfying the uniform cone property, is a bounded domain in $\mathbb{R}^{n_1 +\dots + n_K}$ satisfying the uniform cone property; and Theorem 1.2.2.2 in the book of Grisvard \cite{PG}, which states that a bounded open set in $\mathbb{R}^n$ has the uniform cone property if, and only if, its boundary is Lipschitz.

In the special case of our domain $D$ an alternative proof proceeds by noting that, as a Cartesian product of $K$ bounded open convex sets, $D_i \subset \mathbb{R}^d$, $i=1, \dots, K$, $D$ is a bounded open convex set in $\mathbb{R}^{K d}$ (cf. Hiriart-Urruty \& Lemar{\'e}chal \cite{HUL}, p.23), and then applying Corollary 1.2.2.3 in Grisvard \cite{PG}, which states that a bounded open convex set in $\mathbb{R}^{K d}$ has a Lipschitz boundary.

\section{\!Completeness and separability of $L^2_M(D)$ and $H^1_M(D)$}
\label{AppendixB}
\setcounter{equation}{0}

The completeness of the spaces $H^1_{M}(D)$ and $L^2_{M}(D)$ follows from Theorem 8.10.2
on p.418 in the monograph of Kufner, John \& Fu{\v c}ik \cite{KJF}.

By Theorem 19 in Section IV.8.19 of Dunford \& Schwartz \cite{DS}, $C(\overline{D})$ is dense in $L^2_M(D) = L^2(D; \undertilde{\mu})$, where the Radon measure $\undertilde{\mu}$ is defined on the $\sigma$ algebra of Borel subsets of $D$ by $\undertilde{\mu}(B) := \int_B M(\qt) \dd \qt$; $\undertilde{\mu}$ is absolutely continuous with respect to the Lebesgue measure on $\mathbb{R}^{K d}$
and $\dd \undertilde{\mu} = M(\qt) \dd \qt$. Thus, given $v \in L^2_{M}(D)$ and any $\varepsilon > 0$,
there exists $\varphi \in C(\overline{D})$ such that
\[ \|v - \varphi\|_{L^2_{M}(D)} < \textstyle{\frac{1}{2}}\varepsilon.\]
By the Stone--Weierstrass theorem (cf., for example, Pinkus \cite{AP}), $C(\overline{D})$ is separable: there exists a countable dense
set $\mathcal{P} \subset C(\overline{D})$, where $\mathcal{P}$ is the set of restrictions from $\mathbb{R}^{K d}$ to $\overline{D}$ of all polynomials with rational coefficients (cf. Theorem 1.4.5 on p.30 of Kufner, John \& Fu\v{c}ik \cite{KJF}); hence, given $\varepsilon>0$
there exists $p \in \mathcal{P}$ such that
\[ \|\varphi - p \|_{C(\overline{D})} <
\textstyle{\frac{1}{2}}\varepsilon.\]
Clearly, $C(\overline{D}) \subset L^2_{M}(D)$ and
therefore $\mathcal{P} \subset L^2_{M}(D)$; hence, and since $\mu(D)=1$,
\begin{eqnarray*}
\|v - p\|_{L^2_{M}(D)} &\leq& \|v - \varphi\|_{L^2_{M}(D)} + \|\varphi - p\|_{L^2_{M}(D)}\\
& < & \textstyle{\frac{1}{2}}\,\varepsilon
+ \|\varphi - p \|_{C(\overline{D})} \left(\int_{D} M(\qt)
\dd \qt \right)^{1/2} < \varepsilon.
\end{eqnarray*}
This shows that the countable set $\mathcal{P} \subset L^2_{M}(D)$
is dense in $L^2_{M}(D)$. Therefore $L^2_{M}(D)$ is separable.
By an argument analogous to the one in Section 3.5 on p.61 of Adams \& Fournier \cite{AF:2003},
$H^1_M(D)$ is separable.

\section{Density of $C^\infty(\overline{D})$ in $L^2_M(D)$ and $H^1_M(D)$, and of
$C^\infty(\overline{\Omega \times D})$ in $L^2_M(\Omega \times D)$ and $H^1_M(\Omega \times D)$}
\label{AppendixC}
\setcounter{equation}{0}

Since the set $\mathcal{P}$ defined in the previous subsection is dense in $L^2_M(D)$
and $\mathcal{P} \subset C^\infty(\overline{D}) \subset L^2_M(D)$, we deduce that
$C^{\infty}(\overline{D})$ is dense in $L^2_M(D)$. By an identical argument, $C^\infty(\overline{\Omega \times D})$ is dense in $L^2_M(\Omega \times D)$.

Next we consider the density of $C^\infty(\overline{D})$ in $H^1_M(D)$, and of
$C^\infty(\overline{\Omega \times D})$ in $H^1_M(\Omega \times D)$.
The argument below closely follows the proof of Theorem 1.1 on p.307 in the paper of Ne\v{c}as \cite{Necas}.
We suppose that a fixed Cartesian co-ordinate system, referred to henceforth as the {\em canonical
co-ordinate system}, is introduced in $\mathbb{R}^n$, so that a vector $\qt \in \mathbb{R}^n$ is
represented in terms its co-ordinates as
\[ \qt = (q_1, \dots, q_n).\]
Suppose that $\tilde{\mathbb{A}}$ is an $n \times n$ orthogonal matrix with determinant equal to $+1$ and
$\tilde{C}$ is an $n$-component column vector. The affine transformation
$\qt^{\rm T} \mapsto \tilde{\qt}^{\rm T}:=\tilde{\mathbb{A}} \qt^{\rm T} + \tilde{C}$ defines a new co-ordinate system, which we shall say is {\em equivalent} to the canonical co-ordinate system, and we shall express $\tilde{\qt}$
in terms of its co-ordinates in this second, equivalent, co-ordinate system as
\[ \tilde{\qt} = ({\tilde{q}}_1, \dots, {\tilde{q}}_n).\]
We recall the following definition from Kufner, John \& Fu\v{c}ik \cite{KJF}.

\begin{definition} We say that a bounded open domain $\mathcal{D} \subset \mathbb{R}^n$
is a Lipschitz domain if, and only if, there exist:
\begin{itemize}
\item[(i)] a positive integer $m$ and $m$ (different) Cartesian co-ordinate systems, referred to
as local co-ordinate systems, each of which is equivalent to a fixed canonical co-ordinate system
in $\mathbb{R}^n$.
When an $n$-component vector is expressed in terms of its co-ordinates in the $r$th local co-ordinate
system, we shall write
\[ Q_r = (q_{r1}, \dots, q_{rn}) = (q'_{(r)}, q_{(r)}^n),\]
where $q'_{(r)}:= (q_{r1}, \dots, q_{r n-1})$ and $q_{(r)}^n:= q_{rn}$.
\item[(ii)] a number $\alpha > 0$ and $m$ functions
\[ a_r \in C^{0,1}(\overline \Delta_r),\quad r = 1, \dots, m, \]
where
\[ \Delta_r = \bigg\{ q_{(r)}'\, :\, |q_{(r)}'|:= \left(\sum_{i=1}^{n-1}|q_{r i}|^2\right)^{\frac{1}{2}}
< \alpha \bigg\} .\]

\item[(iii)] a number $\beta>0$ such that
\begin{itemize}
\item[(iii.1)] the sets
\[ \Lambda_r := A_r^{-1}(\{Q_r^{\rm T} = (q_{(r)}',q_{(r)}^n)^{\rm T}\,:\,q_{(r)}' \in \Delta_r\quad
\mbox{and}\quad q_{(r)}^n = a_r(q_{(r)}')\}) \]
are subsets of $\partial \mathcal{D}$ for $r=1, \dots, m$,
\[ \partial \mathcal{D} = \bigcup_{r=1}^m \Lambda_r\]
and $A_r\, : \, Q^{\rm T} \rightarrow Q_r^{\rm T}$ is the affine transformation of co-ordinates $Q_r^{\rm T} := A_r(Q^{\rm T}) = \mathbb{A}_rQ^{\rm T} + C_r$, with $\mathbb{A}_r\in \mathbb{R}^{n \times n}$ orthogonal and determinant equal to $+1$ and $C_r$ an $n$-component column vector, that maps the canonical co-ordinate system in $\mathbb{R}^n$ to the, equivalent, $r$th local co-ordinate system.
\item[(iii.2)] for every $r=1,\dots, m$, the sets $U_r^-$ and $U_r^+$, defined by
\begin{eqnarray*}
&&
\!\!\!\!\!\!\!U_r^- := A_r^{-1}(\{Q_r^{\rm T} = (q'_{(r)}, q^n_{(r)})^{\rm T}\,:\, q'_{(r)} \in \Delta_r,\; a_r(q'_{(r)})-\beta < q^n_{(r)} < a_r(q'_{(r)})\})\\
&&\!\!\!\!\!\!\!U_r^+ := A_r^{-1}(\{Q_r^{\rm T}
= (q'_{(r)}, q^n_{(r)})^{\rm T}\,:\, q'_{(r)} \in \Delta_r,\;
a_r(q'_{(r)}) < q^n_{(r)} < a_r(q'_{(r)})+\beta\})
\end{eqnarray*}
are such that $U_r^-\subset \mathcal{D}$ and $U_r^+ \subset \overline{\mathcal D}^{c}$ (= the complement of $\overline{\mathcal{D}}$ w.r.t. $\mathbb{R}^n$).
\end{itemize}
\end{itemize}
We shall write
\[
U_r := A_r^{-1}(\{Q_r^{\rm T} = (q'_{(r)}, q^n_{(r)})^{\rm T}\,:\, q'_{(r)} \in \Delta_r\;
\mbox{and}\; a_r(q'_{(r)})-\beta < q^n_{(r)} < a_r(q'_{(r)})+\beta\}).
\]
\end{definition}

\begin{theorem}
 $C^\infty(\overline{D})$ is dense in $H^1_M(D)$.
\end{theorem}

\begin{proof}
Let $\qt = (\qt_1, \dots, \qt_K) \in D = D_1 \times \cdots \times D_K$, where
$D_i \subset \mathbb{R}^d$ is a bounded open ball in $\mathbb{R}^d$ centred at the origin,
and let $n:= K d$; thus, after relabelling the co-ordinates of $\qt$,
$\qt = (q_1, \dots, q_n) \in D$, and $D$ is a bounded Lipschitz
domain in $\mathbb{R}^n$ by \ref{AppendixA}.
In particular, by using the notation introduced in the previous definition,
there exist $m$ local co-ordinate systems in $\mathbb{R}^{n}$ and real-valued functions
$a_r$, $r=1, \dots, m$, such that for each $\qt \in \partial D$ there is an
$r \in \{1, \dots, m\}$ for which
\[ \qt^{\rm T} = A_r^{-1}((q_{(r)}', a_r(q'_{(r)}))^{\rm T}), \]
where $q_{(r)}' \in \Delta_r$, $a_r \in C^{0,1}(\overline{\Delta_r})$ and $A_r$ is as in the definition above.
Since the Jacobian of each mapping $A_r$ is equal to $+1$, the specific choice of the matrix
$\mathbb{A}_r$ and the vector ${C}_r$ in the definition of the local co-ordinate transformation
$A_r$ does not affect the argument below. We shall therefore assume for ease of exposition that
$\mathbb{A}_r = \mathbb{I}$ and ${C}_r = 0$, so that $A_r$ is the identity mapping. Thus, for example,
we shall write $u(q'_{(r)}, q^n_{(r)})$ instead of $u((A_r^{-1}((q'_{(r)}, q^n_{(r)}))^{\rm T})^{\rm T})$.

There exist functions $\varphi_r \in C^\infty_0(U_r)$, with $0 \leq \varphi_r \leq 1$, $r=1,\dots, m$,
and a function $\varphi_{m+1} \in C^\infty_0(D)$, with $0 \leq \varphi_{m+1} \leq 1$,  such that
we have the partition of unity
\[
\sum_{i=1}^{m+1} \varphi_r  \equiv 1\quad \mbox{on}\quad \overline{D}.
\]

Now, given $u \in H^1_{M}(D)$, let $u_r:= u \varphi_r$ for $r=1,\dots, m+1$; clearly,
$u_r \in H^1_M(D)$ and $u_{m+1} \in H^1(D) \subset H^1_M(D)$. Having thus decomposed $u \in H^1_M(D)$ as
\begin{equation}\label{partition}
 u = \sum_{r=1}^{m+1} u_r \qquad \mbox{on $D$},
\end{equation}
we give a brief overview of the rest of the proof. Let us first note that on any compact subdomain $B$ of $D$ the Maxwellian $M$ is bounded above and below by positive constants; on such subdomains $B$ we have that  $H^1_M(B) = H^1(B)$. On $D$ itself the situation is quite different: of course, $H^1(D) \subset H^1_M(D)$; however the converse inclusion is false as there exist elements in $H^1_M(D)$, which, due to the decay of the weight function $M$ to $0$ on $\partial D$, tend to infinity at a faster rate than can be accommodated within $H^1(D)$. The idea of the proof is therefore to translate each of the functions $u_r$, $r=1,\dots,m$, in the direction of the boundary in order to shift the locations of potential `bad' behaviour in $u_{r}\in H^1_M(D)$ from $\partial D$ to $\overline{D}^c$; this then ensures that, for any (sufficiently small) positive value of the shift parameter $\lambda$, the shifted function, $u_{r\lambda}$, belongs to $H^1(D)$, and can be therefore well approximated in $H^1(D)$ by functions (denoted $u_{r \lambda h}$ below) that lie in $C^\infty(\overline{D})$, thanks to the density of $C^\infty(\overline{D})$ in $H^1(D)$. As $\varphi_{m+1} \in C^\infty_0(D)$, the function $u_{m+1}=u \varphi_{m+1}$ has compact support in $D$,
and can be therefore well approximated in $H^1_D(D)$ (and thereby also in $H^1_M(D)$)
by smooth functions, without the need to form its shifted counterpart
$u_{m+1 \lambda}$ beforehand and approximate that instead; however, $u_{m+1}$ will be shifted in any case, by the same amount $\lambda$ as the functions $u_r$, $r=1,\dots,m$, simply to ensure that a shifted counterpart of the partition \eqref{partition} holds.

With this motivation in mind, we define shifted counterparts, $u_{r\lambda}$, of the $u_r$ by
\[u_{r\lambda}(q'_{(r)}, q^n_{(r)}) := u_{r}(q'_{(r)}, q^n_{(r)}-\lambda),\qquad \mbox{for $r=1, \dots, m+1$.}\]

\textit{Step 1.} The first step in the argument amounts to showing that
\begin{equation}\label{shift}
  \lim_{\lambda \rightarrow 0_+} u_{r\lambda} = u_r \quad \mbox{in $H^1_{M}(D)$},\quad r=1,\dots, m+1.
\end{equation}
This clearly holds for $r=m+1$. Let us therefore assume that $r \in \{1, \dots, m\}$, let
$g_r$ signify the function $u_r$ or any of its first partial derivatives, and define $V_r : = U_r^-$.
Clearly, $V_r \subset U_r$, $r=1,\dots, m$. By the triangle inequality, with $\qt = (q'_{(r)}, q^n_{(r)})$,
\begin{eqnarray}
&&\left[ \int_{V_r}|g_r(q'_{(r)}, q^n_{(r)}) - g_r(q'_{(r)}, q^n_{(r)} - \lambda)|^2\, {M}(\qt)
\dd \qt \right]^{\frac{1}{2}} \nonumber\\
&&= \left[ \int_{V_r}|g_r(q'_{(r)}, q^n_{(r)})[{M}(\qt)]^{\frac{1}{2}} - g_r(q'_{(r)}, q^n_{(r)} - \lambda) [{M}(\qt)]^{\frac{1}{2}}|^2 \dd \qt \right]^{\frac{1}{2}}
\nonumber\\
&&\leq  \left[ \int_{V_r}|g_r(q'_{(r)}, q^n_{(r)})[{M}(q'_{(r)}, q^n_{(r)})]^{\frac{1}{2}} - g_r(q'_{(r)}, q^n_{(r)} - \lambda) [{M}( q'_{(r)}, q^n_{(r)}-\lambda)]^{\frac{1}{2}}|^2 \dd \qt \right]^{\frac{1}{2}}
\nonumber\\
&&\qquad + \left[ \int_{V_r}|g_r(q'_{(r)}, q^n_{(r)}-\lambda)|^2\, |[{M}(q'_{(r)}, q^n_{(r)}-\lambda)]^{\frac{1}{2}} - [{M}(q'_{(r)}, q^n_{(r)})]^{\frac{1}{2}}|^2 \dd \qt \right]^{\frac{1}{2}}\nonumber \\
&& =: {\rm T}_1 + {\rm T}_2.
\end{eqnarray}
We begin by considering ${\rm T}_2$. Let $\lambda \in (0,\beta/2]$; then, for
$\beta>0$ sufficiently small, there exists $c_0 \in (0,\beta]$ sufficiently small such that
\[ |[{M}(q'_{(r)}, q^n_{(r)}-\lambda)]^{\frac{1}{2}} - [{M}(q'_{(r)}, q^n_{(r)})]^{\frac{1}{2}}|^2 \leq {M}(q'_{(r)}, q^n_{(r)}-\lambda),
\]
for $q'_{(r)} \in \Delta_r$ and $q^n_{(r)} \in (a_r(q'_{(r)}) - c_0,  a_r(q'_{(r)}))$.
This follows on noting that for all $\lambda \in (0,\beta/2]$ and $c_0$ sufficiently small,
\[{M}(q'_{(r)}, q^n_{(r)}-\lambda) > {M}(q'_{(r)}, q^n_{(r)})\quad\mbox{for}\quad
\left\{\begin{array}{ll}q'_{(r)} \in \Delta_r, \\
                        q^n_{(r)} \in (a_r(q'_{(r)}) - c_0,  a_r(q'_{(r)})).\end{array}\right. \]
Hence, and by the absolute continuity of the Lebesgue integral, for any $\varepsilon > 0$ there exists $c_1 \in (0,c_0]$ small (independent of $\lambda \in (0, \beta/2]$) such that
\begin{eqnarray}\label{b1}
&&\int_{\Delta_r} \dd q'_{(r)} \int_{a_r(q'_{(r)})-c_1}^{a_r(q'_{(r)})}
|g_r(q'_{(r)}, q^n_{(r)}-\lambda)|^2\,
|[{M}(q'_{(r)}, q^n_{(r)}-\lambda)]^{\frac{1}{2}} - [{M}(q'_{(r)}, q^n_{(r)})]^{\frac{1}{2}}|^2  \dd q^n_{(r)}\nonumber\\
&&\leq \int_{\Delta_r} \dd q'_{(r)} \int_{a_r(q'_{(r)})-c_1}^{a_r(q'_{(r)})}
|g_r(q'_{(r)}, q^n_{(r)}-\lambda)|^2\,  {M}(q'_{(r)}, q^n_{(r)}-\lambda)  \dd q^n_{(r)}
 < \frac{1}{2}\varepsilon^2.
\end{eqnarray}
Now for $c_1>0$ fixed, there exists $\lambda>0$ sufficiently small such that
\begin{eqnarray}\label{b2}
&&\int_{\Delta_r} \dd q'_{(r)} \int_{a_r(q'_{(r)})-\beta}^{a_r(q'_{(r)})-c_1}
|g_r(q'_{(r)}, q^n_{(r)}-\lambda)|^2
\, |[{M}(q'_{(r)}, q^n_{(r)}-\lambda)]^{\frac{1}{2}} - [{M}(q'_{(r)}, q^n_{(r)})]^{\frac{1}{2}}|^2 \dd q^n_{(r)}\nonumber\\
&& \leq \max_{q'_{(r)} \in \Delta_r,\; a_r(q'_{(r)})-\beta \leq q^n_{(r)} \leq a_r(q'_{(r)})-c_1}
\left|1 - \frac{[{M}(q'_{(r)}, q^n_{(r)})]^{\frac{1}{2}}}{[{M}(q'_{(r)}, q^n_{(r)}-\lambda)]^{\frac{1}{2}}}\right|^2\nonumber\\
&&\qquad \times \int_{\Delta_r} \dd q'_{(r)} \int_{a_r(q'_{(r)})-\beta}^{a_r(q'_{(r)})-c_1}
|g_r(q'_{(r)}, q^n_{(r)}-\lambda)|^2\,  {M}(q'_{(r)}, q^n_{(r)}-\lambda)  \dd q^n_{(r)}
< \frac{1}{2}\varepsilon^2.
\end{eqnarray}
Summing \eqref{b1} and \eqref{b2} and taking the square root of both sides of the resulting inequality,
we deduce that for any $\varepsilon>0$ there exists $\lambda>0$ such that
\[
\left[ \int_{V_r}|g_r(q'_{(r)}, q^n_{(r)}-\lambda)|^2\, |[{M}(q'_{(r)}, q^n_{(r)}-\lambda)]^{\frac{1}{2}} - [{M}(q'_{(r)}, q^n_{(r)})]^{\frac{1}{2}}|^2 \dd \qt \right]^{\frac{1}{2}} < \varepsilon.
\]
Hence,
\begin{equation}\label{b3}
\hspace{8mm}\left[\int_{V_r}|g_r(q'_{(r)}, q^n_{(r)}-\lambda)|^2\, |[{M}(q'_{(r)}, q^n_{(r)}-\lambda)]^{\frac{1}{2}} - [{M}(q'_{(r)}, q^n_{(r)})]^{\frac{1}{2}}|^2 \dd \qt \right]^{\frac{1}{2}}\hspace{-9mm}
\end{equation}
converges to $0$ as $\lambda \rightarrow 0_+$. That concludes the analysis of term ${\rm T}_2$.

Concerning $T_1$, continuity in the $L^2$-norm of the translation operator implies that
\begin{equation}\label{b4}
\hspace{-0.4mm}\left[ \int_{V_r}|g_r(q'_{(r)}, q^n_{(r)})[{M}(q'_{(r)}, q^n_{(r)})]^{\frac{1}{2}} - g_r(q'_{(r)}, q^n_{(r)} - \lambda) [{M}( q'_{(r)}, q^n_{(r)}-\lambda)]^{\frac{1}{2}}|^2 \dd \qt \right]^{\frac{1}{2}}\hspace{-9mm}
\end{equation}
converges to $0$ as $\lambda \rightarrow 0_+$. That concludes the analysis of term ${\rm T}_1$.

Finally, \eqref{b3} and \eqref{b4} imply \eqref{shift}.

\smallskip

\textit{Step 2.}
Having shown \eqref{shift}, it now suffices to prove that each of the functions $u_{r \lambda}$, $r=1,\dots, m+1$, is a limit in $H^1_M(D)$ of $C^\infty(\overline{D})$ functions. To that end, for $r=1,\dots, m+1$ we define the functions
\[ u_{r\lambda h}(\qt) : = \frac{1}{C_0 h^n}  \int_{|\qt - \pt|< h}
\mbox{exp} \left(\frac{|\qt - \pt|^2}{|\qt - \pt|^2 - h^2}\right) u_{r\lambda}(\pt) \dd \pt,\]
(for $h$ sufficiently small, e.g., $0<h \leq \frac{1}{2}\lambda$, to ensure that the integral is correctly defined), with $C_0$ chosen so that
\[\frac{1}{C_0 h^n}  \int_{|\qt - \pt|< h}
\mbox{exp} \left(\frac{|\qt - \pt|^2}{|\qt - \pt|^2 - h^2}\right) \dd \pt = 1.\]
Clearly, $u_{r\lambda h} \in C^\infty(\overline{D})$, $r=1,\dots,m$, and for $\lambda \in (0, \beta/2]$
fixed, we have that
\[ \lim_{h \rightarrow 0_+} u_{r \lambda h} = u_{r \lambda} \quad \mbox{in}\quad H^1(D).\]
Thus, a fortiori (noting that ${M} \in L^\infty(D)$), for $\lambda \in (0,\beta/2]$ fixed,
\[ \lim_{h \rightarrow 0_+} u_{r \lambda h} = u_{r \lambda} \quad \mbox{in}\quad H^1_M(D),\quad
r=1,\dots,m.\]

Hence, given $\varepsilon > 0$ and $\lambda \in (0,\beta/2]$ fixed, there exists $h>0$ such that
\begin{equation}\label{c1}
\| u_{r \lambda h} - u_{r \lambda} \|_{H^1_M(D)} < \frac{\varepsilon}{m+1},\quad
r=1,\dots,m.
\end{equation}
Further, using that ${M} \in L^\infty(D)$ and the closeness of
$u_{m+1\,\lambda h}$ and $u_{m+1\,\lambda}$ in $H^1(D)$, and thereby in $H^1_M(D)$, we also have
\begin{equation}\label{c2}
\| u_{m+1\,\lambda h} - u_{m+1\, \lambda} \|_{H^1_M(D)} < \frac{\varepsilon}{m+1}
\end{equation}
for $h>0$ sufficiently small.

We define
\[ u_{\lambda h} := \sum_{r=1}^{m+1} u_{r \lambda h} \quad \mbox{and}\quad u_{\lambda} := \sum_{r=1}^{m+1}
u_{r \lambda }.\]
The inequalities \eqref{c1} and \eqref{c2} then imply that
\begin{equation}\label{density-partial}
\|u_{\lambda h} - u_\lambda\|_{H^1_M(D)}  \leq  \sum_{r=1}^{m+1} \|u_{r \lambda h }  -  u_{r \lambda}\|_{H^1_M(D)} \leq \varepsilon.
\end{equation}
Since the functions $u_{r \lambda h}$, $r=1,\dots,m+1$,
all belong to $C^\infty(\overline{D})$ the same is true of $u_{\lambda h}$. Finally,
\[ \|u - u_{\lambda h}\|_{H^1_M(D)} \leq  \|u - u_\lambda\|_{H^1_M(D)} +  \|u_{\lambda} -u_{\lambda h}\|_{H^1_M(D)},\]
and therefore the stated density result immediately follows, on recalling \eqref{partition}, \eqref{shift}, the definition of $u_\lambda$, \eqref{density-partial} and that $u_{\lambda h} \in C^\infty(\overline{D})$.
\end{proof}

By an identical argument, $C^\infty(\overline{\Omega \times D})$ is dense in $H^1_M(\Omega \times D)$.

\section{Compact embeddings in Maxwellian weighted spaces}
\label{AppendixD}
\setcounter{equation}{0}

\subsection{Step 1: Compact embedding of $H^1_M(D)$ into $L^2_M(D)$}\label{sec:comptensorise}
\textit{We are grateful to Leonardo Figueroa (University of Oxford) for suggesting the proof
presented in Section \ref{sec:comptensorise}.}

Let $D:=D_1 \times \cdots \times D_K$, where $D_i = B(\zerot, \sqrt{b_i})$, $b_i>0$, $i=1,\dots,K$, and
suppose that $M(\qt):=M_1(\qt_1)\cdots M_K(\qt_K)$. We shall prove that
\begin{equation}\label{compEmbNu}
     H^1_{M}(D) \compEmb L^2_{M}(D).
\end{equation}
We begin by recalling from the Appendix of Barrett \& S\"uli \cite{BS2} that
\begin{equation}\label{partial-compEmb}
H^1_{M_i}(D_i) \compEmb L^2_{M_i}(D_i)
\end{equation}
for $i=1,\dots,K$, which was proved there using a compactness result due to Antoci \cite{FA}.
We shall prove \eqref{compEmbNu} for the case of $K=2$,
with $D=D_1 \times D_2$ and $M(\qt) = M_1(\qt_1) M_2(\qt_2)$. For $K>2$ the proof is completely analogous.

Let $u \in H^1_M(D)$. As $M = M_1 \times M_2$, it follows from Fubini's theorem that, for almost all $\qt_1 \in D_1$,
\begin{equation*}
u(\qt_1,\cdot) \in L^1_{\mathrm{loc}}(D_2) \quad\text{and}\quad \partial^\alpha u(\qt_1,\cdot) \in L^2_{M_2}(D_2),\\
\end{equation*}
where $\alpha$ is any $d$-component multi-index with $0 \leq |\alpha| \leq 1$. Fubini's theorem also
implies that, given $\varphi_2 \in C^\infty_0(D_2)$ and a $d$-component multi-index $\alpha_2$, with
$0 \leq |\alpha_2| \leq 1$, we have
\begin{equation*}
\int_{D_1} \left[ (-1) \int_{D_2} u(\qt_1,\cdot) \partial^{\alpha_2}\varphi_2 \dd  \qt_2 \right] \varphi_1 \dd \qt_1
= \int_{D_1} \left[ \int_{D_2} \partial^{(0,\alpha_2)}u(\qt_1,\cdot) \varphi_2 \dd \qt_2 \right] \varphi_1 \dd \qt_1,
\end{equation*}
for all $\varphi_1 \in C^\infty_0(D_1)$. Therefore,
$\partial^{\alpha_2}[u(\qt_1,\cdot)] = \partial^{(0,\alpha_2)}
u(\qt_1,\cdot)$ in the sense of weak derivatives on $D_2$ for almost all $\qt_1 \in D_1$.
As $\partial^{(0,\alpha_2)} u(\qt_1,\cdot)$ belongs to $L^2_{M_2}(D_2)$ for
almost all $\qt_1 \in D_1$ we have that
\begin{equation}\label{regularity-a.e.}
u(\qt_1,\cdot) \in H^1_{M_2}(D_2) \quad\text{for almost all }\qt_1 \in D_1.
\end{equation}
Analogously,
\[ u(\cdot,\qt_2) \in H^1_{M_1}(D_1) \quad\text{for almost all }\qt_2 \in D_2. \]
As each of the partial Maxwellians, $M_1$ and $M_2$, is bounded from above and below by positive
constants on compact subsets of their respective domains, there exists a sequence $(D_{i,(n)}\,:\,n \in \mathbb{N})$ of open proper Lipschitz subsets of $D_i$, $i=1,2$, such that
\begin{equation*}
D_{i,(n)} \subset D_{i,(n+1)},\ n \in \mathbb{N},\qquad \bigcup_{n=1}^\infty D_{i,(n)} = D_i\quad\text{and}\quad H^1_{M_i}(D_{i,(n)}) \compEmb L^2_{M_i}(D_{i,(n)});
\end{equation*}
e.g., $D_{i,(n)} = B\big(0,\frac{\sqrt{b_i}n}{n+1}\big)$; the compact embeddings stated here
follow by the Rellich--Kondrachov theorem (cf. Adams \& Fournier \cite{AF:2003}, p.168, Theorem 6.3, Part I, eq. (3)) applied on $D_{i,(n)}$, $i=1,2$, $n \in \mathbb{N}$.
Letting, for $n \in \mathbb{N}$, $D_{(n)} := \bigtimes_{i=1}^2 D_{i,(n)} \subsetneq D$, and noting
that, by Appendix A, $D_{(n)}$ is a Lipschitz domain, the above properties get inherited by $D_{(n)}$
from $D_{i,(n)}$:
\begin{equation*}
D_{(n)} \subset D_{(n+1)},\ n \in \mathbb{N},\qquad \bigcup_{n=1}^\infty D_{(n)} = D\quad\text{and}\quad H^1_{M}(D_{(n)}) \compEmb L^2_{M}(D_{(n)}).
\end{equation*}
Let $D_i^{(n)} := D_i \setminus D_{i,(n)}$ and $D^{(n)} := D \setminus
D_{(n)}$. It follows from Opic \cite{Opic}, Theorem 2.4, that the above compact
embeddings on members of a nested covering imply the following characterizations
(the first, for $i \in \{1, 2\}$):
\begin{gather}
\label{equiv-i}
\hspace{-3mm}H^1_{M_i}(D_i) \compEmb L^2_{M_i}(D_i) \iff
\lim_{n \to \infty} \sup_{u \in H^1_{M_i}(D_i) \setminus \{0\}}
\int_{D_i^{(n)}} u^2 \dd M_i / \|u\|_{H^1_{M_i}(D_i)}^2 = 0,\\
\label{equiv-full}
\hspace{-3mm}H^1_M(D) \compEmb L^2_M(D) \iff
\lim_{n \to \infty} \sup_{u \in H^1_M(D) \setminus \{0\}}
\int_{D^{(n)}} u^2 \dd M / \|u\|_{H^1_M(D)}^2 = 0,
\end{gather}
where $\dd M_i:=M_i(\qt_i)\dq_i$, $i=1,2$, and $\dd M:= M(\qt) \dq$.
By virtue of \eqref{partial-compEmb}, the left-hand side of \eqref{equiv-i} holds;
hence, its right-hand side also holds. Using \eqref{regularity-a.e.} and
\eqref{equiv-i} with $i = 2$, we deduce that for any $\varepsilon > 0$
there exists $n=n(\varepsilon) \in \mathbb{N}$ such that
\begin{equation*}
\begin{split}
\int_{D_1 \times D_2^{(n)}} u^2 \dd M & = \int_{D_1} \left[ \int_{D_2^{(n)}} u^2(\qt_1,\cdot) \dd M_2 \right] \dd M_1 \leq \varepsilon \int_{D_1} \|u(\qt_1,\cdot)\|_{H^1_{M_2}(D_2)}^2 \dd M_1\\
& = \varepsilon \int_{D_1} \left[ \int_{D_2} u^2(\qt_1,\cdot)\dd M_2 + \int_{D_2} |{\grad{\qt_2}u(\qt_1,\cdot)}|^2 \dd M_2 \right] \dd M_1\\
& \leq \varepsilon \|u\|_{H^1_M(D)}^2;
\end{split}
\end{equation*}
and similarly for $\int_{D_1^{(n)} \times D_2} u^2 \dd M$. Then,
as $D^{(n)} = (D_1 \times D_2^{(n)}) \cup (D_1^{(n)} \times D_2)$, the right-hand side
of \eqref{equiv-full} holds; therefore, so does its left-hand side; hence \eqref{compEmbNu}.

\subsection{Step 2: Isometric isomorphisms}
Let $\Omega$ be a bounded open Lipschitz domain in $\mathbb{R}^d$.
We now show the isometric isomorphism of
the following pairs of spaces, respectively:
$L^2_M(\Omega \times D)$ and $L^2(\Omega;L^2_M(D))$;
$H^{0,1}_M(\Omega\times D)$ and $L^2(\Omega; H^1_M(D))$;
$H^{1,0}_M(\Omega \times D)$ and $H^1(\Omega;L^2_M(D))$.
The definitions of $H^{0,1}_M(\Omega\times D)$ and $H^{1,0}_M(\Omega \times D)$ are given below.

\subsubsection{Isometric isomorphism of $L^2_M(\Omega \times D)$ and
$L^2(\Omega;L^2_M(D))$}
Let
\[ L^2(\Omega;L^2_M(D)) := \{ v \in \mathcal{M}_w(\Omega,L^2_M(D))\,:\,
\int_\Omega \|v(\xt)\|^2_{L^2_M(D)} \dd \xt < \infty \}, \]
where
\[ \mathcal{M}_w(\Omega,L^2_M(D))
:= \{ v : \Omega \rightarrow L^2_M(D)\,:\, \mbox{$v$ is weakly
measurable on $\Omega$}\}.\]
Let $\{\varphi_j\}_{j=1}^\infty$ be a complete orthonormal system in the (separable)
Hilbert space $L^2_M(D)$ with respect to the inner product
$(\cdot,\cdot)$ of $L^2_M(D)$.
For $v \in L^2(\Omega;L^2_M(D))$, %as above,
we define the function
\[ V_N(\xt,\qt) := \sum_{j=1}^N (v(\xt),\varphi_j)\,\varphi_j(\qt).\]
As $v$ is weakly measurable on $\Omega$, each of the functions
$\xt\mapsto (v(\xt),\varphi_j)$,
$j=1,2,\dots$, is measurable on $\Omega$;
therefore $(\xt,\qt) \mapsto (v(\xt),\varphi_j)$ is
measurable on $\Omega \times D$ for all $j=1,2,\dots$.
Similarly, $\qt \mapsto \varphi_j(\qt)$
is measurable on $D$ for each $j=1,2,\dots$,
and therefore $(\xt,\qt) \mapsto \varphi_j(\qt)$ is
measurable on $\Omega \times D$. Hence, also $V_N$
is a measurable function on $\Omega \times D$. Now,
\[ |V_N(\xt,\qt)|^2 = \sum_{j=1}^N \sum_{m=1}^N (v(\xt),\varphi_j)\,
(v(\xt), \varphi_k) \,\varphi_j(\qt)\, \varphi_k(\qt).\]
By the Cauchy--Schwarz inequality $M \varphi_j\, \varphi_k
= {M}^{\frac{1}{2}}\varphi_j\cdot{M}^{\frac{1}{2}}\varphi_k \in L^1(D)$
for all $j, k \geq 1$;
hence also $M(\cdot) \,|V_N(\xt, \cdot)|^2 \in L^1(D)$ for a.e. $\xt \in \Omega$.
Thus, by the orthonormality
of the $\varphi_j$, $j=1,2,\dots,$ in $L^2_M(D)$,
\[ \int_D M(\qt)\, |V_N(\xt,\qt)|^2 {\dd} \qt =
\sum_{j=1}^N|(v(\xt),\varphi_j)|^2, \qquad \mbox{a.e.
$\xt \in \Omega$}.\]
By Bessel's inequality in $L^2_M(D)$,
the right-hand side of this last equality is bounded by
$\|v(\xt)\|^2_{L^2_M(D)}$ for a.e. $\xt \in \Omega$,
and, by hypothesis, $\xt \mapsto v(\xt) \in L^2(\Omega)$; therefore,
by Fubini's theorem, $M\,|V_N|^2 \in L^1(\Omega \times D)$.
Upon integrating both sides over $\Omega$, and using Fubini's
theorem on the left-hand side to write the
multiple integral over $\Omega$ and $D$ as an integral
over $\Omega \times D$, we have
\begin{equation}\label{eq1}
 \|V_N\|^2_{L^2_M(\Omega \times D)}
 := \int_{\Omega \times D} M(\qt)\, |V_N(\xt,\qt)|^2
 {\dd} \qt {\dd} \xt = \sum_{j=1}^N \int_\Omega |(v(\xt),\varphi_j)|^2
{\dd} \xt.
\end{equation}
Now, let
\[ y_N(\xt) := \sum_{j=1}^N |(v(\xt),\varphi_j)|^2,\qquad \xt \in \Omega.\]
The sequence $\{y_N(\xt)\}_{N=1}^\infty$ is monotonic increasing
for almost all $\xt \in \Omega$; also,
according to Bessel's inequality in $L^2_M(D)$ we have that
\[ 0 \leq y_N(\xt) \leq \|v(\xt)\|^2_{L^2_M(D)}
\qquad \forall N \geq 1, \quad \mbox{a.e. $\xt \in \Omega$}.\]
Thus $\{y_N(\xt)\}_{N=1}^\infty$ is a bounded
sequence of real numbers, for a.e. $\xt \in \Omega$.
Therefore, the sequence $\{y_N(\xt)\}_{N=1}^\infty$
converges in $\mathbb{R}$ for a.e. $\xt \in \mathbb{R}$,
with
\[ y(\xt) = \lim_{N\rightarrow \infty} y_N(\xt)
= \sum_{j=1}^\infty |(v(\xt),\varphi_j)|^2,\qquad
\mbox{a.e. $\xt \in \Omega$}.\]
By the monotone convergence theorem,
\begin{eqnarray}\label{eq2}
 \lim_{N \rightarrow \infty} \sum_{j=1}^N
 \int_\Omega |(v(\xt),\varphi_j)|^2 {\dd} \xt
 &=& \lim_{N \rightarrow \infty} \int_\Omega y_N(\xt) {\dd} \xt \nonumber\\
 &=& \int_\Omega y(\xt) {\dd} \xt = \int_\Omega \sum_{j=1}^\infty
 |(v(\xt),\varphi_j)|^2 {\dd} \xt.
\end{eqnarray}
This implies that
\[ \left\{ \sum_{j=1}^N \int_\Omega |(v(\xt),\varphi_j)|^2 {\dd} \xt \right
\}_{N=1}^\infty\]
is a convergent sequence of real numbers.
Hence, it is also a Cauchy sequence in $\mathbb{R}$.

Since, for any $N > L \geq 1$,
\[\int_{\Omega \times D} |V_N(\xt,\qt) - V_L(\xt,\qt)|^2 {\dd} \qt {\dd} \xt
=  \sum_{j=L+1}^N \int_D |(v(\xt),\varphi_j)|^2 {\dd} \xt,\]
it follows that $\{V_N\}_{N=1}^\infty$ is a
Cauchy sequence in $L^2_M(\Omega \times D)$. Since
$L^2_M(\Omega \times D)$ is a Hilbert space, there exists
a unique $V \in L^2_M(\Omega \times D)$ such that
\begin{equation}\label{eq3}
V = \lim_{N \rightarrow \infty} V_N\qquad \mbox{ in $L^2_M(\Omega \times D)$}.
\end{equation}

Thus we have shown that the mapping
\[ \mathcal{I}\,:\, v \in L^2(\Omega,L^2_M(D))
\mapsto V := \sum_{j=1}^\infty (v(\cdot),\varphi_j)\,
\varphi_j(\cdot) \in L^2_M(\Omega \times D)\]
is correctly defined.
Next, we prove that $\mathcal I$ is a bijective
isometry, and this will imply that the
spaces $L^2(\Omega;L^2_M(D))$ and $L^2_M(\Omega \times D)$
are isometrically isomorphic.

We begin by showing that $\mathcal{I}$ is injective.
As $\mathcal{I}$ is linear it suffices to
prove that if $\mathcal{I}(v)=0$ then $v=0$. Indeed, if $\mathcal{I}(v)=0$, then
\[ \sum_{j=1}^\infty (v(\xt),\varphi_j)\, \varphi_j(\qt) = 0
\qquad \mbox{for a.e. $(\xt,\qt) \in \Omega \times D$}.\]
Since $\{\varphi_j\}_{j=1}^\infty$ is an orthonormal
system in $L^2_M(D)$, it follows that $(v(\xt),\varphi_j) = 0$
for a.e. $\xt \in \Omega$ and all $j=1,2,\dots$. The completeness of the
orthonormal system $\{\varphi_j\}_{j=1}^\infty$ in
$L^2_M(D)$ now implies that $v(\xt)=0$ in $L^2_M(D)$
for a.e. $\xt \in \Omega$, i.e.
$v=0$ in $L^2(\Omega;L^2_M(D))$.

Next we show that $\mathcal{I}$ is surjective.
Suppose that $V \in L^2_M(\Omega \times D)$. Then,
by Fubini's theorem, $V(\xt,\cdot) \in L^2_M(D)$
for a.e. $\xt \in \Omega$. Since
$\{\varphi_j\}_{j=1}^\infty$ is a complete
orthonormal system in $L^2_M(D)$, it follows that
\[ V(\xt,\cdot) = \sum_{j=1}^\infty(V(\xt,\cdot),\varphi_j)\,\varphi_j(\cdot).\]
On defining $v(\xt):= V(\xt,\cdot) \in L^2_M(D)$, we have
that $\mathcal{I}(v) = V$. Hence $\mathcal{I}$ is surjective.

Finally, we show that $\mathcal{I}$ is an isometry. Clearly
\begin{eqnarray*}
\|V\|^2_{L^2_M(\Omega\times D)}
&{(\ref{eq3})\atop{=}}\atop{~}&
\lim_{N \rightarrow \infty} \|V_N\|^2_{L^2_M(\Omega\times D)}
\\
&{(\ref{eq1})\atop{=}}\atop{~}&
\lim_{N \rightarrow \infty} \sum_{j=1}^N
\int_\Omega |(v(\xt),\varphi_j)|^2 {\dd} \xt \\
&{(\ref{eq2})\atop{=}}\atop{~}& \int_\Omega
\sum_{j=1}^\infty |(v(\xt),\varphi_j)|^2 {\dd} \xt.
\end{eqnarray*}
Applying Parseval's identity in $L^2_M(D)$
to the infinite series under the last integral sign, we deduce that
\begin{eqnarray*}
\|V\|^2_{L^2_M(\Omega\times D)}
= \int_\Omega \|v(\xt)\|^2_{L^2_M(D)} {\dd} \xt = \|v\|^2_{L^2(\Omega;
L^2_M(D))}.
\end{eqnarray*}
Thus we have shown that $\|\mathcal{I}v\|_{L^2_M(\Omega\times D)}
= \|v\|_{L^2(\Omega;L^2(D))}$,
whereby $\mathcal{I}$ is an isometry.

\subsubsection{Isometric isomorphism of $H^{0,1}_M(\Omega\times D)$
and $L^2(\Omega;H^1_M(D))$} Let us begin by observing that
 $L^2_M(\Omega \times D) \subset L^1_{\tiny \rm{loc}}(\Omega \times D)$,
and therefore any $V$ in $L^2_M(\Omega \times D)$ can be considered to be an
element of $\mathcal{D}'(\Omega \times D)$, the space of $\mathbb{R}$-valued
distributions on $\Omega \times D$. Let $\nabq$ denote the distributional gradient
with respect to $\qt$, defined on $\mathcal{D}'(\Omega \times D)$.
We define
\[ H^{0,1}_M(\Omega \times D) := \{V \in L^2_M(\Omega \times D)\,:\,
\nabq V \in L^2_M(\Omega \times D)\}.   \]
A completely identical argument to the one above shows that
$H^{0,1}_M(\Omega \times D)$ is isometrically isomorphic
to $L^2(\Omega; H^1_M(D))$; the only change that is required is
to replace $L^2_M(D)$ by $H^1_M(D)$
throughout and to take $\{\varphi_j\}_{j=1}^\infty$ to be a complete
orthonormal system in the inner product
$(\cdot,\cdot)$ of the (separable) Hilbert space $H^1_M(D)$, instead of $L^2_M(D)$.

\subsubsection{Isometric isomorphism of $H^{1,0}_M(\Omega \times D)$ and
$H^1(\Omega;L^2_M(D))$}
Let
\[ H^{1,0}_M(\Omega \times D) := \{V \in L^2_M(\Omega \times D)\,:\,
\nabx V \in L^2_M(\Omega \times D)\}.\]
Concerning the isometric isomorphism of $H^{1,0}_M(\Omega \times D)$
and $H^1(\Omega;L^2_M(D))$
we proceed as follows. Given $v \in H^1(\Omega; L^2_M(D))
\subset L^2(\Omega;L^2_M(D))$, we
define, as in the proof of the isometric
isomorphism of $L^2(\Omega;L^2_M(D))$ and
$L^2_M(\Omega \times D)$ above, the function
\[ V\,:\,(\xt,\qt) \in \Omega \times D \mapsto V(\xt,\qt)
:= \sum_{j=1}^\infty (v(\xt),\varphi_j)\,\varphi_j(\qt)
\in \mathbb{R},\]
where $\{\varphi_j\}_{j=1}^\infty$ is a complete orthonormal
system in $L^2_M(D)$. We showed above
that $V \in L^2_M(\Omega \times D)$,
and $\|V\|_{L^2_M(\Omega \times D)} = \|v\|_{L^2(\Omega;
L^2_M(D))}$.

Now, let $\nabx$ denote the distributional gradient with respect to $\xt$,
defined on $\mathcal{D}'(\Omega \times D)$, and
let $\undertilde{D}_x$ denote the distributional gradient, defined
on $\mathcal{D}'(\Omega;L^2_M(D))$, the space of
$L^2_M(D)$-valued distributions on $\Omega$.
Applying $\nabx$ to
\[ V = \sum_{j=1}^\infty (v,\varphi_j)\,\varphi_j
\quad \mbox{in $\mathcal{D}'(\Omega \times D)$} 
%\]
%
\qquad \mbox{and noting that}\qquad
%
%\[ 
\nabx V = \sum_{j=1}^\infty (\undertilde{D}_x v, \varphi_j)\, \varphi_j,\]
it follows from the isometric isomorphism of
$L^2_M(\Omega\times D)$ and $L^2(\Omega;L^2_M(D))$
that
\begin{eqnarray*}
\|V\|^2_{H^{1,0}_M(\Omega\times D)}&=& \|V\|^2_{L^2_M(\Omega \times D)} +
\|\nabx V\|^2_{L^2_M(\Omega\times D)}\\
&=& \|v\|^2_{L^2_M(\Omega; L^2_M(D))} + \|\undertilde{D}_x v \|^2_{L^2(\Omega;L^2_M(D))}%\\
%&=& 
=\|v\|^2_{H^1(\Omega;L^2_M(D))},
\end{eqnarray*}
which shows that $H^{1,0}_M(\Omega \times D)$ and $H^1(\Omega;L^2_M(D))$
are isometrically isomorphic.

\subsection{Step 3: Compact embedding of
$H^1_M(\Omega \times D)$ into $L^2_M(\Omega \times D)$}\label{sec:shakhmurov}
We use the results of Step 2 to
identify the space $L^2_M(\Omega \times D)$ with $L^2(\Omega;L^2_M(D))$
and the space
$H^1_M(\Omega \times D) =  H^{1,0}_M(\Omega \times D) \cap
H^{0,1}_M(\Omega \times D)$ with
$H^1(\Omega; L^2_M(D)) \cap L^2(\Omega; H^1_M(D))$.
Upon doing so, the compact embedding of
$H^1_M(\Omega \times D)$ into $L^2_M(\Omega \times D)$
directly follows from the compact embedding of $H^1(\Omega;L^2_M(D))\cap L^2(\Omega;H^1_M(D))$
into $L^2(\Omega;L^2_M(D))$, implied by Theorem 2 on p.1499 in the paper of Shakhmurov \cite{VS},
thanks to the compact embedding from Step 1.

~\\

~\hfill{\footnotesize \textit{London \& Oxford}

~\hfill{\footnotesize \textit{Original version: April 1, 2010; Revised version: 15 July 2011}}

~\\

\textbf{Note:} The results contained in this preprint have been published, in an abbreviated form, in our paper \cite{BS2011-fene}.

%%%%%%%%%%%%%%%%%%%%%%%%%%%%%%

\end{document}